\documentclass[12pt,english]{article}
\pdfoutput=1

\usepackage[toc,page,header]{appendix} 
\usepackage{minitoc} 

\usepackage{ae,aecompl}

\usepackage[T1]{fontenc}
\usepackage[latin9]{inputenc}
\usepackage{hhline} 
\usepackage{geometry}
\usepackage{color}
\usepackage{textcomp}

\usepackage{amsmath}
\usepackage{amsthm}
\usepackage{amssymb}
\usepackage{bm} 
\usepackage{graphicx}
\usepackage{mathrsfs}
\usepackage{mathtools}
\usepackage{setspace}

\usepackage[authoryear]{natbib}

\usepackage{verbatim} 
\usepackage[flushleft]{threeparttable} 

\usepackage[unicode=true,pdfusetitle,
 bookmarks=true,bookmarksnumbered=true,bookmarksopen=true,bookmarksopenlevel=1,
 breaklinks=true,pdfborder={0 0 0},pdfborderstyle={},backref=false,colorlinks=true]
 {hyperref}
 \hypersetup{linkcolor=magenta, urlcolor=blue, citecolor=blue,
 pdfstartview={FitH}, unicode=true}

\usepackage[capposition=top]{floatrow} 

\makeatletter

\numberwithin{equation}{section}
\numberwithin{figure}{section}
\theoremstyle{definition}
\newtheorem{example}{\protect\examplename}
\theoremstyle{plain}
\newtheorem{assumption}{\protect\assumptionname}
\numberwithin{assumption}{section}
\theoremstyle{plain}
\newtheorem{thm}{\protect\theoremname}
\numberwithin{thm}{section}
\theoremstyle{definition}
\newtheorem{rem}{\protect\remarkname}
\numberwithin{rem}{section}
\theoremstyle{plain}

\theoremstyle{plain}
\newtheorem{lem}{\protect\lemmaname}
\numberwithin{lem}{section}
\theoremstyle{plain}
\newtheorem{cor}{\protect\corollaryname}
\numberwithin{cor}{section}
\theoremstyle{definition}
\newtheorem{alg}{Algorithm}
\numberwithin{alg}{section}

\def\E{\mathrm{E}}

\def\En{\mathbb{E}_{n}}
\def\G{\mathbb{G}}
\def\Gn{\G_{n}}

\def\N{\mathbb{N}}
\def\P{\mathrm{P}}

\def\R{\mathbb{R}}

\def\bM{\mathbf{M}}
\def\cN{\boldsymbol{\mathcal{N}}}
\def\cS{\boldsymbol{\mathcal{S}}} 
\def\bS{\boldsymbol{S}} 
\def\bU{\boldsymbol{U}}
\def\bv{\boldsymbol{v}}
\def\bV{\boldsymbol{V}}
\def\bw{\boldsymbol{w}}
\def\bW{\boldsymbol{W}}
\def\bx{\boldsymbol{x}}
\def\bX{\boldsymbol{X}}
\def\bfX{\mathbf{X}}

\def\bY{\boldsymbol{Y}}
\def\by{\boldsymbol{y}}
\def\bZ{\boldsymbol{Z}}
\def\bz{\boldsymbol{z}}
\def\balpha{\boldsymbol{\alpha}}
\def\bgamma{\boldsymbol{\gamma}}
\def\bdelta{\boldsymbol{\delta}}
\def\bmu{\boldsymbol{\mu}}
\def\btheta{\boldsymbol{\theta}}
\def\bvartheta{\boldsymbol{\vartheta}}
\def\bSigma{\boldsymbol{\Sigma}}

\newcommand*{\argmin}{\operatornamewithlimits{argmin}\limits}

\newcommand\norm[1]{\left\lVert#1\right\rVert}

\newcommand\paran[1]{(#1)}

\newcommand\absn[1]{\lvert#1\rvert}
\newcommand\abs[1]{\left\lvert#1\right\rvert}
\newcommand\bracn[1]{[#1]}

\newcommand\bracen[1]{\{#1\}}

\newcommand\supp[1]{\mathrm{supp}\left(#1\right)}
\newcommand\suppn[1]{\mathrm{supp}(#1)}

\usepackage{chngcntr}
\usepackage{apptools}
\AtAppendix{\counterwithin{assumption}{section}}
\AtAppendix{\counterwithin{equation}{section}}
\AtAppendix{\counterwithin{lem}{section}}
\AtAppendix{\counterwithin{thm}{section}}

\usepackage{paralist}

\usepackage{layouts}

\makeatother

\providecommand{\assumptionname}{Assumption}
\providecommand{\corollaryname}{Corollary}
\providecommand{\examplename}{Example}
\providecommand{\lemmaname}{Lemma}
\providecommand{\propositionname}{Proposition}
\providecommand{\remarkname}{Remark}
\providecommand{\theoremname}{Theorem}

\begin{document}


\title{Selecting Penalty Parameters of High-Dimensional M-Estimators using Bootstrapping after
Cross-Validation\thanks{Parts of this paper were previously circulated under the title ``Analytic
and Bootstrap-after-Cross-Validation Methods for Selecting Penalty Parameters of High-Dimensional
M-Estimators.'' We thank Richard Blundell, Victor Chernozhukov, Bo Honor\'{e}, Whitney Newey, Joris
Pinkse, Simon Reese, Azeem Shaikh, Mikkel S{\o}lvsten, Sara van de Geer and numerous seminar
participants for their insightful comments and discussions. Bohdan Salahub and Andrei Voronin
provided excellent research assistance. Chetverikov's work was supported by NSF Grant SES -
1628889.}} \author{Denis Chetverikov\footnote{Department of Economics, UCLA; e-mail:
\tt{chetverikov@econ.ucla.edu}.} \and Jesper R.-V. S{\o}rensen\footnote{Department of Economics,
University of Copenhagen; e-mail:
\tt{jrvs@econ.ku.dk}.}}

\maketitle

\begin{abstract}
We develop a new method for selecting the penalty parameter for $\ell_{1}$-penalized M-estimators in
high dimensions, which we refer to as bootstrapping after cross-validation. We derive rates of
convergence for the corresponding $\ell_1$-penalized M-estimator and also for the
post-$\ell_1$-penalized M-estimator, which refits the non-zero entries of the former estimator
without penalty in the criterion function. We demonstrate via simulations that our methods are not
dominated by cross-validation in terms of estimation errors and can outperform cross-validation in
terms of inference. As an empirical illustration, we revisit \citet{fryer_jr_empirical_2019}, who
investigated racial differences in police use of force, and confirm his findings.
\end{abstract}


\section{Introduction}\label{sec:Introduction}

High-dimensional models have attracted substantial attention both in the econometrics and in the
statistics/machine learning literature, see e.g.~\citet{belloni2018highdimensional} and
\citet{hastie_statistical_2015}, and $\ell_{1}$-penalized estimators have emerged among the most
useful methods for learning parameters of such models. However, implementing these estimators
requires a choice of the penalty parameter and with few notable exceptions,
e.g.~$\ell_{1}$-penalized linear mean, quantile and logit regression estimators, the choice of this
penalty parameter in practice often remains unclear. In this paper, we develop a new method to
choose the penalty parameter in the context of $\ell_1$-penalized M-estimation and show that our
method leads to precise estimation and inference in a large variety of models.

We consider a model where the true value $\btheta_{0}$ of some parameter $\btheta$ is given by the
solution to an optimization problem
\begin{equation}
\btheta_{0}=\argmin_{\btheta\in\Theta}\E[m(\bX^{\top}\btheta,\bY)],\label{eq:EstimandIntro}
\end{equation}
where $m:\R\times \mathcal Y \to \R$ is a known (potentially non-smooth) loss function that is
convex in its first argument, $\bX = (X_{1},\dots,X_{p})^{\top}\in\mathcal{X}\subseteq\R^{p}$ a
vector of candidate regressors, $\bY\in\mathcal{Y}$ one or more outcome variables, and
$\Theta\subseteq\R^{p}$ a convex parameter space. Prototypical loss functions are square loss and
negative log-likelihood, but the framework (\ref{eq:EstimandIntro}) also covers many other
cross-sectional models and associated modern as well as classical estimation approaches including
logit and probit models, logistic calibration \citep{tan2020regularized}, covariate balancing
\citep{imai_covariate_2014}, and expectile regression \citep{newey_asymmetric_1987}. It also
subsumes approaches to estimation of panel-data models such as the fixed-effects/conditional logit
for binary outcomes \citep{rasch1960probabilistic}, trimmed least-absolute-deviations and trimmed
least-squares for censored outcomes \citep{honore_trimmed_1992}, and partial likelihood approaches
to heterogeneous panel models for duration \citep{chamberlain_heterogeneity_1985}. We detail some of
these examples in Section \ref{sec:Examples}.

For the purpose of estimation, we assume access to a sample $\{(\bX_{i},\bY_{i})\}_{i=1}^n$ of $n$
independent observations from the distribution $P$ of the pair $(\bX,\bY)$, where the number $p$ of
candidate regressors in each $\bX_{i}=(X_{i,1},\dots,X_{i,p})^{\top}$ may be (potentially much)
larger than the sample size $n$, meaning that we cover high-dimensional models. Following the
literature on high-dimensional models, we assume that the vector
$\btheta_{0}=(\theta_{0,1},\dots,\theta_{0,p})^{\top}$ is at least approximately (also known as
``weakly'') sparse. While we postpone a formal definition to Section \ref{sec:Deterministic-Bounds},
approximate sparsity captures the idea that, even though the number of \emph{candidate} regressors
$p$ can be very large, the number of \emph{relevant} regressors may be substantially smaller. In the
simplest case, known as exact (or ``strong'') sparsity, this assumption amounts to the number of
non-zeros in $\btheta_0$ being much smaller than $n$. Approximate sparsity relaxes this idea to
allow possibly many---but typically small---non-zeros. With sparsity in mind, we study the sparsity
encouraging $\ell_{1}$-penalized M-estimator ($\ell_{1}$-ME)
\begin{equation}
\widehat{\btheta}\left(\lambda\right)\in\widehat{\Theta}\left(\lambda\right):=\argmin_{\btheta\in\Theta}\bigg\{\frac{1}{n}\sum_{i=1}^{n}m(\bX_{i}^{\top}\btheta,\bY_{i})+\lambda\Vert\btheta\Vert_{1}\bigg\},\label{eq:ell1PenalizedMEstimationIntro}
\end{equation}
where $\Vert\btheta\Vert_{1}=\sum_{j=1}^{p}\vert\theta_{j}\vert$ denotes the $\ell_{1}$ norm of
$\btheta$, and $\lambda\in[0,\infty)$ is a penalty parameter.\footnote{Throughout the main text, we
implicitly assume that an estimator exists. Simple conditions under which
$\widehat{\Theta}(\lambda)$ is non-empty (and related properties) are given in Appendix
\ref{sec:ExistenceSparsityAndUniqueness}.} We also study the post-$\ell_1$-penalized M-estimator
(post-$\ell_1$-ME), which refits the coefficients of the variables selected by $\ell_1$-ME without
the penalty in the criterion function in \eqref{eq:ell1PenalizedMEstimationIntro}.

Implementing the estimator $\widehat{\btheta}(\lambda)$ requires choosing $\lambda$. To do so, we
first extend a probabilistic bound from \citet{belloni_l1-penalized_2011}, obtained for
$\ell_{1}$-penalized quantile regression, to our general $\ell_{1}$-penalized M-estimation setting
(\ref{eq:ell1PenalizedMEstimationIntro}). (See also \citet{negahban_unified_2012} for independently
developed and closely related results.) The bound, which we state in Section
\ref{sec:Deterministic-Bounds}, yields a general principle to choose $\lambda$. In particular, it
suggests that, for an arbitrary choice of $c_{0}\in(1,\infty)$, one should choose $\lambda$ as small
as possible subject to the constraint that the event
\begin{equation}
\lambda \geqslant c_{0}\max_{1\leqslant j\leqslant p}\bigg|\frac{1}{n}\sum_{i=1}^{n}m_{1}'(\bX_{i}^{\top}\btheta_{0},\bY_{i})X_{i,j}\bigg|\label{eq:ScoreDominationIntro}
\end{equation}
occurs with probability approaching one, where $m_1'$ denotes the partial derivative of the loss
function with respect to its first argument. We therefore wish to set
$\lambda=c_{0}q_n(1-\alpha)$, where
\begin{equation}
q_n(1-\alpha):=\left(1-\alpha\right)\text{-quantile of }\max_{1\leqslant j\leqslant p}\bigg|\frac{1}{n}\sum_{i=1}^{n}m_{1}'(\bX_{i}^{\top}\btheta_{0},\bY_{i})X_{i,j}\bigg|,\label{eq:ScoreQuantileIntro}
\end{equation}
for some small user-specified probability tolerance level $\alpha = \alpha_n \to
0$ as $n\to\infty$. This choice, however, is typically infeasible since the
random variable in (\ref{eq:ScoreQuantileIntro}) depends on the unknown
$\btheta_{0}$. We thus have a vicious circle: to choose $\lambda$, we need an
estimator of $\btheta_{0}$, but to estimate $\btheta_{0}$, we need to choose
$\lambda$. In this paper, we offer a solution to this problem, which constitutes
our key contribution.

To obtain our solution, we show that even though (as we discuss below) the estimator
$\widehat{\btheta}(\lambda)$ based on $\lambda$ chosen by cross-validation or its variants is
generally difficult to analyze, it can be used to construct provably good, in a certain sense,
estimators of the random vectors $m_{1}'(\bX_{i}^{\top}\btheta_{0},\bY_{i})\bX_{i}$. We are then
able to derive an estimator, say $\widehat{q}(1-\alpha)$, of $q_n(1-\alpha)$ via bootstrapping, as
discussed in \citet{belloni2018highdimensional}, and to set $\lambda=c_{0}\widehat{q}(1-\alpha)$,
which we refer to as the \textit{bootstrap-after-cross-validation} (BCV) \textit{method} to choose
$\lambda$. This method is computationally rather straightforward, applicable in a wide variety of
models, and non-conservative in the sense that it gives $\lambda$ such that $\lambda\approx
c_{0}q_n(1-\alpha)$ rather than $\lambda \gg c_0q_n(1-\alpha)$. We derive convergence rates of
$\ell_1$-ME and post-$\ell_1$-ME based on this choice of $\lambda$ in Section
\ref{sec:Bootstrapping-the-Penalty}. In addition, we show in Section \ref{sec: inference} that, upon
debiasing via the double machine learning approach, these estimators yield simple inference
procedures.

The main alternatives to our method are cross-validation and related sample-splitting techniques.
One of the main complications with these methods is that they are difficult to analyze, at least in
some important dimensions. Sample-splitting techniques yield bounds on the $\ell_{2}$ estimation
error $\|\widehat{\btheta}(\lambda)-\btheta_{0}\|_{2}$, see e.g.~\citet{lecue_mitchell}, but not on
the $\ell_{1}$
estimation error $\|\widehat{\btheta}(\lambda)-\btheta_{0}\|_{1}$.\footnote{Any
two norms on a fixed and finite-dimensional space are equivalent. However, the equivalence
constants generally depend on the dimension (here $p$), which makes translation
of error bounds for one norm into another a non-trivial manner when the dimension
is growing.} In contrast, our method gives bounds on both $\ell_2$ and $\ell_1$
estimation errors. An $\ell_1$ error bound is crucial when we are interested in
estimating dense functionals $\boldsymbol{a}^{\top}\btheta_{0}$ of $\btheta_0$
with $\boldsymbol{a}\in\R^{p}$ being a vector of loadings with many non-zero
components; see \citet{belloni2018highdimensional} for
details.\footnote{Dense functionals $\boldsymbol{a}^{\top}\btheta_0$ may appear
in the analysis, for example, when the vector $\bX$ consists of many dummy
variables and we are interested in making comparisons between two cells,
$(\bx_2 - \bx_1)^{\top}\btheta_0$, where $\bx_1$ and $\bx_2$ represent the first
and the second cell, respectively. In such examples, we can guarantee that
$(\bx_2 - \bx_1)^{\top}\widehat\btheta(\lambda)$ is close to $(\bx_2 -
\bx_1)^{\top}\btheta_0$ only when $\|\widehat\btheta(\lambda) - \btheta_0\|_1$
is small.} Moreover, $\ell_1$ estimation error bounds are needed to perform
inference on components of $\btheta_{0}$ as in Section \ref{sec:
inference}.\footnote{It is possible to replace the requirement on $\ell_1$
estimation error by the requirement on $\ell_2$ estimation error via
cross-fitting, as in \cite{chernozhukov2018double}. However, the combination of
sample-splitting and cross-fitting would require splitting the original sample
into at least three subsamples, which may not lead to accurate inference in
moderate samples.} When $\lambda$ is selected by cross-validation, $\ell_{1}$
and $\ell_{2}$ estimation error bounds are typically both unknown. The only
exception we are aware of is the linear mean regression model estimated by the
LASSO. For this special case, bounds have been derived in
\citet{chetverikov_cross-validated_2016} and \citet{miolane_montanari}, but the
bounds appearing in those references are less sharp than those provided here.
Moreover, and crucially, cross-validation may lead to rather poor inference
results, in the sense of bad size control, even in relatively large samples, and
does not dominate our method even in terms of estimation errors;  see our
simulation results in Section \ref{sec:Simulations} for details.

Another alternative to our method is to base the penalty parameter choice on self-normalized
moderate deviation (SNMD) theory, as proposed in \cite{belloni_sparse_2012} for the linear mean
regression model and extended in \cite{belloni2016post} to the logit model. This method is slightly
conservative, in the sense that it gives $\lambda$ somewhat larger than $c_0 q_n(1-\alpha)$, but
yields estimation and inference results that are comparable in quality with those produced by the
BCV method. The SNMD method can be further extended to cover any Lipschitz-continuous loss function,
but it is not clear how to extend it to a non-Lipschitz setting. For example, the SNMD method can
be applied to the logit model but not to the probit model. In contrast, our BCV method is
nearly universally applicable, and does not require Lipschitz continuity. We provide several other
important examples where the loss function is not Lipschitz-continuous in Section
\ref{sec:Examples}.

To showcase our method using real data, in Section \ref{sec:Application} we revisit the setting of
\citet{fryer_jr_empirical_2019}, who investigated racial differences in police use of force. We
extend Fryer's regression analysis in two ways. First, we change the model from a binary logit to a
binary probit, keeping the regressors as in Fryer's analysis, a relatively small list. Second, we
add a large number of additional (technical) regressors resulting from interactions between the
original regressors. The first change leads to a non-Lipschitz loss (the negative probit
likelihood). The second change brings us into high-dimensional territory, causing classical methods
to break down. Unlike existing methods, the methods developed in this paper can accommodate both
challenges. Our analysis supports the conclusions of \citet{fryer_jr_empirical_2019} in showing that
they are robust to model specification and a much larger set of candidate controls than originally
considered.

The literature on learning parameters of high-dimensional models via $\ell_{1}$-penalized
M-estimation is large. Instead of listing all existing papers, we therefore refer the interested
reader to the excellent textbook treatment in \citet{wainwright_high-dimensional_2019} and focus
here on only a few key references.
\citet{van_de_geer_high-dimensional_2008,van_de_geer_estimation_2016} derives bounds on the
estimation errors of general $\ell_{1}$-penalized M-estimators
(\ref{eq:ell1PenalizedMEstimationIntro}) and provides some choices of the penalty parameter
$\lambda$. However, her penalty formulas give values of $\lambda$ that are so large that the
resulting estimators are typically trivial in moderate samples, with all coefficients being exactly
zero. Recognizing this issue, \citet[p.~621]{van_de_geer_high-dimensional_2008} remarks that her
results should only be seen as an indication that her theory has something to say about finite
sample sizes, and that other methods to choose $\lambda$ should be used in practice.
\citet{negahban_unified_2012} develop error guarantees in a very general setting, and when
specialized to our setting (\ref{eq:ell1PenalizedMEstimationIntro}), their results become quite
similar to those in our Theorem \ref{thm:NonAsymptoticProbabilisticBounds}. The same authors also
note that a challenge to using these results in practice is that the random variable in
(\ref{eq:ScoreDominationIntro}) is usually impossible to compute because it depends on the unknown
vector $\btheta_{0}$ (\emph{ibid.,} p.~547). It is exactly this challenge that we overcome in this
paper. \citet{belloni_l1-penalized_2011} study the high-dimensional quantile regression model and
note that the distribution of the random variable in (\ref{eq:ScoreDominationIntro}) is in this case
pivotal, making the choice of the penalty parameter simple. Similarly, \cite{wang_tuning-free_2020}
study the high-dimensional mean regression model and show that one can obtain pivotality by
replacing the square-loss function by Jaeckel's dispersion function, again making the choice of the
penalty parameter simple. However, these are the only two settings we are aware of in which the
distribution of the random variable in (\ref{eq:ScoreDominationIntro}) is pivotal.\footnote{With a
known censoring propensity, the \cite{buchinsky1998alternative} linear programming estimator for
censored quantile regression boils down to a variant of quantile regression, thus leading to
pivotality.} Finally, \citet{ninomiya_kawano} consider information criteria for the choice of the
penalty parameter $\lambda$ but focus on fixed-$p$ asymptotics, thus precluding high-dimensional
models.

The rest of the paper is organized as follows. In Section \ref{sec:Examples} we
provide a portfolio of examples that constitute possible applications of our
method. In Section \ref{sec:Deterministic-Bounds} we develop bounds on the
estimation error of the $\ell_1$-ME, which motivate our
method for choosing the penalty parameter. In Section
\ref{sec:Bootstrapping-the-Penalty}, we introduce the BCV penalty method and
derive convergence rates for the resulting $\ell_1$-ME and
post-$\ell_1$-ME. In Section \ref{sec: inference}, we show how to perform
inference on individual components of $\btheta_0$ via debiasing. In Section
\ref{sec:Simulations}, we present a simulation study shedding light on the
finite-sample properties of our method and contrast it with cross-validation.
Finally, in Section \ref{sec:Application}, we apply our method to the empirical
setting of \citet{fryer_jr_empirical_2019}. All proofs are relegated to the
Online Appendices.

\subsection*{Notation}

The distribution $P$ of the pair $(\bX,\bY)$ and features thereof, including the dimension $p$ of
the vector $\bX$, may change with the sample size $n$ (that is, we consider triangular array
sampling and asymptotics), but we suppress this potential dependence whenever this does not cause
confusion in order to simplify notation. We use $\E[f(\bX,\bY)]$ (or $\E_{\bX,\bY}[f(\bX,\bY)]$) to
denote the expectation of a function $f$ of the pair $(\bX,\bY)$ computed with respect to $P$, and
we use $\En[f(\bX_i,\bY_i)]:=n^{-1}\sum_{i=1}^{n}f(\bX_i,\bY_i)$ to abbreviate the sample average.
We use $\R$ and $\N$ to denote all real numbers and all positive integers $\{1,2,\dotsc\},$
respectively. For $k\in\N$, we write $\left[k\right]:=\{1,\dotsc,k\}$ for all positive integers up
to and including $k$. When only a non-empty subset $I\subsetneq[n]$ is in use, we write
$\mathbb{E}_{I}[f(\bX_{i},\bY_i)]:=|I|^{-1}\sum_{i\in I}f(\bX_{i},\bY_i)$ for the subsample average.
For a set of indices $I\subseteq\left[n\right]$, we use $I^{c}$ to denote the elements of
$\left[n\right]$ not in $I.$ For $k\in\N$, we use $\mathbf 0_k$ to denote the vector in $\R^k$ whose
components are all zero. Given a vector $\bdelta\in\R^{k},$ we denote its $\ell_{r}$ norms, $r\in
[1,\infty]$, by $\norm{\bdelta}_{r}$. We write $\mathrm{supp}(\bdelta):=\{j\in[k];\bdelta_j\neq0\}$
for the support of $\bdelta$, and use the $\ell_0$ ``norm''
$\|\bdelta\|_0:=|\mathrm{supp}(\bdelta)|$ to denote the number of non-zero elements of $\bdelta$,
where $|J|$ denotes the cardinality of the set $J$. For any function $f\colon \R\times \mathcal
Z\to\R$, whose first argument is a scalar, we use $f'_1$, $f''_{11}$ and $f'''_{111}$ to denote its
partial derivatives with respect to the first argument of the first, second and third order,
respectively. We abbreviate $a\lor b:=\max\{a,b\} $ and $a\wedge b:=\min\{ a,b\} $. Unless
explicitly stated otherwise, limits are understood as $n\to\infty$. For numbers $a_n$ and positive
numbers $b_n,n\in\N,$ we write $a_n=o(1)$ if $a_n\to0$, and $a_n\lesssim b_n,$ if the sequence
$a_n/b_n$ is bounded. For random variables $V_n$ and positive numbers $b_n,$ we write
$V_n\lesssim_{\P} b_n,$ if the sequence $V_n/b_n$ is bounded in probability. We denote
$\eta_n:=\sqrt{\ln(pn)/n}$. We use the word ``constant'' to refer to non-random quantities that do
not depend on $n$. Finally, we take $n\geqslant3$ and $p\geqslant2$ throughout and introduce more
notation as needed in the appendices.

\section{Examples\label{sec:Examples}}

In this section, we discuss a variety of models that fit into the M-estimation
framework (\ref{eq:EstimandIntro}) with the loss function $m(t,\by)$ being convex
in its first argument. The following examples cover both discrete and continuous
outcomes in likelihood and non-likelihood settings with smooth as well as kinked
loss functions. Additional examples can be found in Appendix
\ref{sec:Additional-Examples}.

\begin{example}
[\textbf{Binary Response Model}]\label{exa:Logit} 
A relatively simple model fitting our framework is the \emph{binary} \emph{response model}, i.e. a
model for an outcome $Y\in\{0,1\} $ with
\[
\P(Y=1\mid\bX)=F(\bX^{\top}\btheta_{0}),
\]
for a known cumulative distribution function (CDF) $F:\R\to\left(0,1\right)$. The log-likelihood of
this model yields the following loss function:
\begin{equation}
m\left(t,y\right)=-y\ln F\left(t\right)-\left(1-y\right)\ln\left(1-F\left(t\right)\right).\label{eq:LossBinaryResponse}
\end{equation}
The \emph{logit} model arises from setting
$F\left(t\right)=1/\left(1+\mathrm{e}^{-t}\right)=:\Lambda\left(t\right)$, the standard logistic
CDF, and the loss function reduces in this case to
\begin{equation}
m\left(t,y\right)=\ln\left(1+\mathrm{e}^{t}\right)-yt.\label{eq:LossLogit}
\end{equation}
The \emph{probit} model arises from setting
$F\left(t\right)=\int_{-\infty}^{t}\left(2\pi\right)^{-1/2}\allowbreak\mathrm{e}^{-u^{2}/2}\mathrm{d}u=:\Phi\left(t\right)$,
the standard normal CDF, and the loss function in this case becomes
\begin{equation}
m\left(t,y\right)=-y\ln\Phi\left(t\right)-\left(1-y\right)\ln\left(1-\Phi\left(t\right)\right).\label{eq:LossProbit}
\end{equation}
Both loss functions (\ref{eq:LossLogit}) and (\ref{eq:LossProbit}) are convex in $t$.

More generally, any binary response model with both $F$ and complementary CDF $1-F$ being
log-concave leads to a loss (\ref{eq:LossBinaryResponse}) that is convex in $t$. For these
log-concavities it suffices that $F$ admits a probability density function (PDF) $f=F'$, which is
itself positive and log-concave \citep[Section 5]{pratt_concavity_1981}. Both the standard logistic
and standard normal PDFs are log-concave. Also, $\ln f$ is concave whenever $f$ is of the (Subbotin)
form $f(t)\propto\mathrm{e}^{-\vert t\vert^{a}/a}$ for some $a\in[1,\infty)$, the extreme case being
the Laplace distribution. See \citet[Section 6]{pratt_concavity_1981} for additional examples. We
focus on the logit and probit cases for concreteness.\qed
\end{example}

\begin{example}[\textbf{Ordered Response Model}]\label{exa:LogConcaveOrderedResponse} 
Consider the \emph{ordered response model}, i.e.~a model for an outcome $Y\in\{0,1,\dots,V\}$ with
\[
\P(Y=v\mid\bX)=F(\alpha_{v+1}-\bX^{\top}\btheta_{0})-F(\alpha_{v}-\bX^{\top}\btheta_{0}),\quad v\in\left[V\right] ,
\]
for a known CDF $F:\R\to(0,1)$ and known cut-off points
$-\infty=\alpha_{0}<\alpha_{1}<\cdots<\alpha_{V}<\alpha_{V+1}=+\infty$. (We interpret
$F\left(-\infty\right)$ as zero and $F\left(+\infty\right)$ as one to subsume the end cases.) The
log-likelihood of this model yields the loss function
\begin{equation}
m\left(t,y\right)=-\sum_{v=0}^{V}\mathbf{1}\left(y=v\right)\ln\left(F\left(\alpha_{v+1}-t\right)-F\left(\alpha_{v}-t\right)\right),\label{eq:LossOrderedResponse}
\end{equation}
which is convex in $t$ for any distribution $F$ admitting a positive and log-concave PDF $f=F'$
\citep[Section 3]{pratt_concavity_1981}. See Example \ref{exa:Logit} for specific distributions
satisfying this criterion. As for binary response, we focus on the logit and probit cases.\qed
\end{example}

\begin{example}
[\textbf{Expectile Model}]\label{exa:Expectile} \citet{newey_asymmetric_1987}
study the conditional ($\tau$th) \textit{expectile model}
$\mu_{\tau}(Y\mid\bX)=\bX^{\top}\btheta_{0}$, where
$\tau\in(0,1)$ is a known number, and propose the\emph{ asymmetric least
squares} (ALS) estimator of $\btheta_{0}$ in this model. This estimator
can be understood as an M-estimator with loss of the form
\begin{equation}
m\left(t,y\right)=\rho_{\tau}\left(y-t\right),\label{eq: ALS}
\end{equation}
with $\rho_{\tau}:\R\to\R$ being the ``swoosh'' function given by
\[
\rho_{\tau}\left(u\right)=\left|\tau-\mathbf{1}\left(u<0\right)\right|u^{2}=\begin{cases}
\left(1-\tau\right)u^{2}, &\text{if}\; u<0,\\
\tau u^{2}, &\text{if}\; u\geqslant0,
\end{cases}
\]
a piecewise quadratic and continuously differentiable analogue of the ``check'' function known from
the quantile regression literature. The ALS estimator can be interpreted as a maximum likelihood
estimator when model disturbances arise from a normal distribution with unequal weights placed on
positive and negative disturbances \citep{aigner_estimation_1976, philipps2022mle}. Note that
$m(\cdot,y)$ in (\ref{eq: ALS}) is convex but not twice differentiable (at $y$) unless
$\tau=1/2$.\qed
\end{example}

\begin{example}
[\textbf{Panel Censored Model}]
\label{exa:PanelCensoredRegressionAndTrimming}
Consider the \textit{panel censored model} 
\[
Y_{\tau}=\max\left(0,\gamma+\bX_{\tau}^{\top}\btheta_0+\varepsilon_{\tau}\right),\quad \tau\in\{1,2\},
\]
where $\bY=(Y_{1},Y_{2})^{\top}\in[0,\infty)^2$ is a pair of outcome variables,
$(\bX_{1}^{\top},\bX_{2}^{\top})^{\top}$ is a vector of regressors, $\gamma$ is a unit-specific
(possibly random) unobserved fixed effect, and $\varepsilon_{1}$ and $\varepsilon_{2}$ are
unobserved error terms, which may or may not be centered. \citet{honore_trimmed_1992} shows that
under certain conditions, including exchangeability of $\varepsilon_{1}$ and $\varepsilon_{2}$
conditional on $(\bX_{1},\bX_{2},\gamma)$, $\btheta_{0}$ in this model can be identified by
$\btheta_{0}=\argmin\nolimits_{\btheta\in\R^{p}}\E[m(\bX^{\top}\btheta,\bY)],$ with $\bX := \bX_1 -
\bX_2$ and $m$ being the \emph{trimmed loss }function
\begin{equation}
m\left(t,\by\right)=\begin{cases}
\Xi\left(y_{1}\right)-\left(y_{2}+t\right)\xi\left(y_{1}\right), &\text{if}\; t\in\left(-\infty,-y_2\right],\\
\Xi\left(y_{1}-y_{2}-t\right), &\text{if}\; t\in\left(-y_{2},y_{1}\right),\\
\Xi\left(-y_{2}\right)-\left(t-y_{1}\right)\xi\left(-y_{2}\right), &\text{if}\; t\in\left[y_{1},\infty\right),
\end{cases}\label{eq:TrimmedLoss}
\end{equation}
and either $\Xi=\left\vert\cdot\right\vert$ or $\Xi=(\cdot)^{2}$ and $\xi$ its derivative (when
defined).\footnote{When $\Xi=\left|\cdot\right|$, we set $\xi\left(0\right):=0$ to make
\eqref{eq:TrimmedLoss} consistent with formulas in \citet{honore_trimmed_1992}.} These choices lead
to \emph{trimmed least absolute deviations} (trimmed LAD) and \emph{trimmed least squares} (trimmed
LS) estimators, respectively, both of which are based on loss functions convex in $t$. Note that
trimmed LAD is based on a non-differentiable loss $m(\cdot,\by)$, while trimmed LS is based on a
continuously differentiable but not twice differentiable loss.\qed
\end{example}

\section{Non-Asymptotic Bounds on Estimation Error\label{sec:Deterministic-Bounds}}

In this section, we derive probabilistic bounds on the error of the
$\ell_{1}$-ME (\ref{eq:ell1PenalizedMEstimationIntro}) in the $\ell_{1}$ and
$\ell_{2}$ norms. The bounds reveal which quantities one needs to control in
order to ensure good behavior of the estimator, motivating the choice of the
penalty parameter $\lambda$ in the next section. 

Our bounds will be based on the following assumptions. Since Assumptions
\ref{as: diff and int}, \ref{assu:Margin}, and
\ref{assu:LossLocallyLipschitzAndMore} stated below are high level, we verify
these assumptions under more low-level conditions in the familiar case of the linear
model with square loss in Appendix \ref{sec:VerificationLinearModelSquareLoss}
and for all examples in Section \ref{sec:Examples}  in Appendix \ref{sec:
verification}.
\begin{assumption}
[\textbf{Parameter Space}]\label{assu:ParameterSpace} The parameter space
$\Theta$ is a non-empty convex subset of $\R^{p}$ for which $\btheta_{0}$ is
interior.
\end{assumption}
\begin{assumption}
[\textbf{Convexity}]\label{assu:Convexity} The function
$m\left(\cdot,\by\right)$ is convex for all $\by\in\mathcal{Y}$.
\end{assumption}

\begin{assumption}
[\textbf{Differentiability and Integrability}]\label{as: diff and int} The
derivative $m_{1}'(\bX^{\top}\btheta,\bY)$ exists almost surely for all
$\btheta\in\Theta$, and $\E\bracn{\absn{m\paran{\bX^{\top}\btheta,\bY}}}<\infty$ for
all $\btheta\in\Theta$.
\end{assumption}

Assumption \ref{assu:ParameterSpace} is a minor regularity condition. Both
convexity and interiority follow trivially in the case of a full parameter space
$\Theta=\R^p$. Assumption \ref{assu:Convexity} is satisfied in all examples from
the previous section, as discussed there. In the same examples, Assumption
\ref{as: diff and int} imposes minor integrability conditions on the random
vectors $\bX$ and $\bY$. In addition, in the case of Example
\ref{exa:PanelCensoredRegressionAndTrimming} with trimmed LAD loss function,
this assumption requires that the conditional distribution of $Y_1 - Y_2$ given
$(\bX, Y_1 + Y_2 > 0)$ is continuous; see Appendix \ref{sec: verification} for
details.

Further, define the \emph{excess risk function} $\mathcal{E}:\Theta\to[0,\infty)$ by
\[
\mathcal{E}\left(\btheta\right):=\E\left[m\left(\bX^{\top}\btheta,\bY\right)-m\left(\bX^{\top}\btheta_{0},\bY\right)\right],\quad\btheta\in\Theta.
\]
By definition of $\btheta_0$ in \eqref{eq:EstimandIntro}, this function is non-negative and takes
value zero at $\btheta = \btheta_0$. The next assumption requires that it grows sufficiently fast
as $\btheta$ moves away from $\btheta_0$.
\begin{assumption}
[\textbf{Margin}]\label{assu:Margin} There are constants $c_{M}\in(0,1]$ and $c_{M}'\in(0,\infty]$
such that for all $\btheta\in\Theta$ satisfying $\Vert\btheta-\btheta_{0}\Vert_{2}\leqslant c_{M}'$,
we have $\mathcal{E}\left(\btheta\right)\geqslant c_{M}\Vert\btheta-\btheta_{0}\Vert_{2}^{2}$.
\end{assumption}
In addition to some technical regularity conditions, this assumption requires the matrix
$\E[\bX\bX^\top]$ to be non-singular, which means that there should be no perfect regressor
multicollinearity in the population. In the context of Example
\ref{exa:PanelCensoredRegressionAndTrimming}, it also requires $Y_1$ and $Y_2$ to be different with
positive probability. Also, our formal analysis reveals that Assumption \ref{assu:Margin} could be
relaxed by requiring the bound $\mathcal{E}\left(\btheta\right)\geqslant
c_{M}\Vert\btheta-\btheta_{0}\Vert_{2}^{2}$ to hold only for certain \emph{sparse} vectors
$\btheta$. We have opted for a less general statement to avoid additional technicalities.

The following assumption requires additional technical regularity of the loss function.
\begin{assumption}[\textbf{Local Loss}]\label{assu:LossLocallyLipschitzAndMore} 
There are constants $c_{L}\in(0,\infty]$, $C_{L}\in[1,\infty)$ and $r\in(4,\infty)$, a non-random
sequence $B_n$ in $[1,\infty)$, and a function $L:\mathcal{X}\times\mathcal Y\to[1,\infty)$ such
that
\begin{enumerate}
\item \label{enu:LossLocallyLipschitz} for all
$(\bx,\by)\in\mathcal{X}\times\mathcal Y$ and all $(t_1,t_2)\in\R^2$ satisfying
$\abs{t_1}\lor\abs{t_2}\leqslant c_{L},$
\begin{align}
\abs{m\left(\bx^{\top}\btheta_{0}+t_1,\by\right)-m\left(\bx^{\top}\btheta_{0}+t_2,\by\right)} & \leqslant L\left(\bx,\by\right)\abs{t_1-t_2}\label{eq:LocallyLipschitz}
\end{align}
with $\max_{1\leqslant j\leqslant p}\E[|L(\bX,\bY)X_j|^2]\leqslant C_L^2$ and
$\E[|L(\bX,\bY)\|\bX\|_{\infty}|^r]\leqslant B_n^r$;
\item \label{enu:LossMeanSquareEll2Conts} for all $\btheta\in\Theta$ satisfying
$\|\btheta - \btheta_0\|_2\leqslant c_L$, we have
\begin{align*}
\mathrm{E}\left[\left|m\left(\bX^{\top}\btheta,\bY\right)-m\left(\bX^{\top}\btheta_{0},\bY\right)\right|^2\right] & \leqslant C_L^2 \|\btheta - \btheta_0\|_2^2;
\end{align*}
\item \label{enu:ResidualMeanSquareEll2Conts} for all $\btheta\in\Theta$ satisfying $\|\btheta -
\btheta_0\|_2\leqslant c_L$, we have
\begin{equation}\label{eq: derivative mean square differentiability}
\E\left[\left| m_{1}'\left(\bX^{\top}\btheta,\bY\right)-m_{1}'\left(\bX^{\top}\btheta_{0},\bY\right)\right|^{2}\right]\leqslant C_{L}^{2}\|\btheta - \btheta_0\|_2.
\end{equation}
\end{enumerate}
\end{assumption}

Assumption
\ref{assu:LossLocallyLipschitzAndMore}.\ref{enu:LossLocallyLipschitz} states that the loss function
is locally Lipschitz in the first argument with the Lipschitz ``constant'' $L(\bx,\by)$ being
sufficiently well-behaved. The local Lipschitzness required in \eqref{eq:LocallyLipschitz} actually
follows from the loss convexity in Assumption \ref{assu:Convexity} \citep[][Theorem
10.4]{rockafellar_convex_1970}, so Assumption
\ref{assu:LossLocallyLipschitzAndMore}.\ref{enu:LossLocallyLipschitz} should be regarded as a mild
moment condition. Assumptions
\ref{assu:LossLocallyLipschitzAndMore}.\ref{enu:LossMeanSquareEll2Conts} and
\ref{assu:LossLocallyLipschitzAndMore}.\ref{enu:ResidualMeanSquareEll2Conts} essentially state that,
viewed as functions of $\btheta$, both the loss and its derivative are mean-square continuous at
$\btheta_0$.  When the loss $m(\cdot,\by)$ is \emph{globally} Lipschitz uniformly in $\by$ (thus
allowing the choice $c_L=\infty$), Assumption
\ref{assu:LossLocallyLipschitzAndMore}.\ref{enu:LossLocallyLipschitz} boils down to the regressors
having sufficiently many absolute moments, and Assumption
\ref{assu:LossLocallyLipschitzAndMore}.\ref{enu:LossMeanSquareEll2Conts} reduces to the requirement
that the largest eigenvalue of $\E[\bX\bX^\top]$ is bounded from above.\footnote{Boundedness of
eigenvalues is a standard assumption in the semi- and non-parametric estimation literature. See
e.g.~\citet[Condition A.2]{belloni_new_2015} and \citet[Assumption 5]{soerensen_2024}.} Examples of
globally Lipschitz losses are the logit likelihood loss in Example \ref{exa:Logit} and the trimmed
LAD loss in Example \ref{exa:PanelCensoredRegressionAndTrimming}.

\begin{assumption}
[\textbf{Approximate Sparsity}]\label{assu:Approximate-Sparsity} 
There is a constant $q\in\left[0,1\right]$ and a non-random sequence $s_{q}:=s_{q,n}$ in
$[1,\infty)$ such that $\sum_{j=1}^{p}|\theta_{0,j}|^{q}\leqslant s_{q}.$
\end{assumption}

This assumption is a sparsity condition, stating that $\btheta_{0}$ lies in an
$\ell_{q}$-``ball'' of ``radius'' $s_{q}^{1/q}$. We interpret the $q=0$ case in the
limiting sense
$\lim_{q\to0_{+}}\sum_{j=1}^{p}|\theta_{0,j}|^{q}=\sum_{j=1}^{p}\mathbf{1}(\theta_{0,j}\neq0)$
so as to nest the case of \textit{exact} sparsity with (at most) $s_{0}$
non-zero entries. When $q>0$, we have only \textit{approximate} sparsity, allowing
possibly many---but typically small---non-zero entries. Related notions of
sparsity appear in many papers on estimation of high-dimensional models. See
Remark \ref{rem:SparsityNotions} for further discussion.

Under Assumption \ref{as: diff and int}, we can (almost surely) define
$\bS_n\in\R^p$ by
\begin{equation}
\bS_n:=\En\left[\left.\frac{\partial}{\partial\btheta}m(\bX_{i}^{\top}\btheta,\bY_{i})\right|_{\btheta = \btheta_0}\right]=\En\left[m_1'(\bX_{i}^{\top}\btheta_0,\bY_{i})\bX_{i}\right].\label{eq:Score}
\end{equation}
In this paper we refer to $\bS_n$ as the \emph{score}.

We are now ready to present a
theorem that provides probabilistic guarantees for $\ell_1$ and $\ell_2$ estimation errors of the
$\ell_{1}$-ME. The proof, given in Appendix \ref{sec:ProofsDeterministicBounds}, builds on arguments
of \cite{belloni_l1-penalized_2011}. Related statements appear also in
\cite{van_de_geer_high-dimensional_2008}, \cite{bickel_simultaneous_2009}, and
\cite{negahban_unified_2012}, among others. Although we could not find the exact same version of the
theorem in the literature, we make no claims of originality for these bounds and include the theorem
for expositional purposes and in order to motivate our method for choosing the penalty parameter
$\lambda$.

To state the theorem, recall that we denote $\eta_n=\sqrt{\ln(pn)/n}$. 

\begin{thm}
[\textbf{Non-Asymptotic Error Bounds for $\boldsymbol{\ell_1}$-ME}]\label{thm:NonAsymptoticProbabilisticBounds} Let Assumptions
\ref{assu:ParameterSpace}--\ref{assu:Approximate-Sparsity} hold, let
$\overline{\lambda}_{n}$ be a non-random sequence in $(0,\infty)$, let $c_0\in(1,\infty)$, and define
\begin{align*}
\ u_{n} & :=\frac{4c_{0}\sqrt{s_{q}\eta_{n}^{-q}}}{(c_{0}-1)c_M}\left(C_{L}\eta_{n}+\overline{\lambda}_{n}\right).
\end{align*}
Then there is a universal constant $C\in[1,\infty)$ such that for all $n\in \N$
and $t\in[1,\infty)$ satisfying
$$
\eta_{n}\leqslant1,\  
   Cu_{n}\leqslant c_{M}', \ 
   \frac{B_n^2\ln(pn)}{\sqrt n}\leqslant C_L^2 \text{ and } \  
   t n^{1/r} B_n\left(Cu_{n}\sqrt{s_{q}\eta_{n}^{-q}}+s_{q}\eta_{n}^{1-q}\right)\leqslant\frac{\left(c_{0}-1\right)c_{L}}{2c_{0}},
$$
 we have
\begin{align*}
\sup_{\mathclap{\widehat{\btheta}\in\widehat{\Theta}\left(\lambda\right)}}\|\widehat{\btheta}-\btheta_{0}\|_{2} & \leqslant Cu_{n}
   \quad\text{and}\quad\sup_{\mathclap{\widehat{\btheta}\in\widehat{\Theta}\left(\lambda\right)}}\|\widehat{\btheta}-\btheta_{0}\|_{1}\leqslant\frac{2c_{0}}{c_{0}-1}\left(Cu_{n}\sqrt{s_{q}\eta_{n}^{-q}}+s_{q}\eta_{n}^{1-q}\right)
\end{align*}
with probability at least
$1-\mathrm{P}(\lambda<c_{0}\|\bS_n\|_{\infty})-\mathrm{P}(\lambda>\overline{\lambda}_{n})-
4t^{-r} - C/\ln^2(pn) - n^{-1}.$
\end{thm}

This theorem motivates our choice of the penalty parameter $\lambda$.
Specifically, it demonstrates that we want a level of regularization sufficient
to overrule the score $(\lambda\geqslant c_{0}\|\bS_n\|_{\infty})$ with high
probability, without making the penalty ``too large''
$(\lambda>\overline{\lambda}_n).$ An interested reader can also find an analogue
of Theorem \ref{thm:NonAsymptoticProbabilisticBounds} for the post-$\ell_1$-ME
in Appendix \ref{sec:analysis post estimator}, but the general principle for
choosing $\lambda$ remains the same.

\begin{rem}[\textbf{Non-Uniqueness}]
Like similar statements appearing in the literature, Theorem
\ref{thm:NonAsymptoticProbabilisticBounds} concerns the entire set $\widehat\Theta(\lambda)$ of
optimizers for the convex minimization problem (\ref{eq:ell1PenalizedMEstimationIntro}). While the
objective function is presumed convex, it need not be strictly convex, and the global minimum may be
attained at more than one point. For example, no matter the choice of $\Xi$, the trimmed loss
function \eqref{eq:TrimmedLoss} in Example \ref{exa:PanelCensoredRegressionAndTrimming} will have
linear pieces and need therefore not produce a strictly convex objective function. The bounds stated
here (and in what follows) hold for any of these optimizers. See also Appendix
\ref{sec:ExistenceSparsityAndUniqueness} for sufficient conditions for solution existence and
uniqueness as well as related (sparsity) properties. Despite the possible multiplicity, we sometimes
refer to any element $\widehat{\btheta}\in\widehat{\Theta}\left(\lambda\right)$ as \emph{the}
$\ell_1$-ME. \qed
\end{rem}

\begin{rem}[\textbf{Margin}]
Our convexity, interiority and differentiability assumptions suffice to show
that the excess risk function $\mathcal E(\btheta)$ is differentiable at
$\btheta_0$, and so the estimand $\btheta_{0}$ must satisfy the population
first-order condition $\nabla \mathcal E\left(\btheta_{0}\right)=\mathbf{0}$.
Assumption \ref{assu:Margin} therefore amounts to assuming that $\mathcal
E(\btheta)$ admits a quadratic \emph{margin} near $\btheta_{0}$. The name
\emph{margin condition} appears to originate from \citet[Assumption
A1]{tsybakov_optimal_2004}, who invokes a similar assumption in a classification
context. \citet[Assumption B]{van_de_geer_high-dimensional_2008} contains a more
general formulation of margin behavior for estimation purposes. We consider the
(focal) quadratic case for the sake of simplicity.\qed
\end{rem}

\begin{rem}
[\textbf{Sparsity Notions}]\label{rem:SparsityNotions} In \citet[Section
4.3]{negahban_unified_2012} the sparsity in Assumption
\ref{assu:Approximate-Sparsity} is referred to as \emph{strong} for $q=0$ and
\emph{weak }for $q>0.$ \citet[Chapter 7]{wainwright_high-dimensional_2019}
distinguishes between \emph{strong }$\ell_{q}$-balls (like
$\{\btheta\in\R^{p};\sum_{j=1}^{p}|\theta_{j}|^{q}\leqslant s_{q}\}$
implicitly considered here) and \emph{weak }$\ell_{q}$-balls, which impose a
polynomial decay in the non-increasing rearrangement of the absolute values of
the coefficients. In \citet{belloni2018highdimensional}, restricting $\btheta_{0}$ to a weak $\ell_{q}$-ball is referred to as \emph{approximate
sparsity}, and a $\btheta_{0}$ having bounded $\ell_{1}$ norm (i.e.~belonging to
a strong $\ell_{1}$-ball) is called \emph{dense}. Both strong $(q>0)$ and weak
ball restrictions formalize the idea of ``weak'' or ``approximate''
sparsity.\qed
\end{rem}

\begin{rem}[\textbf{Free Parameter}]
The free parameter $c_0\in(1,\infty)$ in Theorem \ref{thm:NonAsymptoticProbabilisticBounds} serves
as a trade-off between the likelihood of score domination on the one hand and the bound quality on
the other. A smaller $c_0\in(1,\infty)$ makes the event $\lambda\geqslant c_{0}\|\bS_n\|_{\infty}$
more probable but also worsens the bounds. Note that the free parameter $c_0$ appears, either
explicitly or implicitly, in existing bounds as well.\footnote{A free parameter is explicit in both
\cite{belloni_l1-penalized_2011} and \cite{van_de_geer_high-dimensional_2008}. In deriving their
bounds both \cite{bickel_simultaneous_2009} (for the LASSO) and \cite{negahban_unified_2012} set
$c_0=2$.} While asymptotic theory provides no guidance on the choice of $c_0$, our
finite-sample experiments in Section \ref{sec:Simulations} indicate that increasing $c_0$ away from
one worsens performance but setting $c_0$ to any value near one, including one itself, does not
impact the results by much (cf.~Figures \ref{fig:MeanEll2ErrorBCVVaryingc0BCCHrule} and
\ref{fig:MeanEll2ErrorPostBCVVaryingc0BCCHrule}). Similar observations were made by \citet[Footnote
7]{belloni_sparse_2012} in the context of the LASSO.\qed
\end{rem}

\section{Bootstrapping after Cross-Validation\label{sec:Bootstrapping-the-Penalty}}

We next provide a method for choosing the penalty parameter which is broadly
available yet amenable to theoretical analysis. We split the section into two
parts. In Section \ref{sub: boot penalty level}, we discuss a generic bootstrap
method that allows for choosing the penalty parameter $\lambda$ under
availability of some generic estimators $\widehat{U}_i$ of
$U_i:=m'_1(\bX_i^\top\btheta_0,\bY_i),i\in[n]$. In Section
\ref{sec:ResidualEstimationViaCrossValidation}, we show how to obtain suitable
estimators $\widehat U_i$ via cross-validation. By analogy with linear mean
regression, we refer to $U:=m_1'(\bX^\top\btheta_0,\bY)$ as the
\emph{residual}.\footnote{The linear mean model $\E[Y|\bX]=\bX^\top\btheta_0$
and (half) square loss imply $U=\bX^\top\btheta_0-Y$. The name ``residual''
stems from $U$ agreeing with the deviation $Y-\E[Y|\bX]$ from the mean up to a
sign.}

\subsection{Bootstrapping the Penalty Level}\label{sub: boot penalty level}
Suppose for the moment that residuals $U_{i}=m_1'\left(\bX_{i}^{\top}\btheta_{0},\bY_{i}\right)$ are
observable. In this case, we can estimate the $(1-\alpha)$-quantile $q_n(1-\alpha)$ of
$\|\bS_n\|_{\infty} = \| \En[U_i\bX_i] \|_{\infty}$ via the Gaussian multiplier
bootstrap.\footnote{Recall that $\bS_n$ is well-defined a.s.~under Assumption \ref{as: diff and
int}. We omit the qualifier throughout this section.} To this end, let $\{e_i\}_{i=1}^{n}$ be
independent standard normal random variables that are independent of the data
$\{(\bX_i,\bY_i)\}_{i=1}^{n}$. Given that $\E[U\bX]=\mathbf 0_p$ under mild regularity conditions,
the Gaussian multiplier bootstrap estimates $q_n(1-\alpha)$ by
\[
\widetilde{q}_n\left(1-\alpha\right):=\left(1-\alpha\right)\text{-quantile of }\max_{1\leqslant j\leqslant p}\left|\En\left[e_{i}U_{i}X_{i,j}\right]\right|\text{ given }\{(\bX_{i},\bY_i)\}_{i=1}^{n}.
\]
Under certain regularity conditions, $\widetilde q_n(1-\alpha)$ delivers a good
approximation to $q_n(1-\alpha)$, even if the dimension $p$ of the vectors
$\bX_i$ is much larger than the sample size $n$. To see why this is the case,
let $\bZ:=(Z_1,\dotsc,Z_p)^\top$ be a centered random vector in $\R^{p}$ and let
$\{\bZ_i\}_{i=1}^n$ be independent copies of $\bZ$. 
As established in \citet{chernozhukov_gaussian_2013,chernozhukov_central_2017},
the random vectors $\{\bZ_i\}_{i=1}^n$ satisfy the following high-dimensional
versions of the central limit and Gaussian multiplier bootstrap theorems: If for
some constant $b\in (0,\infty)$ and a non-random sequence $\widetilde{B}_n$ in
$[1,\infty)$, possibly growing to infinity, one has
\[
\min_{1\leqslant j\leqslant p}\E[Z_{j}^{2}]\geqslant b,\quad\max_{k\in\{1,2\}}\max_{1\leqslant j\leqslant p}\E\left[|Z_{j}|^{2+k}\right]/\widetilde{B}_n^k\leqslant 1\quad\text{and}\quad\E\Big[\max_{1\leqslant j\leqslant p}Z_{j}^{4}\Big]\leqslant \widetilde{B}_n^4,
\]
then there is a constant $C_b\in(0,\infty)$, depending only on $b$, such that
\begin{equation}
\sup_{A\in\mathcal{A}_{p}}\left|\P\left(\frac{1}{\sqrt{n}}\sum_{i=1}^{n}\bZ_i\in A\right)-\P(\cN_n\in A)\right|\leqslant C_{b}\left(\frac{\widetilde{B}_n^4\ln^{7}\left(pn\right)}{n}\right)^{1/6},\label{eq:CCK_HDCLT}
\end{equation}
and, with probability approaching one,
\begin{equation}
\sup_{A\in\mathcal{A}_{p}}\left|\P\left(\left.\frac{1}{\sqrt{n}}\sum_{i=1}^{n}e_{i}\bZ_{i}\in A\right|\left\{\bZ_{i}\right\} _{i=1}^{n}\right)-\P(\cN_n\in A)\right|\leqslant C_{b}\left(\frac{\widetilde{B}_n^4\ln^{7}\left(pn\right)}{n}\right)^{1/6},\label{eq:CCK_HDBootstrap}
\end{equation}
where $\mathcal{A}_{p}$ denotes the collection of all (hyper)rectangles in
$\R^{p}$, and $\cN_n$ is a centered Gaussian random vector in $\R^p$ with
covariance matrix $\E[\bZ\bZ^{\top}]$. Provided
$\widetilde{B}_n^4\ln^{7}(pn)/n\to0$, applying these two results with $\bZ_i =
U_i\bX_i$ for all $i\in[n]$ and noting that sets of the form $A_t =
\{\boldsymbol{u}\in\R^p;\max_{1\leqslant j\leqslant p}|u_j|\leqslant t\}$,
$t\in[0,\infty)$, are indeed rectangles, suggest that the Gaussian multiplier
bootstrap estimator $\widetilde q_n(1-\alpha)$ provides a good approximation to
$q_n(1-\alpha)$.

As we typically do not observe the residuals $U_i =
m_{1}'(\bX_{i}^{\top}\btheta_{0},\bY_{i})$, the method described above is
infeasible. Fortunately, the result (\ref{eq:CCK_HDBootstrap}) continues to hold
upon replacing $\{\bZ_{i}\}_{i=1}^n$ with estimators
$\{\widehat{\bZ}_{i}\}_{i=1}^{n}$, provided these estimators are ``sufficiently
good,'' in the sense to be defined below; see \eqref{eq: residual estimation side condition}. Suppose
therefore that residual estimators $\{\widehat{U}_{i}\}_{i=1}^{n}$ are
available. We then compute
\begin{equation}
\widehat{q}^{\texttt{bm}}\left(1-\alpha\right):=\left(1-\alpha\right)\text{-quantile of }\max_{1\leqslant j\leqslant p}\big|\En\big[e_{i}\widehat{U}_{i}X_{i,j}\big]\big|\text{ given }\{(\bX_{i},\bY_i,\widehat{U}_{i})\}_{i=1}^{n},\label{eq:qhat1minusAlpha}
\end{equation}
and a penalty level follows as
\begin{equation}
\widehat{\lambda}^{\mathtt{bm}}_{\alpha}:=c_{0}\widehat{q}^{\texttt{bm}}\left(1-\alpha\right).\label{eq:BootstrapPenaltyLevel}
\end{equation}
We refer to this method for obtaining a penalty level as the \emph{bootstrap
method} (BM) and to $\widehat{\lambda}^{\mathtt{bm}}_{\alpha}$ itself as the
\emph{bootstrap penalty level}.

To ensure that $\widehat{q}^{\texttt{bm}}\left(1-\alpha\right)$ indeed delivers a good
approximation to $q_n(1-\alpha)$, we invoke the following
assumption.

\begin{assumption}
[\textbf{Residuals}]\label{assu:ResidualBootstrapMethod} 
There are constants $c_{U}\in(0,\infty)$ and $C_U\in[1,\infty)$ and a non-random
sequence $\widetilde{B}_n$ in $[1,\infty)$, such that
\begin{inparaenum}[(1)] 
\item\label{enu:ZijSecondMomentsBndAwayZero}
$c_U^2\leqslant\E[|UX_j|^2]\leqslant C_{U}^2$ for all $j\in[p]$, 
\item\label{enu:Zj2pluskMomentBnd} $\E[|UX_j|^{4}]\leqslant\widetilde{B}_n^2$
for all $j\in[p]$, and
\item\label{enu:maxZijFourthMomentBnd}
$\E\left[\|U\bX\|_{\infty}^4\right]\leqslant \widetilde{B}_n^4.$\end{inparaenum}
\end{assumption}

This assumption imposes a few minor regularity conditions. It requires, in
particular, that all components of the vector $\bX$ are normalized to be on the
same scale. Since this assumption is high level, we verify it under low-level
conditions in Appendix \ref{sec: verification} for the examples in Section
\ref{sec:Examples}.

Our next result provides convergence rates for the $\ell_1$-ME based on the bootstrap method.

\begin{lem}
[\textbf{Convergence Rates: Generic Bootstrap Method}]\label{thm:RatesBM} Let
Assumptions \ref{assu:ParameterSpace}--\ref{assu:Approximate-Sparsity} and
\ref{assu:ResidualBootstrapMethod} hold, let $\delta_n$ be a non-random sequence
in $[0,\infty)$ such that 
\begin{equation}\label{eq: residual estimation side condition}
\P\Big(\En[(\widehat{U}_{i}-U_{i})^{2}] > \delta_{n}^{2}/\ln^{2}\left(pn\right)\Big)\to 0,
\end{equation}
let $\widehat{\Theta}(\widehat{\lambda}^{\mathtt{bm}}_{\alpha})$ be the
solutions to the $\ell_{1}$-penalized M-estimation problem
(\ref{eq:ell1PenalizedMEstimationIntro}) with penalty level
$\lambda=\widehat{\lambda}^{\mathtt{bm}}_{\alpha}$ given in
(\ref{eq:BootstrapPenaltyLevel}) and $\alpha=\alpha_n\in(0,1)$ satisfying $\alpha_n\to0$
and $\ln(1/\alpha_n)\lesssim \ln(pn)$, and suppose that
\begin{equation}\label{eq: simple restriction bm}
   n^{1/r}B_{n}\left(\delta_{n} + s_{q}\eta_{n}^{1-q}\right)\to0,\quad
   \frac{B_{n}^{4}\ln^2(pn)}{n}\to0\quad\text{and}\quad
   \frac{\widetilde{B}_{n}^{4}\ln^{7}\left(pn\right)}{n}\to0.
\end{equation}
Then 
\[
\sup_{\mathclap{\widehat\btheta\in\widehat{\Theta}(\widehat{\lambda}^{\mathtt{bm}}_{\alpha})}}\|\widehat{\btheta}-\btheta_{0}\|_{2} \lesssim_{\P} \sqrt{s_{q}\eta_{n}^{2-q}}\quad\text{and}\quad
\sup_{\mathclap{\widehat\btheta\in\widehat{\Theta}(\widehat{\lambda}^{\mathtt{bm}}_{\alpha})}}\|\widehat{\btheta}-\btheta_{0}\|_{1}\lesssim_{\P} s_{q}\eta_{n}^{1-q}.
\]
\end{lem}

The idea of using a bootstrap procedure to select the penalty level in
high-dimensional estimation is in itself not new.
\cite{chernozhukov_gaussian_2013} use a Gaussian multiplier bootstrap to tune
the Dantzig selector \citep{candes_dantzig_2007} for the high-dimensional linear
model allowing both non-Gaussian and heteroskedastic errors. Note, however, that
\citet[Theorem 4.2]{chernozhukov_gaussian_2013} presumes access to a preliminary
Dantzig selector, which is used to estimate residuals. The condition (\ref{eq:
residual estimation side condition}) is similarly high-level in the sense that
it does not specify how one performs residual estimation in practice. Our
primary contribution lies in providing a method for coming up with good residual
estimators. We turn to this task in the next subsection, where we also compare the
rates with those appearing in the literature and discuss the side conditions
under which they are derived; see Remarks \ref{rem:ConvergenceRates} and
\ref{rem:DenseCase}.

\subsection{Cross-Validating Residuals\label{sec:ResidualEstimationViaCrossValidation}}

In this subsection, we explain how residual estimation can be performed via
cross-validation (CV). To describe our CV residual estimator, fix any integer
$K\geqslant2$, and let $\{I_{k}\}_{k=1}^K$ partition the observation indices
$[n]$. Provided $n$ is divisible by $K$, the even
partition
\begin{equation}
I_{k}=\left\{ \left(k-1\right)n/K+1,\dotsc,kn/K\right\} ,\quad k\in\left[K\right] ,\label{eq:CVEvenPartition}
\end{equation}
is natural, but not necessary. For the formal results below, we only require
that each $I_{k}$ specifies a ``substantial'' subsample; see Assumption
\ref{assu:DataPartition} below.

Let $\Lambda_{n}$ denote a finite subset of $(0,\infty)$ composed by candidate
penalty levels. In Assumption \ref{assu:CandidatePenalties} below, we require
$\Lambda_{n}$ to be ``sufficiently rich.'' Our CV procedure then goes as
follows. First, estimate the vector of parameters $\btheta_{0}$ by 
\begin{equation}\label{eq:SubsampleEstimator}
\widehat{\btheta}_{I_{k}^{c}}\left(\lambda\right) \in \widehat{\Theta}_{I_{k}^{c}}\left(\lambda\right) := \argmin_{\btheta\in\Theta}\left\{ \mathbb{E}_{I_k^c}[m(\bX_{i}^{\top}\btheta,\bY_{i})]+\lambda\norm{\btheta}_{1}\right\},
\end{equation}
for each candidate penalty level $\lambda\in\Lambda_{n}$ and holding out each
subsample $k\in[K]$ in turn. Second, determine the penalty level 
\begin{equation}\label{eq:LambdaCVDefn}
\widehat{\lambda}^{\mathtt{cv}}\in\argmin_{\lambda\in\Lambda_{n}}\sum_{k=1}^{K}\sum_{i\in I_{k}}m\big(\bX_{i}^{\top}\widehat{\btheta}_{I_{k}^{c}}\left(\lambda\right),\bY_{i}\big)
\end{equation}
by minimizing the out-of-sample loss over the set of candidate penalties. Third,
estimate residuals $U_{i}=m_{1}'(\bX_{i}^{\top}\btheta_{0},\bY_{i}),i\in[n]$, by
predicting out of each estimation subsample, i.e.,
\begin{equation}\label{eq:ResidualEstimatorCV}
\widehat{U}_{i}^{\mathtt{cv}}:=m_{1}'\big(\bX_{i}^{\top}\widehat{\btheta}_{I_{k}^{c}}(\widehat{\lambda}^{\mathtt{cv}}),\bY_{i}\big),\quad i\in I_{k},\quad k\in\left[K\right].
\end{equation}
Note here that since $I_k$ and $I_k^c$ have no elements in common, the
derivative $m_1'(\bX^{\top}_i\widehat\btheta_{I_k^c}(\lambda),\bY_i)$ exists for
all $i\in I_k$, $k\in[K]$, and $\lambda\in\Lambda_n$ almost surely by Assumption
\ref{as: diff and int}. The residual estimates $\{\widehat
U_i^{\mathtt{cv}}\}_{i=1}^n$ are therefore almost-surely well-defined even
though the function $m(\cdot,\by)$ is not necessarily differentiable.

Combining the bootstrap penalty level
$\widehat{\lambda}^{\mathtt{bm}}_\alpha=c_{0}\widehat{q}^{\texttt{bm}}\left(1-\alpha\right)$
from the previous subsection with the CV residual estimates
$\widehat{U}_i=\widehat{U}_{i}^{\mathtt{cv}}$ from this subsection, we obtain
the {\em bootstrap-after-cross-validation} (BCV) \emph{method} for estimating
the quantile $q_n(1-\alpha)$,
\begin{equation}
\widehat{q}^{\mathtt{bcv}}\left(1-\alpha\right):=\left(1-\alpha\right)\text{-quantile of }\max_{1\leqslant j\leqslant p}\big|\En\big[e_{i}\widehat{U}_{i}^{\mathtt{cv}}X_{i,j}\big]\big|\text{ given }\{(\bX_{i},\bY_i)\}_{i=1}^{n},\label{eq:qhat1minusAlpha CV}
\end{equation}
and the \emph{BCV penalty level} follows as
\begin{equation}
\widehat{\lambda}^{\mathtt{bcv}}_{\alpha}:=c_{0}\widehat{q}^{\mathtt{bcv}}\left(1-\alpha\right).\label{eq:BootstrapPenaltyLevel CV}
\end{equation}
To analyze the $\ell_1$-ME implied by BCV, we invoke the following two assumptions.
\begin{assumption}
[\textbf{Data Partition}]\label{assu:DataPartition} The number of folds $K\in\left\{
2,3,\dotsc\right\}$ is constant. There is a constant
$c_{D}\in\left(0,1\right)$ such that $\min_{1\leqslant k\leqslant
K}\left|I_{k}\right|\geqslant c_{D}n$.
\end{assumption}
\begin{assumption}
[\textbf{Candidate Penalties}]\label{assu:CandidatePenalties} There are
constants $c_{\Lambda}$ and $C_{\Lambda}$ in $(0,\infty)$ and
$a\in\left(0,1\right)$ such that
\[
\Lambda_{n}=\left\{ C_{\Lambda}a^{\ell};a^{\ell}\geqslant c_{\Lambda}/n,\ell\in\left\{ 0,1,2,\dotsc\right\} \right\} .
\]
\end{assumption}

Assumption \ref{assu:DataPartition} means that we rely upon classical $K$-fold cross-validation with
fixed $K$. This assumption does rule out leave-one-out cross-validation, since $K=n$ and
$I_{k}=\left\{ k\right\} $ imply $|I_k|/n\to 0$. Assumption \ref{assu:CandidatePenalties} allows for
a rather large candidate set $\Lambda_{n}$ of penalty values. Note that the largest penalty value,
$C_{\Lambda}$, can be set arbitrarily large and the smallest value, $c_{\Lambda}/n$, converges
rapidly to zero. As a part of the proof of Theorem \ref{cor: convergence rate bootstrap after cv}
below, we show that these properties ensure that the set $\Lambda_{n}$ eventually contains a
``good'' penalty candidate, say $\lambda_{\ast}$, in the sense of leading to a uniform bound on the
excess risk of subsample estimators
$\widehat{\btheta}_{I_{k}^{c}}\left(\lambda_{\ast}\right),k\in[K]$ and, because of that, the CV
residual estimators are reasonable inputs for the bootstrap method, in the sense of satisfying
\eqref{eq: residual estimation side condition}. Combining this finding with Lemma \ref{thm:RatesBM},
we obtain convergence rates for the $\ell_1$-ME implied by BCV.
\begin{thm}
[\textbf{Convergence Rates: BCV Method, Penalized Estimator}]\label{cor: convergence rate bootstrap after cv} Let Assumptions
\ref{assu:ParameterSpace}--\ref{assu:Approximate-Sparsity} and
\ref{assu:ResidualBootstrapMethod}--\ref{assu:CandidatePenalties} hold, let
$\widehat\Theta(\widehat \lambda_{\alpha}^{\mathtt{bcv}})$ be the
solutions to the $\ell_1$-penalized M-estimation problem
\eqref{eq:ell1PenalizedMEstimationIntro} with penalty level $\lambda =
\widehat\lambda_{\alpha}^{\mathtt{bcv}}$ given in
\eqref{eq:BootstrapPenaltyLevel CV} and $\alpha=\alpha_n\in(0,1)$ satisfying $\alpha_n\to0$ and
$\ln(1/\alpha_n)\lesssim \ln(pn)$, and suppose that
\begin{equation}\label{eq: difficult restriction}
n^{1/r}B_ns_q\eta_n^{1-q}\to0,\quad
\frac{B_n^4 s_q(\ln(pn))^{5-q/2}(\ln n)^2}{n^{1-q/2-4/r}}\to0\quad\text{and}\quad
\frac{\widetilde{B}_n^4\ln^7\left(pn\right)}{n}\to0.
\end{equation}
Then
\begin{equation}\label{eq: rates ell1 me}
   \sup_{\mathclap{\widehat\btheta\in\widehat{\Theta}(\widehat{\lambda}^{\mathtt{bcv}}_{\alpha})}}\|\widehat{\btheta}-\btheta_{0}\|_{2}\lesssim_{\P} \sqrt{s_{q}\eta_{n}^{2-q}}\quad\text{and}\quad
   \sup_{\mathclap{\widehat\btheta\in\widehat{\Theta}(\widehat{\lambda}^{\mathtt{bcv}}_{\alpha})}}\|\widehat{\btheta}-\btheta_{0}\|_{1}\lesssim_{\P} s_{q}\eta_{n}^{1-q}.
\end{equation}
\end{thm}

\begin{rem}[\textbf{Convergence Rates}]\label{rem:ConvergenceRates} 
The Theorem \ref{cor: convergence rate bootstrap after cv} (and Lemma \ref{thm:RatesBM}) convergence
rates are as one would expect in high-dimensional settings. For example, the $\ell_2$ rate in
\eqref{eq: rates ell1 me} coincides with that obtained for the LASSO in \citet[Corollary
3]{negahban_unified_2012} in the context of linear mean regression with $\btheta_0$ belonging to an
$\ell_q$-ball (Assumption \ref{assu:Approximate-Sparsity}). The rate is known to be minimax optimal
in the context of sparse linear mean regression \citep{ye2010rate,raskutti12}. We expect it to
remain optimal in the general high-dimensional M-estimation framework as well. In the special case
of exact sparsity $(q=0)$, the $\ell_2$ and $\ell_1$ rates in \eqref{eq: rates ell1 me} become
$\sqrt{s_0\ln(pn)/n}$ and $\sqrt{s_0^2\ln(pn)/n}$, respectively.\qed
\end{rem}

\begin{rem}[{\textbf{Dense Case}}]\label{rem:DenseCase} 
In the dense case $(q=1)$ $\eta_n^{1-q}$ does not vanish, and so the side condition $n^{1/r}B_n s_q
\eta_n^{1-q}\to0$ in \eqref{eq: difficult restriction} fails. However, inspection of the proof
reveals that we actually require $n^{1/r}B_n s_q \eta_n^{1-q} / c_L\to0$ for $c_L\in(0,\infty]$ in
Assumption \ref{assu:LossLocallyLipschitzAndMore}. The latter condition is trivially satisfied when
$c_L=\infty$, which is allowed when the loss $m(t,\by)$ is globally Lipschitz in its first argument
uniformly in $\by\in\mathcal Y$. Hence, provided the loss is globally Lipschitz, even in the dense
case Theorem \ref{cor: convergence rate bootstrap after cv} can produce the $\ell_2$ rate of
convergence $s_1^{1/2}(\ln(pn)/n)^{1/4}$. Examples of globally Lipschitz losses are the logit
likelihood loss in Example \ref{exa:Logit} and the trimmed LAD loss in Example
\ref{exa:PanelCensoredRegressionAndTrimming}. The side condition $n^{1/r}B_n s_q \eta_n^{1-q}\to0$
may also be relaxed in the special case of generalized linear models; see \citet[Section
4.4]{negahban_unified_2012} and \citet[Section 9.5]{wainwright_high-dimensional_2019} for
details.\qed
\end{rem}

\begin{rem}[\textbf{Model Sparsity and Regressor Regularity}]
The side condition
\[
   \frac{B_n^4 s_q(\ln(pn))^{5-q/2}(\ln n)^2}{n^{1-q/2-4/r}}\to 0   
\]
in \eqref{eq: difficult restriction} necessitates $r>8/(2-q)$, which reveals an
interplay between the model sparsity as captured by the constant $q$ in
Assumption \ref{assu:Approximate-Sparsity} and the regressor integrability as
captured by the constant $r$ in Assumption
\ref{assu:LossLocallyLipschitzAndMore}.\ref{enu:LossLocallyLipschitz}. In the
special case of exact sparsity $(q=0)$ the regressors are required to have more
than four finite moments.\qed
\end{rem}

We next consider the post-$\ell_1$-penalized M-estimator (post-$\ell_1$-ME). The main motivation for
the post-$\ell_1$-ME is that the $\ell_1$-ME may be severely biased because it shrinks coefficients
toward zero. By refitting the non-zero coefficients of the $\ell_1$-ME without the penalty in the
criterion function in \eqref{eq:ell1PenalizedMEstimationIntro}, the post-$\ell_1$-ME attempts to
reduce this bias. 

To define the post-$\ell_1$-ME, for any
$\overline\btheta\in\Theta$, we define the set
$\widetilde\Theta(\suppn{\overline\btheta})\subseteq\Theta$ by
\begin{equation}\label{eq:post-ell1-ME-thetabar}
   \widetilde\Theta(\suppn{\overline\btheta})
      := \argmin_{\mathclap{\substack{\btheta\in\Theta,\\\supp\btheta\subseteq\suppn{\overline\btheta}}}}\En[m(\bX_i^\top\btheta,\bY_i)].   
\end{equation}
Then for any $\ell_1$-ME, i.e.~a solution $\widehat\btheta\in\widehat\Theta(\lambda)$ to the
optimization problem in \eqref{eq:ell1PenalizedMEstimationIntro}, the corresponding post-$\ell_1$-ME
is defined as any element $\widetilde\btheta$ of the set
$\widetilde\Theta(\suppn{\widehat\btheta})$. Note that there could be multiple
post-$\ell_1$-MEs.\footnote{As in our treatment of $\ell_1$-ME, we implicitly assume that a
post-$\ell_1$-ME exists.} Our treatment below covers the set $\widetilde\Theta(\lambda)$ of
\emph{all} post-$\ell_1$-MEs, which we denote
\begin{equation}\label{eq:AllPostEstimators}
   \widetilde\Theta(\lambda):=
      \bigcup_{\mathclap{\widehat\btheta\in\widehat\Theta(\lambda)}} \widetilde\Theta(\suppn{\widehat\btheta}).
\end{equation}
To analyze the post-$\ell_1$-ME, we will use the following two additional assumptions.

\begin{assumption}[\textbf{Smoothness}]\label{as: post estimator smoothness} The function
$m(\cdot,\by)$ is differentiable for all $\by\in\mathcal Y$ with its derivative being
Lipschitz-continuous, i.e.~$|m_{1}'(t_2,\by) - m_1'(t_1,\by)|\leqslant C_m|t_2 - t_1|$ for all
$(t_1,t_2,\by)\in\mathbb R\times\mathbb R\times\mathcal Y$ and some constant $C_m\in(0,\infty)$.
\end{assumption}

\begin{assumption}[\textbf{Moments}]\label{as: post estimator moments}
There is a constant
      $C_{ev}\in[1,\infty)$ such that $\E[(\bX^\top \bdelta)^4]\leqslant
      C_{ev}\|\bdelta\|_2^4$ for all $\bdelta\in\mathbb R^p$.
\end{assumption}

Assumption \ref{as: post estimator smoothness} strengthens the almost-sure differentiability in
Assumption \ref{as: diff and int}. The stronger smoothness requirement precludes the trimmed LAD
loss function in Example \ref{exa:PanelCensoredRegressionAndTrimming}, but it can be easily verified
under more low-level conditions
for the trimmed LS loss function in the same example as well as for all other examples from Section
\ref{sec:Examples}; see Appendix \ref{sec: verification}. Assumption \ref{as: post
estimator moments} is satisfied if the entries of $\bX$ are independent standard Gaussian, for
example. Related assumptions appear in the existing literature on high-dimensional estimation.

With these added assumptions, we can derive the convergence rates for the post-$\ell_1$-ME.
\begin{thm}[\textbf{Convergence Rates: BCV Method, Post-Penalized Estimator}]
\label{thm: convergence rates post estimator}
Let Assumptions
\ref{assu:ParameterSpace}--\ref{assu:Approximate-Sparsity} and
\ref{assu:ResidualBootstrapMethod}--\ref{as: post estimator moments} hold, let
$\widetilde\Theta(\widehat \lambda_{\alpha}^{\mathtt{bcv}})$ be the set of post-$\ell_1$-penalized M-estimators \eqref{eq:AllPostEstimators} with penalty level $\lambda =
\widehat\lambda_{\alpha}^{\mathtt{bcv}}$ given in
\eqref{eq:BootstrapPenaltyLevel CV} and $\alpha=\alpha_n\in(0,1)$ satisfying $\alpha_n\to0$ and
$\ln(1/\alpha_n)\lesssim \ln(pn)$, and suppose that
\begin{equation}\label{eq: even more difficult restriction}
n^{1/r}B_ns_q\eta_n^{1-q}\ln(p n)\to0,\ \frac{B_n^4 s_q(\ln(pn))^{5-q/2}(\ln n)^2}{n^{1-q/2-4/r}}\to 0\ \text{and }  \;\frac{\widetilde{B}_n^4\ln^7\left(pn\right)}{n}\to 0.
\end{equation}
Then
\begin{equation}\label{eq: rates post ell1 me}
   \sup_{\mathclap{\widetilde\btheta\in\widetilde{\Theta}(\widehat{\lambda}^{\mathtt{bcv}}_{\alpha})}}\|\widetilde{\btheta}-\btheta_{0}\|_{2}\lesssim_{\P} \sqrt{s_{q}\eta_{n}^{2-q}\ln(pn)}\quad\text{and}\quad
   \sup_{\mathclap{\widetilde\btheta\in\widetilde{\Theta}(\widehat{\lambda}^{\mathtt{bcv}}_{\alpha})}}\|\widetilde{\btheta}-\btheta_{0}\|_{1}\lesssim_{\P} s_{q}\eta_{n}^{1-q}\ln(pn).
\end{equation}
\end{thm}

\begin{rem}[{\textbf{Comparison of Rates for $\ell_1$-ME and Post-$\ell_1$-ME}}]
The convergence rates for the post-$\ell_1$-ME we derive here are slightly slower than those we
derived for the $\ell_1$-ME itself in Theorem \ref{cor: convergence rate bootstrap after cv}. The
technical reason for this difference is that the analysis of the post-$\ell_1$-ME requires not only
that the penalty parameter $\lambda$ is not too \textit{large} but also that it is not too
\textit{small}, as we may end up with ``too many'' selected variables; see Appendix
\ref{sec:analysis post estimator} for details. In turn, our BCV method may yield low values of the
penalty level if there is substantial correlation between regressors in the vector $\bX$. A simple
solution to this issue would be to censor the BCV penalty parameter
$\widehat\lambda_{\alpha}^{\mathtt{bcv}}$ from below so that it shrinks to zero no faster than
$\eta_n$, i.e.~to replace $\widehat\lambda_{\alpha}^{\mathtt{bcv}}$ by
$\max\{\widehat\lambda_{\alpha}^{\mathtt{bcv}}, c\eta_n\}$ for some small user-chosen constant
$c\in(0,\infty)$. In this case, the rates in \eqref{eq: rates post ell1 me} would coincide with the
rates in \eqref{eq: rates ell1 me}. However, we prefer to state a slightly slower rate, as in
\eqref{eq: rates post ell1 me}, over the necessity to introduce another tuning parameter $c$ (for
which there is no obvious guiding principle). Moreover, under the additional assumption that the
elements of the regressor vector $\bX$ are not too correlated, it is possible to derive the same
rates as in \eqref{eq: rates ell1 me} for the post-$\ell_1$-ME even without censoring, in which case
the rates for post-$\ell_1$-ME are as good as those for $\ell_1$-ME.\qed
\end{rem}

\section{Debiased Estimation and Inference}\label{sec: inference}

In this section, we describe how to construct $\sqrt n$-consistent and asymptotically normal
estimators of individual components of the vector $\btheta_0$ defined in \eqref{eq:EstimandIntro}.
Since these estimators are asymptotically unbiased and have easily estimable asymptotic variance,
they lead to standard inference procedures for testing hypotheses about and building confidence
intervals for individuals components of $\btheta_0$. Our approach here is based on the concept of
Neyman orthogonal equations and closely follows the literature on double/debiased machine learning
\citep{chernozhukov2018double}. We note that the tools developed in this section rule out the
trimmed LAD loss in Example \ref{exa:PanelCensoredRegressionAndTrimming}, as this function is not
sufficiently smooth to satisfy our Assumption \ref{as: smoothness inference}.

Without loss of generality, suppose that we are interested in the first
component of the vector $\btheta_0$, so that $\btheta_0 =
(\beta_0,\bgamma_0^{\top})^{\top}$, where $\beta_0\in\R$ is the scalar parameter
of interest and $\bgamma_0\in\R^{p-1}$ is a vector of nuisance parameters. To
derive a $\sqrt n$-consistent and asymptotically normal estimator of $\beta_0$,
write $\bX = (D,\bW^{\top})^{\top}$, so that $\bX^{\top}\btheta_0 = D\beta_0 +
\bW^{\top}\bgamma_0$, and let $\bmu_0\in\R^{p-1}$ be a vector that is defined as
a solution to the following system of equations:
\begin{equation}\label{eq: definition of mu0}
\E[m''_{11}(\bX^{\top}\btheta_0,\bY)(D-\bW^{\top}\bmu_0)\bW] = \mathbf{0}_{p-1}.
\end{equation}
Note that this system has a solution $\bmu_0$ and this solution is unique as
long as the matrix $\E[m''_{11}(\bX^{\top}\theta_0,\bY)\bW\bW^{\top}]$ is
non-singular, which is the case under our assumptions.\footnote{See Lemma
\ref{lem:ExistenceAndUniquenessOfMu0} in the appendix for the precise
statement.} With this definition in mind, by inspecting the first-order
conditions associated with \eqref{eq:EstimandIntro}, we have
\begin{equation}\label{eq: estimating equation}
\E[m'_1(D\beta_0 + \bW^{\top}\bgamma_0,\bY)(D-\bW^{\top}\bmu_0)] = 0.
\end{equation}
We obtain an estimator of $\beta_0$ by solving an empirical version of this
equation, where we replace the (high-dimensional) vectors $\bgamma_0$ and
$\bmu_0$ by suitable estimators. Here, $\sqrt n$-consistent and asymptotically
normal estimation of $\beta_0$ is possible due to \eqref{eq:
estimating equation} being Neyman orthogonal with respect to $\bgamma_0$ and
$\bmu_0$, which means that this equation is first-order insensitive with respect
to perturbations in $\bgamma_0$ and $\bmu_0$. Specifically, we have
\begin{align*}
&\left.\frac{\partial}{\partial\bgamma}\E[m'_1(D\beta_0 + \bW^{\top}\bgamma,\bY)(D-\bW^{\top}\bmu_0)]\right|_{\bgamma=\bgamma_0} = \mathbf 0_{p-1}\quad\text{and}\\
&\left.\frac{\partial}{\partial\bmu}\E[m'_1(D\beta_0 + \bW^{\top}\bgamma_0,\bY)(D-\bW^{\top}\bmu)]\right|_{\bmu=\bmu_0} = \mathbf 0_{p-1},
\end{align*}
which follows from \eqref{eq: definition of mu0} and \eqref{eq:EstimandIntro},
respectively. Neyman orthogonality thus facilitates simple inference for the
low-dimensional $\beta_0$ despite possibly complicated estimation of the
high-dimensional $\bgamma_0$ and $\bmu_0$. Formally, we consider the following
procedure:
\begin{alg}[\textbf{Three-Step Debiasing}]\label{alg:ThreeStepDebiasing} 
Given rules for choosing penalty levels
$\lambda_1,\lambda_2\in(0,\infty)$,\footnote{Here, $\lambda_1$ can be chosen via the BCV method, and
$\lambda_2$ can be chosen either via the BCV method or via the SNMD theory for weighted LASSO, as
discussed in \cite{belloni2016post}.} follow the steps below to obtain a debiased estimator
$\widehat\beta$ of $\beta_0$:
\begin{itemize}
   \item[\emph{Step 1 (Initiate):}] 
      \begin{inparaenum}[a.]
         \item Define the (preliminary) estimator $\widetilde\btheta =
            (\widetilde\beta,\widetilde\bgamma^{\top})^{\top}$~of $\btheta_0 =
            (\beta_0,\bgamma_0^{\top})^{\top}$ by
            $$
            \widetilde\btheta \in \argmin_{\btheta\in\Theta}\left\{\En[m(\bX_i^{\top}\btheta,\bY_i)] + \lambda_1\|\btheta\|_1\right\}.
            $$
         \item (Optional): Define $\widetilde{T}_1:=\mathrm{supp}(\widetilde\btheta)$ and recast $\widetilde{\btheta}$ as the refitted estimator of $\btheta_0$:
            $$
            \widetilde\btheta \in \argmin_{\mathclap{\substack{\btheta\in\Theta,\\\mathrm{supp}(\btheta)\subseteq\widetilde{T}_1}}}\En[m(\bX_i^{\top}\btheta,\bY_i)].
            $$
      \end{inparaenum} 
      
   \item[\emph{Step 2 (Orthogonalize):}]
   \begin{inparaenum}[a.] 
      \item Based on $\widetilde{\btheta}$ from Step 1, define an estimator $\widetilde\bmu$ of $\bmu_0$ by
         \begin{equation}\label{eq: estimator mu}
         \widetilde\bmu \in \argmin_{\bmu\in\R^{p-1}}\left\{\En\big[ m''_{11}(\bX_i^{\top}\widetilde\btheta,\bY_i)(D_i - \bW_i^{\top}\bmu)^2\big] + \lambda_2\|\bmu\|_1\right\}.
         \end{equation}
      \item (Optional): Define $\widetilde{T}_2:=\mathrm{supp}(\widetilde\bmu)$ and recast $\widetilde{\bmu}$ as the refitted estimator of $\bmu_0$:
         \begin{equation}\label{eq: estimator mu post}
         \widetilde\bmu \in \argmin_{\mathclap{\substack{\bmu\in\R^{p-1},\\\mathrm{supp}(\bmu)\subseteq\widetilde{T}_2}}}\En\big[ m''_{11}(\bX_i^{\top}\widetilde\btheta,\bY_i)(D_i - \bW_i^{\top}\bmu)^2\big].
         \end{equation}
   \end{inparaenum} 
   \item[\emph{Step 3 (Update):}] Define the (debiased) estimator $\widehat\beta$ of $\beta_0$ as the
   one-step update of $\widetilde\beta$:
      \begin{equation}\label{eq: debiased estimator betahat}
      \widehat\beta := \widetilde\beta - \frac{\En\big[m'_1(\bX_i^{\top}\widetilde\btheta, \bY_i)(D_i - \bW_i^{\top}\widetilde\bmu)\big]}{\En\big[m''_{11}(\bX_i^{\top}\widetilde\btheta, \bY_i)(D_i - \bW_i^{\top}\widetilde\bmu)D_i\big]}.\footnote{In Examples \ref{exa:Expectile} and \ref{exa:PanelCensoredRegressionAndTrimming}, the second derivative $m_{11}''(\bX_i^{\top}\widetilde\btheta,\bY_i)$ may not exist for some observations $i\in[n]$, in which case the estimator $\widehat\beta$ may not be well-defined. To make it well-defined, we replace $m_{11}''(\bX_i^{\top}\widetilde\btheta,\bY_i)$ in \eqref{eq: estimator mu}, \eqref{eq: estimator mu post} and \eqref{eq: debiased estimator betahat} for such observations by zero.}
      \end{equation}
\end{itemize}
\end{alg}

Note that (even without refitting) this procedure gives two estimators
of $\beta_0$: $\widetilde\beta$ on the first step and $\widehat\beta$ on the
third step. As it turns out, the estimator $\widehat\beta$ is better, in the
sense that it can be established as both asymptotically unbiased, $\sqrt n$-consistent
and asymptotically normal. To derive these properties, we impose the following
assumptions.
\begin{assumption}[\textbf{Identifiability}]\label{as: identifiability inference} There exists a constant $c_I\in(0,\infty)$ such that we have
$\E[|m_1'(\bX^{\top}\btheta_0,\bY)(D-\bW^{\top}\bmu_0)|^2]\geqslant c_I$. 
\end{assumption}
\begin{assumption}[\textbf{Integrability}]\label{as: integrability inference}
There are constants $C_M\in(0,\infty)$ and $\widetilde r\in(4,\infty)$ such that
$\max_{1\leqslant j\leqslant p}(\E[|X_{j}|^{\widetilde r}])^{1/\widetilde
r}\leqslant C_M$, $(\E[|D - \bW^{\top}\bmu_0|^{\widetilde r}])^{1/\widetilde
r}\leqslant C_M$, and $(\E[|m_1'(\bX^{\top}\btheta_0,\bY)|^{\widetilde
r}])^{1/\widetilde r}\leqslant C_M$.
\end{assumption}
Assumption \ref{as: identifiability inference} essentially means that there is
non-trivial variation in the variable of interest $D$ after partialling out the
controls $\bW$. In the familiar case of the linear mean model with square loss,
non-trivial variation follows from the usual rank condition for identification
of $\btheta_0$; see Appendix \ref{sec:VerificationLinearModelSquareLoss} for
details.\footnote{More generally, Assumption \ref{as: identifiability inference}
is implied by the eigenvalues of the matrix $\E[U^2\bX \bX^{\top}]$ being bounded
away from zero, which is a non-degeneracy condition.}
Assumption \ref{as: integrability inference} imposes minor regularity
conditions requiring a certain amount of integrability of the random variables
in the model and transformations thereof.

For inference purposes, we also invoke a stronger smoothness condition.
\begin{assumption}[\textbf{Smoothness}]\label{as: smoothness inference} There
are constants $C_m\in(0,\infty)$, $J\in\mathbb N$, and a possibly $\by$-dependent
partition $-\infty=t_{\by,0} < t_{\by,1} \leqslant \dots \leqslant t_{\by,J-1} <
t_{\by,J} = \infty$ of $\mathbb R$ such that for all $\by\in\mathcal Y$, the
function $m(\cdot,\by)$ is continuously differentiable on $\R$ and three-times
differentiable on each $(t_{\by,j-1},t_{\by,j})$, $j\in[J]$, with second and
third derivatives satisfying $|m''_{11}(t,\by)|\leqslant C_m$ and
$|m'''_{111}(t,\by)|\leqslant C_m$. In addition,
$m_{11}''(\bX^{\top}\btheta_0,\bY)$ exists almost surely.
\end{assumption}

This assumption strengthens Assumption \ref{as: post estimator smoothness} from Section
\ref{sec:Bootstrapping-the-Penalty} (which is why we reuse the symbol $C_m$ for the constant). Note
that Assumption \ref{as: smoothness inference} does not hold for the trimmed LAD loss function in
Example \ref{exa:PanelCensoredRegressionAndTrimming}, which means that our inference approach does not apply for this loss function. In addition, Assumption \ref{as: smoothness inference} does not hold for the trimmed LS loss function in the same example whenever $\btheta_0 = \mathbf 0_{p}$. Although we believe it should be possible to perform inference in these cases using methods from \cite{BCK17} developed for the case of a high-dimensional linear quantile regression model, we leave this line of work for the future. In
Appendix \ref{sec: verification}, we verify Assumption \ref{as: smoothness inference} for all other
examples from Section \ref{sec:Examples} including Example
\ref{exa:PanelCensoredRegressionAndTrimming} with the trimmed LS loss function whenever $\btheta_0 \neq \mathbf 0_p$.

The next assumption controls the impact of points of non-smoothness in the loss, if any.
\begin{assumption}[\textbf{Density}]\label{as: conditional density}
Provided $J\geqslant 2$, there is a constant $C_f\in(0,\infty)$ and a non-random
sequence $\overline\Delta_n$ in $(0,\infty)$ such that $\P(\bX^{\top}\btheta_0 -
\overline\Delta_n \leqslant t_{\bY,j}\leqslant \bX^{\top}\btheta_0 +
\overline\Delta_n) \leqslant C_f\overline\Delta_n$ for all
$j\in[J-1]$.
\end{assumption}

Assumption \ref{as: conditional density} holds trivially in Examples \ref{exa:Logit} and
\ref{exa:LogConcaveOrderedResponse}, as for those examples Assumption \ref{as: smoothness inference}
holds with $J=1$. When combined with the requirement that $\overline\Delta_n\to0$ (sufficiently
fast), Assumption \ref{as: conditional density} does impose quite a bit of structure in Examples
\ref{exa:Expectile} and \ref{exa:PanelCensoredRegressionAndTrimming}, however. In Example
\ref{exa:Expectile}, this assumption is satisfied if the conditional distribution of $Y$ given $\bX$
is absolutely continuous with bounded PDF. In Example \ref{exa:PanelCensoredRegressionAndTrimming}
with the trimmed LS loss function, it is satisfied if the conditional distribution of $Y_1$ given
$(\bX,Y_1>0)$, the conditional distribution of $Y_2$ given $(\bX,Y_2>0)$, and the (unconditional)
distribution of $\bX^{\top}\btheta_0$ are all absolutely continuous with bounded (uniformly over
$n$) PDFs; see Appendix \ref{sec: verification} for details.\footnote{Note here that the requirement
that the distribution of $\bX^{\top}\btheta_0$ is absolutely continuous with bounded PDF implies
that $\btheta_0$ is sufficiently well separated from $\mathbf 0_p$.}

\begin{assumption}[\textbf{Convergence Rates}]\label{as: convergence rates inference} 
There is a non-random sequence $a_n$ in $(0,\infty)$ such that
$a_n\to 0$ and $\|\widetilde\btheta - \btheta_0\|_1 + \|\widetilde\bmu -
\bmu_0\|_1 \lesssim_{\P} a_n$.
\end{assumption}
Assumption \ref{as: convergence rates inference} is a high-level assumption placed on the estimators
from Steps 1 and 2 in Algorithm \ref{alg:ThreeStepDebiasing}. When
$\widetilde\btheta$ is $\ell_1$-ME or post-$\ell_1$-ME based on BCV, we can lean on the
bounds from Theorem \ref{thm: convergence rates post estimator}. For the estimation error of
$\widetilde\bmu$, however, we cannot use Theorem \ref{thm: convergence rates post estimator}, as
this estimator does not fit into our framework because of the presence of estimated weights in the
optimization problems \eqref{eq: estimator mu} and \eqref{eq: estimator mu post}. However, these
optimization problems correspond to LASSO and post-LASSO with estimated weights, and such estimators
are well studied in the literature. See e.g.~\cite{belloni2016post}, where one can find the
appropriate rates for the estimation error of $\widetilde\bmu$ in terms of the sparsity of $\bmu_0$.

We next present a theorem on the asymptotic distribution of the debiased estimator $\widehat\beta$.
\begin{thm}[\textbf{Asymptotic Distribution}]\label{thm: asymptotic distribution} Let Assumptions
\ref{assu:ParameterSpace}--\ref{assu:Approximate-Sparsity} and \ref{as: identifiability
inference}--\ref{as: convergence rates inference} hold, and suppose that $\sqrt n a_n^2\to0$,
$a_n(n^{1/r}B_n + \sqrt{\ln(pn)})\to 0$, and $B_n^2\ln(pn) = o(n^{1-4/(r\wedge\widetilde r)})$.
If $J\geqslant 2$, suppose also that $(1+\sqrt n B_na_n)(\overline\Delta_n^{1/2} +
(B_na_n/\overline\Delta_n)^{r/2})\to0$. Then
\begin{equation}\label{eq: asy normality}
\frac{\sqrt n(\widehat\beta - \beta_0)}{\sigma_{0}}\overset{D}\to \mathrm{N}(0,1),\quad\text{where}\quad \sigma_{0}^2:=\frac{\E\big[(m'_1(\bX^{\top}\btheta_0,\bY)(D-\bW^{\top}\bmu_0))^2\big]}{\big(\E[m''_{11}(\bX^{\top}\btheta_0,\bY)(D-\bW^{\top}\bmu_0)D]\big)^2}.
\end{equation}
\end{thm}

This theorem shows that the estimator $\widehat\beta$ is asymptotically unbiased and normal under
plausible regularity conditions.  The ``asymptotic'' variance $\sigma_0^2$ appearing in this theorem, which
depends on $n$ in general via the distribution $P$ of $(\bX,\bY)$, is easily estimable. For example,
one can use a plug-in estimator
\begin{equation}\label{eq: sigma hat 1}
\widehat{\sigma}^2:=\frac{\En\big[(m'_1(D_i\widetilde\beta + \bW_i^{\top}\widetilde\bgamma, \bY_i)(D_i - \bW_i^{\top}\widetilde\bmu))^2\big]}{\big(\En\big[m''_{11}(D_i\widetilde\beta + \bW_i^{\top}\widetilde\bgamma, \bY_i)(D_i - \bW_i^{\top}\widetilde\bmu)D_i\big]\big)^2}
\end{equation}
with the estimators $\widetilde\beta,\widetilde\bgamma$ and $\widetilde\bmu$ stemming from Steps 1
and 2 of Algorithm \ref{alg:ThreeStepDebiasing} (possibly with refitting) using BCV as the
penalty rule in both steps. Alternatively, one can incorporate Step 3 of the same algorithm and use
\begin{equation}\label{eq: sigma hat 2}
\widehat{\sigma}^2:=\frac{\En\big[(m'_1(D_i\widehat\beta + \bW_i^{\top}\widetilde\bgamma, \bY_i)(D_i - \bW_i^{\top}\widetilde\bmu))^2\big]}{\big(\En\big[m''_{11}(D_i\widehat\beta + \bW_i^{\top}\widetilde\bgamma, \bY_i)(D_i - \bW_i^{\top}\widetilde\bmu)D_i\big]\big)^2}.\footnote{For both variance estimators \eqref{eq: sigma hat 1} and \eqref{eq: sigma hat 2}, in case the second derivative $m_{11}''(\bX_i^{\top}\btheta,\bY_i)$ fails to exist at $\btheta=(\widetilde\beta,\widetilde\bgamma^\top)^\top$ or $\btheta=(\widehat\beta,\widetilde\bgamma^\top)^\top$ and some $i\in[n]$, we replace those second derivatives by zero.}
\end{equation}
It is rather standard to derive consistency of these estimators. Also, because of asymptotic
normality of $\widehat\beta$, it is then straightforward to perform inference on $\beta_0$. For
example, an asymptotically valid $(1-\alpha)\times100\%$ confidence interval for $\beta_0$ takes the
standard form $[\widehat\beta - z_{\alpha/2}\widehat\sigma/\sqrt n,\widehat\beta +
z_{\alpha/2}\widehat\sigma/\sqrt n]$, where $\widehat\sigma$ is given either by \eqref{eq: sigma hat
1} or by \eqref{eq: sigma hat 2}, and $z_{\alpha/2}$ is the $(1-\alpha/2)$-quantile of the standard
normal distribution.

\begin{rem}[{\textbf{Relation to Literature}}]
As discussed in the beginning of this section, our approach to inference in this
section closely follows the developments in the literature. In particular, our
estimator $\widehat\beta$ is essentially the same as that proposed in
\cite{van_de_geer_high-dimensional_2014}, the only difference being that we
allow refitting in the optional parts of Steps 1 and 2 of Algorithm
\ref{alg:ThreeStepDebiasing}. As we will see in the next section, this refitting
can substantially improve inference, in terms of size control, even in
approximately sparse models. More importantly, however, is that Theorem
\ref{thm: asymptotic distribution} is different from the corresponding theorem
in \cite{van_de_geer_high-dimensional_2014}, as we tune the assumptions of our
theorem toward the examples from Section \ref{sec:Examples}. Specifically, we do
not require the function $m(\cdot,\by)$ to be strictly convex (see \textit{ibid.},
p.~1179) or for it to be everywhere twice differentiable with a
Lipschitz-continuous second derivative (see \textit{ibid.}, Condition (C1)). No
matter the choice of ``trimmer'' $\Xi$, the trimmed loss \eqref{eq:TrimmedLoss} in
Example \ref{exa:PanelCensoredRegressionAndTrimming} has linear pieces and is
therefore not strictly convex. Moreover, neither the asymmetric LS (Example
\ref{exa:Expectile} with $\tau\neq1/2$) nor trimmed LS loss functions have
Lipschitz-continuous second derivatives. We also provide a detailed verification
of our assumptions for all Section \ref{sec:Examples} examples in Appendix
\ref{sec: verification}. Related approaches to debiasing of high-dimensional
estimators were also proposed in \cite{javanmard_confidence_2013} in a
likelihood framework and in \cite{belloni2016post} for generalized linear
models.\qed
\end{rem}

\section{Simulations}\label{sec:Simulations}

In this section we investigate the finite-sample behavior of our estimators
based on the bootstrap-after-cross-validation (BCV) method for obtaining penalty
levels proposed in Section \ref{sec:Bootstrapping-the-Penalty}. We also compare
our estimation and inference methods to ($K$-fold) cross-validation, which lacks general
theoretical justification but is a popular method in practice.

\subsection{Simulation Design}

We consider a master data-generating process (DGP) of the form
\[
Y_{i}=\mathbf{1}\left(\beta_{0}D_{i}+\sum_{j=1}^{p-1}\gamma_{0j}W_{i,j}+\varepsilon_{i}>0\right),\quad\varepsilon_{i}\mid D_{i},\bW_{i}\sim\mathrm{N}\left(0,1\right),\quad i\in\left[n\right],
\]
thus implying a binary probit model as in Example \ref{exa:Logit}. The
regressors $\bX=(D,\bW^{\top})^{\top}$ are distributed jointly centered Gaussian
$\bX\sim\mathrm{N}(\boldsymbol{0},\bSigma(\rho))$ with covariances (and
correlations)
\[
\Sigma_{j,k}(\rho):=\mathrm{cov}\left(X_j,X_k\right)=\mathrm{E}\left[X_jX_k\right]=\rho^{\left|j-k\right|},\quad\left(j,k\right)\in\left[p\right]^2.
\]
Hence, the regressor covariance matrix $\bSigma\left(\rho\right)$ takes a
Toeplitz form with the overall correlation level being dictated by $\rho$. We
allow $\rho\in\left\{ 0,.2,\dotsc,.8\right\} $, thus running the gamut of
(positive) correlation levels. Since $\varepsilon_{i}$'s are standard normal,
the ``noise'' $\mathrm{var}\left(\varepsilon\right)$ in our DGP is fixed at one.
Hence, the signal-to-noise ratio (SNR) equals the ``signal,''
\[
\mathrm{SNR}:=\frac{\mathrm{var}(\bX^{\top}\theta_{0})}{\mathrm{var}(\varepsilon)}
   =\btheta_{0}^{\top}\bSigma\left(\rho\right)\btheta_{0}.
\]
which depends on both the correlation level and coefficient pattern. We consider
the patterns:
\begin{align*}
\mathrm{Pattern\,1}:\quad & \btheta_{0}=\left(1,1,0,\dotsc,0\right)^{\top},\tag{Exactly Sparse}\\
\mathrm{Pattern\,2}:\quad & \theta_{0,j}=(1/\sqrt{2})^{j-1}\mathbf{1}\left(j\leqslant5\right),\quad j\in\left[p\right],\tag{Intermediate}\\
\mathrm{Pattern\,3}:\quad & \theta_{0,j}=(1/\sqrt{2})^{j-1},\quad j\in\left[p\right].\tag{Approximately Sparse}
\end{align*}
The \emph{exactly sparse} pattern has only non-zero coefficients for the first
couple of regressors $(s_{0}=2)$, and both non-zeros are clearly separated from
zero, thus allowing perfect variable selection. The implied signals (hence SNRs)
are
\begin{equation}
\mathrm{var}\left(\bX^{\top}\btheta_{0}\right)=2\left(1+\rho\right)\in\left\{ 2,2.4,2.8,3.2,3.8\right\} .\label{eq:SimulationsSignal}
\end{equation}
Compared to existing simulation studies for high-dimensional binary response
models, the SNRs considered here are relatively low.\footnote{For example, the
binary logit designs in \citet[Section 5.2]{friedman_regularization_2010} and
\citet[Section 5]{ng_feature_2004} imply SNRs of three and over 30,
respectively.} 

Note that the SNR is increasing with the regressor correlation, such that
sampling from a high-$\rho$ DGP tends to produce an easier estimation problem
compared to sampling from a low-$\rho$ DGP, keeping all other things equal. When
reporting results below for $\rho=0$ (our baseline), we are thus considering the
\emph{worst} correlation scenario.\footnote{The same comments apply to the other
coefficient patterns albeit with the more complicated signal
\[
\mathrm{var}\left(\bX^{\top}\btheta_{0}\right)=\sum_{j=1}^{p}\theta_{0,j}^{2}
   +2\sum_{j=1}^{p-1}\sum_{k=j+1}^{p}\theta_{0,j}\theta_{0,k}\rho^{k-j}.
\]
} 

In contrast to the exactly sparse pattern, the \emph{approximately sparse} 
pattern involves all non-zeros $(s_{0}=p)$, which are not bounded away from zero,
such that variable selection mistakes are bound to happen. To see that this
pattern is in fact approximately sparse, note that for every $q\in(0,1]$ one has
$
\sum_{j=1}^{p}\left|\theta_{0,j}\right|^{q}\leqslant\sum_{j=1}^{\infty}\left|\theta_{0,j}\right|^{q}=1/(1-2^{-q/2}).
$
Hence, for the purpose of Assumption \ref{assu:Approximate-Sparsity}, we can
choose $q\in(0,1]$ freely and pair it with $s_q=1/(1-2^{-q/2})$.
The base $1/\sqrt{2}$ of the approximately sparse pattern was here chosen to
(approximately) equate the signals arising from the approximately and exactly
sparse coefficient patterns in the baseline case of uncorrelated regressors
$(\rho=0)$, which amounts to $\|\btheta_{0}\|_{2}^{2}$. The relevance of a
regressor, as measured by its coefficient, is rapidly decaying in the regressor
index $j$, such that the vast majority of the signal is captured by a small
fraction of the regressors. For example, in the baseline case of uncorrelated
regressors $(\rho=0)$, the first 10 regressors account for 99.9 percent of the
signal (two). 

In between these two extremes lies the \emph{intermediate} pattern. This pattern
was created by cutting off the approximately sparse coefficient sequence at the
smallest regressor index $j^{\ast}$, such that regressors $[j^{\ast}]$ account
for at least 95 percent of the baseline signal. (Here: $j^{\ast}=5$.) For this
pattern, perfect variable selection is possible but unlikely.

We consider sample sizes $n\in\left\{ 100,200,400\right\} $ and limit attention
to the high-dimensional regime by fixing $p=n$ throughout.

\begin{rem}[\textbf{Sparsity of Debiasing Coefficient Vector}]
One may wonder whether the above patterns for the structural coefficients
$\btheta_{0}=(\beta_{0},\bgamma_{0}^{\top})^{\top}$ agree or conflict with sparsity of the
non-primitive debiasing coefficient vector $\bmu_{0}$ in any sense of the word. While a thorough
investigation of this question is beyond the scope of this paper, we can provide some insights for
our concrete DGPs. Specifically, in Appendix \ref{sec:Sparsity-of-Debiasing-Vector}, we show
that---in our collection of DGPs---the number of non-zeros in $\bmu_{0}$ is bounded by the number of
non-zeros in $\bgamma_{0}$, $\|\bmu_{0}\|_{0}\leqslant\|\bgamma_{0}\|_{0}$, i.e.~the number of
relevant controls. Hence, when $\bgamma_0$ is (exactly) sparse, so is $\bmu_0$. Moreover, we show
via simulation that when $\bgamma_{0}$ is only approximately sparse, the sorted absolute values of
the elements of $\bmu_{0}$ are rapidly decaying and approaching zero, cf.~Figure
\ref{fig:SortedAbsoluteDebiasingVector}. Such a decay is in line with the notion of approximate
sparsity.\qed
\end{rem}

\subsection{Estimation and Implementation}

We consider the following four estimators arising from $\ell_1$-ME
\eqref{eq:ell1PenalizedMEstimationIntro} and post-$\ell_1$-ME \eqref{eq:AllPostEstimators} based on
either the CV or BCV penalty levels in \eqref{eq:LambdaCVDefn} and \eqref{eq:BootstrapPenaltyLevel
CV}, respectively:
\begin{itemize}
\item $\ell_1$-ME based on bootstrapping after cross-validation (``BCV''),
\item post-$\ell_1$-ME based on bootstrapping after cross-validation (``post-BCV''),
\item $\ell_1$-ME based on cross-validation (``CV''), and
\item post-$\ell_1$-ME based on cross-validation (``post-CV'').
\end{itemize}
When discussing normal approximations based on three-step debiasing (Algorithm
\ref{alg:ThreeStepDebiasing}), we use the same method in both Steps 1 and 2. For
example, the ``post-BCV'' inference procedure refers to post-BCV in the first
step, followed by post-BCV in the second step (i.e.~both optional steps are
taken).

Our BCV and post-BCV estimation methods require us to specify a score markup
$c_{0}\in(1,\infty)$ and probability tolerance rule $\alpha=\alpha_{n}$. We here
follow the recommendation in \citet[p.~2380]{belloni_sparse_2012} for the LASSO
and post-LASSO and take $c_{0}=1.1$ and $\alpha_{n}=.1/\ln(p\lor n)$
as our benchmark. The latter function, slowly decaying in $p\lor n$, leads to
$\alpha\approx2.2\%,1.9\%$ and $1.7\%$ for $n=100,200$ and $400$, respectively.
We also look at the alternative score markups $\{1,1.05\} $, the
first one being excluded by the theory in Section
\ref{sec:Bootstrapping-the-Penalty}. The alternative probability tolerance rule
$\alpha_{n}=10/n$ leads to qualitatively identical conclusions, cf.~Appendix
\ref{sec:Additional-Simulation-Results}. We stress that the benchmark choices of
$c_0$ and $\alpha_n$ are only rules of thumb that tend to perform well in the
simulation designs considered here. Other choices of score markups and
probability tolerance rules may have better properties in other DGPs.

We have previously treated all coefficients in the same manner, in that they are
all penalized and with equal weight. However, in an empirical application one is
typically confident that an intercept belongs in the model. For this reason, the
(intercept) coefficient on the constant regressor is usually not penalized
during estimation. Moreover, to justify equal penalty weighting, prior to
estimation one typically brings the (non-constant) regressors onto the same scale
by dividing them by their respective sample standard deviations. To align our
simulation study with these empirical practices, we include unpenalized
intercepts in both Steps 1 and 2 of Algorithm \ref{alg:ThreeStepDebiasing} and
rescale regressors. (The intercepts are still suppressed in our notation.) That
is, we treat neither the zero (true) intercept nor equivariant regressors as
information known to the researcher. In these aspects our simulations are
therefore empirically calibrated.

For each sample size $n(=p)$, each correlation level $\rho$, and each
coefficient pattern, we use 2,000 independent simulation draws and 1,000
independent standard Gaussian bootstrap draws per simulation draw and per
estimation step (when relevant). We assign observations to $K$ approximately
equally large folds $\{ I_{k}\} _{k=1}^{K}$ for both the first and
second steps, shuffling the assignments in between. We keep $K=3$ throughout
and use the same folds for all estimators to facilitate
comparison.\footnote{Three folds is the minimum value allowed by
\texttt{cv.glmnet}. Preliminary and unreported simulation experiments suggest
that using 5-fold (instead of 3-fold) CV only affects the average errors
reported below at the third decimal. Similarly, using 2,000 Gaussian bootstraps
(instead of 1,000) appears to only affect these averages at the fourth decimal.} 

All simulations are carried out in \texttt{R} with cross-validation done using
\texttt{glmnet::cv.glmnet}, and refitting done using
\texttt{stats::glm}.\footnote{We use \texttt{R} version 4.2.2 and \texttt{glmnet
}version 4.1-6.} When constructing the candidate penalty set $\Lambda_{n}$, we
use the \texttt{glmnet} default settings, which creates a log-scale equi-distant
grid of a 100 candidate penalties from the threshold penalty level to
essentially zero. The threshold is the (approximately) smallest level of
penalization needed to set every coefficient to zero, thus resulting in a
trivial (null) model.\footnote{Log-scale equi-distance from a ``large''
candidate value to essentially zero fits well with the form of $\Lambda_{n}$ in
our Assumption \ref{assu:CandidatePenalties} (interpreting
$c_{\Lambda}/n\approx0$). However, the threshold penalty is a function of the
data and, thus, random. The resulting candidate penalty set used in our
simulations is therefore also random, and thus, strictly speaking, not allowed
by Assumption \ref{assu:CandidatePenalties}. Moreover, the number of candidate
values $|\Lambda_{n}|$ is here held fixed. We believe these deviations from our
theory to be only a minor issue.} 

Note that \texttt{cv.glmnet} calculates and stores the out-of-fold linear forms
$\bX_{i}^{\top}\widehat{\btheta}_{I_{k}^{c}}(\lambda)$ (with an intercept, if
relevant) for each $i\in I_{k}$, fold $k$ and candidate penalty $\lambda$, and
allows for extraction of estimates for penalty levels off the regularization
grid via linear interpolation. Hence, compared to CV, there is essentially zero
added computational burden associated with using BCV.

\subsection{Results}

\subsubsection{Non-Existence and Treatment of Missing Values}

While the $\ell_{1}$-penalized probit estimators BCV and CV always exist (cf.~Section
\ref{sec:ExistenceSparsityAndUniqueness}), refitting after variable selection based on either of
these estimator can fail. For example, in our binary response setting, without any penalty one may
encounter complete separation of the outcomes based on the fitted probabilities, in which case the
refitted estimates fail to exist (as real numbers).\footnote{Strictly speaking, the
$\ell_{1}$-penalized probit estimator fails to exist when all outcomes are of the same label and
some coefficient (here: the intercept) goes unpenalized. In none of our simulated datasets did we
encounter all zeros or all ones. See Appendix \ref{sec:ExistenceSparsityAndUniqueness} and, in
particular, Remark \ref{rem:UnpenalizedCoefficients} for more discussion.}

Across all simulation designs and draws, refitting after CV fails in nearly 15\%
of all cases. The fraction of such non-existent post-CV cases varies with the DGP
and can be higher than 47\%. Since post-CV estimation and debiasing procedures
do not appear well-defined in our context, we drop them from further
consideration. 

In contrast, refitting after BCV fails to converge in only about 0.01\% of all
cases.\footnote{Specifically, convergence fails in 74 out of a total of 540,000
cases, where the total equals the product of the numbers of simulation draws
(2,000), correlation levels (5), sample/problem sizes (3), coefficient patterns
(3), score markups (3) and probability tolerance rules (2).} Since we find this
fraction miniscule, when reporting results below we choose to simply omit the
problematic cases from the relevant post-BCV statistics; see also the figure
notes.

\subsubsection{Estimation Error}

Figure \ref{fig:MeanEll2ErrorEstimatorComparisonBCCHTuning} shows the mean
$\ell_{2}$ estimation error (for the slope coefficients, averaging over the
2,000 simulation draws) arising from BCV, post-BCV and CV, respectively,
using benchmark tuning.
\begin{figure}
\caption{Mean $\ell_{2}$ Estimation Error by Method with $c_{0}=1.1$ and
$\alpha_{n}=.1/\ln\left(p\lor
n\right)$\label{fig:MeanEll2ErrorEstimatorComparisonBCCHTuning}}

\centering{}\includegraphics[viewport=5bp 5bp 463bp
416bp,clip,width=0.7\textwidth]{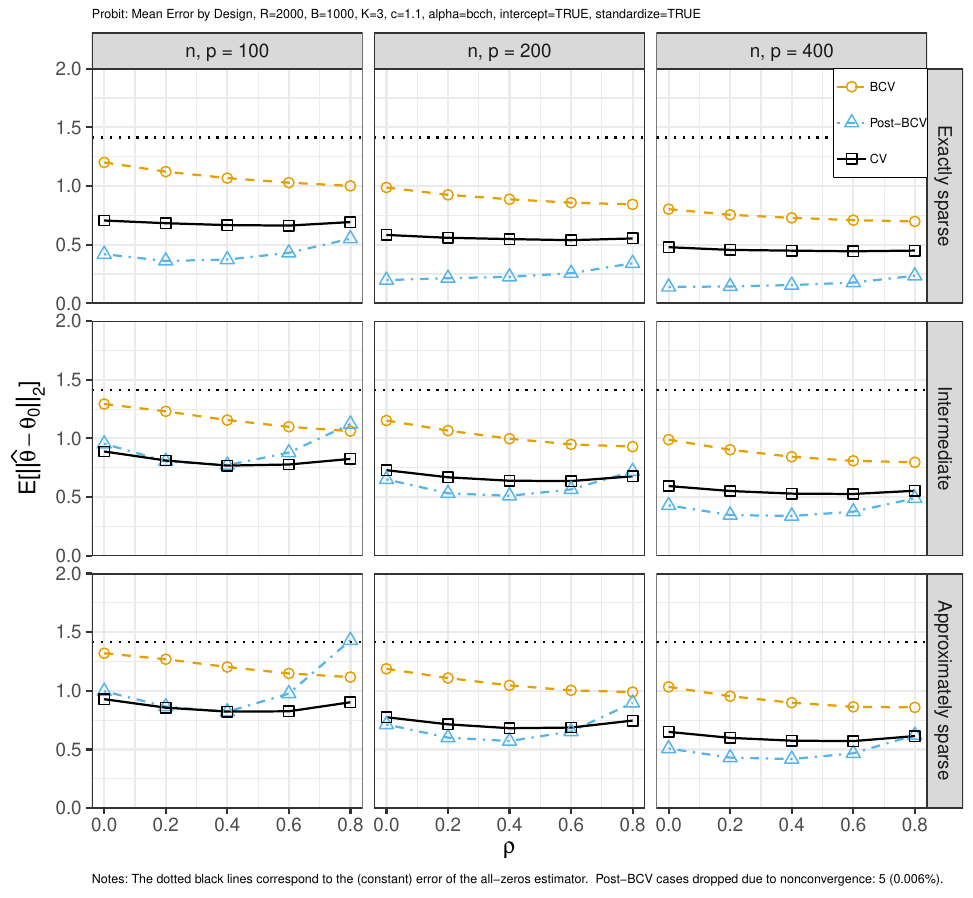}
\end{figure}
Mean $\ell_{2}$ estimation error is here depicted as a function of the
sample/problem size (the tile column), coefficient pattern (the tile row), and
correlation level (the horizontal axis in each tile). The horizontal line at
$\|\btheta_{0}\|_{2}$ facilitates comparison with the trivial ``estimator''
$\widehat{\btheta}\equiv\boldsymbol{0}_p$.

One observation evident from this figure is that the error curves of BCV,
post-BCV and CV can cross. Hence, these estimators cannot be ranked in terms of
mean $\ell_{2}$ estimation error, in general. However, for the largest
sample/problem size considered, post-BCV outperforms CV for small to medium
levels of correlation, and CV outperforms BCV.\footnote{We reach qualitatively
identical conclusions from inspecting the \emph{median }$\ell_{2}$ estimation
errors. Hence, these findings are not limited to one particular feature of the
error distributions. We omit the corresponding median plots due to their
similarity with the mean error plots. Figures are available upon request.}

Increasing the sample size (moving from left to right) leads to a downward shift
in mean estimation error for all three estimators, which is indicative of
convergence. Convergence appears to take place no matter the coefficient pattern
or regressor correlation level even though the number of candidate regressors
matches the sample size. Increasing the number of non-zeros in $\btheta_{0}$
(moving from top to bottom) leads to an upward shift in mean estimation error.
This finding is consistent with convergence slowing down with $q$ and $s_q$
as predicted by Theorems \ref{cor: convergence rate bootstrap after cv} and \ref{thm: convergence rates post estimator}.

We next investigate the impact of the choice of score markup $c_{0}$. Figures
\ref{fig:MeanEll2ErrorBCVVaryingc0BCCHrule} and
\ref{fig:MeanEll2ErrorPostBCVVaryingc0BCCHrule} plot the mean $\ell_{2}$
estimation error for $c_{0}=1,1.05$ and (the previously used) $1.1$, each
sample/problem size and coefficient pattern, and for the BCV and post-BCV
estimators, respectively.
\begin{figure}[p]
\caption{Mean $\ell_{2}$ BCV Estimation Error by Score Markup with
$\alpha_{n}=.1/\ln\left(p\lor
n\right)$\label{fig:MeanEll2ErrorBCVVaryingc0BCCHrule}}

\centering{}\includegraphics[viewport=5bp 5bp 463bp
416bp,clip,width=0.7\textwidth]{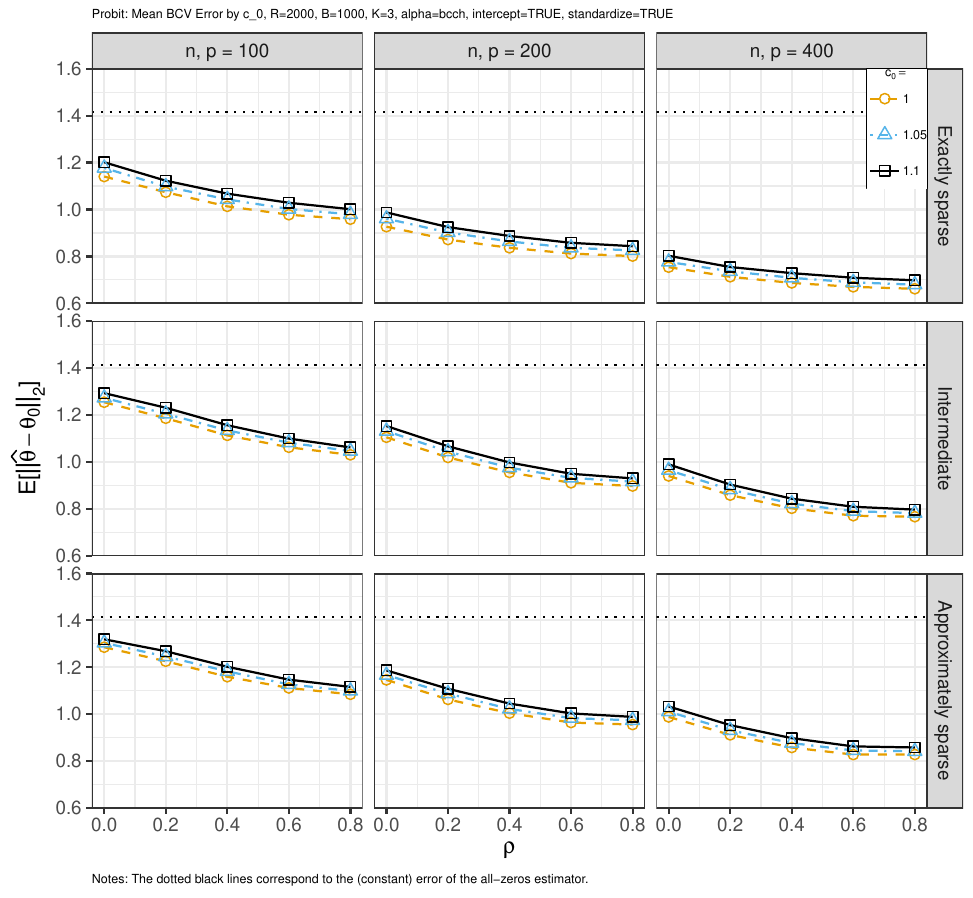}
\end{figure}
\begin{figure}
\caption{Mean $\ell_{2}$ Post-BCV Estimation Error by Score Markup with
$\alpha_{n}=.1/\ln\left(p\lor
n\right)$\label{fig:MeanEll2ErrorPostBCVVaryingc0BCCHrule}}

\centering{}\includegraphics[viewport=5bp 5bp 463bp
416bp,clip,width=0.7\textwidth]{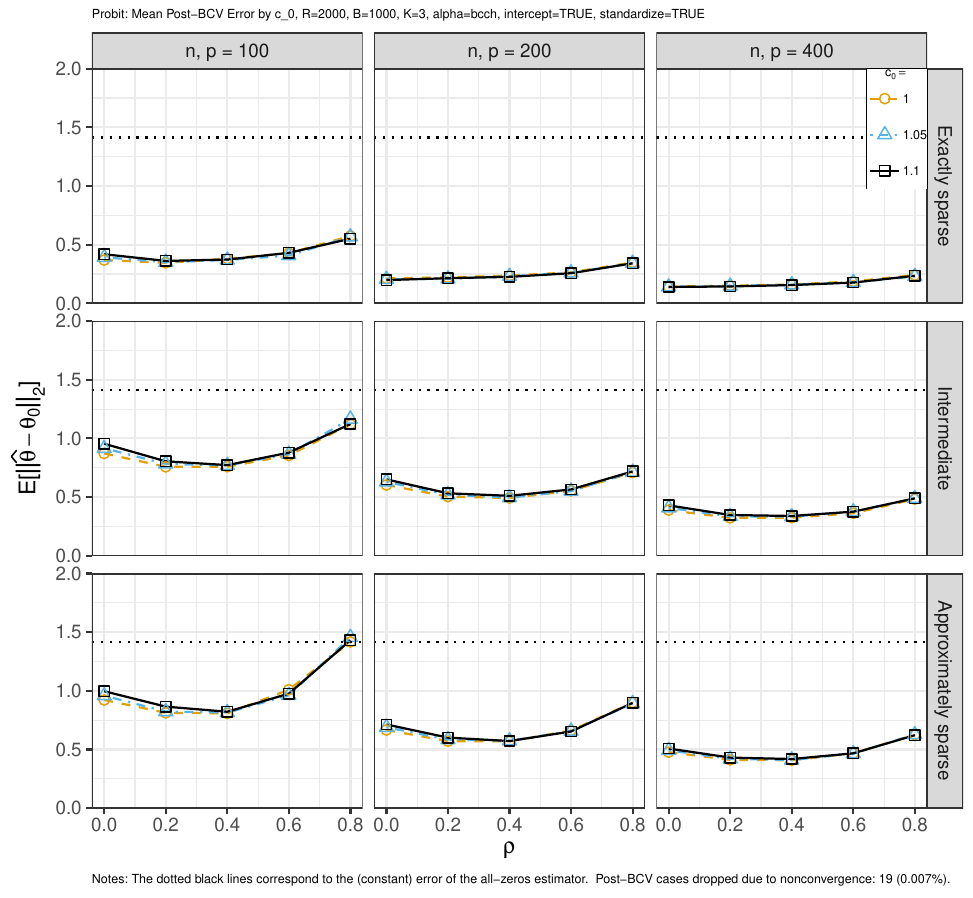}
\end{figure}
Figure \ref{fig:MeanEll2ErrorBCVVaryingc0BCCHrule} suggests that increasing $c_{0}$ away from one
slightly worsens (mean estimation error) performance of BCV. While our theory takes $c_{0}$ strictly
greater than one, any value near one---including the limit case of one itself--- appears to lead to
near identical results.\footnote{That mean BCV error is downward sloping for small to moderate
$\rho$ levels is due to the signal being increasing in $\rho$ and need not translate to other
correlation or coefficient patterns.} For post-BCV (Figure
\ref{fig:MeanEll2ErrorPostBCVVaryingc0BCCHrule}), the findings are similar. In fact, at least for
the largest sample/problem size, the exact value of $c_{0}\in\left\{ 1,1.05,1.1\right\} $ has little
to no impact on mean error. Note that our findings for post-BCV apply even with the approximately
sparse coefficient pattern, where variable selection mistakes are bound to occur. 

To conclude this subsection, we note that it is a well-known puzzle in the LASSO literature that the
theory typically requires that $c_0$ is strictly bigger than one, with the estimation error bounds
deteriorating as $c_0$ approaches one, while simulation experience suggests that the estimation
errors are insensitive with respect to $c_0$ when $c_0$ is close to one. We believe that solving
this puzzle remains one of the key challenges in this literature.

\subsubsection{Normal Approximation}

We next assess the normal approximations resulting from three-step debiasing
(Algorithm \ref{alg:ThreeStepDebiasing}) using either BCV, post-BCV or CV.
Instead of looking at the \emph{standardized} estimate
$\sqrt{n}(\widehat{\beta}-\beta_{0})/\sigma_{0}$ for the true asymptotic
variance $\sigma_0^2$ given in \eqref{eq: asy normality},
we form an estimate $\widehat{\sigma}^{2}$ and consider the \emph{studentized
}estimate $\sqrt{n}(\widehat{\beta}-\beta_{0})/\widehat{\sigma}$. That is, we
take into account the unknown nature of the $\sigma_{0}^{2}$, as required in an
empirical application. 

To construct the estimate $\widehat{\sigma}^{2}$, we first leverage the binary
response model to establish the (conditional information) equality
$
   \mathrm{E}[m_{1}'\left(\bX^{\top}\btheta_{0},\bY\right)^{2}\mid \bX]=\mathrm{E}[m_{11}''\left(\bX^{\top}\btheta_{0},\bY\right)\mid \bX].
$
We then use the definition of $\bmu_{0}$ to establish the (weighted projection)
equality
\[
   \mathrm{E}\big[m_{11}''\left(\bX^{\top}\btheta_{0},\bY\right)\left(D-\bW^{\top}\bmu_{0}\right)^2\big]=\mathrm{E}\left[m_{11}''\left(\bX^{\top}\btheta_{0},\bY\right)\left(D-\bW^{\top}\bmu_{0}\right)D\right].
\]
Again using the binary response model, we can evaluate
\begin{align*}
\E\left[ m_{11}''\left(\bX^{\top}\btheta_{0},\bY\right) \mid \bX^\top\btheta_0 = t \right]
   = \frac{f\left( t \right)^2 }{F\left( t \right)\left( 1 - F \left( t \right) \right)}
   =: \omega_{F}\left(t\right),
\end{align*}
where $f$ and $F$ denote the PDF and CDF, respectively, associated with the
binary response model.\footnote{In our current binary probit setting, these
functions are the standard normal PDF and CDF, respectively. In the empirical
application in Section \ref{sec:Application}, we also use the logistic
distribution, leading to the binary logit.} This
allows us to simplify the expression for $\sigma_{0}^{2}$ to
\[
\sigma_{0}^{2}=1\big\slash\mathrm{E}\left[\omega_{F}\left(\bX^{\top}\btheta_{0}\right)\left(D-\bW^{\top}\bmu_{0}\right)D\right]
\]
and leads to the following estimator of $\sigma_0^2$:
\[
\widehat{\sigma}^{2}:=1\big\slash\En\big[\omega_{F}\big(\widehat{\beta}D_{i}+\bW_{i}^{\top}\widetilde{\bgamma}\big)\big(D_{i}-\bW_{i}^{\top}\widetilde{\bmu}\big)D_{i}\big],
\]
with $\widetilde{\bgamma},\widetilde{\bmu}$ and $\widehat{\beta}$ given by Steps
1, 2 and 3, respectively, of Algorithm \ref{alg:ThreeStepDebiasing} given
different rules for choosing the penalties $\lambda_1$ and
$\lambda_2$.\footnote{Alternatively, one can use the
``sandwich'' estimators \eqref{eq: sigma hat 1} and \eqref{eq: sigma hat 2}.
Experimenting with these estimators, we obtained numerically similar results as
reported below for the estimator $\widehat{\sigma}^2$. We prefer
$\widehat{\sigma}^2$ since it leverages both the binary response and projection
structure.}

Figure \ref{fig:NormalApproxBCVPostBCVCVVaryingSampleSizeBCCHtuningZeroCorr}
shows the (kernel) densities of the studentized estimates using
benchmark tuning and $\rho=0$.\footnote{All kernel densities are created using
the \texttt{R} package \texttt{ggplot2} with \texttt{geom\_density}. In
expectation of an approximately normal distribution, we use a Gaussian kernel
and the \citet[Equation (3.31)]{silverman_density_1986} rule-of-thumb bandwidth
(both \texttt{geom\_density }defaults).}
\begin{figure}
\caption{{\small{}Densities of Studentized Estimates by $n(=p)$ with $\rho=0$,
$c_{0}=1.1$ and $\alpha_{n}=.1/\ln\left(p\lor
n\right)$.}\label{fig:NormalApproxBCVPostBCVCVVaryingSampleSizeBCCHtuningZeroCorr}}

\centering{}\includegraphics[viewport=5bp 5bp 463bp
416bp,clip,width=0.68\textwidth]{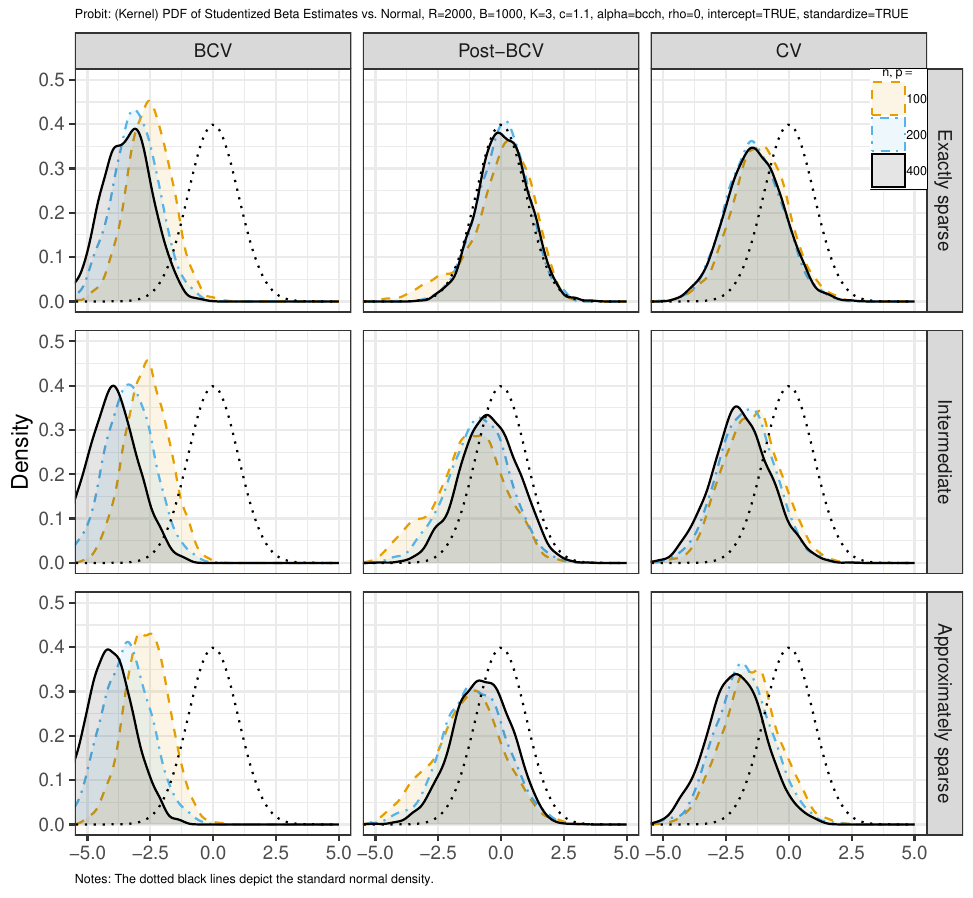}
\end{figure}
The densities arising from BCV, post-BCV and CV, respectively, are here depicted
as columns of tiles, where each tile row corresponds to a coefficient pattern
and each graph within a tile a sample/problem size. Starting with the exactly
sparse coefficient pattern (the top row), we see that both BCV and CV lead to
considerable shrinkage bias even after debiasing the initial estimate of the
focal parameter $\beta_{0}$. This feature is seen from the leftward shifts in
the resulting densities compared to the standard normal density, here
represented by the dotted line. These biases do not seem to disappear as $n$ increases,
holding $n=p$. If anything, these distributions shift further left, 
which indicates that BCV requires a larger sample size.
In constrast, the post-BCV density essentially collapses to the standard normal
one, at least for $n=p=200$ and $400$.

As the coefficient pattern becomes less and less sparse (moving down),
all approximations deteriorate, as is to be expected. While imperfect, the
post-BCV densities are still decent approximations to the normal for both the
intermediate and approximately sparse coefficient patterns. Moreover, only these
densities appear to approach the standard normal as the sample/problem size
increases.

While Figure
\ref{fig:NormalApproxBCVPostBCVCVVaryingSampleSizeBCCHtuningZeroCorr} depicts
the normal approximations for the worst-correlation case $\rho=0$, in Figure
\ref{fig:NormalApproxPostBCVCVBCCHtuningVaryingCorr} we display the normal
approximations as a function of $\rho$.
\begin{figure}[!htb]
\caption{{\footnotesize{}Densities of Studentized Post-BCV and CV Estimates for
Different $\rho$ with $n(=p)=400,$ Approximately Sparse Coefficient Pattern, $c_{0}=1.1$ and
$\alpha_{n}=.1/\ln\left(p\lor
n\right)$}\label{fig:NormalApproxPostBCVCVBCCHtuningVaryingCorr}}

\centering{}\includegraphics[viewport=5bp 5bp 463bp
416bp,clip,width=0.68\textwidth]{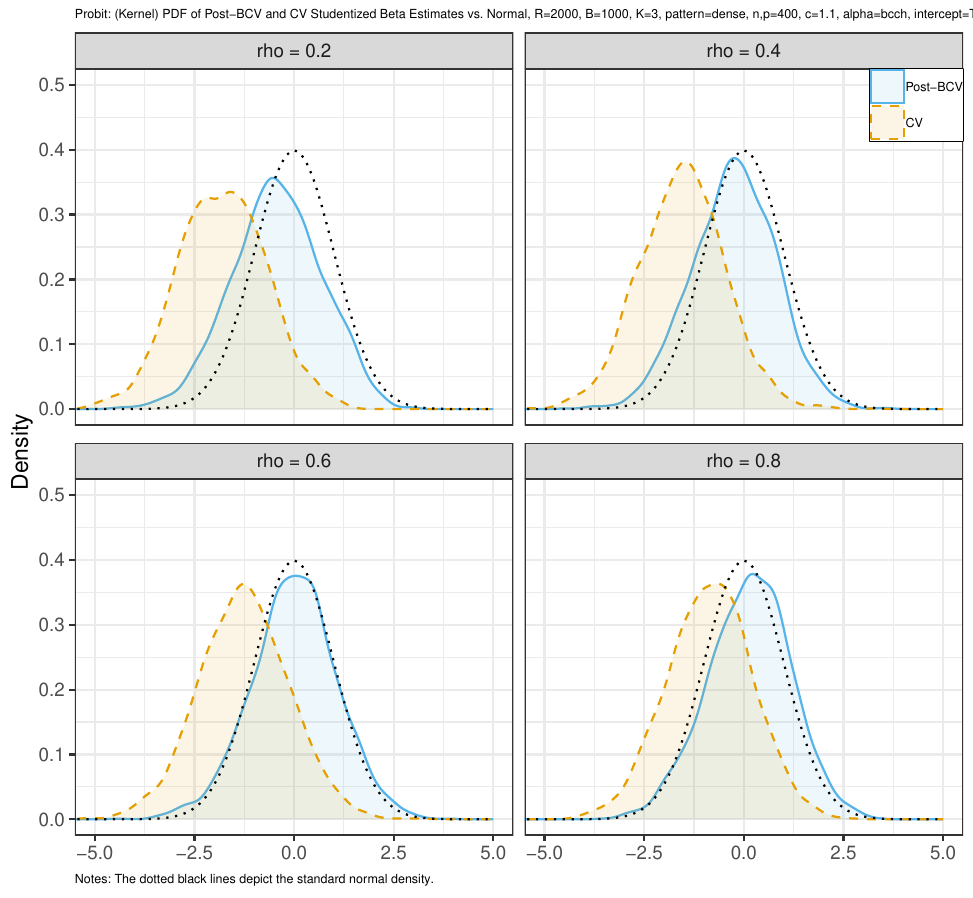}
\end{figure}
We here focus on the largest sample size $n(=p)=400$ and the (more challenging)
approximately sparse coefficient pattern, again using benchmark tuning. Since
post-BCV and CV appear to lead to better normal approximations than BCV, we
display only results from the former two methods. Post-BCV leads to a relatively
accurate normal approximation for every correlation level considered. Moreover,
while the normal approximation stemming from CV appears to improve as $\rho$
increases, at no correlation level considered does CV lead to a visually better
approximation than post-BCV.\footnote{We also investigated the robustness of the
post-BCV-resulting normal approximations with respect to the markup
$c_{0}\in\{1,1.05,1.1\}$. Parallelling our findings for mean estimation error in
Figure \ref{fig:MeanEll2ErrorPostBCVVaryingc0BCCHrule}, the exact markup value
appears to make little difference. Figures are available upon request.}

\section{Revisiting Racial Differences in Police Use of Force}\label{sec:Application}

In this section we revisit the empirical setting in
\citet{fryer_jr_empirical_2019} (henceforth: Fryer), who explored racial
differences in police use of force. We here focus on the part of Fryer's
regression analysis invoking the full Police-Public Contact Survey (PPCS)
dataset with the outcome being an indicator for any use of force by the police
(conditional on an encounter), thus leading to a binary response model as in our
Example \ref{exa:Logit}.\footnote{See \citet{fryer_jr_empirical_2019} and the
associated online appendix for alternative outcome variables and data sources as
well as a detailed discussion of their relative merits and drawbacks.} Specifically,
Fryer estimates models of the form
\begin{equation}
\mathrm{P}\left(\mathrm{Force}=1\mid\mathbf{Race},\bW\right)=F\left(\mathbf{Race}^{\top}\balpha_{0}+\bW^{\top}\bgamma_{0}\right),\label{eq:FryerModel}
\end{equation}
where $\mathrm{Force}$ indicates whether any force was used by the police when
encountering a civilian,
$\mathbf{Race}=(\mathrm{Black},\mathrm{Hisp},\mathrm{Other})^{\top}$ indicate
the race of the civilian (black, hispanic and other than white, with white being
the reference race), and $\bW$ is a list of control variables (including a
constant) capturing both civilian (e.g.~gender and age), officer (e.g.~majority
race) and encounter characteristics (e.g.~whether the civilian disobeyed,
resisted or otherwise misbehaved).\footnote{See the \citet[Table
2.B]{fryer_jr_empirical_2019} notes for the full variable list and his online
appendix A for descriptions.} Here $F$ is a placeholder for a strictly
increasing known cumulative distribution function (CDF), which Fryer takes to be
the logistic CDF $\Lambda$, thus leading to the binary logit model.

The PPCS logistic regression results reported in \citet[Table
2.B]{fryer_jr_empirical_2019} show that black and hispanic subjects are
statistically significantly more likely to experience some form of force in
interactions with the police, controlling for context and civilian behavior. We
here look into the robustness of this finding by employing also
\begin{inparaenum}[(i)]
   \item an alternative binary response model and
   \item large(r) sets of candidate regressors, in combination with
      $\ell_{1}$-penalization.
\end{inparaenum}
For brevity, we here single out the black dummy $(\mathrm{Black})$ and its
coefficient $\beta_{\mathrm{Black}}$ and group the other non-white dummies
($\mathrm{Hisp}$ and $\mathrm{Other}$) with the controls (thus recasting $\bW$
and $\bgamma_0$). We interpret the statement ``there are no racial differences
in police use of force'' as there being no difference in the probability of
force being used for black civilians relative to white subjects, holding
everything else equal. That is,
\[
   \mathrm{P}\left(\mathrm{Force}=1\mid\mathrm{Black}=1,\bW=\bw\right)
      =\mathrm{P}\left(\mathrm{Force}=1\mid\mathrm{Black}=0,\bW=\bw\right)
\]
for all realizations $\bw$ of the controls (with $\mathrm{Hisp}$ and
$\mathrm{Other}$ both zero). Since the race dummies enter the strictly
increasing $F$ in \eqref{eq:FryerModel} in an additive manner, using the model,
such a zero probability difference is equivalent to a zero coefficient on the
dummy for being black, i.e.~$\beta_{\mathrm{Black}}=0$. We therefore take the
latter as the hypothesis to be tested.

To this end, we first use the \citet{fryer_jr_empirical_2019} supplementary files and descriptions
in his online appendix to recollect and recreate the PPCS dataset. Using the same supplementary
files, we then replicate the PPCS logistic regression results in \citet[Table
2.B]{fryer_jr_empirical_2019} to all reported digits, which leaves us confident that we are indeed
considering the original dataset.

We next apply three-step debiasing (Algorithm \ref{alg:ThreeStepDebiasing}) with the loss function
being either the negative logit or probit log-likelihood. Our simulation findings indicate that
post-BCV debiasing outperforms both the BCV and CV equivalents. We therefore only consider the
former.\footnote{For both Steps 1 and 2, we here use 10-fold cross-validation, i.e.~$K=10$.} We use
two sets of regressors. The first set (Basic Controls) corresponds to that in \citet[Table 2.B, Row
l]{fryer_jr_empirical_2019}, and is Fryer's largest set of controls. The only difference is that we
include categorical regressors via dummies for their different levels, leaving one reference
category for each. The second set of regressors (Basic Controls + Interactions) builds on the first
by adding all first-order pairwise interactions between the controls (excluding the race dummies
$\mathrm{Hisp}$ and $\mathrm{Other}$). After eliminating variables with zero variance or perfect
correlation, the two sets include 30 and 327 non-constant regressors, respectively, which should be
compared to a sample size of $n=59,668$ civilian--police encounters.\footnote{We note in passing
that the 59,668 equals the \emph{total} number of police-civilian encounters in the PPCS dataset
covering the six surveys 1996, 1999, 2002, 2005, 2008 and 2011. The number of \emph{complete} cases
with respect to the regressors used is 9,930 and only leaves the years 2002 and 2011. For a clean
comparison, we follow Fryer's approach to missing values.}

Table \ref{tab:RevisitingFryerTstats} displays the $t$-values associated
with testing the null hypothesis using either unpenalized or
$\ell_{1}$-penalized methods.
\begin{table}
\caption{$t$-Values for Testing a Zero Coefficient
$\beta_{\mathrm{Black}}$ on $\mathrm{Black}$\label{tab:RevisitingFryerTstats}}

\centering{}%
\begin{tabular}{lrrrrr}
\hhline{======}
 & \multicolumn{2}{c}{Unpenalized (ML)} &  &
\multicolumn{2}{c}{Post-BCV}\tabularnewline \cline{2-3} \cline{3-3} \cline{5-6}
\cline{6-6} Controls \textbackslash{} Loss & Logit & Probit &  & Logit &
Probit\tabularnewline
\hline 
Basic Controls & 8.8 & 8.7 &  & 10.5 & 9.6\tabularnewline + Interactions & n.a.
& n.a. &  & 20.7 & 18.9\tabularnewline
\hline 
\end{tabular}
\end{table}
Using only basic controls (the logit case being covered in Fryer), the
$t$-statistics take on similar values for both unpenalized maximum likelihood
(ML) and post-BCV methods. Thus, with only 30 non-constant candidate regressors,
regularization has little impact. In contrast, the specification including both
basic controls and interactions thereof leads to complete separation in the data,
such that the (unpenalized) maximum likelihood estimates do not exist (as real
numbers). For this case, some regularization is necessary. Hence, although the
numbers of regressors considered here may not appear overwhelmingly large
when compared to the sample size, the set of regressors is of great importance. Even
when we include all first-order interactions between the controls, the
$t$-statistics resulting from our three-step debiasing procedure remain of the
same order as before.\footnote{The increase in the $t$-values for
post-BCV upon inclusion of interactions is for both the logit and probit
loss due to both a somewhat larger point estimate and a somewhat smaller
standard error. The values underlying Table \ref{tab:RevisitingFryerTstats} 
thus illustrate that more candidate regressors need not lead to
larger standard error.} The $t$-tests based on our post-BCV debiasing
lead us to reject the null hypothesis of no racial differences in police use of
force at any reasonable significance level. This conclusion in
\citet{fryer_jr_empirical_2019} therefore appears robust to the choice of
controls.

To gauge the economic impact of our change in estimation procedures, we estimate
the average partial effect (APE) of changing the civilian race from white to
black. Iterating expectations and using (\ref{eq:FryerModel}), the APE can be
expressed as the average probability difference
\begin{align}
\mathrm{APE}_{\mathrm{Black}} & :=\mathrm{E}\big[\mathrm{P}\left(\mathrm{Force}=1\mid\mathrm{Black}=1,\mathrm{Hisp}=0,\mathrm{Other}=0,\bW\right)\nonumber \\
 & \quad\quad-\mathrm{P}\left(\mathrm{Force}=1\mid\mathrm{Black}=0,\mathrm{Hisp}=0,\mathrm{Other}=0,\bW\right)\big]\nonumber \\
 & =\mathrm{E}\left[F\left(\beta_{\mathrm{Black}}+\bW^{\top}\bgamma_{0}\right)-F\left(\bW^{\top}\bgamma_{0}\right)\right],\label{eq:APE}
\end{align}
where we bring back the other (non-white) civilian race dummies to clarify the
comparison made. We estimate this APE by
$
\widehat{\mathrm{APE}}_{\mathrm{Black}}:=\mathbb{E}_{n}[F\big(\widehat{\beta}_{\mathrm{Black}}+\bW_{i}^{\top}\widehat{\bgamma}\big)-F\left(\bW_{i}^{\top}\widehat{\bgamma}\right)],
$
for point estimates $\widehat{\beta}_{\mathrm{Black}}$ and $\widehat{\bgamma}$
of $\beta_{\mathrm{Black}}$ and $\bgamma_{0}$, respectively. When results stem
from (unpenalized) ML, we use the ML estimates. When results stem from
three-step post-BCV debiasing, we use the debiased third-step estimate
$\widehat{\beta}_{\mathrm{Black}}$ and the (biased) first-step estimate
$\widehat{\bgamma}$. Table \ref{tab:RevisitingFryerAPEs} reports the APE
estimates (in percentage points) corresponding to these procedures including
either basic controls or basic controls with interactions.

\begin{table}[htb]
\caption{Estimates of the Average Partial Effect of $\text{Black}$ (in Percentage
Points)\label{tab:RevisitingFryerAPEs}}
\centering{}%
\begin{tabular}{lrrrrr}
\hhline{======}
 & \multicolumn{2}{c}{Unpenalized (ML)}
&  & \multicolumn{2}{c}{Post-BCV}\tabularnewline \cline{2-3} \cline{3-3}
\cline{5-6} \cline{6-6} Controls \textbackslash{} Loss & Logit & Probit &  &
Logit & Probit\tabularnewline
\hline 
Basic Controls & 1.1 & 1.1 &  & 1.4 & 1.3\tabularnewline + Interactions & n.a. &
n.a. &  & 3.2 & 2.8\tabularnewline
\hline 
\end{tabular}
\end{table}
Using only basic controls, (unpenalized) ML and post-BCV lead to APE estimates
in the range of 1.1--1.4 percentage points regardless of the CDF used. (For
context, the unconditional average of contacts in which PPCS respondents
reported any force being used for white civilians is .7 percent.) Including
interactions of basic controls, the post-BCV APE estimates roughly double in
size to about 3 percentage points. Of course, these relatively large APE
estimates may come with relatively large estimation error. However, as the APE
in (\ref{eq:APE}) depends on many coefficients, it remains a non-trivial task to
assign standard errors to these point estimates---a task falling outside the
scope of this paper.

Finally, to get a feel for the computational burden associated with the methods
proposed in this paper when applied to real data, in Table \ref{tab:Timings} we
report the computing time used by the above-mentioned estimation routines.
\begin{table}[htb]
   \begin{threeparttable}
      \caption{Estimation Routine Timings (in Seconds) \label{tab:Timings}}
      \centering{}%
      \begin{tabular}{lrrrrr}
      \hhline{======}
      & \multicolumn{2}{c}{Unpenalized (ML)} &  &
      \multicolumn{2}{c}{Post-BCV}\tabularnewline \cline{2-3} \cline{3-3}
      \cline{5-6} \cline{6-6} Controls \textbackslash{} Loss & Logit & Probit &
      & Logit & Probit\tabularnewline
      \hline 
      Basic Controls & 1.5 & 1.5 &  & 21 & 40\tabularnewline + Interactions
      & $\infty$ & $\infty$ &  & 210 & 199\tabularnewline
      \hline 
      \end{tabular}
      \begin{tablenotes}
         \footnotesize
         \item\emph{Notes:} All timings were carried out on an Intel Core
         i7-8700 3.20GHz CPU. When using \texttt{cv.glmnet}, we use the parallel
         computing option with all 12 virtual cores available.
      \end{tablenotes}
   \end{threeparttable}
\end{table}
With only basic controls, the three-step post-BCV debiasing procedure takes at least ten times as
long as (unpenalized) ML. This is not surprising, as the former method involves two rounds of
(10-fold) CV, bootstrapping and refitting---no task of which is undertaken by ML. However, including
also interactions, the ranking of the two approaches is reversed. The about tenfold increase in
number of controls increases the computing time associated with post-BCV logit debiasing
approximately linearly. For post-BCV probit debiasing, the corresponding increase is almost five
fold. In contrast, as the ML estimates are not real numbers, without proper checks for solution
existence, any (gradient-based) optimizer would iterate indefinitely in search of the ML estimates.
We represent the non-existence of an ML estimate by infinite computing time.\footnote{While infinity
may appear overly dramatic, we warn that \texttt{glmnet} does \emph{not} check for optimizer
divergence \citep[cf.][p.~9]{friedman_regularization_2010}. We therefore opted for
\texttt{stats::glm} for ML estimation and refitting.}

Of course, the Table \ref{tab:Timings} runtimes are only single observations arising from our
particular \texttt{R} implementation of our procedures, using a specific dataset, and our specific
computing environment. As such, they need not translate to other settings.

\bibliographystyle{ecta}
\bibliography{My_Library}

\newpage{}

\appendix


\doparttoc 
\faketableofcontents 
\part{Online Appendices} 
\singlespacing 
\parttoc 
\onehalfspacing 

\section{Verification of High-Level Assumptions}\label{sec:Verification}

In this section, we discuss the high-level assumptions from the main text. We
first consider the case of the linear model and square loss, and then turn to
the examples from Section \ref{sec:Examples}. In particular, Corollaries
\ref{cor: examples} and \ref{cor: examples inference} state convergence rate and
inference results under low-level conditions for each of the examples from
Section \ref{sec:Examples}.

\subsection{Verification for Linear Model with Square Loss}
\label{sec:VerificationLinearModelSquareLoss}
It is instructive to illustrate the contents of Assumptions
\ref{assu:ParameterSpace}--\ref{assu:LossLocallyLipschitzAndMore},
\ref{assu:ResidualBootstrapMethod}, \ref{as: post estimator smoothness},
\ref{as: post estimator moments}, \ref{as: identifiability inference}, \ref{as:
smoothness inference}, and \ref{as: conditional density} in the familiar case of
the linear mean regression model. We here consider the following (strong)
version of the linear model with independent Gaussian errors,
\[
Y=\bX{^\top}\btheta_0+\varepsilon,\quad \varepsilon\mid \bX\sim \mathrm{N}\left(0,\sigma_0^2\right),
\] 
where $\btheta_0\in\R^p$ and $\sigma_0^2\in(0,\infty)$ are model parameters, and
we use the (one-half) square loss $m(t,y)=(1/2)(y-t)^2$. Less restrictive
dependence and distributional assumptions can be accommodated. We focus on the
independent Gaussian case for simplicity.

We take $\Theta$ to be the full space $\R^p$, such that Assumption
\ref{assu:ParameterSpace} is trivial. Convexity of the loss (Assumption
\ref{assu:Convexity}) follows from differentiating twice with respect to $t$ and
observing that $m''_{11}(t,y)=1>0$ no matter $(t,y)\in\R^2$. For Assumption
\ref{as: diff and int}, observe that for $\btheta\in\Theta$ arbitrary and
denoting $\bdelta:=\btheta-\btheta_0$, the derivative
\[
   m_1'(\bX^{\top}\btheta,Y)=\bX^\top\btheta-Y=\bX^\top\bdelta-\varepsilon
\]   
surely exists. The finiteness of
\[
\E\left[\left|m\left(\bX^{\top}\btheta,Y\right)\right|\right]
   =\frac{1}{2}\E\big[\left(\varepsilon-\bX^{\top}\bdelta\right)^2\big]
   =\frac{1}{2}\left(\sigma_0^2+\bdelta^{\top}\E\left[\bX\bX^{\top}\right]\bdelta\right)
\]
boils down to finiteness of the matrix $\E[\bX\bX^{\top}]$. Assumption \ref{as:
diff and int} can then be guaranteed by finiteness of second moments
$\E[X_{j}^{2}]$ for all $j\in[p]$, for example. Turning to Assumption
\ref{assu:Margin}, note that the excess risk at $\btheta$ is
\[
\mathcal{E}\left(\btheta\right)
   =\frac{1}{2}\E\big[\left(\varepsilon-\bX^{\top}\bdelta\right)^2-\varepsilon^2\big]
   =\frac{1}{2}\bdelta^{\top}\E\left[\bX\bX^{\top}\right]\bdelta
   \geqslant \frac{1}{2}\lambda_{\min}\left(\E\left[\bX\bX^{\top}\right]\right)\left\|\bdelta\right\|_2^2,
\]
with $\lambda_{\min}(\E[\bX\bX^{\top}])$ denoting the smallest eigenvalue of
$\E[\bX\bX^{\top}]$. As long as the eigenvalues are bounded away from zero,
Assumption \ref{assu:Margin} holds with $c_M=1\wedge
(1/2)\inf_{n\in\N}\lambda_{\min}(\E[\bX\bX^{\top}])$ and $c_M'=\infty$.
Positivity of $\lambda_{\min}(\E[\bX\bX^{\top}])$ is precisely the rank
condition for identification of $\btheta_0$. Without this condition $\btheta_0$
does not uniquely minimize the expected loss. For Assumption
\ref{assu:LossLocallyLipschitzAndMore}.\ref{enu:LossLocallyLipschitz}, a
calculation shows that for all $(\bx,y)\in\R^{p+1}$ and all $(t_1,t_2)\in\R^2$,
\begin{align*}
   \abs{m\left(\bx^{\top}\btheta_{0}+t_1,y\right)-m\left(\bx^{\top}\btheta_{0}+t_2,y\right)}
      &=\frac{1}{2}\left|t_1+t_2+2\left(\bx^{\top}\btheta_0-y\right)\right|\left|t_1-t_2\right|\\
      &\leqslant \frac{1}{2}\left(\left|t_1\right|+\left|t_2\right|+2\left|y-\bx^{\top}\btheta_0\right|\right)\left|t_1-t_2\right|.
\end{align*}
Choosing $c_L=1$ and $L(\bx,y)=1+|y-\bx^{\top}\btheta_0|$, we see that
(\ref{eq:LocallyLipschitz}) holds for all $(\bx,y)\in\R^{p+1}$ and all
$(t_1,t_2)\in\R^2$ for which $|t_1|\lor |t_2|\leqslant c_L$. For this $L$, using
independence and $(a+b)^2\leqslant 2a^2+2b^2$, we get
\[
\max_{1\leqslant j\leqslant p}\E\big[\left|L\left(\bX,Y\right)X_{j}\right|^2\big]
   =\E\left[\left(1+\left|\varepsilon\right|\right)^2\right]\max_{1\leqslant j\leqslant p}\E\left[X_{j}^2\right]
   \leqslant 2\left(1+\sigma^2_0\right)\max_{1\leqslant j\leqslant p}\E\left[X_{j}^2\right].
\]
As long as the error has bounded variance and the regressors bounded second
moments, the previous displays suggests
\[
   C_{L,1}^2:=1\lor 2\sup_{n\in\N}\left\{\left(1+\sigma_0^2\right)\max_{1\leqslant j\leqslant p}\E[X_{j}^2]\right\}.
\]
Also, picking $r=8$, we get
\[
   \E\big[\left|L\left(\bX,Y\right)\left\|\bX\right\|_{\infty}\right|^8\big]
      =\E\left[\left(1+\left|\varepsilon\right|\right)^8\right]\E\left[\left\|\bX\right\|_{\infty}^8\right]
      \leqslant 2^7 \left(1+105\sigma_0^8\right) \E\left[\left\|\bX\right\|_{\infty}^8\right],
\]
where the last inequality uses normality to get
$\E[\varepsilon^8]=105\sigma_0^8$. Provided the right-hand side is finite, this
calculation suggests
\[
B_n^8 = 1\lor 2^7\left(1+105\sigma_0^8\right)\E\left[\left\|\bX\right\|_{\infty}^8\right].
\]

For Assumption \ref{assu:LossLocallyLipschitzAndMore}.\ref{enu:LossMeanSquareEll2Conts},
note that
\begin{align*}
\mathrm{E}\left[\left|m\left(\bX^{\top}\btheta,Y\right)-m\left(\bX^{\top}\btheta_{0},Y\right)\right|^2\right]
   &=\frac{1}{4}\mathrm{E}\big[\big|\left(\varepsilon-\bX^{\top}\bdelta\right)^2-\varepsilon^2\big|^2\big]\\
   &=\frac{1}{4}\mathrm{E}\big[\left(\bX^{\top}\bdelta\right)^4\big]
      +\mathrm{E}\big[\varepsilon^2\left(\bX^{\top}\bdelta\right)^2\big]
      -\mathrm{E}\big[\varepsilon\left(\bX^{\top}\bdelta\right)^3\big]\\
   &=\frac{1}{4}\mathrm{E}\big[\left(\bX^{\top}\bdelta\right)^4\big]
      +\sigma_0^2\mathrm{E}\big[\left(\bX^{\top}\bdelta\right)^2\big]\\
   &\leqslant \frac{1}{4}\mathrm{E}\big[\left(\bX^{\top}\bdelta\right)^4\big]
      + \sigma_0^2\lambda_{\max}\left(\mathrm{E}\big[\bX\bX^{\top}\big]\right)\|\bdelta\|_2^2,
\end{align*}
with $\lambda_{\max}(\mathrm{E}[\bX\bX^{\top}])$ denoting the largest eigenvalue
of $\mathrm{E}[\bX\bX^{\top}]$. The term
$(1/4)\mathrm{E}[(\bX^{\top}\bdelta)^4]$ is not generally of the form required
by Assumption
\ref{assu:LossLocallyLipschitzAndMore}.\ref{enu:LossMeanSquareEll2Conts}.
However, if the regressors themselves are jointly Gaussian
$\bX\sim\mathrm{N}(\mathbf{0}_p,\mathrm{E}[\bX\bX^{\top}])$ (or at least centered
sub-Gaussian to obtain an inequality), then
$\bX^\top\bdelta\sim\mathrm{N}(0,\E[(\bX^{\top}\bdelta)^2])$
and, thus,
\begin{equation}\label{eq:GaussianFourthMomentLinearVerification}
   \mathrm{E}\big[\left(\bX^{\top}\bdelta\right)^4\big]=3\big(\mathrm{E}\big[\left(\bX^{\top}\bdelta\right)^2\big]\big)^2
   \leqslant 3\lambda_{\max}\left(\mathrm{E}\big[\bX\bX^{\top}\big]\right)^2\|\bdelta\|_2^4.  
\end{equation}
Hence, when $\|\bdelta\|_2\leqslant c_L=1$,
$\mathrm{E}[(\bX^{\top}\bdelta)^4]\leqslant 3\lambda_{\max}(\mathrm{E}[\bX\bX^{\top}])^2\|\bdelta\|_2^2$,
and, thus,
\[
\mathrm{E}\left[\left|m\left(\bX^{\top}\btheta,Y\right)-m\left(\bX^{\top}\btheta_{0},Y\right)\right|^2\right]
   \leqslant \left(\frac{3}{4}\lambda_{\max}\left(\mathrm{E}\big[\bX\bX^{\top}\big]\right)^2
      +\sigma_0^2\lambda_{\max}\left(\mathrm{E}\big[\bX\bX^{\top}\big]\right)\right)\|\bdelta\|_2^2.
\]
Provided also the eigenvalues of $\mathrm{E}[\bX\bX^{\top}]$ are bounded from above,
the previous display suggests
\[
C_{L,2}^2=1\lor \sup_{n\in\N}\left\{\frac{3}{4}\lambda_{\max}\left(\mathrm{E}\big[\bX\bX^{\top}\big]\right)^2
   +\sigma_0^2\lambda_{\max}\left(\mathrm{E}\big[\bX\bX^{\top}\big]\right)\right\}.   
\]
Finally, for Assumption \ref{assu:LossLocallyLipschitzAndMore}.\ref{enu:ResidualMeanSquareEll2Conts},
observe that no matter $\|\bdelta\|_2$,
\[
   \E\big[\left| m_{1}'\left(\bX^{\top}\btheta,Y\right)-m_{1}'\left(\bX^{\top}\btheta_{0},Y\right)\right|^{2}\big]
   = \E\big[\left(\bX^{\top}\bdelta\right)^{2}\big] \leqslant \lambda_{\max}\left(\mathrm{E}\big[\bX\bX^{\top}\big]\right)\|\bdelta\|_2^2,
\]
which suggests
\[
   C_{L,3}^2=1\lor \sup_{n\in\N}\lambda_{\max}\left(\mathrm{E}\big[\bX\bX^{\top}\big]\right).
\]
Assumption \ref{assu:LossLocallyLipschitzAndMore} now follows from setting
$C_L=\max\{C_{L,1},C_{L,2},C_{L,3}\}$, which lies in $[1,\infty)$ under the previously
stated assumptions.

With an eye on Assumption \ref{assu:ResidualBootstrapMethod}, from
$m'_1(\bX^\top\btheta_0,Y)=-\varepsilon$ and independence of $\varepsilon$ and
$\bX$, we see that $\E[|UX_{j}|^2]=\sigma_0^2\E[X_{j}^2]$, and so
\[
   \sigma_0^2 \min_{1\leqslant j\leqslant p}\E[X_{j}^2]\leqslant \E[|UX_{j}|^2] \leqslant \sigma_0^2 \max_{1\leqslant j\leqslant p}\E[X_{j}^2],
\]
Hence, as long as the lower and upper bounds are bounded away from zero and
infinity, respectively, for the purpose of Assumption
\ref{assu:ResidualBootstrapMethod}.\ref{enu:ZijSecondMomentsBndAwayZero}, we can
take
\[
c_U^2=\inf_{n\in\N}\left\{\sigma_0^2 \min_{1\leqslant j\leqslant p}\E[X_{j}^2]\right\}
   \;\text{and}\;C_U^2=1\lor\sup_{n\in\N}\left\{\sigma_0^2 \max_{1\leqslant j\leqslant p}\E[X_{j}^2]\right\}.
\]
Using also normality, $\E[|UX_{j}|^4]=3\sigma_0^4\E[X_{j}^4]$ and
$\E[\|U\bX\|^4_{\infty}]=3\sigma_0^4\E[\|\bX\|_{\infty}^4]$, which suggest
\[
\widetilde{B}_{n,1}^2=1\lor 3\sigma_0^4\max_{1\leqslant j\leqslant p}\E[X_{j}^4]
   \;\text{and}\;\widetilde{B}_{n,2}^4=1\lor 3\sigma_0^4\E[\|\bX\|_{\infty}^4],
\]
respectively. The remainder of Assumption \ref{assu:ResidualBootstrapMethod} now follows from
setting $\widetilde B_{n}=\widetilde{B}_{n,1}\lor \widetilde{B}_{n,2}$, which lies in $[1,\infty)$
as long as $\E[\|\bX\|_{\infty}^4]<\infty$, which holds when all regressors have finite fourth
moments. Assumptions \ref{assu:DataPartition} and \ref{assu:CandidatePenalties} have nothing to do
with the model and loss function. Moving to Assumption \ref{as: post estimator smoothness},
observing that $m'_1(t,y)=t-y$, setting $C_m:=1$ we have |$m'_1(t_1,y)-m'_1(t_2,y)|=C_m|t_1-t_2|$
for all $t_1,t_2,y\in\R$. As shown in \eqref{eq:GaussianFourthMomentLinearVerification}, if $\bX$ is
Gaussian with $\lambda_{\max}(\E[\bX\bX^\top])$ bounded, then Assumption \ref{as: post estimator
moments} holds with 
\[
   C_{ev}:=3\sup_{n\in\N}\lambda_{\max}(\E[\bX\bX^\top])^2.
\]
More generally, joint sub-Gaussianity with bounded (joint) sub-Gaussian norm
here suffices.

Next, $\bmu_0$ is uniquely given by $\bmu_0=(\E[\bW\bW^\top])^{-1}\E[\bW D]$
provided $\lambda_{\min}(\E[\bW\bW^\top])>0$, which is implied by the previously
discussed rank condition for identification of $\btheta_0$. It follows that
\[
 \E\left[|m_1'(\bX^\top\btheta_0,Y)(D-\bW^\top\bmu_0)|^2\right]=\sigma^2_0 \E\left[(D-\bW^\top\bmu_0)^2\right],
\]
and Assumption \ref{as: identifiability inference} is satisfied if both right-hand side
terms are bounded away from zero. To this end, note that
\[
\E\left[(D-\bW^\top\bmu_0)^2\right]=
\left[\begin{array}{cc}
1 & -\boldsymbol{\mu}_{0}^{\top}\end{array}\right]\mathrm{E}\left[\boldsymbol{X}\boldsymbol{X}^{\top}\right]\left[\begin{array}{c}
1\\
-\boldsymbol{\mu}_{0}
\end{array}\right]\geqslant \lambda_{\min}(\E[\bX\bX^\top]).
\]
Our identifiability condition (Assumption \ref{as: identifiability
inference}) then essentially follows from the previously invoked
$\inf_{n\in\N}\lambda_{\min}(\E[\bX\bX^\top])>0$, which, again, is only slightly
stronger than the rank condition for identification of $\btheta_0$.
The square loss is everywhere thrice differentiable. Hence,
for the purpose of Assumption \ref{as: smoothness inference}, we can
take $J=1$. The second derivative of the loss is $m''_{11}(\cdot,\cdot)\equiv 1$
and the third derivative is $m'''_{111}(\cdot,\cdot)\equiv0$, and, thus, the remainder of 
Assumption \ref{as: smoothness inference} follows. Since we can pick
$J=1$, Assumption \ref{as: conditional density} is trivially satisfied.

\subsection{Verification for Examples in Section \ref{sec:Examples}}\label{sec: verification} In
this section, we state convergence rate and inference results for each of the examples from the main
text under low-level assumptions. The main task of this section is to prove the following results.
\begin{cor}
[\textbf{Convergence Rates in Examples from Main Text}]\label{cor: examples}
Let $c_{de}$, $c_{ev}$, $c_f$, $c_{eps}$, $C_0$, $C_{ev}$ and $C_{pdf}$ be some constants in
$(0,\infty)$, let $\bar r$ be a constant in $(4,\infty)$, and let $\bar B_n$ be a non-random
sequence in $[1,\infty)$. Assume the setting of one of the examples from the main text: Example
\ref{exa:Logit} (logit or probit), Example \ref{exa:LogConcaveOrderedResponse} (logit or probit),
Example \ref{exa:Expectile}, or Example \ref{exa:PanelCensoredRegressionAndTrimming} (trimmed LS or
trimmed LAD). In all of these examples, take $\Theta = \R^p$ and assume that
\begin{equation}\label{eq: verification of high level assumptions key condition}
\E[|\bX^{\top}\bdelta|^2]\geqslant c_{ev}\|\bdelta\|_2^2,\quad\text{and}\quad\E[|\bX^{\top}\bdelta|^4]\leqslant C_{ev}^2\|\bdelta\|_2^4,\quad\text{for all }\bdelta\in\R^p,
\end{equation}
\begin{equation}\label{eq: verification of high level assumptions key condition 2}
\max\left\{\E[|\bX^{\top}\btheta_0|^8], \E[|D - \bW^{\top}\bmu_0|^8], \max_{j\in[p]}\E[|X_j|^8]\right\}\leqslant C_0^8,
\end{equation}
\begin{equation}\label{eq: verification of high level assumptions key condition 3}
\E\left[\left(1 + |\bX^{\top}\btheta_0|\right)^{\bar r}\|\bX\|_{\infty}^{\bar r}\right]\leqslant \bar B_n^{\bar r}.
\end{equation}
For Example \ref{exa:Expectile}, assume in addition that $Y$ is continuously distributed given
$\bX$, i.e., that
\begin{equation}\label{eq: expectile Y conditionally continuous}
\P(Y=t\mid\bX=\bx)=0\quad\text{for all}\quad(t,\bx)\in\mathcal Y\times\mathcal X,
\end{equation}
and that
\begin{equation}\label{eq: y square integrable example}
\E[|Y|^8]\leqslant C_{0}^8,\quad \min_{j\in[p]}\E[|Y - \bX^{\top}\btheta_0|^2X_j^2]\geqslant c_{eps},\quad\text{and}\quad \E[\|Y\bX\|_{\infty}^{\bar r}] \leqslant \bar B_n^{\bar r}.
\end{equation}
In Example \ref{exa:PanelCensoredRegressionAndTrimming} (trimmed LS), assume in addition
that the conditional distributions of $Y_1$ and $Y_2$ given $\bX$ are continuous on $(0,\infty)$, i.e., that
\begin{equation}\label{eq: y1 and y2 conditionally continuous on positive reals}
\P(Y_j=t\mid\bX=\bx)=0\quad\text{for all}\quad(j, t,\bx)\in\{1,2\}\times(0,\infty)\times\mathcal X,
\end{equation}
and that
\begin{equation}\label{eq: y square integrable example 4}
\E[|\overline Y|^8]\leqslant C_0^8,\quad \E[\|\overline Y\bX\|_{\infty}^{\bar r}] \leqslant \bar B_n^{\bar r},
\end{equation}
\begin{equation}\label{eq: even more stuff}
\min_{j\in[p]}\E\left[X_j^2\left(Y_1^2 \mathbf 1\{Y_2 \leqslant - \bX^{\top}\btheta_0\} + Y_2^2 \mathbf 1\{Y_1 \leqslant  \bX^{\top}\btheta_0\}\right)\right]\geqslant c_{eps},
\end{equation}
where $\overline Y := Y_1\lor Y_2$, and
\begin{align}
\P\Big( \min\big\{ &\P(Y_1 - Y_2 > \bX^{\top}\btheta_0 + c_{de} \mid \bX),\nonumber \\
 &\P(Y_1 - Y_2 < \bX^{\top}\btheta_0 - c_{de} \mid \bX)\big\} \geqslant 2c_f \Big) \geqslant 1 - \left(\frac{c_{ev}}{2\sqrt 2C_{ev}}\right)^2. \label{eq: nontrivial tail bound assumption}
\end{align}
In Example \ref{exa:PanelCensoredRegressionAndTrimming} (trimmed LAD), assume
in addition that
\begin{equation}\label{eq: y square integrable example 5}
\E[|\overline Y|^8]\leqslant C_0^8,\quad  \min_{j\in[p]}\E\left[|X_j|^2\mathbf 1\{Y_1>0,Y_2>0\}\right] \geqslant c_{eps},
\end{equation}
where $\overline Y:=Y_1\lor Y_2$, and that the conditional distribution of $Y_1 - Y_2$ given $(\bX, Y_1
> 0, Y_2 > 0)$ as well as the conditional distributions of $Y_1$ given $(\bX, Y_1>0, Y_2 =0)$ and of
$Y_2$ given $(\bX, Y_1 = 0, Y_2 > 0)$ are absolutely continuous with bounded PDFs, i.e., that
\begin{equation}\label{eq: bounded conditional distributions}
f_{Y_1 - Y_2\mid \bX, Y_1 > 0, Y_2 > 0}(t\mid \bx)\leqslant C_{pdf},\quad\text{for all }(\bx,t)\in\mathcal X\times\R
\end{equation}
and
\begin{equation}\label{eq: bounded conditional distributions 2}
f_{Y_j\mid \bX, Y_j > 0, Y_{3-j} = 0}(t\mid \bx)\leqslant C_{pdf},\quad\text{for all }(j,\bx,t)\in\{1,2\}\times\mathcal X\times(0,\infty),
\end{equation} 
and that
\begin{align}
& \P\Big(\inf_{|t|\leqslant c_{de}}f_{Y_1 - Y_2\mid \bX, Y_1>0,Y_2>0}(\bX^{\top}\btheta_0 + t\mid \bX)\P(Y_1 > 0, Y_2 > 0\mid \bX) \geqslant 2c_f\Big) \nonumber\\
& \qquad\qquad\qquad\qquad\qquad\qquad\qquad\qquad\qquad\qquad\qquad\qquad\qquad \geqslant 1 - \left(\frac{c_{ev}}{2\sqrt 2 C_{ev}}\right)^2. \label{eq: nontrivial tail bound assumption 5}
\end{align}
Moreover, let Assumptions \ref{assu:Approximate-Sparsity}, \ref{assu:DataPartition} and
\ref{assu:CandidatePenalties} hold. Finally, let $\widehat\Theta(\widehat
\lambda_{\alpha}^{\mathtt{bcv}})$ be the $\ell_1$-MEs in \eqref{eq:ell1PenalizedMEstimationIntro}
arising from the BCV penalty level $\widehat\lambda_{\alpha}^{\mathtt{bcv}}$ in \eqref{eq:BootstrapPenaltyLevel CV} and
$\alpha=\alpha_n$ satisfying $\alpha_n\to0$ and $\ln(1/\alpha_n)\lesssim \ln(pn)$, and suppose that
$$
n^{1/\bar r}\bar B_ns_q\eta_n^{1-q}\to0,\quad\frac{\bar B_n^4 s_q(\ln(pn))^{5-q/2}(\ln n)^2}{n^{1-q/2-4/\bar r}}\to 0\quad\text{and}\quad \;\frac{\bar{B}_n^4\ln^7\left(pn\right)}{n}\to0.
$$
Then
\[
   \sup_{\mathclap{\widehat\btheta\in\widehat{\Theta}(\widehat{\lambda}^{\mathtt{bcv}}_{\alpha})}}\|\widehat{\btheta}-\btheta_{0}\|_{2}\lesssim_{\P} \sqrt{s_{q}\eta_{n}^{2-q}}\quad\text{and}\quad
   \sup_{\mathclap{\widehat\btheta\in\widehat{\Theta}(\widehat{\lambda}^{\mathtt{bcv}}_{\alpha})}}\|\widehat{\btheta}-\btheta_{0}\|_{1}\lesssim_{\P} s_{q}\eta_{n}^{1-q}.
\]
Let $\widetilde\Theta(\widehat\lambda_{\alpha}^{\mathtt{bcv}})$ be the post-$\ell_1$-MEs in
\eqref{eq:post-ell1-ME-thetabar}--\eqref{eq:AllPostEstimators} resulting from
$\widehat\Theta(\widehat\lambda_{\alpha}^{\mathtt{bcv}})$. If, in addition, in all examples except
for Example \ref{exa:PanelCensoredRegressionAndTrimming} (trimmed LAD), also $n^{1/\bar
r}\bar B_ns_q\eta_n^{1-q}\ln(pn)\to0$, then
\[
   \sup_{\mathclap{\widetilde\btheta\in\widetilde{\Theta}(\widehat{\lambda}^{\mathtt{bcv}}_{\alpha})}}\|\widetilde{\btheta}-\btheta_{0}\|_{2}\lesssim_{\P} \sqrt{s_{q}\eta_{n}^{2-q}\ln(pn)}\quad\text{and}\quad
   \sup_{\mathclap{\widetilde\btheta\in\widetilde{\Theta}(\widehat{\lambda}^{\mathtt{bcv}}_{\alpha})}}\|\widetilde{\btheta}-\btheta_{0}\|_{1}\lesssim_{\P} s_{q}\eta_{n}^{1-q}\ln(pn).
\]
\end{cor}

\begin{cor}
[\textbf{Inference in Examples from Main Text}]
\label{cor: examples inference}
Assume the setting of one of the examples from the main text: Example \ref{exa:Logit} (logit or
probit), Example \ref{exa:LogConcaveOrderedResponse} (logit or probit), Example \ref{exa:Expectile},
or Example \ref{exa:PanelCensoredRegressionAndTrimming} (trimmed LS), and let all
assumptions of Corollary \ref{cor: examples} related to these examples hold. In addition, in Example
\ref{exa:Expectile}, assume that the conditional distribution of $Y$ given $\bX$ is absolutely
continuous with bounded PDF, i.e., that
\begin{equation}\label{eq: y given x continuous example}
f_{Y|\bX}(y\mid\bx) \leqslant C_{pdf}\quad\text{for all}\quad(y,\bx)\in\R\times\mathcal X,
\end{equation}
and in Example \ref{exa:PanelCensoredRegressionAndTrimming} (trimmed LS), assume that for
all $j \in \{1,2\}$, the conditional distribution of $Y_j$ given $(\bX,Y_j>0)$ as well as the
unconditional distribution of $\bX^{\top}\btheta_0$ are absolutely continuous with bounded PDFs, i.e., that
\begin{equation}\label{eq: y given x continuous example 4}
f_{Y_j\mid\bX, Y_j > 0}(t\mid\bx) \leqslant C_{pdf}\ \text{for all }(t,\bx)\in (0,\infty)\times \mathcal X\ \text{and}
\ f_{\bX^{\top}\btheta_0}(t) \leqslant C_{pdf} \ \text{for all }t\in\R.
\end{equation}
Also, let Assumptions \ref{as: identifiability inference} and \ref{as: convergence rates
inference} hold and suppose that $\sqrt n a_n^2\to0$, $a_n(n^{1/\bar r}\bar B_n + \sqrt{\ln(pn)})\to 0$ and
$\bar B_n^2\ln(pn) = o(n^{1-4/(\bar r\wedge 8)})$. Finally, in Examples \ref{exa:Expectile} and
\ref{exa:PanelCensoredRegressionAndTrimming} (trimmed LS), suppose also that $\sqrt n(\bar B_n
a_n)^{(3\bar r+2)/(2\bar r+2)}\to 0$. Then the debiased estimator $\widehat \beta$ given in 
\eqref{eq: debiased estimator betahat} satisfies
$$
\frac{\sqrt n(\widehat\beta - \beta_0)}{\sigma_{0}}\overset{D}\to \mathrm{N}(0,1),\quad\text{where}\quad \sigma_{0}^2:=\frac{\E\big[(m'_1(\bX^{\top}\btheta_0,\bY)(D-\bW^{\top}\bmu_0))^2\big]}{\big(\E[m''_{11}(\bX^{\top}\btheta_0,\bY)(D-\bW^{\top}\bmu_0)D]\big)^2}.
$$
\end{cor}
Corollaries \ref{cor: examples} and \ref{cor: examples inference} follow immediately from Theorems
\ref{cor: convergence rate bootstrap after cv}, \ref{thm: convergence rates post estimator} and
\ref{thm: asymptotic distribution} as long as we can verify Assumptions \ref{assu:ParameterSpace},
\ref{assu:Convexity}, \ref{as: diff and int}, \ref{assu:Margin},
\ref{assu:LossLocallyLipschitzAndMore}, \ref{assu:ResidualBootstrapMethod}, \ref{as: post estimator
smoothness}, \ref{as: post estimator moments}, \ref{as: integrability inference}, \ref{as:
smoothness inference} and \ref{as: conditional density} from the main text under low-level
example-specific assumptions of these corollaries with $B_n \lesssim \bar B_n$, $\widetilde B_n
\lesssim \bar B_n$, $r = \bar r$, and $\widetilde r = 8$.\footnote{Our verification in the Examples
\ref{exa:Logit}, \ref{exa:LogConcaveOrderedResponse}, \ref{exa:Expectile} and
\ref{exa:PanelCensoredRegressionAndTrimming} (trimmed LS) leaves $\Delta_n\in(0,\infty)$
unrestricted. The Corollary \ref{cor: examples inference} growth condition $\sqrt n(\bar B_n
a_n)^{(3\bar r+2)/(2\bar r+2)}\to 0$ results from optimizing this choice subject to the growth
conditions stated in Theorem \ref{thm: asymptotic distribution}.} Also, Assumption
\ref{assu:ParameterSpace} holds trivially as we set $\Theta = \R^p$ and Assumption
\ref{assu:Convexity} holds trivially in all examples as we discussed in the main text. In addition,
Assumptions \ref{as: post estimator smoothness} and \ref{as: post estimator moments} follow
immediately from Assumption \ref{as: smoothness inference} and \eqref{eq: verification of high level
assumptions key condition}, respectively, so we do not have to verify them separately. In the rest
of this section, we verify all the remaining assumptions.

Before starting the verification process, however, we note that there are various sets of sufficient
low-level example-specific assumptions. In particular, throughout this section, we do not impose any
restrictions on the set $\Theta$ and assume that $\Theta=\R^p$, which is convenient from the
implementation point of view, but one could assume, for example, that the set $\Theta$ is
$\ell_1$-bounded in the sense that $\sup_{\btheta\in\Theta}\|\btheta\|_1$ is finite but possibly
growing with $n$, and relax some of the assumptions above. For brevity, we provide results only for
assumptions that are listed in Corollaries \ref{cor: examples} and \ref{cor: examples inference}.

Let $f\colon\R\times\mathcal X\to\R$ be the function defined by $f(t,\bx):=\E[m(t,\bY)\mid\bX=\bx]$
for $(t,\bx)\in\R\times\mathcal X$. We will show that it exists under our conditions. Also, in this
section, to emphasize dependence between constants, we will use function arguments. For example, we
will write $C = C(C_{ev},C_0)$ when the constant $C\in(0,\infty)$ may depend on $C_{ev}$ and $C_0$. 

To streamline the verification process, we first state five lemmas, whose proofs can be found at the
end of this section.

\begin{lem}\label{lem: verification of margin assumption}
Let Assumptions \ref{assu:ParameterSpace}, \ref{assu:Convexity} and \ref{as: diff and int}
hold, and suppose that inequalities \eqref{eq: verification of high level assumptions key condition} and
\eqref{eq: verification of high level assumptions key condition 2} are satisfied. Let $c_f$ and
$c_M'$ be some constants in $(0,\infty)$. In addition, suppose that for all $\bx\in\mathcal X$,
$t\mapsto f(t,\bx)$ exists as a differentiable function from $\R$ to $\R$, with its first derivative
$t\mapsto f_1'(t,\bx)$ being Lipschitz continuous on compacta, so that (per Rademacher's theorem)
there is a (possibly empty) Lebesgue null set $N(\bx)\subset\R$ for which also the second
derivatives $f_{11}''(t,\bx),t\in\R\backslash N(\bx)$, exist. Moreover, suppose that
$\E[|f_1'(\bX^{\top}\btheta,\bX)X_j|]<\infty$ for all $\btheta\in\Theta$ and $j\in[p]$, and that at
least one of the following two conditions is satisfied:
\begin{equation}\label{eq: verifying assumption 3.4 key condition}
\P\left(\inf_{\substack{t\in[-C,C]\\t\notin N(\bX)}}f_{11}''(t,\bX)\geqslant 4c_f\right) \geqslant 1-\left(\frac{c_{ev}}{2\sqrt 2C_{ev}}\right)^2\;\text{for}\quad C: = \frac{4C_{ev}}{c_{ev}}\sqrt{C_0^2+C_{ev}(c_M')^2},
\end{equation}
\begin{equation}\label{eq: verifying assumption 3.4 key condition 2}
\P\left(\inf_{\substack{t=\bX^{\top}\btheta_0+u\\u\in[-C,C]\\t\notin N(\bX)}}f_{11}''(t,\bX)\geqslant 4c_f\right) \geqslant 1-\left(\frac{c_{ev}}{2\sqrt 2C_{ev}}\right)^2\;\text{for}\quad C: = \frac{2\sqrt 2c_M'C_{ev}^{3/2}}{c_{ev}}.
\end{equation}
Then Assumption \ref{assu:Margin} holds with the given $c_M'$ and
$c_M:=(c_{ev}c_f)\wedge1$.
\end{lem}

\begin{lem}\label{lem: verification of some high level assumptions}
Suppose that inequalities \eqref{eq: verification of high level assumptions key condition},
\eqref{eq: verification of high level assumptions key condition 2} and \eqref{eq: verification of
high level assumptions key condition 3} are satisfied, and let $c_{m,1}$, $c_{m,2}$, and $c_{m,3}$
be some constants in $[0,\infty)$. Also, suppose that the function $m(\cdot,\by)$ is continuously
differentiable for all $\by\in\mathcal Y$ with the first derivative $m(\cdot,\by)$ satisfying
\begin{equation}\label{eq: bounded derivative lemma}
|m_1'(t,\by)| \leqslant c_{m,1} + c_{m,2}|t|,\quad\text{for all }(t,\by)\in\R\times\mathcal Y
\end{equation}
and
\begin{equation}\label{eq: lipschitz derivative lemma}
|m_1'(t_2,\by) - m_1'(t_1,\by)| \leqslant c_{m,3}|t_2 - t_1|,\quad\text{for all }(t_1,t_2,\by)\in\R\times\R\times\mathcal Y.
\end{equation}
Then Assumption \ref{assu:LossLocallyLipschitzAndMore} is satisfied with
$L(\bx,\by):=(1 + c_{m,1}+c_{m,2})(1 + |\bx^{\top}\btheta_0|)$ for all
$(\bx,\by)\in\mathcal X\times\mathcal Y$, $c_L:=1$, $r:=\bar r$, $B_n := (1 +
c_{m,1}+c_{m,2})\bar B_n$, and 
$C_L^2:=\max\{1,
   2(1+c_{m,1}+c_{m,2})^2C_{ev}(1+C_0^2),
   3(c_{m,1}^2C_{ev}+c_{m,2}^2 C_0^2 C_{ev}+c_{m,2}^2C_{ev}^2),
   c_{m,3}^2C_{ev}\}$.
\end{lem}

\begin{lem}\label{lem: verification of high level assumptions section 4}
Suppose that all conditions of Lemma \ref{lem: verification of some high level assumptions} are
satisfied. In addition, suppose that there is a constant $c_U\in(0,\infty)$ such that $\E[|
m_1'(\bX^{\top}\btheta_0,\bY) X_j |^2]\geqslant c_U^2$ for all $j\in[p]$. Then Assumption
\ref{assu:ResidualBootstrapMethod} is satisfied with the given $c_U$, $C_U^2:=\max\{1,
2(c_{m,1}^2C_0^2 + c_{m,2}^2C_0^4)\}$ and $\widetilde{B}_n^2:= \max\{1,8(c_{m,1} + c_{m,2})^4 C_0^4
(1+C_0^4), (c_{m,1}+c_{m,2})^2 \bar{B}_n^2\}$.
\end{lem}

\begin{lem}\label{lem: verification of high level assumptions section 5} 
Suppose that all conditions of Lemma \ref{lem: verification of some high level assumptions} are
satisfied. Then Assumption \ref{as: integrability inference} is satisfied with $\widetilde r := 8$
and $C_M:=\max\{C_0,(c_{m,1} + c_{m,2})(1+C_0)\}$.
\end{lem}

\begin{lem}\label{lem: lipschitz derivative}
Let $Z$ be a random variable satisfying $\E[|Z|]<\infty$. Then the function $f$ defined by
$
f(t):= \E[(Z-t)\mathbf 1(Z\geqslant t)]
$
or, equivalently, $f(t):=\E[(Z-t)\mathbf 1(Z > t)]$, is a Lipschitz continuous mapping from $\R$ to
$\R$. If, in addition, $Z$ is continuously distributed, i.e.~$\P(Z=t)=0$ for all
$t\in\R$, then $f$ is differentiable with derivative $f'(t) = - \P(Z>t)$ for all $t\in\R$.
\end{lem}

We are now ready to verify Assumptions \ref{as: diff and int}, \ref{assu:Margin},
\ref{assu:LossLocallyLipschitzAndMore}, \ref{assu:ResidualBootstrapMethod}, \ref{as: integrability
inference}, \ref{as: smoothness inference} and \ref{as: conditional density} in each of the examples
from the main text.

\medskip
\noindent
\textbf{Example \ref{exa:Logit} (Binary Response Model, Continued).} 
We first consider the case of the \emph{logit} loss function \eqref{eq:LossLogit}. In this case, the
differentiability part of Assumption \ref{as: diff and int} is trivial. In addition, for all
$\btheta\in\Theta$, since $1+\mathrm{e}^t\leqslant 2\mathrm{e}^t$ for $t\geqslant0$, by \eqref{eq:
verification of high level assumptions key condition} we have
$$
\E[|m(\bX^{\top}\btheta,Y)|] \leqslant \E\left[\ln(1+\mathrm{e}^{|\bX^{\top}\btheta|})\right] + \E[|\bX^{\top}\btheta|] \leqslant \ln 2 + 2\E[|\bX^{\top}\btheta|] <\infty,
$$
which gives the integrability part of Assumption \ref{as: diff and int}.

Next, recalling that $\P(Y=1|\bX=\bx)=\Lambda(\bx^\top\btheta_0)$ for the standard logistic CDF
$\Lambda(t)=1/(1+\mathrm{e}^{-t})$, we see that the function $f$ defined by
$$
f(t,\bx)=\E[m(t,Y)\mid \bX = \bx] = \ln(1 + \mathrm{e}^t) - \Lambda(\bx^{\top}\btheta_0)t,\quad(t,\bx)\in\R\times\mathcal X,
$$ 
is twice continuously differentiable in its first argument. Hence, for all $\bx\in\mathcal X$,
$t\mapsto f_1'(t,\bx)$ is Lipschitz continuous on compacta. Moreover, $f$ satisfies
\begin{align*}
&f_1'(t,\bx) = \Lambda(t) - \Lambda(\bx^{\top}\btheta_0)\quad\text{and}\quad f_{11}''(t,\bx) = \Lambda'(t) = \Lambda(t)(1-\Lambda(t)),
\end{align*}
which, in particular, shows that $|f'_1(t,\bx)|\leqslant1$, that $f_{11}''(t,\bx)$ is strictly
positive for all $(t,\bx)\in\R\times\mathcal X$ and that $f_{11}''(t,\bx)$ does not actually depend
on $\bx$. Thus, for all $\btheta\in\Theta$ and $j\in[p]$, we have
$\E[|f_1'(\bX^{\top}\btheta,\bX)X_j|]\leqslant \E[|X_j|] < \infty$ by \eqref{eq: verification of
high level assumptions key condition 2}. Now, for arbitrary $C\in(0,\infty)$, defining
$c_f(C):=(1/4)\inf_{|t|\leqslant C}\Lambda'(t)=(1/4)\Lambda'(C)\in(0,\infty)$, we see that
$f_{11}''(t,\bx)\geqslant 4c_f(C)$ for all $(t,\bx)\in[-C,C]\times\mathcal X$. Choosing $c_M':=1$
and the specific $C:=C(C_0,c_{ev},C_{ev},c_M')$ in \eqref{eq: verifying assumption 3.4 key
condition}, by Lemma \ref{lem: verification of margin assumption} with the implied $c_f:=c_f(C)$, we
see that Assumption \ref{assu:Margin} holds with $c_M' = 1$ and $c_M = c_M(C_0,c_{ev},C_{ev})$.

Next, for all $(t,y)\in\R\times\mathcal Y$, the loss has derivatives in its first argument of all
orders with
\begin{alignat*}{2}
m_1'(t,y) &= \Lambda(t) - y &{}={}& \frac{\mathrm{e}^t}{1+\mathrm{e}^t} - y,\\
m_{11}''(t,y) &= \Lambda'(t) &{}={}& \frac{\mathrm{e}^t}{(1+\mathrm{e}^t)^2}\quad\text{and}\\
m_{111}'''(t,y)  &= \Lambda''(t) &{}={}& \frac{\mathrm{e}^t (1-\mathrm{e}^t)}{(1+\mathrm{e}^t)^3},
\end{alignat*}
and so $|m_1'(t,y)|\leqslant 1$, $|m_{11}''(t,y)|\leqslant 1$, and $|m_{111}'''(t,y)|\leqslant 1$.
Hence, \eqref{eq: bounded derivative lemma} and (by way of the mean-value theorem) \eqref{eq:
lipschitz derivative lemma} are satisfied with $c_{m,1} = 1$, $c_{m,2} = 0$, and $c_{m,3} = 1$.
Assumption \ref{assu:LossLocallyLipschitzAndMore} now follows from Lemma \ref{lem: verification of
some high level assumptions} with $L(\bx,y) = 2(1+| \bx^{\top}\btheta_0 |)$ for
$(\bx,y)\in\mathcal X\times\mathcal Y$, $c_L=1$, $r=\bar r$, $B_n = 2\bar B_n$, and $C_L =
C_L(C_0,C_{ev})$. Also, Assumption \ref{as: smoothness inference} is satisfied with $J=1$ and $C_m
=1$. In addition, since $J=1$, Assumption \ref{as: conditional density} is satisfied trivially.
Moreover, Assumption \ref{as: integrability inference} follows from Lemma \ref{lem: verification of
high level assumptions section 5} with $\widetilde r = 8$ and $C_M = C_M(C_0)$.

Finally, letting $C:=C_0^2/\sqrt{c_{ev}/2}$, noting that $U=m_1'(\bX^{\top}\btheta_0,Y) =
\Lambda(\bX^{\top}\btheta_0) - Y$ and letting $c=c(C_0,c_{ev})$ be the constant
$c:=\inf_{|t|\leqslant C}\Lambda'(t)=\Lambda'(C)\in(0,\infty)$, since $Y|\bX=\bx$ is here Bernoulli
distributed with $\P(Y=1|\bX=\bx)=\Lambda(\bx^\top\btheta_0)$, iterating expectations, we have
for all $j\in[p]$ that
\begin{align*}
\E[|UX_j|^2]
& = \E\left[|X_j|^2\E\left[|Y-\Lambda(\bX^\top\btheta_0)|^2\middle|\bX\right]\right]
=\E\left[|X_j|^2\Lambda(\bX^{\top}\btheta_0)\left(1-\Lambda(\bX^{\top}\btheta_0)\right)\right] \\
& \geqslant c\E\left[|X_j|^2\mathbf{1}\left\{|\bX^{\top}\btheta_0|\leqslant C\right\}\right] = c\left( \E[|X_j|^2] - \E\left[|X_j|^2\mathbf 1\left\{|\bX^{\top}\btheta_0|>C\right\}\right] \right)\\
& \geqslant c\left(c_{ev} - \E\left[|X_j|^2|\bX^{\top}\btheta_0|^2\right]/C^2\right) \geqslant c\left(c_{ev} - \Big( \E[|X_j|^4]\E[|\bX^{\top}\btheta_0|^4] \Big)^{1/2}/C^2\right) \\
& \geqslant c\left( c_{ev} - C_0^4/C^2 \right) = cc_{ev}/2,
\end{align*}
by \eqref{eq: verification of high level assumptions key condition}, \eqref{eq: verification of high
level assumptions key condition 2}, and the Cauchy-Schwarz and H{\"o}lder inequalities. The final
inequality in the previous display shows that the lower bound in Assumption
\ref{assu:ResidualBootstrapMethod}.\ref{enu:ZijSecondMomentsBndAwayZero} is satisfied with $c_U =
c_U(C_0,c_{ev})$. The remaining parts of Assumption \ref{assu:ResidualBootstrapMethod} follow from
Lemma \ref{lem: verification of high level assumptions section 4} with $C_U = C_U(C_0)$ and
$\widetilde B_n = \bar B_n \lor C_1$ for some $C_1=C_1(C_0)$. This observation
completes the logit case.

Next, we consider the case of the \emph{probit} loss function \eqref{eq:LossProbit}. Let $\Phi$ and
$\phi$ denote the standard normal CDF and PDF, respectively.  We start with deriving some basic
inequalities. By Proposition 2.5(b) in \cite{dudley2014uniform}, for all $t\in[1,\infty)$, we have
$
\phi(t)/(1 - \Phi(t)) \leqslant 2t,
$
so
$$
\frac{1}{1 - \Phi(t)} \leqslant 2\sqrt{2\pi} t \mathrm{e}^{t^2/2} \leqslant 2\sqrt{2\pi}\mathrm{e}^{t+t^2/2}.
$$
Also, for all $t\in(-\infty,1)$, we have $1/(1-\Phi(t)) \leqslant 1/(1-\Phi(1))$. Hence,
\begin{equation}\label{eq: gauss 1}
|\ln(1-\Phi(t))| \leqslant |\ln(1 - \Phi(1))| + \ln(2\sqrt{2\pi}) + |t| + \frac{t^2}{2},\quad\text{for all } t\in\R,
\end{equation}
and, using the symmetry of the standard normal distribution to get $\Phi(t) = 1 - \Phi(-t)$,
\begin{equation}\label{eq: gauss 2}
|\ln(\Phi(t))| \leqslant |\ln(1 - \Phi(1))| + \ln(2\sqrt{2\pi}) + |t| + \frac{t^2}{2},\quad\text{for all } t\in\R.
\end{equation}
In addition, for all $t\in(-\infty,1)$, we have $\phi(t)/(1 - \Phi(t)) \leqslant \phi(0)/(1 -
\Phi(1))$. Hence, both
\begin{equation}\label{eq: gauss 3}
\frac{\phi(t)}{1 - \Phi(t)} \leqslant \frac{\phi(0)}{1 - \Phi(1)} + 2|t|\quad\text{and}\quad\frac{\phi(t)}{\Phi(t)} \leqslant \frac{\phi(0)}{1 - \Phi(1)} + 2|t|,\quad\text{for all } t\in\R,
\end{equation}
where the second inequality follows from the first, $\Phi(t) = 1 - \Phi(-t)$ and $\phi(t) =
\phi(-t)$. Moreover, by (1.2.2) in \cite{adler_random_2007}, for all $t\in(0,\infty)$, we have
$\phi(t)/(1-\Phi(t))>t$, and so
\begin{equation}\label{eq: gauss 4}
\frac{\phi(t)}{1 - \Phi(t)} - t>0\quad\text{and}\quad \frac{\phi(t)}{ \Phi(t)} + t>0, \quad\text{for all } t\in\R,
\end{equation}
where the second inequality again follows from the first and symmetry. Again, by (1.2.2) in
\cite{adler_random_2007}, for all $t\in(1,\infty)$, we have
$$
0<\frac{\phi(t)}{1 - \Phi(t)} - t \leqslant \frac{1}{1/t-1/t^3} - t = \frac{t}{t^2 - 1}.
$$
The previous display shows that for $t\in(2,\infty)$,
$$
\frac{\phi(t)}{1 - \Phi(t)}\cdot\left| \frac{\phi(t)}{1 - \Phi(t)} - t \right|\leqslant 2t\cdot\frac{t}{t^2 - 1}=\frac{2t^2}{t^2-1}\leqslant \frac{8}{3}.
$$
The left-hand side (continuous) function is bound on compacta, including $[-2,2]$. Finally, for
$t\in(-\infty,2)$, as $\phi$ is bounded, $1-\Phi(\cdot)$ is bounded away from zero, and $\phi(t)$
decays more rapidly than $-t$ grows as $t\to-\infty$, the same left-hand side function remains
bounded on $(-\infty,2)$ as well. Conclude that there is a universal constant $C\in[1,\infty)$ such that 
\begin{equation}\label{eq: gauss 5}
\frac{\phi(t)}{1 - \Phi(t)}\left| \frac{\phi(t)}{1 - \Phi(t)} - t \right| \leqslant C\quad\text{and}\quad \frac{\phi(t)}{\Phi(t)}\left| \frac{\phi(t)}{\Phi(t)} + t \right| \leqslant C,\quad\text{for all } t\in\R,
\end{equation}
where the second inequality follows from symmetry and parallel reasoning. Finally, both
\begin{equation}\label{eq: gauss 6}
\lim_{t\to\infty}t^2\left(\frac{\phi(t)}{1 - \Phi(t)} - t - \frac{1}{t}\right)=0\quad\text{and}\quad \lim_{t\to-\infty}t^2\left(\frac{\phi(t)}{\Phi(t)} + t + \frac{1}{t}\right) = 0
\end{equation}
which both follow from repeated application of L'H{\^o}pital's rule.

With these inequalities in mind, we now verify the required assumptions. The differentiability part
of Assumption \ref{as: diff and int} is trivial. In addition, for all $\btheta\in\Theta$, we have
$$
\E[|m(\bX^{\top}\btheta,Y)|] \leqslant \E[|\ln(\Phi(\bX^{\top}\btheta))| + |\ln(1 - \Phi(\bX^{\top}\btheta))|] <\infty
$$
by \eqref{eq: gauss 1}, \eqref{eq: gauss 2}, the Cauchy-Schwarz inequality and \eqref{eq:
verification of high level assumptions key condition}, which yields also the integrability part of
Assumption \ref{as: diff and int}.

Next, viewed as a function of its first argument, the function $f$ defined by
$$
f(t,\bx)=\E[m(t,Y)\mid \bX = \bx] = -\Phi(\bx^{\top}\btheta_0)\ln(\Phi(t)) - (1 - \Phi(\bx^{\top}\btheta_0))\ln(1 - \Phi(t)),
$$ 
for $(t,\bx)\in\R\times\mathcal X$, is seen to be twice continuously differentiable in its first argument, with
\begin{align*}
f_1'(t,\bx) &= -\Phi(\bx^{\top}\btheta_0)\frac{\phi(t)}{\Phi(t)} + (1 - \Phi(\bx^{\top}\btheta_0))\frac{\phi(t)}{1 - \Phi(t)}\quad\text{and}\\
f_{11}''(t,\bx) &= \Phi(\bx^{\top}\btheta_0)\frac{\phi(t)}{\Phi(t)}\left(t + \frac{\phi(t)}{\Phi(t)}\right) + (1 - \Phi(\bx^{\top}\btheta_0))\frac{\phi(t)}{1 - \Phi(t)}\left(\frac{\phi(t)}{1 - \Phi(t)} - t\right)>0.
\end{align*}
Thus, for all $\btheta\in\Theta$ and $j\in[p]$, we have 
$$
\E[|f_1'(\bX^{\top}\btheta,\bX)X_j|]\leqslant \E\left[\left( \frac{\phi(\bX^{\top}\btheta)}{\Phi(\bX^{\top}\btheta)} + \frac{\phi(\bX^{\top}\btheta)}{1 - \Phi(\bX^{\top}\btheta)} \right)|X_j|\right] < \infty
$$ 
by \eqref{eq: gauss 3}, \eqref{eq: verification of high level assumptions key condition}, \eqref{eq:
verification of high level assumptions key condition 2}, and the Cauchy-Schwarz inequality. Also,
by the positivities in \eqref{eq: gauss 4}, for arbitrary $C\in(0,\infty)$, defining
$$
c_f(C):=\frac{1}{4}\min\left\{\inf_{|t|\leqslant C}\frac{\phi(t)}{\Phi(t)}\left(t + \frac{\phi(t)}{\Phi(t)}\right),\inf_{|t|\leqslant C}\frac{\phi(t)}{1 - \Phi(t)}\left(\frac{\phi(t)}{1 - \Phi(t)} - t\right)\right\}\in(0,\infty)
$$
we see that $f_{11}''(t,\bx)\geqslant 4c_f(C)$ for all $t\in[-C,C]\times\mathcal X$. Choosing
$c_M':=1$ and the specific $C:=C(C_0,c_{ev},C_{ev},c_M')$ in \eqref{eq: verifying assumption 3.4 key
condition}, by Lemma \ref{lem: verification of margin assumption} with the implied $c_f:=c_f(C)$, we
see that Assumption \ref{assu:Margin} holds with $c_M' = 1$ and $c_M = c_M(C_0,c_{ev},C_{ev})$.

Next, for all $(t,y)\in\R\times\mathcal Y$, differentiation shows that
\begin{align*}
m_1'(t,y) &= -y\frac{\phi(t)}{\Phi(t)} + (1 - y)\frac{\phi(t)}{1 - \Phi(t)},\\
m_{11}''(t,y) &= y\frac{\phi(t)}{\Phi(t)}\left(t + \frac{\phi(t)}{\Phi(t)}\right) + (1 - y)\frac{\phi(t)}{1 - \Phi(t)}\left(\frac{\phi(t)}{1 - \Phi(t)} - t\right)\quad\text{and}\\
m_{111}'''(t,y) &= y\frac{\phi(t)}{\Phi(t)}\left\{1 -  \left( \frac{\phi(t)}{\Phi(t)} + t \right)\left( \frac{2\phi(t)}{\Phi(t)} + t \right) \right\}  \\
&\quad +  (1-y)\frac{\phi(t)}{1 - \Phi(t)}\left\{ \left( \frac{\phi(t)}{1 - \Phi(t)} - t \right)\left( \frac{2\phi(t)}{1 - \Phi(t)} - t \right) - 1 \right\}.
\end{align*}
Therefore, for all $(t,y)\in\R\times\mathcal Y$,
$
|m_1'(t,y)|\leqslant \phi(0)/(1 - \Phi(1)) + 2|t|
$
by \eqref{eq: gauss 3} and
$
|m_{11}''(t,y)| \leqslant C
$
for some universal constant $C\in[1,\infty)$ by \eqref{eq: gauss 5}. Hence, both \eqref{eq: bounded
derivative lemma} and (via the mean-value theorem) \eqref{eq: lipschitz derivative lemma} are
satisfied with $c_{m,1} = \phi(0)/(1 - \Phi(1))$, $c_{m,2} = 2$, and $c_{m,3} = C$, and so
Assumption \ref{assu:LossLocallyLipschitzAndMore} follows from Lemma \ref{lem: verification of some
high level assumptions} with $L(\bx,y) = (3 + \phi(0)/[1-\Phi(1)])(1+| \bx^{\top}\btheta_0 |)$ for
 $(\bx,y)\in\mathcal X\times\mathcal Y$, $c_L=1$, $r=\bar r$, $B_n = (3 +
\phi(0)/[1-\Phi(1)])\bar B_n$, and $C_L = C_L(C_0,C_{ev})$. Also, as $t\to\infty$, uniformly over
$y\in\mathcal Y$,
$$
|m_{111}'''(t,y)| = o(1) + (t + o(t))\left| (1/ t + o(1/t^2))(t + 2/t + o(1/t^2)) - 1 \right| =o(1)
$$
and 
as $t\to-\infty$, uniformly over $y\in\mathcal Y$,
$$
|m_{111}'''(t,y)| = (-t + o(|t|))\left|1 - (-1/ t + o(1/t^2))(-t - 2/t + o(1/t^2)) \right| + o(1) =o(1)
$$
using the limits in \eqref{eq: gauss 6}. Since each $t\mapsto m_{111}'''(t,y),y\in\mathcal Y$, is
continuous, from the previous two displays we deduce boundedness of $t\mapsto
m_{111}'''(t,y)$ uniformly in $y\in\mathcal Y$. Hence, Assumption \ref{as: smoothness inference} is
satisfied with $J=1$ and some universal constant $C_m \in [1,\infty)$. Since $J=1$, Assumption
\ref{as: conditional density} is trivially satisfied. Moreover, Assumption \ref{as: integrability
inference} follows from Lemma \ref{lem: verification of high level assumptions section 5} with
$\widetilde r = 8$ and $C_M = C_M(C_0)$. 

Finally, let $C:=C_0^2/\sqrt{c_{ev}/2}$ and let $c=c(C_0,c_{ev})$ be the constant
$$
c:=\inf_{|t|\leqslant C}\frac{\phi(t)^2}{\Phi(t)\left[1-\Phi(t)\right]}\in(0,\infty).
$$
Since $Y|\bX=\bx$ is here Bernoulli distributed with $\P(Y=1|\bX=\bx)=\Phi(\bx^\top\btheta_0)$,
we have
$$
U=m_1'(\bX^{\top}\btheta_0,Y) = \frac{\phi(\bX^{\top}\btheta_0)}{\Phi(\bX^{\top}\btheta_0)\left[1-\Phi(\bX^{\top}\btheta_0)\right]}\left[\Phi(\bX^{\top}\btheta_0)-Y\right].
$$
Iterating expectations, we see that for all $j\in[p]$,
\begin{align*}
\E[|UX_j|^2]
& = \E\left[\frac{\phi(\bX^{\top}\btheta_0)^2}{\Phi(\bX^{\top}\btheta_0)^2\left[1-\Phi(\bX^{\top}\btheta_0)\right]^2}|X_j|^2\E\left[\left|Y-\Phi(\bX^{\top}\btheta_0)\right|^2\middle|\bX\right]\right] \\
& = \E\left[\frac{\phi(\bX^{\top}\btheta_0)^2}{\Phi(\bX^{\top}\btheta_0)\left[1-\Phi(\bX^{\top}\btheta_0)\right]}|X_j|^2\right] \\
& \geqslant c\E\left[|X_j|^2\mathbf{1}\left\{|\bX^{\top}\btheta_0|\leqslant C\right\}\right]
= c\left( \E[|X_j|^2] - \E\left[|X_j|^2\mathbf 1\left\{|\bX^{\top}\btheta_0|>C\right\}\right] \right)\\
& \geqslant c\left(c_{ev} - \E\left[|X_j|^2|\bX^{\top}\btheta_0|^2\right]/C^2\right)
\geqslant c\left(c_{ev} - \left( \E[|X_j|^4]\E\left[|\bX^{\top}\btheta_0|^4\right] \right)^{1/2}/C^2\right) \\
& \geqslant c\left( c_{ev} - C_0^4/C^2 \right) = cc_{ev}/2,
\end{align*}
by \eqref{eq: verification of high level assumptions key condition}, \eqref{eq: verification of high
level assumptions key condition 2}, the Cauchy-Schwarz and H{\"o}lder inequalities, and the choice
of $C$. The final inequality in the previous display shows that the lower bound in Assumption
\ref{assu:ResidualBootstrapMethod}.\ref{enu:ZijSecondMomentsBndAwayZero} is satisfied with $c_U =
c_U(C_0,c_{ev})$. The remaining parts of Assumption \ref{assu:ResidualBootstrapMethod} follow from
Lemma \ref{lem: verification of high level assumptions section 4} with $C_U = C_U(C_0)$, and
$\widetilde B_n = C_1\bar B_n \lor C_2$, where $C_1\in[1,\infty)$ is a universal constant and
$C_2=C_2(C_0)$. This observation completes the probit case and, thus, the example.
\qed

\medskip
\noindent
\textbf{Example \ref{exa:LogConcaveOrderedResponse} (Ordered Response Model, Continued).} 
We first consider the \emph{logit} case. Here the CDF $F$ is of the logistic form, $F(t) =
\Lambda(t) = 1/(1+\mathrm{e}^{-t})$, which has PDF
$\Lambda'(t)=\Lambda(t)[1-\Lambda(t)]=\mathrm{e}^{-t}/(1+\mathrm{e}^{-t})^2$. In this case, the
differentiability part of Assumption \ref{as: diff and int} is trivial. Also, for all $t_1,t_2\in\R$
such that $t_1<t_2$, by the Mean Value Theorem, there is a $\tau\in(t_1,t_2)$ such that
\begin{align}
|\ln(\Lambda(t_2) - \Lambda(t_1))|
&=\ln\left(\frac{1}{\Lambda(t_2) - \Lambda(t_1)}\right)
=\ln\left(\frac{1}{\Lambda'(\tau)(t_2 - t_1)}\right)\nonumber\\
&=\ln\left(\frac{(1+\mathrm{e}^{\tau})^2}{\mathrm{e}^{\tau}}\right) - \ln(t_2 - t_1)\nonumber\\
&\leqslant 2\ln(1 + \mathrm{e}^{|\tau|}) + |\tau| -\ln(t_2 - t_1)\nonumber\\
&\leqslant 2\ln 2 + 3(|t_1|\vee|t_2|) - \ln(t_2 - t_1).\label{eq: log of logistic CDF difference}
\end{align}
Moreover, since $1+\mathrm{e}^t\leqslant 2\mathrm{e}^t$ for $t\geqslant0$, for any $t\in\R$, we have
\begin{align}
\left|\ln(\Lambda(t))\right|\lor\left|\ln(1-\Lambda(t))\right|\leqslant \ln 2 + |t|.\label{eq: log of logistic CDF}
\end{align}
Hence, for all $\btheta\in\Theta$, using the bound \eqref{eq: log of logistic CDF difference} to
control the terms with $v\in[V-1]$ and \eqref{eq: log of logistic CDF} for the 
end cases $v\in\{0,V\}$, from \eqref{eq: verification of high level assumptions key condition}, we
see that
\begin{align*}
\E\left[\left|m(\bX^{\top}\btheta,Y)\right|\right]
& \leqslant \sum_{v=0}^V \E\left[\left|\ln(\Lambda(\alpha_{v+1} - \bX^{\top}\btheta) - \Lambda(\alpha_v - \bX^{\top}\btheta))\right|\right] < \infty.
\end{align*}
(Recall that we interpret $\Lambda(+\infty)$ as one and $\Lambda(-\infty)$ as zero.) The previous display
shows the integrability part of Assumption \ref{as: diff and int}.

For each $t\in\R$, $m(t,Y)$ is integrable. Hence, for any $\bx\in\mathcal X$, $t\mapsto f(t,\bx)
:=\E[m(t,Y)|\bX = \bx]$ is a well-defined function from $\R$ to $\R$, given by
\begin{align*}
f(t,\bx) & = -\sum_{v=0}^V \left[\Lambda(\alpha_{v+1} - \bx^{\top}\btheta_0) - \Lambda(\alpha_{v} - \bx^{\top}\btheta_0)\right]\ln\left(\Lambda(\alpha_{v+1} - t) - \Lambda(\alpha_{v} - t)\right).
\end{align*}
We see that $t\mapsto f(t,\bx)$ is twice continuously differentiable, with derivatives
\begin{align*}
f_1'(t,\bx) &= \sum_{v=0}^V \left[\Lambda(\alpha_{v+1} - \bx^{\top}\btheta_0) - \Lambda(\alpha_{v} - \bx^{\top}\btheta_0)\right]\left[1 - \Lambda(\alpha_{v+1} - t) - \Lambda(\alpha_{v} - t)\right]\quad\text{and}\\
f_{11}''(t,\bx)  &= \sum_{v=0}^V \left[\Lambda(\alpha_{v+1} - \bx^{\top}\btheta_0) - \Lambda(\alpha_{v} - \bx^{\top}\btheta_0)\right]\left[\Lambda'(\alpha_{v+1} - t)+\Lambda'(\alpha_{v} - t)\right].
\end{align*}
Thus, for each $\btheta\in\Theta$ and $j\in[p]$, we have
$\E[|f_1'(\bX^{\top}\btheta,\bX)X_j|]\leqslant \E[|X_j|] < \infty$ by \eqref{eq: verification of
high level assumptions key condition 2}. Also, for arbitrary $C\in(0,\infty)$, gathering the cut-off
points in $\balpha:=(\alpha_1,\dots,\alpha_V)^\top$ and defining
\[
   c_f(\balpha,C):=\frac{1}{4}\inf_{|t|\leqslant C}\min_{v\in\{0,1,\dotsc,V\}}\left\{\Lambda'(\alpha_{v+1} - t)+\Lambda'(\alpha_{v} - t)\right\}\in(0,\infty),
\]
we see that, for all $(t,\bx)\in[-C,C]\times\mathcal X$,
\begin{align*}
   f_{11}''(t,\bx)
   &\geqslant \sum_{v=0}^V \left[\Lambda(\alpha_{v+1} - \bx^{\top}\btheta_0) - \Lambda(\alpha_{v} - \bx^{\top}\btheta_0)\right]\left[\Lambda'(\alpha_{v+1} - t)+\Lambda'(\alpha_{v} - t)\right]\\
   &\geqslant 4c_f(\balpha,C)\times \sum_{v=0}^V \left[\Lambda(\alpha_{v+1} - \bx^{\top}\btheta_0) - \Lambda(\alpha_{v} - \bx^{\top}\btheta_0)\right]=4c_f(\balpha,C),
\end{align*}
where the equality follows from the probability differences summing to one. Choosing the specific
$C:=C(C_0,c_{ev},C_{ev},c_M')$ in \eqref{eq: verifying assumption 3.4 key
condition} and $c_M':=1$, by Lemma \ref{lem: verification of margin assumption} with the implied
$c_f:=c_f(\balpha,C)$, we see that Assumption \ref{assu:Margin} holds with $c_M' = 1$ and $c_M =
c_M(\balpha,C_0,c_{ev},C_{ev})$.

Next, differentiating thrice shows that, for all $(t,y)\in\R\times\mathcal Y$,
\begin{align*}
m_1'(t,y) &= \sum_{v=0}^V \mathbf 1(y = v)\left[1 - \Lambda(\alpha_{v+1} - t) - \Lambda(\alpha_v - t)\right],\\
m_{11}''(t,y) &= \sum_{v=0}^V \mathbf 1(y = v)\left[\Lambda'(\alpha_{v+1} - t) +  \Lambda'(\alpha_{v} - t)\right]\quad\text{and}\\
m_{111}'''(t,y) &=- \sum_{v=0}^V \mathbf 1(y = v)\big(\Lambda'(\alpha_{v+1} - t)\left[1-2\Lambda(\alpha_{v+1} - t)\right]+\Lambda'(\alpha_{v} - t)\left[1-2\Lambda(\alpha_{v} - t)\right]\big),
\end{align*}
from which we deduce that $|m_1'(t,y)|\leqslant 1$, $|m_{11}''(t,y)|\leqslant 2$, and
$|m_{111}'''(t,y)|\leqslant 2$ for all $(t,y)\in\R\times\mathcal Y$. Hence, \eqref{eq: bounded
derivative lemma} and \eqref{eq: lipschitz derivative lemma} are satisfied with $c_{m,1} = 1$,
$c_{m,2} = 0$, and $c_{m,3} = 2$, and so Assumption \ref{assu:LossLocallyLipschitzAndMore} follows
from Lemma \ref{lem: verification of some high level assumptions} with $L(\bx,y) = 2(1+|
\bx^{\top}\btheta_0 |)$ for $(\bx,y)\in\mathcal X\times\mathcal Y$, $c_L=1$, $r=\bar r$, $B_n =
2\bar B_n$, and $C_L = C_L(C_0,C_{ev})$. Also, Assumption \ref{as: smoothness inference} is
satisfied with $J=1$ and $C_m =2$, and Assumption \ref{as: conditional density} holds
trivially (because $J=1$). Moreover, Assumption \ref{as: integrability inference} follows from Lemma
\ref{lem: verification of high level assumptions section 5} with $\widetilde r = 8$ and $C_M =
C_M(C_0)$.

For arbitrary $C\in(0,\infty)$, defining
\[
   c(\balpha,C):=\inf_{|t|\leqslant C}\left\{\Lambda(\alpha_{1} - t)\left[1 - \Lambda(\alpha_{1} - t)\right]^2\right\}\in(0,\infty),
\]
since all summands are non-negative, we must have
\begin{align*}
&\inf_{|t|\leqslant C} \sum_{v=0}^V \left[\Lambda(\alpha_{v+1} - t) - \Lambda(\alpha_{v} - t)\right]\left[1 - \Lambda(\alpha_{v+1} - t) - \Lambda(\alpha_{v} - t)\right]^2\\
&\geqslant \inf_{|t|\leqslant C} \left[\Lambda(\alpha_{1} - t) - \Lambda(\alpha_{0} - t)\right]\left[1 - \Lambda(\alpha_{1} - t) - \Lambda(\alpha_{0} - t)\right]^2 = c(\balpha,C),
\end{align*}
where we have used that $\alpha_0=-\infty$ implies $\Lambda(\alpha_{0} - t) = 0$ for all $t\in\R$. For
this example, 
\begin{align*}
   U=m_1'(\bX^\top\btheta_0,Y)=\sum_{v=0}^V \mathbf{1}\left(Y=v\right)\left[1 - \Lambda(\alpha_{v+1} - \bX^{\top}\btheta_0) - \Lambda(\alpha_{v} - \bX^{\top}\btheta_0)\right],
\end{align*}
such that its square
\begin{align*}
   U^2=\sum_{v=0}^V \mathbf{1}\left(Y=v\right)\left[1 - \Lambda(\alpha_{v+1} - \bX^{\top}\btheta_0) - \Lambda(\alpha_{v} - \bX^{\top}\btheta_0)\right]^2.
\end{align*}
Hence, for the specific $C:=C_0^2/\sqrt{c_{ev}/2}$ and the implied $c:=c(\balpha,C)=c(\balpha,C_0,c_{ev})$, for all
$j\in[p]$, by iterated expectations, \eqref{eq: verification of high level assumptions key
condition}, \eqref{eq: verification of high level assumptions key condition 2}, and the
Cauchy-Schwarz inequality, we get
\begin{align*}
\E[|UX_j|^2]
& = \E\Bigg[|X_j|^2 \sum_{v=0}^V \left[\Lambda(\alpha_{v+1} - \bX^{\top}\btheta_0) - \Lambda(\alpha_{v} - \bX^{\top}\btheta_0)\right]\\
&\qquad\qquad\qquad\times \left[1 - \Lambda(\alpha_{v+1} - \bX^{\top}\btheta_0) - \Lambda(\alpha_{v} - \bX^{\top}\btheta_0)\right]^2 \Bigg] \\
& \geqslant c\E\big[|X_j|^2\mathbf{1}\left\{|\bX^{\top}\btheta_0|\leqslant C\right\}\big]
= c\left( \E[|X_j|^2] - \E\left[|X_j|^2\mathbf 1\left\{|\bX^{\top}\btheta_0|>C\right\}\right] \right)\\
& \geqslant c\left(c_{ev} - \E\left[|X_j|^2|\bX^{\top}\btheta_0|^2\right]/C^2\right)
\geqslant c\left(c_{ev} - \left( \E[|X_j|^4]\E[|\bX^{\top}\btheta_0|^4] \right)^{1/2}/C^2\right) \\
& \geqslant c\left( c_{ev} - C_0^4/C^2 \right) = cc_{ev}/2.
\end{align*}
The final inequality provides the lower bound in Assumption
\ref{assu:ResidualBootstrapMethod}.\ref{enu:ZijSecondMomentsBndAwayZero} for $c_U =
c_U(\balpha,C_0,c_{ev})$. The remaining parts of Assumption \ref{assu:ResidualBootstrapMethod}
follow from Lemma \ref{lem: verification of high level assumptions section 4} with $C_U = C_U(C_0)$
and $\widetilde B_n = \bar B_n\lor C_1$ for $C_1=C_1(C_0)$. This observation completes the logit
case.

Next, we consider the \emph{probit} case. Here the CDF $F$ is of the standard normal form $F(t) =
\Phi(t)$. We first derive several basic inequalities for later reference.
First, we show that
\begin{equation}\label{eq: difference gauss tails 0}
|\ln\left(\Phi(t_2) - \Phi(t_1)\right)| \leqslant \frac{(|t_1|\vee|t_2|)^2}{2} - \ln\left(\frac{t_2 - t_1}{\sqrt{2\pi}}\right)\;\text{for}\;t_1,t_2\in\R\text{ such that }t_1<t_2.
\end{equation}
To this end, note that by the Mean Value Theorem, for some $\tau\in(t_1,t_2)$ we have
\begin{align*}
|\ln\left(\Phi(t_2) - \Phi(t_1)\right)| & = \ln\left(\frac{1}{\Phi(t_2) - \Phi(t_1)}\right) = \ln\left(\frac{1}{\phi(\tau)(t_2 - t_1)}\right) \\
& = \frac{\tau^2}{2} - \ln\left(\frac{t_2 - t_1}{\sqrt{2\pi}}\right) \leqslant \frac{(|t_1|\vee|t_2|)^2}{2} - \ln\left(\frac{t_2 - t_1}{\sqrt{2\pi}}\right),
\end{align*}
where we have inserted $\phi(t)=(2\pi)^{-1/2}\mathrm{e}^{-t^2/2}$. The inequality \eqref{eq:
difference gauss tails 0} follows.

Second, we show that
\begin{equation}\label{eq: difference gauss tails 1}
\frac{\phi(t_1) - \phi(t_2)}{\Phi(t_2) - \Phi(t_1)} \geqslant t_1\;\text{for}\; t_1,t_2\in\R\text{ such that }t_1<t_2.
\end{equation}
To do so, fix any $t_1\in \R$ and let $g\colon(t_1,\infty)\to\R$ be the function defined by
\begin{equation}\label{eq: function g verification example 2}
g(t):=\frac{\phi(t_1) - \phi(t)}{\Phi(t) - \Phi(t_1)},\quad t\in(t_1,\infty).
\end{equation}
By L'H{\^o}pital's rule, we see that 
\begin{equation}\label{eq: lower tail function g}
\lim_{t\downarrow t_1}g(t) = t_1,
\end{equation}
which allows us to continuously extend the domain of $g$ to include $t_1$. Henceforth, we therefore
interpret $g$ as the resulting function from $[t_1,\infty)$ to $\R$.

Suppose for the moment that $g'(t) \geqslant 0$ for all $t\in(t_1,\infty)$. By the Fundamental
Theorem of Calculus, for $\underline{t},t\in(t_1,\infty)$ satisfying $\underline{t}\leqslant t$, we must have
\[
g(t) - g(\underline{t}) = \int_{\underline{t}}^{t}g'(u)\,\mathrm{d}u \geqslant \int_{\underline{t}}^{t}0\,\mathrm{d}u = t - \underline{t} \geqslant 0.
\]
Rearranging this inequality to get $g(t)\geqslant g(\underline{t})$ and taking limits as
$\underline{t}\downarrow t_1$ gives $g(t) \geqslant g(t_1)=t_1$. To establish \eqref{eq: difference
gauss tails 1}, it thus suffices to show that $g'(t) \geqslant 0$ for all $t\in(t_1,\infty)$.
Differentiating $g$ shows
\[
g'(t) = \frac{\phi(t)}{\Phi(t) - \Phi(t_1)}\left[t-g(t)\right]\;\text{for}\;t\in(t_1,\infty),
\]
so, on $(t_1,\infty)$, $g'(t) \geqslant 0$ if and only if $t\geqslant g(t)$. To show the latter
claim, Let $f:[t_1,\infty)\to\R$ be the function defined by $f(t):=g(t)-t$ for $t\in[t_1,\infty)$.
Then $f$ is continuous, and $f(t_1) = 0$. Seeking a contradiction, suppose that there is a
$t_2\in(t_1,\infty)$ such that $f(t_2) > 0$. Let $\mathcal T:=\{t\in[t_1,t_2];f(t)=0\}$ be the roots
of $f$ in $[t_1,t_2]$, which includes at least $t_1$. Per continuity of $f$, $\mathcal T$ is closed,
so $\overline{t}:=\sup\mathcal T$ lies in $\mathcal T$. Since $f(t_2)>0$ by supposition, per
continuity of $f$ and definition of $\overline{t}$, we must have $f(t) > 0$ for all
$t\in(\overline{t},t_2)$. The Mean Value Theorem then implies that there is a $\tilde
t\in(\overline{t},t_2)$ such that
\[
f'(\tilde t) = \frac{f(t_2) - f(\overline{t})}{t_2 - \overline{t}} = \frac{f(t_2)}{t_2 - \overline{t}} > 0.
\]
On the other hand, for any such mean value $\tilde t$, differentiation and $f(\tilde t)>0$ imply that
\[
f'(\tilde t) = g'(\tilde t) - 1 = \frac{\phi(\tilde t)}{\Phi(\tilde t) - \Phi(t_1)}\left[-f(\tilde t)\right] - 1 < 0,
\]
a contradiction. Hence, it must be that $f(t) \leqslant 0$ for all $t\in(t_1,\infty)$, which means that $g(t) \leqslant
t$ for all $t\in(t_1,\infty)$, as desired. This observation completes the proof of \eqref{eq:
difference gauss tails 1}.

Third, we show that
\begin{equation}\label{eq: difference gauss tails 2}
\frac{t_2\phi(t_2) - t_1\phi(t_1)}{\Phi(t_2) - \Phi(t_1)} + \left( \frac{\phi(t_2) - \phi(t_1)}{\Phi(t_2) - \Phi(t_1)} \right)^2 > 0,\quad\text{for all }t_1,t_2\in \R\text{ such that }t_1 < t_2.
\end{equation}
To do so, we consider the three possible cases: (i) $t_2 > t_1 \geqslant 0$, (ii) $t_1 < t_2
\leqslant 0$, and (iii) $t_1 < 0 < t_2$. In the first case, we have
\begin{align*}
\frac{t_2\phi(t_2) - t_1\phi(t_1)}{\Phi(t_2) - \Phi(t_1)} + \left( \frac{\phi(t_2) - \phi(t_1)}{\Phi(t_2) - \Phi(t_1)} \right)^2
& > \frac{t_1\phi(t_2) - t_1\phi(t_1)}{\Phi(t_2) - \Phi(t_1)} + \left( \frac{\phi(t_2) - \phi(t_1)}{\Phi(t_2) - \Phi(t_1)} \right)^2 \\
& = \frac{\phi(t_1) - \phi(t_2)}{\Phi(t_2) - \Phi(t_1)}\left( \frac{\phi(t_1) - \phi(t_2)}{\Phi(t_2) - \Phi(t_1)} - t_1 \right) \geqslant 0
\end{align*}
by \eqref{eq: difference gauss tails 1}. In the second case, we have
$$
\frac{t_2\phi(t_2) - t_1\phi(t_1)}{\Phi(t_2) - \Phi(t_1)} + \left( \frac{\phi(t_2) - \phi(t_1)}{\Phi(t_2) - \Phi(t_1)} \right)^2 
= \frac{|t_1|\phi(|t_1|) - |t_2|\phi(|t_2|)}{\Phi(|t_1|) - \Phi(|t_2|)} + \left( \frac{\phi(|t_1|) - \phi(|t_2|)}{\Phi(|t_1|) - \Phi(|t_2|)} \right)^2 > 0,
$$
arguing as in the first case and using $|t_1| > |t_2|\geqslant 0$. In the third case, \eqref{eq:
difference gauss tails 2} is immediate from the first fraction being positive. Since we have
considered all cases, \eqref{eq: difference gauss tails 2} follows.

Fourth, we show that
\begin{equation}\label{eq: difference gauss tails 3}
\frac{\left| \phi(t_2) - \phi(t_1) \right|}{\Phi(t_2) - \Phi(t_1)} \leqslant \frac{\phi(0)}{1 - \Phi(1)} + 2\left(|t_1|\wedge |t_2|\right),\quad \text{for all }t_1,t_2\in\R\text{ such that }t_1<t_2.
\end{equation}
To do so, we again consider the three possible cases: (i) $t_2 > t_1 \geqslant 0$, (ii) $t_1 < t_2
\leqslant 0$, and (iii) $t_1 < 0 < t_2$. In the first case, we have
$$
\frac{\left| \phi(t_2) - \phi(t_1) \right|}{\Phi(t_2) - \Phi(t_1)} = \frac{\phi(t_1) - \phi(t_2)}{\Phi(t_2) - \Phi(t_1)} \leqslant \frac{\phi(t_1)}{1 - \Phi(t_1)} \leqslant  \frac{\phi(0)}{1 - \Phi(1)} + 2\left(|t_1|\wedge |t_2|\right),
$$
where the first inequality follows from $t_2\mapsto g(t_2)$ in \eqref{eq: function g verification
example 2} being non-decreasing on $[t_1,\infty)$ and taking the limit as $t_2\to\infty$, and the
second from \eqref{eq: gauss 3}. In the second case, we have
$$
\frac{\left| \phi(t_2) - \phi(t_1) \right|}{\Phi(t_2) - \Phi(t_1)} = \frac{\phi(|t_2|) - \phi(|t_1|)}{\Phi(|t_1|) - \Phi(|t_2|)} \leqslant \frac{\phi(|t_2|)}{1 - \Phi(|t_2|)} \leqslant  \frac{\phi(0)}{1 - \Phi(1)} + 2\left(|t_1|\wedge |t_2|\right)
$$
by the same arguments. Consider now the third case, where $t_1 < 0 < t_2$. If $-t_1=t_2$, then
the left-hand side of \eqref{eq: difference gauss tails 3} is zero per symmetry of the standard
normal distribution about zero, and the bound is trivial. The remaining two subcases $-t_1<t_2$ or
$-t_1>t_2$, can be handled exactly as in the above. Since we have considered all cases, \eqref{eq:
difference gauss tails 3} follows.

Fifth, we show that for any $\Delta\in(0,\infty)$,
\begin{equation}\label{eq: difference gauss tails 4}
\lim_{t\to\infty} t^2\left( \frac{\phi(t) - \phi(t+\Delta)}{\Phi(t+\Delta) - \Phi(t)} - t - \frac{1}{t} \right) = 0.
\end{equation}
To do so, fix any $\Delta\in(0,\infty)$ and observe that for all $t\in[1,\infty)$, we have
$t\leqslant \phi(t)/[1-\Phi(t)]\leqslant 2t$ by Proposition 2.5(b) in \cite{dudley2014uniform}. Hence,
again for all $t\in[1,\infty)$,
$$
\frac{1 - \Phi(t + \Delta)}{1 - \Phi(t)} = \frac{1 - \Phi(t+\Delta)}{\phi(t+\Delta)}\cdot\frac{\phi(t)}{1 - \Phi(t)}\cdot\frac{\phi(t+\Delta)}{\phi(t)}\leqslant \frac{1}{t+\Delta}\cdot2t\cdot\frac{\phi(t+\Delta)}{\phi(t)},
$$
such that, given that $\phi(t+\Delta) / \phi(t) = o(1/t^3)$ as $t\to\infty$, we get
$$
\frac{1 - \Phi(t + \Delta)}{1 - \Phi(t)} = o(1/t^3)\quad\text{as }t\to\infty.
$$
It follows that
$$
\left| \frac{\phi(t) - \phi(t+\Delta)}{\Phi(t+\Delta) - \Phi(t)} - \frac{\phi(t)}{1 - \Phi(t)} \right|
= \frac{\phi(t)}{1 - \Phi(t)}\left|\frac{1 - \frac{\phi(t+\Delta)}{\phi(t)}}{1 - \frac{1 - \Phi(t+\Delta)}{1 - \Phi(t)}} - 1\right| = o(1/t^2)\quad\text{as }t\to\infty,
$$
which in combination with \eqref{eq: gauss 6} gives \eqref{eq: difference gauss tails 4}.

Sixth, we show that for any $\Delta\in(0,\infty)$,
\begin{equation}\label{eq: gauss tails million}
\frac{t\Delta\phi(t+\Delta)}{\Phi(t+\Delta) - \Phi(t)} \to0\quad\text{as }t\to\infty.
\end{equation}
To do so, observe that by a change of variables, for $t\in(0,\infty)$ we have
$$
\frac{\Phi(t+\Delta) - \Phi(t)}{\phi(t+\Delta)} = \int_0^{\Delta} \frac{\phi(t+s)}{\phi(t+\Delta)}\mathrm{d}s \geqslant \int_0^{\Delta} \mathrm{e}^{t\Delta - ts}\mathrm{d}s = \int_0^{\Delta} \mathrm{e}^{ts}\mathrm{d}s = \frac{\mathrm{e}^{t\Delta} - 1}{t}.
$$
The claim \eqref{eq: gauss tails million} follows from rearranging, multiplying by $t\Delta$, and
taking the limit as $t\to\infty$.

With these bounds in mind, we now verify the required assumptions. (Throughout, we continue to
interpret $\Phi(+\infty)$ as one and $\Phi(-\infty)$, $\phi(-\infty)$ and $\phi(+\infty)$ as zero.)
The differentiability part of Assumption \ref{as: diff and int} is immediate. In addition, for all
$\btheta\in\Theta$, we have
$$
\E\left[\left|m(\bX^{\top}\btheta,Y)\right|\right] \leqslant \sum_{v=0}^V \E\left[\left|\ln\left(\Phi(\alpha_{v+1} - \bX^{\top}\btheta) - \Phi(\alpha_{v} - \bX^{\top}\btheta)\right)\right|\right] <\infty
$$
where we use \eqref{eq: gauss 1} and \eqref{eq: gauss 2} for the cases $v\in\{0,V\}$
and \eqref{eq: difference gauss tails 0} for the cases $v\in[V-1]$, and \eqref{eq:
verification of high level assumptions key condition}. The integrability part of
Assumption \ref{as: diff and int} follows.

Next, for all $t\in\R$, $\bx\in\mathcal X$ and $v\in\{0,1,\dots,V\}$, denote
$$
\psi_v(t): = \frac{\phi(\alpha_v - t) - \phi(\alpha_{v+1} - t)}{\Phi(\alpha_{v+1} - t) - \Phi(\alpha_{v} - t)}\quad\text{and}\quad \kappa_v(\bx) := \Phi(\alpha_{v+1} - \bx^{\top}\btheta_0) -  \Phi(\alpha_{v} - \bx^{\top}\btheta_0),
$$
so that
\begin{equation}\label{eq: psi limit simplify}
\lim_{t\to-\infty}(\alpha_v - t)^2\left(\psi_v(t) - (\alpha_v-t) - \frac{1}{\alpha_v - t}\right)=0
\end{equation}
for all $v\in[V-1]$ by \eqref{eq: difference gauss tails 4}. For each $t\in\R$, $m(t,Y)$ is
integrable. Hence, for any $\bx\in\mathcal X$, $t\mapsto f(t,\bx) :=\E[m(t,Y)|\bX = \bx]$ is a
well-defined function from $\R$ to $\R$, namely
\begin{align*}
f(t,\bx) = -\sum_{v=0}^V \kappa_v(\bx)\ln\left(\Phi(\alpha_{v+1} - t) - \Phi(\alpha_v - t)\right).
\end{align*}
We see that $t\mapsto f(t,\bx)$ is twice continuously differentiable, with derivatives
\begin{align*}
f_1'(t,\bx) &=  \sum_{v=0}^V \kappa_v(\bx)\left[-\psi_v(t)\right]\quad\text{and}\\
f_{11}''(t,\bx) &= \sum_{v=0}^V\kappa_v(\bx)\left\{ \frac{(\alpha_{v+1} - t)\phi(\alpha_{v+1} - t) - (\alpha_v - t)\phi(\alpha_v - t)}{\Phi(\alpha_{v+1} - t) - \Phi(\alpha_v - t)} + \psi_v(t)^2 \right\},
\end{align*}
where the terms $(\alpha_{V+1} - t)\phi(\alpha_{V+1} - t)$ and $(\alpha_0 - t)\phi(\alpha_0 - t)$
are understood to be zero for all $t\in\R$. Thus, for all $\btheta\in\Theta$ and $j\in[p]$, we have 
$$
\E\left[\left|f_1'(\bX^{\top}\btheta,\bX)X_j\right|\right]\leqslant \E\left[\left|X_j\right|\sum_{v=0}^V \left|\psi_v(\bX^\top\btheta_0)\right|\right] < \infty
$$ 
by \eqref{eq: difference gauss tails 3}, \eqref{eq: verification of high level assumptions key
condition}, \eqref{eq: verification of high level assumptions key condition 2}, and the
Cauchy-Schwarz inequality. Also, by \eqref{eq: gauss 4} and \eqref{eq: difference gauss tails 2},
for arbitrary $C\in(0,\infty)$, defining
\[
c_f(\balpha,C) := \frac{1}{4}\min_{v\in\{0,1,\dotsc,V\}}\inf_{|t|\leqslant C}\left\{ \frac{(\alpha_{v+1} - t)\phi(\alpha_{v+1} - t) - (\alpha_v - t)\phi(\alpha_v - t)}{\Phi(\alpha_{v+1} - t) - \Phi(\alpha_v - t)} + \psi_v(t)^2 \right\},
\]
we have $c_f(\balpha,C)\in(0,\infty)$ and $f_{11}''(t,\bx)\geqslant 4c_f(\balpha,C)$ for all
$(t,\bx)\in[-C,C]\times\mathcal X$. Choosing the specific $C:=C(C_0,c_{ev},C_{ev},c_M')$
in \eqref{eq: verifying assumption 3.4 key condition}, with $c_M':=1$ and the implied
$c_f:=c_f(\balpha,C)=c_f(\balpha,C_0,c_{ev},C_{ev})$, Lemma \ref{lem: verification of margin
assumption} now shows that Assumption \ref{assu:Margin} holds with $c_M' = 1$ and $c_M =
c_M(\balpha,C_0,c_{ev},C_{ev})$.

Next, differentiating thrice shows that for all $(t,y)\in\R\times\mathcal Y$,
\begin{align*}
m_1'(t,y) &= \sum_{v=0}^V \mathbf 1(y = v)\left[-\psi_v(t)\right],\\
m_{11}''(t,y) &= \sum_{v=0}^V \mathbf 1(y = v) \left\{ \frac{(\alpha_{v+1} - t)\phi(\alpha_{v+1}-t) - (\alpha_{v} - t)\phi(\alpha_{v} - t)}{\Phi(\alpha_{v+1} - t) - \Phi(\alpha_v - t)} + \psi_v(t)^2 \right\}\quad\text{and}\\
m_{111}'''(t,y)  &= \sum_{v=0}^V\mathbf 1(y = v)\bigg\{\psi_v(t)  + \frac{(\alpha_{v+1} - t)^2\phi(\alpha_{v+1} - t) - (\alpha_v - t)^2\phi(\alpha_v - t)}{\Phi(\alpha_{v+1} - t) - \Phi(\alpha_v - t)} \\
& \qquad\qquad\qquad -\psi_v(t) \frac{(\alpha_{v+1} - t)\phi(\alpha_{v+1} - t) - (\alpha_{v} - t)\phi(\alpha_v - t)}{\Phi(\alpha_{v+1} - t) - \Phi(\alpha_v - t)} \\
& \qquad\qquad\qquad - 2\psi_v(t) \bigg[ \psi_v(t)^2  + \frac{(\alpha_{v+1} - t)\phi(\alpha_{v+1} - t) - (\alpha_{v} - t)\phi(\alpha_v - t)}{\Phi(\alpha_{v+1} - t) - \Phi(\alpha_v - t)}  \bigg] \bigg\},
\end{align*}
where the terms $(\alpha_{V+1} - t)^2\phi(\alpha_{V+1} - t)$, $(\alpha_{V+1} - t)\phi(\alpha_{V+1} -
t)$, $(\alpha_0 - t)^2\phi(\alpha_0 - t)$, and $(\alpha_0 - t)\phi(\alpha_0 - t)$ are all understood
to be zero for all $t\in\R$. Therefore, for all $(t,y)\in\R\times\mathcal Y$,
$
|m_1'(t,y)|\leqslant \phi(0)/[1 - \Phi(1)] + 2\max_{v\in[V]}|\alpha_v| + 2|t|
$
by \eqref{eq: gauss 3} and \eqref{eq: difference gauss tails 3}. Also, for all $t\in
(-\infty,\alpha_1]$, using the Mean Value Theorem for the denominator, we see that
$$
\frac{(\alpha_{v+1} - \alpha_v)\phi(\alpha_{v+1}-t)}{\Phi(\alpha_{v+1} - t) - \Phi(\alpha_v - t)} \leqslant 1,\quad \text{for all}\quad v\in[V-1],
$$
and applying this upper bound to the terms $v\in[V-1]$, for all $t\in
(-\infty,\alpha_1]$, we get
\begin{align*}
|m_{11}''(t,y)| & \leqslant 
\mathbf 1(y = 0)\frac{\phi(\alpha_1 - t)}{\Phi(\alpha_1 - t)}\left| \frac{\phi(\alpha_1 - t)}{\Phi(\alpha_1 - t)} + (\alpha_1 - t) \right| \\
& \qquad + \sum_{v=1}^{V-1} \mathbf 1(y = v)\big(1 + |\psi_v(t)|\cdot\left|\psi_v(t) - (\alpha_v - t)\right|\big) \\
& \qquad + \mathbf 1(y = V)\frac{\phi(\alpha_V - t)}{1 - \Phi(\alpha_V - t)}\left| \frac{\phi(\alpha_V - t)}{1 - \Phi(\alpha_V - t)} - (\alpha_V -t ) \right|.
\end{align*}
Hence, we have $\max_{y\in\mathcal Y}|m_{11}''(t,y)| = O(1)$ as $t\to-\infty$ by \eqref{eq: gauss 5}
(for the cases $v\in\{0,V\}$) and \eqref{eq: psi limit simplify} (for the cases
$v\in[V-1]$). Also, $\max_{y\in\mathcal Y}|m''_{11}(t,y)| = O(1)$ as $t\to\infty$ by a
similar argument. Thus, there is a constant $C(\balpha)\in(0,\infty)$ such
that for all $(t,y)\in\R\times\mathcal Y$, we have
$
|m_{11}''(t,y)| \leqslant C(\balpha).
$
Hence, \eqref{eq: bounded derivative lemma} and \eqref{eq: lipschitz derivative lemma} are satisfied
with $c_{m,1} = \phi(0)/[1 - \Phi(1)] + 2\max_{v\in[V]}|\alpha_v|$, $c_{m,2} = 2$,
and $c_{m,3} = C(\balpha)$, and so Assumption \ref{assu:LossLocallyLipschitzAndMore} follows from
Lemma \ref{lem: verification of some high level assumptions} with $L(\bx,y) = (3 +
\phi(0)/[1-\Phi(1)] + 2\max_{v\in[V}|\alpha_v|)(1+| \bx^{\top}\btheta_0 |)$ for
$(\bx,y)\in\mathcal X\times\mathcal Y$, $c_L=1$, $r=\bar r$, $B_n = (3 +
\phi(0)/[1-\Phi(1)]+2\max_{v\in[V]}|\alpha_v|)\bar B_n$, and $C_L =
C_L(\balpha,C_0,C_{ev})$. Further, for all $v\in[V-1]$, as $t\to-\infty$, 
\begin{equation}\label{eq: last resort 1}
\frac{\phi(\alpha_{v+1}-t)}{\Phi(\alpha_{v+1} - t) - \Phi(\alpha_v - t)} \to 0 \quad \text{and} \quad \frac{(\alpha_v - t)\phi(\alpha_{v+1}-t)}{\Phi(\alpha_{v+1} - t) - \Phi(\alpha_v - t)} \to 0
\end{equation}
by \eqref{eq: gauss tails million},
\begin{equation}\label{eq: last resort 2}
\frac{\psi_v(t)\phi(\alpha_{v+1}-t)}{\Phi(\alpha_{v+1} - t) - \Phi(\alpha_v - t)} \to 0
\end{equation}
by \eqref{eq: gauss tails million} and \eqref{eq: psi limit simplify}, Hence, uniformly over
$y\in\mathcal Y$,
$$
|m_{111}'''(t,y)| \leqslant o(1) + \sum_{v=1}^{V-1}|\psi_v(t)|\times \left| (\psi_v(t) - (\alpha_v - t))(2\psi_v(t) - (\alpha_v - t)) - 1 \right|
$$
as $t\to-\infty$, where the terms $v\in\{0,V\}$ are treated by applying the same arguments as
in the binary probit derivations from Example \ref{exa:LogConcaveOrderedResponse}, and the terms
$v\in[V-1]$ are handled by decomposing $\alpha_{v+1} - t = \alpha_{v+1} - \alpha_v + \alpha_v -
t$ and applying \eqref{eq: last resort 1} and \eqref{eq: last resort 2} to all the terms containing
$\alpha_{v+1} - \alpha_v$. Thus, $\max_{y\in\mathcal Y}|m_{111}'''(t,y)|\to 0$ as
$t\to-\infty$ by \eqref{eq: psi limit simplify} and
$
\max_{y\in\mathcal Y}|m_{111}'''(t,y)| \to 0
$
as $t\to\infty$ by a similar argument. Therefore, Assumption \ref{as: smoothness inference} is
satisfied with $J=1$ and some $C_m = C_m(\balpha)$, and Assumption \ref{as: conditional density}
follows trivially (from $J=1$). Moreover, Assumption \ref{as: integrability inference} holds
from Lemma \ref{lem: verification of high level assumptions section 5} with $\widetilde r = 8$ and
$C_M = C_M(\balpha,C_0)$. 

Finally, for arbitrary $C\in(0,\infty)$, define
$$
c(\balpha,C):=\inf_{|t|\leqslant C} \sum_{v=0}^V \frac{|\phi(\alpha_v - t) - \phi(\alpha_{v+1} - t)|^2}{\Phi(\alpha_{v+1} - t) - \Phi(\alpha_{v} - t)}\in(0,\infty).
$$
In this example and case, the residual takes the form
$$
U=m'_1(\bX^\top\btheta_0,Y)= \sum_{v=0}^V \mathbf 1(Y = v)\left[-\psi_v(\bX^\top\btheta_0)\right]
$$
so that its square
$$
U^2 = \sum_{v=0}^V \mathbf 1(Y = v)\psi_v(\bX^\top\btheta_0)^2,
$$
Consider now the specific $C:=C_0^2/\sqrt{c_{ev}/2}$ and the implied $c:=c(\balpha,C)=c(\balpha,C_0,c_{ev})$.
Then by iterated expectations, we get
\begin{align*}
\E\left[|UX_j|^2\right]
& = \E\left[|X_j|^2 \sum_{v=0}^V \kappa_v(\bX) \psi_v(\bX^\top\btheta_0)^2\right]\\
& = \E\left[|X_j|^2\sum_{v=0}^V \frac{\left|\phi(\alpha_v - \bX^{\top}\btheta_0) - \phi(\alpha_{v+1} - \bX^{\top}\btheta_0)\right|^2}{\Phi(\alpha_{v+1} - \bX^{\top}\btheta_0) - \Phi(\alpha_{v} - \bX^{\top}\btheta_0)}\right] \\
& \geqslant c\E\left[|X_j|^2\mathbf1\left\{|\bX^{\top}\btheta_0|\leqslant C\right\}\right]
 = c\left( \E[|X_j|^2] - \E\left[|X_j|^2\mathbf 1\left\{|\bX^{\top}\btheta_0|>C\right\}\right] \right)\\
& \geqslant c\left(c_{ev} - \E\Big[|X_j|^2|\bX^{\top}\btheta_0|^2\Big]/C^2\right)
 \geqslant c\left(c_{ev} - \Big( \E[|X_j|^4]\E[|\bX^{\top}\btheta_0|^4] \Big)^{1/2}/C^2\right) \\
& \geqslant c\left( c_{ev} - C_0^4/C^2 \right) = cc_{ev}/2,
\end{align*}
by \eqref{eq: verification of high level assumptions key condition}, \eqref{eq: verification of high
level assumptions key condition 2}, and the Cauchy-Schwarz inequality. The previous display shows
that the lower bound in Assumption
\ref{assu:ResidualBootstrapMethod}.\ref{enu:ZijSecondMomentsBndAwayZero} holds for
$c_U=c_U(\balpha,C_0,c_{ev})$. The remaining parts of Assumption \ref{assu:ResidualBootstrapMethod}
follow from Lemma \ref{lem: verification of high level assumptions section 4} with $C_U =
C_U(\balpha,C_0)$, and $\widetilde B_n = (C_1\bar B_n) \lor C_2$, where $C_1 = C_1(\balpha)\in[1,\infty)$ and $C_2 =
C_2(\balpha,C_0)$. This observation completes the probit case and, thus, the
example.
\qed

\medskip
\noindent
\textbf{Example \ref{exa:Expectile} (Expectile Model, Continued).}
The loss \eqref{eq: ALS} is continuously differentiable in $t$, thus implying the differentiability
part of Assumption \ref{as: diff and int}. In addition, for all $\btheta\in\Theta$, the basic
inequality $(a+b)^2\leqslant 2(a^2 + b^2)$, \eqref{eq: y square integrable example} and \eqref{eq:
verification of high level assumptions key condition} combine to show that
$$
\E[|m(\bX^{\top}\btheta,Y)|] \leqslant \E[|Y - \bX^{\top}\btheta|^2] \leqslant 2\E[Y^2] + 2\E[|\bX^{\top}\btheta|^2] <\infty,
$$
which shows also the integrability part of Assumption \ref{as: diff and int}.

Next, by integrability of $Y$ (implied by \eqref{eq: y square integrable example}),
for all $\bx\in\mathcal X$, $t\mapsto f(t,\bx) = \E[m(t,Y)\mid \bX = \bx]$
is a well-defined function from $\R$ to $\R$. Its partial derivative is
$$
f_1'(t,\bx) = 2\E[(1 - \tau)(t-Y)\mathbf 1(Y<t) + \tau(t-Y)\mathbf 1(Y\geqslant t) \mid \bX = \bx],\quad (t,\bx)\in\R\times\mathcal X.
$$
By the conditional continuity \eqref{eq: expectile Y conditionally continuous}, two applications of
Lemma \ref{lem: lipschitz derivative} (conditional on $\bX=\bx$, with $Z$ there set to $Y$ and, in
turn, $-Y$) show that this partial derivative is itself differentiable in $t$ with derivative
$$
f_{11}''(t,\bx) = 2(1-\tau)\P(Y < t\mid \bX = \bx) + 2\tau\P(Y > t\mid \bX = \bx),\quad (t,\bx)\in\R\times\mathcal X.
$$
Thus, for all $\btheta\in\Theta$ and $j\in[p]$, we have 
$$
\E[|f_1'(\bX^{\top}\btheta,\bX)X_j|]\leqslant 2\E[|X_j\bX^{\top}\btheta|] + 2\E[|YX_j|] < \infty
$$ 
by the law of iterated expectations, the Jensen and Cauchy-Schwarz inequalities, \eqref{eq:
verification of high level assumptions key condition}, \eqref{eq: verification of high level
assumptions key condition 2} and \eqref{eq: y square integrable example}. Using the conditional
continuity \eqref{eq: expectile Y conditionally continuous} again, we see that,
for all $(t,\bx)\in\R\times\mathcal X$,
$$
f_{11}''(t,\bx)\geqslant 2(\tau\wedge(1-\tau))\left[\P(Y < t\mid \bX = \bx)+\P(Y > t\mid \bX = \bx)\right]
= 2(\tau\wedge(1-\tau)).
$$
The previous display and Lemma \ref{lem: verification of margin assumption} with
$c_f:=(1/2)(\tau\wedge(1-\tau))$, combine to show that Assumption \ref{assu:Margin} holds
for $c_M = c_M(c_{ev},\tau)$ and any $c_M'\in(0,\infty)$.

Next, for all $(t,y)\in\R\times\mathcal Y$, we have
\begin{align*}
& m_1'(t,y) = \begin{cases}
2(1-\tau)(t-y), & \text{if }y<t,\\
2\tau(t-y), & \text{if }y\geqslant t,
\end{cases}
\end{align*}
which satisfies
$$
|m_1'(t,y)|\leqslant 2|y|+2|t|,\quad (t,y)\in\R\times\mathcal Y.
$$
Moreover, going case by case, we see that no matter the ordering, for all
$$
   |m_1'(t,y)-m_2'(t,y)| \leqslant 2|t_1-t_2|,\quad(t_1,t_2,y)\in\R^2\times\mathcal Y,
$$
i.e.~$t\mapsto m_1'(t,y)$ is Lipschitz continuous in $t$ with Lipschitz constant $2$ for all
$y\in\mathcal Y$. The previous two displays show that \eqref{eq: bounded derivative lemma} and
\eqref{eq: lipschitz derivative lemma} hold with $c_{m,1}(y) = 2|y|$, $c_{m,2} = 2$, and $c_{m,3} =
2$. While $c_{m,1}(y)=2|y|$ is no constant, from minor modifications to the proof of Lemma \ref{lem:
verification of some high level assumptions}, leveraging now the added \eqref{eq: y square
integrable example}, we deduce that Assumption \ref{assu:LossLocallyLipschitzAndMore} is satisfied
with $L(\bx,y) = 2(1 + |y|+| \bx^{\top}\btheta_0 |)$ (now depending on $y$) for $(\bx,y)\in\mathcal
X\times\mathcal Y$, $c_L=1$, $r=\bar r$, $B_n = 4\bar B_n$, and $C_L = C_L(C_0,C_{ev})$.

Next, for all $(t,y)\in\R\times\mathcal Y$, such that $t\neq y$, we have
$$
m_{11}''(t,y) =
\begin{cases}
2(1-\tau), & \text{if }y<t,\\
2\tau, & \text{if }y > t,
\end{cases}
\quad\text{and, thus,}\quad
m_{111}'''(t,y) =
\begin{cases}
0, & \text{if }y<t,\\
0, & \text{if }y > t.
\end{cases}
$$
The conditional continuity \eqref{eq: expectile Y conditionally continuous} now shows that
Assumption \ref{as: smoothness inference} is satisfied with $J=2$, $t_{y,1}=y$ and $C_m =2$. Under
the (stronger) bounded conditional density assumption \eqref{eq: y given x continuous example},
Assumption \ref{as: conditional density} is satisfied with any $\overline\Delta_n\in(0,\infty)$ and
$C_f = 2C_{pdf}$. Moreover, Assumption \ref{as: integrability inference} follows from \eqref{eq:
verification of high level assumptions key condition 2} and \eqref{eq: y square integrable example}
with $\widetilde r = 8$ and $C_M = C_M(C_0)$.

Finally, we have for all $j\in[p]$ that, by \eqref{eq: y square integrable example},
\begin{align*}
\E[|UX_j|^2]
& = \E[|m_1'(\bX^{\top}\btheta_0,Y)X_j|^2] \\
& \geqslant 4(\tau\wedge(1-\tau))^2 \E[(Y - \bX^{\top}\btheta_0)^2X_j^2] \geqslant 4(\tau\wedge(1-\tau))^2 c_{eps},
\end{align*}
showing that the lower bound in Assumption
\ref{assu:ResidualBootstrapMethod}.\ref{enu:ZijSecondMomentsBndAwayZero} is satisfied with $c_U =
c_U(c_{eps},\tau)$. While Lemma \ref{lem: verification of high level assumptions section 4} does not
apply directly (due to the non-constant $c_{m,1}(y)$), from minor modifications to the arguments
used in the proof of Lemma \ref{lem: verification of some high level assumptions}, leveraging the
added \eqref{eq: y square integrable example}, we deduce that the rest of Assumption
\ref{assu:ResidualBootstrapMethod} is satisfied with $C_U = C_U(C_0)$, and $\widetilde B_n = (C_1\bar
B_n)\lor C_2$, where $C_1 = C_1(C_0)\in[1,\infty)$ and $C_2 = C_2(C_0)$.
\qed

\medskip
\noindent
\textbf{Example \ref{exa:PanelCensoredRegressionAndTrimming} (Panel Censored Model, Continued).} 
We consider \emph{trimmed LS} and \emph{trimmed LAD}, in turn. For the former loss,
\eqref{eq:TrimmedLoss} takes the form
$$
m\left(t,\by\right)=\begin{cases}
y_1^2-2y_1(y_2 + t), &\text{if}\;t\in\left(-\infty,-y_2\right],\\
(y_{1}-y_{2}-t)^2, &\text{if}\; t\in\left(-y_{2},y_{1}\right),\\
y_2^2 + 2y_2(t-y_1), &\text{if}\; t\in\left[y_{1},\infty\right),
\end{cases}
$$
which is continuously differentiable in $t$ for each $\by\in\mathcal Y=[0,\infty)^2$. The
differentiability part of Assumption \ref{as: diff and int} follows. In addition, for all
$\btheta\in\Theta$, we have
$$
\E[|m(\bX^{\top}\btheta,\bY)|] \leqslant 9\E[Y_1^2] + 9\E[Y_2^2] + 9\E[|\bX^{\top}\btheta|^2] <\infty
$$
by \eqref{eq: y square integrable example 4} and \eqref{eq: verification of high level assumptions
key condition}, which gives the integrability part of Assumption \ref{as: diff and int}.

A similar argument shows integrability of $m(t,\bY)$ for each $t\in\R$. Hence, for all
$\bx\in\mathcal X$, $t\mapsto f(t,\bx)=\E[m(t,\bY)| \bX = \bx]$ is well-defined as a function
from $\R$ to $\R$. Direct calculation and case inspection show that it is differentiable with
derivative given by
$$
f_1'(t,\bx) = 
\begin{cases}
2\E[(Y_2 + t)\mathbf 1(Y_2 + t > 0) - Y_1\mid \bX = \bx], &\text{if }t\in(-\infty,0],\\
2\E[(t - Y_1)\mathbf 1(Y_1 - t > 0) + Y_2\mid \bX = \bx], &\text{if }t\in(0,\infty).
\end{cases}
$$
From the continuities in \eqref{eq: y1 and y2 conditionally continuous on positive reals} and two
applications of Lemma \ref{lem: lipschitz derivative} conditional on $\bX=\bx$, $t\mapsto
f_1'(t,\bx)$ is seen to be Lipschitz-continuous on $\R$, and differentiable on $\R\backslash\{0\}$
with
$$
f_{11}''(t,\bx) = 
\begin{cases}
2\P(Y_2 > -t\mid \bX = \bx), &\text{if }t\in(-\infty,0),\\
2\P(Y_1 > t\mid \bX = \bx), &\text{if }t\in(0,\infty).
\end{cases}
$$
Thus, for all $\btheta\in\Theta$ and $j\in[p]$, we have 
$
\E[|f_1'(\bX^{\top}\btheta,\bX)X_j|] < \infty
$ 
by the law of iterated expectations, the Cauchy-Schwarz inequality, \eqref{eq: verification of high
level assumptions key condition}, \eqref{eq: verification of high level assumptions key condition
2}, and \eqref{eq: y square integrable example 4}. Also, for $t\neq0$, we have
$$
f_{11}''(t,\bx) \geqslant 
\begin{cases}
2\P(Y_1 - Y_2 < t \mid \bX = \bx), &\text{if }t\in(-\infty,0),\\
2\P(Y_1 -Y_2 > t\mid \bX = \bx), &\text{if }t\in(0,\infty),
\end{cases}
$$
which implies that
\begin{align*}
\inf_{\substack{t\in[-c_{de},c_{de}]\\\bx^\top\btheta_0+t\neq0}}f_{11}''(\bx^{\top}\btheta_0 + t, \bx) 
\geqslant 2\min\Big\{&\P\left( Y_1 - Y_2 < \bx^{\top}\btheta_0 - c_{de} \mid \bX = \bx\right) ,\\ 
& \P\left( Y_1 - Y_2 > \bx^{\top}\btheta_0 + c_{de}  \mid \bX = \bx\right)\Big\}.
\end{align*}
Therefore, by \eqref{eq: nontrivial tail bound assumption}, we have that the probability statement
in \eqref{eq: verifying assumption 3.4 key condition 2} is satisfied with $C = c_{de}$, the provided
$c_f$, and $N(\bx)=\{0\}$ for all $\bx\in\mathcal X$. Adjusting $c_M'$ (which is here free to vary)
to match this $C$, we see that Assumption \ref{assu:Margin} follows from Lemma \ref{lem:
verification of margin assumption} with $c_M' = c_M'(c_{de},c_{ev},C_{ev})$ and $c_M =
c_M(c_{ev},c_f)$.

Next, for all $(t,\by)\in\R\times\mathcal Y$, we have
\begin{align*}
m_1'(t,\by) &= \begin{cases}
-2y_1, &\text{if}\; t\in(-\infty, -y_2],\\
2(t + y_2 - y_1), &\text{if}\; t\in(-y_2,y_1),\\
2y_2, &\text{if}\; t \in [y_1,\infty),
\end{cases}
\end{align*}
which implies both
$$
|m_1'(t,\by)|\leqslant 2(y_1\lor y_2)+2|t|,\quad\text{for all}\quad (t,\by)\in\R\times\mathcal Y.
$$
and 
$$
|m_1'(t_1,\by)-m_1'(t_2,\by)|\leqslant 2|t_1-t_2|,\quad\text {for all}\quad (t_1,t_2,\by)\in\R^2\times\mathcal Y.
$$
The previous two displays show that \eqref{eq: bounded derivative lemma} and \eqref{eq: lipschitz
derivative lemma} hold with $c_{m,1}(\by) = 2(y_1\lor y_2)$, $c_{m,2} = 2$, and $c_{m,3} = 2$. While
$c_{m,1}(\by)$ is no constant, with minor modifications to the proof of Lemma \ref{lem: verification
of some high level assumptions} and leveraging the added \eqref{eq: y square integrable example 4},
it follows that Assumption \ref{assu:LossLocallyLipschitzAndMore} is satisfied with $L(\bx,\by) =
2(1 + (y_1\lor y_2) +| \bx^{\top}\btheta_0 |)$ for $(\bx,\by)\in\mathcal X\times\mathcal Y$, $c_L=1$,
$r=\bar r$, $B_n = C\bar B_n$ for a universal constant $C\in[1,\infty)$, and $C_L =
C_L(C_0,C_{ev})$. 

If $(y_1,y_2)=(0,0)$, then the trimmed LS loss $m(\cdot,0,0)$ is identically zero. If
$(y_1,y_2)\neq (0,0)$, i.e.~at least one element in $(y_1,y_2)$ is positive, then $(-y_2,y_1)$ is non-empty, and
\begin{align*}
m_{11}''(t,\by) &= \begin{cases}
0, &\text{if}\; t\in(-\infty,-y_2),\\
2, &\text{if}\; t\in(-y_2,y_1),\\
0, &\text{if}\; t\in(y_1,\infty),
\end{cases}
\quad\text{and}\quad m_{111}'''(t,\by) = 0,\quad\text{if}\; t\notin\{-y_2,y_1\}.
\end{align*}
with the second (and thus the third) derivative being undefined at $t\in\{-y_2,y_1\}$. This gives all conditions of Assumption \ref{as: smoothness inference} except for the last one with $J=3$, $t_{\by,1}=-y_2$, $t_{\by,2}=y_1$, and $C_m =2$. The last condition of Assumption \ref{as: smoothness inference} follows by noting that
\begin{align*}
\P(\bX^{\top}\btheta_0 = Y_1) &= \P(\bX^{\top}\btheta_0 = Y_1,Y_1 = 0) + \P(\bX^{\top}\btheta_0 = Y_1, Y_1 > 0) \\
&\leqslant \P(\bX^{\top}\btheta_0 = 0) + \P(\bX^{\top}\btheta_0 = Y_1 \mid Y_1 > 0)\P(Y_1 > 0) = 0
\end{align*}
and, similarly, $\P(\bX^{\top}\btheta_0 = -Y_2) = 0$ under the continuities provided by \eqref{eq: y
given x continuous example 4}. Also, by similar reasoning,
\begin{align*}
&\P(\bX^{\top}\btheta_0 - \overline\Delta_n \leqslant Y_1 \leqslant \bX^{\top}\btheta_0 + \overline\Delta_n) \\
& \leqslant \P(\bX^{\top}\btheta_0 - \overline\Delta_n \leqslant 0 \leqslant \bX^{\top}\btheta_0 + \overline\Delta_n) \\
& \quad + \P(\bX^{\top}\btheta_0 - \overline\Delta_n \leqslant Y_1 \leqslant \bX^{\top}\btheta_0 + \overline\Delta_n \mid Y_1 > 0)\cdot\P(Y_1 > 0) \\
& \leqslant 2C_{pdf}\overline\Delta_n + 2C_{pdf}\overline\Delta_n \P(Y_1>0) \leqslant 4C_{pdf}\overline\Delta_n
\end{align*}
and, analogously, 
$$
\P(\bX^{\top}\btheta_0 - \overline\Delta_n \leqslant -Y_2 \leqslant \bX^{\top}\btheta_0 + \overline\Delta_n)\leqslant 4C_{pdf}\overline\Delta_n,
$$
showing that Assumption \ref{as: conditional density} is satisfied with any
$\overline\Delta_n\in(0,\infty)$ and $C_f = 4C_{pdf}$. Moreover, Assumption \ref{as: integrability
inference} follows from \eqref{eq: verification of high level assumptions key condition 2} and
\eqref{eq: y square integrable example 4} with $\widetilde r = 8$ and $C_M = C_M(C_0)$.

Finally, by \eqref{eq: even more stuff}, we have for all $j\in[p]$ that
\begin{align*}
\E[|UX_j|^2]
& = \E\left[\left|m_1'(\bX^{\top}\btheta_0,\bY)X_j\right|^2\right] \\
& \geqslant 4\E\left[X_j^2\left(Y_1^2 \mathbf 1\left\{Y_2 \leqslant - \bX^{\top}\btheta_0\right\} + Y_2^2 \mathbf 1\left\{Y_1 \leqslant  \bX^{\top}\btheta_0\right\}\right)\right]  \geqslant 4 c_{eps},
\end{align*}
which shows that the lower bound in Assumption
\ref{assu:ResidualBootstrapMethod}.\ref{enu:ZijSecondMomentsBndAwayZero} is satisfied with $c_U =
c_U(c_{eps})$. While Lemma \ref{lem: verification of high level assumptions section 4} is not
directly applicable (due to the non-constant $c_{m,1}(\by)$), from minor modifications to the proof
of that lemma, we see that the rest of Assumption \ref{assu:ResidualBootstrapMethod} is satisfied
with $C_U = C_U(C_0)$ and $\widetilde B_n = (C_1\bar B_n)\lor C_2$, where $C_1 =
C_1(C_0)\in[1,\infty)$ and $C_2 = C_2(C_0)$.

Next, we consider \emph{trimmed LAD}, in which case the loss \eqref{eq:TrimmedLoss} can be written
as
$$
m\left(t,\by\right)=\begin{cases}
|y_1 - y_2 - t|, & \text{if } y_1>0,y_2>0,\\
(y_2 + t)\mathbf 1(y_2 + t > 0), & \text{if }y_1 = 0, y_2 >0,\\
(y_1 - t)\mathbf 1(y_1 - t > 0), & \text{if }y_1>0, y_2 = 0,\\
0, &\text{if }y_1 =0, y_2 = 0.
\end{cases}
$$
The differentiability part of Assumption \ref{as: diff and int} follows from the absolute continuity
provided by \eqref{eq: bounded conditional distributions} and \eqref{eq: bounded conditional
distributions 2}. In addition, for all $\btheta\in\Theta$, we have
$$
\E[|m(\bX^{\top}\btheta,\bY)|] \leqslant \E[Y_1] + \E[Y_2] + \E[|\bX^{\top}\btheta|] <\infty
$$
by \eqref{eq: y square integrable example 5} and \eqref{eq: verification of high level assumptions
key condition}, which gives the integrability part of Assumption \ref{as: diff and int}.

A similar argument shows that each $m(t,\bY)$ is integrable, which implies that, for all
$\bx\in\mathcal X$, $t\mapsto f(t,\bx)=\E[m(t,\bY)| \bX = \bx]$ is well-defined as a function from
$\R$ to $\R$. By \eqref{eq: bounded conditional distributions} and \eqref{eq: bounded conditional
distributions 2}, $t\mapsto f(t,\bx)$ is seen to be differentiable with derivative
\begin{align*}
&f_1'(t,\bx)
 = \E[(\mathbf 1(Y_1 - Y_2 < t) - \mathbf 1(Y_1 - Y_2 > t))\mathbf 1\{Y_1>0,Y_2>0\}\mid \bX = \bx]\\
& \quad + \E[\mathbf 1\{Y_2 > - t\}\mathbf 1\{Y_1 = 0, Y_2 >0\}\mid \bX = \bx]
 - \E[\mathbf 1\{Y_1 > t\}\mathbf 1\{Y_1>0, Y_2 = 0\}\mid \bX = \bx].
\end{align*}
Again by \eqref{eq: bounded conditional distributions} and \eqref{eq: bounded conditional
distributions 2}, we deduce that $t\mapsto f_1'(t,\bx)$ is Lipschitz-continuous and differentiable.
Inspection reveals that its derivative satisfies the lower bound
$$
f_{11}''(t,\bx) \geqslant 2f_{Y_1 - Y_2\mid \bX, Y_1>0,Y_2>0}(t\mid \bx)\P(Y_1 > 0, Y_2 > 0\mid \bX = \bx).
$$
Thus, for all $\btheta\in\Theta$ and $j\in[p]$, we have  $\E[|f_1'(\bX^{\top}\btheta,\bX)X_j|] <
\infty$ by \eqref{eq: verification of high level assumptions key condition}. Also, by \eqref{eq:
nontrivial tail bound assumption 5}, we have that \eqref{eq: verifying assumption 3.4 key condition
2} is satisfied with $C = c_{de}$ and the provided $c_f$. Adjusting $c_M'$ (which is here free to
vary), Assumption \ref{assu:Margin} follows from Lemma \ref{lem: verification of margin assumption}
with $c_M' = c_M'(c_{de},c_{ev},C_{ev})$ and $c_M = c_M(c_{ev},c_f)$.

Next, given that for all $\by\in\mathcal Y$, the function $m(\cdot,\by)$ is Lipschitz-continuous
with Lipschitz constant one, using \eqref{eq: verification of high level assumptions key condition},
\eqref{eq: verification of high level assumptions key condition 2} and \eqref{eq: verification of
high level assumptions key condition 3}, Assumptions
\ref{assu:LossLocallyLipschitzAndMore}.\ref{enu:LossLocallyLipschitz} and
\ref{assu:LossLocallyLipschitzAndMore}.\ref{enu:LossMeanSquareEll2Conts} follow upon taking
$L(\bx,\by)=1$ for $(\bx,\by)\in\mathcal X\times\mathcal Y$, $c_L=1$, $r = \bar r$, $B_n = \bar B_n$
and $C_L = C_L(C_0,C_{ev})$. For the remaining part of this assumption, note that for all
$\btheta\in\Theta$, we have
\begin{align*}
 \E\left[\left| m_1'(\bX^{\top}\btheta,\bY) - m_1'(\bX^{\top}\btheta_0,\bY)\right|^2\right] 
&  \leqslant 4\P(\bX^{\top}\btheta_0 \leqslant Y_1 - Y_2\leqslant \bX^{\top}\btheta,Y_1 + Y_2 >0) \\
& \quad + 4\P(\bX^{\top}\btheta \leqslant Y_1 - Y_2\leqslant \bX^{\top}\btheta_0,Y_1 + Y_2 >0) \\
& \leqslant 8C_{pdf} \E[|\bX^{\top}\btheta - \bX^{\top}\btheta_0|]  \leqslant 8C_{pdf}\sqrt{C_{ev}}\|\btheta - \btheta_0\|_2,
\end{align*}
where the first inequality follows from the fact that a jump in (the piecewise constant and Lebesgue
almost everywhere defined) function $t\mapsto m_1'(t,\by)$ occurs only at $t = y_1 - y_2$ and only
when $y_1 + y_2 >0$, and from observing that such a jump can be at most two (corresponding to a sign
flip), the second from \eqref{eq: bounded conditional distributions} and \eqref{eq: bounded
conditional distributions 2}, and the third from Jensen's inequality and \eqref{eq: verification of
high level assumptions key condition}. The previous display shows that Assumption
\ref{assu:LossLocallyLipschitzAndMore}.\ref{enu:ResidualMeanSquareEll2Conts} holds with $C_L =
C_L(C_{ev},C_{pdf})$, such that Assumption \ref{assu:LossLocallyLipschitzAndMore} as a whole is
satisfied with $C_L=C_L(C_0,C_{ev},C_{pdf})$.

Finally, we have for all $j\in[p]$ that
\begin{align*}
\E[|UX_j|^2] \geqslant \E[|m_1'(\bX^{\top}\btheta_0,\bY)X_j|^2\mathbf 1(Y_1>0,Y_2>0)] = \E[|X_j|^2\mathbf 1(Y_1>0,Y_2>0)] \geqslant c_{eps}
\end{align*}
by \eqref{eq: bounded conditional distributions} and \eqref{eq: y square integrable example 5},
which shows that the lower bound in Assumption
\ref{assu:ResidualBootstrapMethod}.\ref{enu:ZijSecondMomentsBndAwayZero} is satisfied with $c_U =
\sqrt{c_{eps}}$. Given that the loss $m(\cdot,\by)$ is Lipschitz-continuous with Lipschitz constant
one for all $\by\in\mathcal Y$, using  \eqref{eq: verification of high level assumptions key
condition 2} and \eqref{eq: verification of high level assumptions key condition 3}, straightforward
arguments show that the rest of Assumption \ref{assu:ResidualBootstrapMethod} is satisfied with $C_U
= C_0$, and $\widetilde B_n = \bar B_n\lor C_0^2$.
\qed

We end this section by providing the proofs for Lemmas \ref{lem: verification of margin
assumption}--\ref{lem: lipschitz derivative}, in turn.

\begin{proof}
[\sc{Proof of Lemma \ref{lem: verification of margin assumption}}] We first
show that $\E[f_1'(\bX^{\top}\btheta_0,\bX)\bX]=\mathbf 0_{p}$. To do so, note that since
$\btheta_{0}$ is interior to $\Theta$ (Assumption \ref{assu:ParameterSpace}),
there is a radius $\overline{r}_{n}\in(0,\infty)$ such that the ball $\mathcal{B}_{\btheta_0}(\overline{r}_{n}):= \left\{\btheta\in\R^{p};\|\btheta-\btheta_{0}\|_{2}\leqslant\overline{r}_{n}\right\}$
is a subset of $\Theta$. Fix any $\btheta\in \mathcal B_{\btheta_0}(\overline r_n)$ and let
$$
g(\tau):=\E[f(\bX^{\top}\btheta_0 + \tau(\bX^{\top}\btheta - \bX^{\top}\btheta_0),\bX)],\quad \tau\in(-1,1).
$$
Note that since $\E[|m(\bX^{\top}\btheta_0,\bY)|]<\infty$ by Assumption \ref{as: diff and int}, it
follows from \citet[Theorem 10.1.1]{dudley2004real} that the conditional expectation
$\E[m(\bX^{\top}\btheta_0,\bY)|\bX]$ exists and hence is equal to $f(\bX^{\top}\btheta_0,\bX)$ almost
surely. Thus,
$$
\E[|f(\bX^{\top}\btheta_0,\bX)|] = \E[|\E[m(\bX^{\top}\btheta_0\bY)|X]|] \leqslant \E[|m(\bX^{\top}\btheta_0,\bY)|] < \infty.
$$
Also, since it follows from Assumption \ref{assu:Convexity} that the function
$t\mapsto f(t,\bx)$ is convex,  we have for any $\tau\in(-1,1)\backslash \{0\}$ that
\begin{align*}
& \tau^{-1}\Big| f(\bX^{\top}\btheta_0 + \tau(\bX^{\top}\btheta - \bX^{\top}\btheta_0),\bX) - f(\bX^{\top}\btheta_0,\bX) \Big| \\
&\qquad \leqslant\Big( |f_1'(\bX^{\top}\btheta_0,\bX)| + |f_1'(\bX^{\top}\btheta,\bX)| \Big) \times |\bX^{\top}\btheta - \bX^{\top}\btheta_0|,
\end{align*}
where the right-hand side is integrable by assumption. Along with the differentiability and
integrability presumed in the statement of the lemma, the previous two displays suffice to satisfy
both Condition (A.2) and the difference quotient domination condition of \citet[Corollary
A.3]{dudley2014uniform} when applied to the function $\tau\mapsto f(\bx^{\top}\btheta_0 +
\tau(\bx^{\top}\btheta - \bx^{\top}\btheta_0),\bx)$ in a neighborhood of $\tau_0=0$.
Hence, it follows from the same corollary that $g$ is differentiable at $\tau_0=0$ with derivative
$$
g'(0)=\E[f_1'(\bX^{\top}\btheta_0,\bX)(\bX^{\top}\btheta - \bX^{\top}\btheta_0)].
$$
Therefore, by Taylor's theorem (with Peano's form of remainder),
$$
g(\tau) = g(0)+g'(0)\tau+h(\tau)\tau,\quad \tau\in(-1,1),
$$
where $h\colon(-1,1)\to\R$ is a function such that $h(\tau)\to0$ as $\tau\to0$. On the
other hand, by the definition of $\btheta_0$, we have 
$$
g(\tau) =
\E[m(\bX^\top\btheta_0 + \tau(\bX^{\top}\btheta -
\bX^{\top}\btheta_0),\bY)]\geqslant \E[m(\bX^{\top}\btheta_0,\bY)]=g(0)$$
for all $\tau\in(-1,1)$. Thus,
$$
g'(0) + h(\tau) \leqslant 0\quad\text{for}\quad\tau\in(-1,0)
$$
and
$$
g'(0) + h(\tau) \geqslant 0\quad\text{for}\quad\tau\in(0,1).
$$
Taking the limits as $\tau\to0_-$ and $\tau\to0_+$, respectively, implies that $g'(0)=0$. Since
$\btheta\in \mathcal B_{\btheta_0}(\overline r_n)$  used in defining $f$ and, thus, $g$ was
arbitrarily chosen, varying $\btheta\in \mathcal B_{\btheta_0}(\overline r_n)$ produces the desired
$\E[f'_1(\bX^{\top}\btheta_0,\bX)\bX]=\mathbf 0_{p}$. (See the proof of Lemma \ref{lem:EofUXis0} for
details.)

Let $\btheta\in\Theta$ satisfy $\|\btheta - \btheta_0\|_2\leqslant c_M'$, and introduce the
abbreviations $\bdelta:=\btheta - \btheta_0$ and
\[
   g(r,\bx):=
   \begin{cases}
      f_{11}''(\bx^{\top}(\btheta_0 + r\bdelta),\bx), &\text{if } \bx^{\top}(\btheta_0 +
      r\bdelta)\notin N(\bx),\\
      4c_f, &\text{if } \bx^{\top}(\btheta_0 +
      r\bdelta)\in N(\bx),
   \end{cases}
\]
for $r\in[0,1]\times\mathcal X$. Using $\E[f'_1(\bX^{\top}\btheta_0,\bX)\bX]=\mathbf 0_{p}$, by a
first-order Taylor expansion with remainder in the integral form (which is allowed per Lipschitz
continuity of $t\mapsto f_1'(t,\bx)$ on compacta), using that $N(\bx)$ is a Lebesgue null set
for all $\bx\in\mathcal X$, we get
\begin{align}
\mathcal E(\btheta)
&=\E[f(\bX^{\top}\btheta,\bX) - f(\bX^{\top}\btheta_0,\bX)] \nonumber\\
&= \E\left[\int_{\bX^{\top}\btheta_0}^{\bX^{\top}\btheta}f''_{11}(t,\bX)(\bX^{\top}\btheta - t)dt\right]\nonumber\\
&= \E\left[\int_0^1 f''_{11}(\bX^\top(\btheta_0+r\bdelta),\bX)|\bX^{\top}\bdelta|^2(1-r)dr\right] \nonumber\\
&= \E\left[\int_0^1 g(r,\bX)|\bX^{\top}\bdelta|^2(1-r)dr\right]\nonumber\\
&\geqslant 4c_f \E\left[ \int_0^1 |\bX^{\top}\bdelta|^2  \mathbf{1}\{g(r,\bX) \geqslant 4c_f\} (1-r)  dr \right]. \label{eq: verifying assumption 3.4 last line, working hard}
\end{align}
Splitting the expectation in \eqref{eq: verifying assumption 3.4 last line, working hard} in two, we get
\begin{align*}
& \E\left[ \int_0^1 |\bX^{\top}\bdelta|^2(1-r)dr \right] - \E\left[ \int_0^1 |\bX^{\top}\bdelta|^2 \mathbf 1\{g(r,\bX)<4c_f\} (1-r)  dr \right] \\
& \geqslant \frac{c_{ev}\|\btheta-\btheta_0\|_2^2}{2} - \int_0^1\sqrt{\E[|\bX^{\top}\bdelta|^4]}\sqrt{\P(g(r,\bX)<4c_f)}(1-r)dr \\
& \geqslant \frac{\|\btheta - \btheta_0\|_2^2}{2}\left(c_{ev} - C_{ev}\sup_{r\in[0,1]}\sqrt{\P(g(r,\bX)<4c_f)}\right)
\end{align*}
where we have used Tonelli's theorem, the Cauchy-Schwarz inequality and \eqref{eq: verification of
high level assumptions key condition}. Continuing with the right-hand side probability, for all
$r\in[0,1]$, for $C$ stated in \eqref{eq: verifying assumption 3.4 key condition}, we see that
\begin{align*}
\P\left(g(r,\bX)<4c_f\right)
& \leqslant \P\left(\inf_{t\in[-C,C]\backslash N(\bX)}f_{11}''(t,\bX)<4c_f\right) + \P\left(|\bX^{\top}(\btheta_0 + r\bdelta)| > C\right) \\
& \leqslant \left(\frac{c_{ev}}{2\sqrt 2 C_{ev}}\right)^2 + \frac{\E[(\bX^{\top}\btheta_0 + r\bX^{\top}(\btheta - \btheta_0))^2]}{C^2}\\
& \leqslant \left(\frac{c_{ev}}{2\sqrt 2 C_{ev}}\right)^2 + \frac{2\E[|\bX^{\top}\btheta_0|^2] + 2\E[|\bX^{\top}(\btheta - \btheta_0)|^2]}{C^2}\\
& \leqslant \left(\frac{c_{ev}}{2\sqrt 2 C_{ev}}\right)^2 + \frac{2(C_0^2 + C_{ev}(c_M')^2)}{C^2} = \left(\frac{c_{ev}}{2C_{ev}}\right)^2
\end{align*}
where we have used the union bound, \eqref{eq: verifying assumption 3.4 key condition}, Markov's
inequality, the basic inequality $(a+b)^2\leqslant 2(a^2+b^2)$, H{\"o}lder's inequality,
\eqref{eq: verification of high level assumptions key condition}
and \eqref{eq: verification of high level assumptions key condition 2}. 

Also, for all $r\in[0,1]$,
with $C$ stated in \eqref{eq: verifying assumption 3.4 key condition 2}, we have
\begin{align*}
\P\left(g(r,\bX)<4c_f\right)
& \leqslant \P\Bigg(\inf_{\substack{t\in[-C,C]\\\bX^{\top}\btheta_0 + t\notin N(\bX)}}f_{11}''(\bX^{\top}\btheta_0 + t,\bX)<4c_f\Bigg) + \P\Big( |\bX^{\top}(\btheta - \btheta_0)| > C \Big) \\
& \leqslant \left(\frac{c_{ev}}{2\sqrt 2 C_{ev}}\right)^2 + \frac{\E[(\bX^{\top}(\btheta - \btheta_0))^2]}{C^2}\\
& \leqslant \left(\frac{c_{ev}}{2\sqrt 2 C_{ev}}\right)^2 + \frac{C_{ev}(c_M')^2}{C^2} = \left(\frac{c_{ev}}{2C_{ev}}\right)^2
\end{align*}
by similar arguments. Combining these chains of inequalities gives the asserted claim.
\end{proof}

\begin{proof}
[\sc{Proof of Lemma \ref{lem: verification of some high level assumptions}}]
First, for all $(\bx,\by)\in\mathcal X\times\mathcal Y$ and all $(t_1,t_2)\in\R^2$ satisfying
$|t_1|\vee|t_2|\leqslant 1$, we have for some $\tau\in[0,1]$ that
\begin{align*}
& |m(\bx^{\top}\btheta_0 + t_1,\by) - m(\bx^{\top}\btheta_0 + t_2,\by)| 
= |m_1'(\bx^{\top}\btheta_0 + t_1 + \tau(t_2 - t_1),\by)(t_1 - t_2)| \\
&\qquad\leqslant \left(c_{m,1} + c_{m,2}(| \bx^{\top}\btheta_0 + t_1 |\vee| \bx^{\top}\btheta_0 + t_2 |)\right)|t_1 - t_2| \\
& \qquad\leqslant \left( c_{m,1} + c_{m,2}(| \bx^{\top}\btheta_0 | + 1) \right)|t_1 - t_2|
\leqslant (1 + c_{m,1} + c_{m,2})(1 + |\bx^{\top}\btheta_0|)|t_1 - t_2|,
\end{align*}
where the first line follows from the Mean Value Theorem, the second from \eqref{eq: bounded
derivative lemma} and the triangle inequality, and the third from $|t_1|\vee|t_2|\leqslant 1$ and
the triangle inequality. This gives \eqref{eq:LocallyLipschitz} with the provided $c_L$ and $L$.
Also, for this $L$, by \eqref{eq: verification of high level assumptions key condition 3},
\begin{align*}
\E\left[| L(\bX,\bY)\|\bX\|_{\infty} |^{\bar r}\right]
= (1 + c_{m,1} + c_{m,2})^{\bar r} \E\left[\left(1+|\bX^{\top}\btheta_0|\right)^{\bar r}\|\bX\|_{\infty}^{\bar r}\right]
\leqslant (1 + c_{m,1} + c_{m,2})^{\bar r}\bar B_n^{\bar r},
\end{align*}
which justifies the choices $r = \bar r$ and $B_n = (1 + c_{m,1} + c_{m,2})\bar B_n$. In addition,
for all $j\in[p]$, continuing with the above $L$, we have
\begin{align*}
\E[|L(\bX,\bY)X_j|^2] 
& = (1 + c_{m,1} + c_{m,2})^2\E[(1+|\bX^{\top}\btheta_0|)^2X_j^2] \\
& \leqslant 2(1 + c_{m,1} + c_{m,2})^2\left(\E[X_j^2] + \sqrt{\E[|\bX^{\top}\btheta_0|^4]\E[X_j^4]}\right) \\
& \leqslant 2(1 + c_{m,1} + c_{m,2})^2 C_{ev}(1+C_0^2) \leqslant C_L^2,
\end{align*}
where the second line follows from the basic inequality $(a+b)^2\leqslant2(a^2+b^2)$ and the
Cauchy-Schwarz inequality, and the third from \eqref{eq: verification of high level assumptions key
condition}, \eqref{eq: verification of high level assumptions key condition 2} and the definition of
$C_L$. The previous display yields Assumption
\ref{assu:LossLocallyLipschitzAndMore}.\ref{enu:LossLocallyLipschitz}.

Second, for all $\btheta\in\Theta$ satisfying $\|\btheta - \btheta_0\|_2
\leqslant 1$, we have for some $\tau\in[0,1]$ that
\begin{align*}
&\E\left[\left|m(\bX^{\top}\btheta,\bY) - m(\bX^{\top}\btheta_0,\bY)\right|^2\right] \\
&\qquad =\E\left[\left|m_1'(\bX^{\top}\btheta_0 + \tau\bX^{\top}(\btheta - \btheta_0),\bY)\right|^2\left|\bX^{\top}\btheta - \bX^{\top}\btheta_0\right|^2\right] \\
&\qquad \leqslant \E\left[ \left(c_{m,1} + c_{m,2}(| \bX^{\top}\btheta_0 | + | \bX^{\top}(\btheta - \btheta_0) |)\right)^2\left|\bX^{\top}\btheta - \bX^{\top}\btheta_0\right|^2 \right] \\
&\qquad \leqslant 3\E\left[ \left(c_{m,1}^2 + c_{m,2}^2|\bX^{\top}\btheta_0|^2 + c_{m,2}^2|\bX^{\top}(\btheta - \btheta_0)|^2\right)\left|\bX^{\top}\btheta - \bX^{\top}\btheta_0\right|^2\right] \\
&\qquad\leqslant 3\left(c_{m,1}^2 C_{ev} + c_{m,2}^2C_0^2C_{ev} + c_{m,2}^2C_{ev}^2\right)\|\btheta - \btheta_0\|_2^2 \leqslant C_L^2\|\btheta - \btheta_0\|_2^2,
\end{align*}
where the second line follows from the Mean Value Theorem, the third from \eqref{eq: bounded
derivative lemma} and the triangle inequality, the fourth from the basic inequality
$(a+b+c)^2\leqslant 3(a^2+b^2+c^2)$, and the fifth from the Cauchy-Schwarz inequality, \eqref{eq:
verification of high level assumptions key condition}, \eqref{eq: verification of high level
assumptions key condition 2} and the definition of $C_L$. The previous display yields Assumption
\ref{assu:LossLocallyLipschitzAndMore}.\ref{enu:LossMeanSquareEll2Conts}.

Third, for all $\btheta\in\Theta$ satisfying $\|\btheta - \btheta_0\|_2\leqslant 1$, we have
\begin{align*}
&\E\left[\left| m_1'(\bX^{\top}\btheta,\bY) - m_1'(\bX^{\top}\btheta_0,\bY) \right|^2\right]
 \leqslant c_{m,3}^2\E\left[\left| \bX^{\top}\btheta - \bX^{\top}\btheta_0 \right|^2\right] \\
&\qquad \leqslant c_{m,3}^2 C_{ev} \| \btheta - \btheta_0 \|_2^2 \leqslant C_L^2\|\btheta - \btheta_0\|_2^2 \leqslant C_L^2\|\btheta - \btheta_0\|_2
\end{align*}
by \eqref{eq: lipschitz derivative lemma}, \eqref{eq: verification of high level assumptions key
condition}, and the definition of $C_L$. The previous display yields Assumption
\ref{assu:LossLocallyLipschitzAndMore}.\ref{enu:ResidualMeanSquareEll2Conts} and completes the proof
of the lemma.
\end{proof}

\begin{proof}
[\sc{Proof of Lemma \ref{lem: verification of high level assumptions section 4}}] 
First, for all $j\in[p]$, we have
\begin{align*}
\E[| m_1'(\bX^{\top}\btheta_0,\bY)X_j |^2]
& \leqslant \E\left[ (c_{m,1} + c_{m,2}|\bX^{\top}\btheta_0|)^2X_j^2 \right] \\
&  \leqslant 2\E\left[ (c_{m,1}^2 + c_{m,2}^2|\bX^{\top}\btheta_0|^2) X_j^2 \right]
 \leqslant 2(c_{m,1}^2C_0^2 + c_{m,2}^2C_0^4) \leqslant C_U^2
\end{align*}
by \eqref{eq: bounded derivative lemma}, the basic inequality $(a+b)^2\leqslant 2(a^2+b^2)$,
\eqref{eq: verification of high level assumptions key condition 2}, the Cauchy-Schwarz inequality,
and the definition of $C_U$. Together with the assumptions of the lemma, the previous display yields
Assumption \ref{assu:ResidualBootstrapMethod}.\ref{enu:ZijSecondMomentsBndAwayZero}.

Second, for all $j\in[p]$, we have
\begin{align*}  
   \E\left[ | m_1'(\bX^{\top}\btheta_0,\bY)X_j |^4 \right] 
   & \leqslant \E\left[ (c_{m,1} + c_{m,2}|\bX^{\top}\btheta_0|)^4X_j^4 \right] \\
   & \leqslant (c_{m,1} + c_{m,2})^4 \E\left[ (1+ |\bX^{\top}\btheta_0|)^4X_j^4 \right] \\  
   & \leqslant 8(c_{m,1} + c_{m,2})^4 \E\left[ (1+ |\bX^{\top}\btheta_0|^4)X_j^4 \right] \\  
   & \leqslant 8(c_{m,1} + c_{m,2})^4 C_0^4 (1+C_0^4) \leqslant \widetilde B_n^2.
\end{align*}
by \eqref{eq: bounded derivative lemma}, the basic inequality $(a+b)^2\leqslant 2(a^2+b^2)$,
\eqref{eq: verification of high level assumptions key condition 2}, the Cauchy-Schwarz inequality,
and the definition of $\widetilde B_n$. The previous display yields Assumption
\ref{assu:ResidualBootstrapMethod}.\ref{enu:Zj2pluskMomentBnd}.

Third, by \eqref{eq: bounded derivative lemma}, H{\"o}lder's inequality and \eqref{eq: verification
of high level assumptions key condition 3}, and the definition of $\widetilde B_n$,
\begin{align*}
   \E\left[ \| m_1'(\bX^{\top}\btheta_0,\bY)\bX\|_{\infty}^4 \right] 
   & \leqslant \E\left[ \left(c_{m,1} + c_{m,2}|\bX^{\top}\btheta_0|\right)^4\|\bX\|_{\infty}^4 \right] \\   
   & \leqslant \left(c_{m,1}+c_{m,2}\right)^4 \E\left[ \left(1+ |\bX^{\top}\btheta_0|\right)^4\|\bX\|_{\infty}^4 \right] \\
   & \leqslant \left(c_{m,1}+c_{m,2}\right)^4 \bar{B}_n^4 \leqslant \widetilde{B}_n^4.
\end{align*}
The previous display shows Assumption
\ref{assu:ResidualBootstrapMethod}.\ref{enu:maxZijFourthMomentBnd} and completes the proof of the
lemma.
\end{proof}

\begin{proof}
[\sc{Proof of Lemma \ref{lem: verification of high level assumptions section 5}}] 
By \eqref{eq: bounded derivative lemma}, the triangle inequality, \eqref{eq: verification of high
level assumptions key condition 2}, and the definition of $C_M$,
\begin{align*}
   \left( \E\left[| m_1'(\bX^{\top}\btheta_0,\bY) |^8\right] \right)^{1/8}
   & \leqslant \left( \E\left[ \left(c_{m,1} + c_{m,2}|\bX^{\top}\btheta_0|\right)^8 \right] \right)^{1/8} \\
   & \leqslant \left(c_{m,1} + c_{m,2}\right)\left( \E\left[ \left(1+|\bX^{\top}\btheta_0|\right)^8 \right] \right)^{1/8} \\
   & \leqslant \left(c_{m,1} + c_{m,2}\right)\left( 1 + \left(\E\left[|\bX^{\top}\btheta_0|^8\right]\right)^{1/8} \right) \\
   & \leqslant \left(c_{m,1} + c_{m,2}\right)(1+C_0)\leqslant C_M.
\end{align*}   
The previous display gives one part of Assumption \ref{as: integrability inference} for $\widetilde
r = 8$. Since the remaining parts follow trivially for the same $\widetilde r$ by \eqref{eq:
verification of high level assumptions key condition 2}, the asserted claim follows.
\end{proof}

\begin{proof}
[\sc{Proof of Lemma \ref{lem: lipschitz derivative}}] 
First, $f$ is a well-defined and real-valued function on $\R$ by integrability of $Z$. For the
alternative expression for $f$, considering the three cases $t<z$, $t=z$ and $t>z$, we see that the
integrands $(z-t)\mathbf{1}(z\geqslant t)$ and $(z-t)\mathbf{1}(z>t)$ used in the two definitions
are actually one and the same. Taking the expectation over $Z$ gives the desired equivalence.

To argue Lipschitzness, let $t_1,t_2\in\R$. Consider the case $t_1 < t_2$. Then
\begin{align}
f(t_2) - f(t_1)
&=\E\left[(Z - t_2)\mathbf 1(Z\geqslant t_2)\right] - \E\left[(Z - t_1)\mathbf 1(Z\geqslant t_1)\right] \nonumber\\
&=\E\left[(Z - t_1)\left(\mathbf 1(Z\geqslant t_2) - \mathbf 1(Z\geqslant t_1)\right)\right] - (t_2-t_1)\E\left[\mathbf 1(Z\geqslant t_2)\right] \nonumber\\
&=-\underbrace{\big(\E\left[(Z - t_1)\mathbf 1(t_1\leqslant Z < t_2)\right] + (t_2-t_1)\E\left[\mathbf 1(Z\geqslant t_2)\right]\big)}_{\geqslant 0}, \label{eq: lipschitz property derivativation}
\end{align}
which implies
\begin{align*}
|f(t_2) - f(t_1)|
& = \E\left[(Z - t_1)\mathbf 1(t_1\leqslant Z < t_2)\right] + (t_2-t_1)\E\left[\mathbf 1(Z\geqslant t_2)\right] \\
&\leqslant (t_2-t_1)\E\left[\mathbf 1(t_1\leqslant Z < t_2)\right] + (t_2-t_1)\E\left[\mathbf 1(Z\geqslant t_2)\right] \leqslant (t_2-t_1) = |t_2-t_1|.
\end{align*}
The case $t_1 > t_2$ can be handled analogously, thus showing Lipschitzness of $f$.

For the differentiability assertion, let $t\in\R$ and $\Delta\in(0,\infty)$ be arbitrary.
Then \eqref{eq: lipschitz property derivativation} and the continuity of the distribution of $Z$
combine to show that
\begin{align*}
& |f(t+\Delta) - f(t)-\Delta\left[-\P(Z>t)\right]|\\
& = \E\left[(Z - t)\mathbf 1(t\leqslant Z < t+\Delta)\right] + \Delta\left[\P(Z\geqslant t+\Delta)-\P(Z > t)\right] \\
& \leqslant \Delta\left[\P(t\leqslant Z < t+\Delta) + \P(Z\geqslant t+\Delta)-\P(Z > t)\right] \\
& = \Delta\P(Z = t) = 0.
\end{align*}
It follows that
$$
\lim_{\Delta\to0_+}\frac{f(t+\Delta) - f(t)}{\Delta} = -\P(Z>t).
$$
The case $\Delta\to0_-$ can be handled analogously, thus completing the proof.
\end{proof}

\section{Proofs for Statements in Main Text}\label{sec:ProofsMainText}

\subsection{Proofs for Section
\ref{sec:Deterministic-Bounds}}\label{sec:ProofsDeterministicBounds}

For the arguments in this section, we introduce some additional notation. Let
\[
T_{0}:=\mathrm{supp}\left(\btheta_{0}\right)=\left\{ j\in\left[p\right];\left|\theta_{0,j}\right|>0\right\} 
\]
be the support of $\btheta_{0}$, and let $T(\eta)\subseteq T_{0}$ be the
$\eta$-thresholded version thereof, i.e.
\[
T(\eta):=\left\{ j\in\left[p\right] ;\left|\theta_{0,j}\right|>\eta\right\} ,\quad\eta\in[0,\infty),
\]
such that $T(0)=T_{0}.$ Given a vector $\bdelta\in\R^{p}$ and a set
of indices $J\subseteq\left[p\right]$, we let $\bdelta_{J}$ denote the vector in
$\R^{p}$ with coordinates given by $\delta_{J,j}=\delta_{j}$ if $j\in J$ and
$\delta_{J,j}=0$ otherwise. Also, for
$\widetilde{c},\eta\in[0,\infty)$, let $\mathcal{R}(\widetilde{c},\eta)$ denote
the \emph{restricted set}
\[
\mathcal{R}\left(\widetilde{c},\eta\right):=\left\{ \bdelta\in\R^{p};\left\Vert \bdelta_{T\left(\eta\right)^{c}}\right\Vert _{1}\leqslant\widetilde{c}\left\Vert \bdelta_{T\left(\eta\right)}\right\Vert _{1}+\left(1+\widetilde{c}\right)\left\Vert \btheta_{0T\left(\eta\right)^{c}}\right\Vert _{1}\text{ and }\btheta_{0}+\bdelta\in\Theta\right\} .
\]
In addition, for a constant $c_0\in(1,\infty)$, define the (random) \emph{empirical error} function
$\epsilon_{n}\colon [0,\infty)\to[0,\infty)$ by
\begin{align*}
\epsilon_{n}\left(u\right) & :=\sup_{\substack{\bdelta\in\mathcal{R}\left(\overline{c}_{0},\eta_{n}\right),\\
\left\Vert \bdelta\right\Vert _{2}\leqslant u
}
}\left|\left(\mathbb{E}_{n}-\mathrm{E}\right)\left[m\left(\bX_{i}^{\top}\left(\btheta_{0}+\bdelta\right),\bY_{i}\right)-m\left(\bX_{i}^{\top}\btheta_{0},\bY_{i}\right)\right]\right|,
\end{align*}
where $\overline{c}_{0}=(c_{0}+1)/(c_{0}-1)$, and
$\eta_{n}=\sqrt{\ln(pn)/n}$. Finally, for any $\lambda\in(0,\infty)$, any
non-random sequence $\overline\lambda_n$ in $(0,\infty)$, and non-random
sequences $a_{\epsilon,n}$ and $b_{\epsilon,n}$ in $(0,\infty)$, to be specified
later, define the events
\begin{equation}\label{eq: three events}
\mathscr{S}_n :=\left\{ \lambda\geqslant c_{0}\|\bS_n\|_{\infty}\right\} ,\quad \mathscr{L}_n :=\left\{ \lambda\leqslant\overline{\lambda}_{n}\right\}, \quad\text{and}\quad \mathscr{E}_n  :=\left\{ \epsilon_{n}\left(\widetilde{u}_{n}\right)\leqslant a_{\epsilon,n}\widetilde{u}_{n}+b_{\epsilon,n}\right\},
\end{equation}
where
\begin{equation}\label{eq: definition un tilde key}
\widetilde{u}_{n}:=\frac{2}{c_{M}}\left(a_{\epsilon,n}+\left(1+\overline{c}_{0}\right)\overline{\lambda}_{n}\sqrt{s_{q}\eta_{n}^{-q}}\right).
\end{equation}
(That $\bS_n$ exists is left implicit in the definition of $\mathscr{S}_n$.) In
proving Theorem \ref{thm:NonAsymptoticProbabilisticBounds}, we rely on the
following four lemmas, whose proofs can be found at the end of this section.

\begin{lem}
[\textbf{Strong Ball Consequences}]\label{lem:StrongBallConsequences} Let
Assumption \ref{assu:Approximate-Sparsity} hold. Then for any
$\eta\in(0,\infty)$, we have $|T(\eta)|\leqslant s_{q}\eta^{-q}$  and
$\|\btheta_{0T\left(\eta\right)^{c}}\|_{1}\leqslant s_{q}\eta^{1-q}$.
\end{lem}

\begin{lem}
[\textbf{Restricted Set Consequence}]\label{lem:RestrictedSetConsequence} Let
Assumption \ref{assu:Approximate-Sparsity} hold. Then for any
$\widetilde{c},\eta\in(0,\infty)$, $\bdelta\in\mathcal{R}(\widetilde{c},\eta)$
implies
$\|\bdelta\|_{1}\leqslant\left(1+\widetilde{c}\right)(\sqrt{s_{q}\eta^{-q}}\left\Vert
\bdelta\right\Vert _{2}+s_{q}\eta^{1-q}).$
\end{lem}

\begin{lem}[\textbf{Non-Asymptotic Deterministic
Bounds}]\label{lem:NonasymptoticDeterministicBounds} Let Assumptions
\ref{assu:ParameterSpace}--\ref{assu:Margin} and \ref{assu:Approximate-Sparsity}
hold and suppose that $\widetilde{u}_{n}\leqslant c_{M}'$ and
\begin{equation}\label{eq: lemma c3 additional constraint}
\Big( a_{\epsilon,n}+(1+\overline c_0)\overline\lambda_n\sqrt{s_q\eta^{-q}_n} \Big)^2 \geqslant c_M\Big( b_{\epsilon,n} + (1+\overline c_0)\overline\lambda_ns_q \eta_n^{1-q} \Big).
\end{equation}
Then on the event $\mathscr{S}_n\cap\mathscr{L}_n\cap\mathscr{E}_n,$ for all
$\widehat{\btheta}\in\widehat{\Theta}(\lambda)$, we have
$\widehat{\btheta}-\btheta_0\in\mathcal{R}(\overline{c}_{0},\eta_{n})$,
$$
\|\widehat{\btheta} -\btheta_{0}\|_{2}  \leqslant \widetilde{u}_{n}\ \text{and }\ \|\widehat{\btheta} - \btheta_{0}\|_{1}  \leqslant\left(1+\overline{c}_{0}\right)\left(\widetilde{u}_{n}\sqrt{s_{q}\eta_{n}^{-q}}+s_{q}\eta_{n}^{1-q}\right).
$$
\end{lem}

\begin{lem}[\textbf{Empirical Error Bound}]\label{lem:EmpiricalErrorBound} Let
Assumptions \ref{as: diff and int}, \ref{assu:LossLocallyLipschitzAndMore} and
\ref{assu:Approximate-Sparsity} hold. Then
there is a universal constant $C\in[1,\infty)$, such that for any $n\in\N$,
$t\in[1,\infty)$ and $u\in(0,\infty)$ satisfying
\begin{equation}\label{eq: extra condition to apply contraction principle}
   \eta_{n}\leqslant1,\quad\frac{B_n^2 \ln(pn)}{\sqrt n}\leqslant C_L^2\quad\text{and}\quad
   t n^{1/r} B_n\Big(u\sqrt{s_{q}\eta_{n}^{-q}}+s_{q}\eta_{n}^{1-q}\Big)\leqslant\frac{c_{L}}{1+\overline{c}_{0}},
\end{equation}
we have
\begin{align*}
\epsilon_{n}\left(u\right) & \leqslant C(1+\overline{c}_{0})C_{L}\left(u\sqrt{s_{q}\eta_{n}^{2-q}}+s_{q}\eta_{n}^{2-q}\right)
\end{align*}
with probability at least $1-4t^{-r} - C/\ln^2(p n) - n^{-1}.$
\end{lem}

\begin{rem}[\textbf{Alternative Non-Asymptotic Bounds}]
If the loss function $m$ is (globally) Lipschitz in its first argument with
Lipschitz constant not depending on $\by$, and the regressors are bounded, then
symmetrization, contraction, and concentration arguments may be used to bound
the modified empirical error
\[
\widetilde{\epsilon}_n\left(u\right):=\sup_{\substack{\bdelta\in\R^p;\\ \Vert\bdelta\Vert_{1}\leqslant u}}\left|\left(\En-\E\right)\left[m\left(\bX_{i}^{\top}\left(\btheta_{0}+\bdelta\right),\bY_{i}\right)-m\left(\bX_{i}^{\top}\btheta_{0},\bY_{i}\right)\right]\right|,\quad u\in\left(0,\infty\right),
\]
now defined with respect to the $\ell_{1}$ norm and without the restricted set.
This is the approach taken by \citet{van_de_geer_high-dimensional_2008}, who
shows that, under the above assumptions, there exists a constant
$\widetilde{C}\in\left(0,\infty\right)$ such that with probability approaching one,
\[
   \sup_{u\in\left(0,\infty\right)}\frac{\widetilde\epsilon_n\left(u\right)}{u}\leqslant\widetilde{C}\left(\sqrt{\frac{\ln p}{n}}+\frac{\ln p}{n}\right).
\]
\citet{van_de_geer_high-dimensional_2008} demonstrates that bounds on the
estimation error of the $\ell_1$-ME can be derived if $\lambda$ is
chosen to exceed the right-hand side of this inequality, which motivates
alternative methods to choose $\lambda$. Unfortunately, $\widetilde{C}$
typically relies on design constants unknown to the researcher. Moreover, even
if these constants were known, the resulting values of $\widetilde C$ would
typically be prohibitively large, yielding choices of $\lambda$ leading to
trivial estimates of the vector $\btheta_0$ in moderate samples. Our bounds
therefore seem more suitable for devising methods to choose $\lambda$.\qed
\end{rem}

\begin{proof}
[\sc{Proof of Theorem \ref{thm:NonAsymptoticProbabilisticBounds}}] We will prove
the theorem with the universal constant $C\in[1,\infty)$ appearing in the
statement of Lemma \ref{lem:EmpiricalErrorBound}. Let
$\widehat{\btheta}\in\widehat\Theta(\lambda)$ be arbitrary, fix any
$t\in[1,\infty)$ satisfying the requirements of the theorem, and specify the
sequences $a_{\epsilon,n}:=C(1+\overline{c}_{0})C_{L}\sqrt{s_{q}\eta_{n}^{2-q}}$
and $b_{\epsilon,n}:=C(1+\overline{c}_{0})C_{L}s_{q}\eta_{n}^{2-q}$, so that
\eqref{eq: lemma c3 additional constraint} is satisfied (because
$C(1+\overline{c}_{0})C_{L} > 1\geqslant c_M$) and $\widetilde{u}_{n}\leqslant
Cu_{n}$ with $\widetilde{u}_{n}$ given in \eqref{eq: definition un tilde key}. 
Since $Cu_n\leqslant c_M'$ by assumption, we therefore obtain from Lemma
\ref{lem:NonasymptoticDeterministicBounds} that on the event
$\mathscr{S}_n\cap\mathscr{L}_n\cap\mathscr{E}_n$, we have
\begin{align*}
\|\widehat{\btheta}-\btheta_{0}\|_{2} & \leqslant Cu_{n}\quad\text{and}\quad\|\widehat{\btheta}-\btheta_{0}\|_{1}\leqslant\frac{2c_{0}}{c_{0}-1}\left(Cu_{n}\sqrt{s_{q}\eta_{n}^{-q}}+s_{q}\eta_{n}^{1-q}\right),
\end{align*}
where we have used $1+\overline{c}_{0}=2c_{0}/(c_{0}-1).$ The asserted claim now
follows from
\begin{align*}
\mathrm{P}\left(\left(\mathscr{S}_n\cap\mathscr{L}_n\cap\mathscr{E}_n\right)^{c}\right) & \leqslant\mathrm{P}\left(\mathscr{S}^{c}_n\right)+\mathrm{P}\left(\mathscr{L}^{c}_n\right)+\mathrm{P}\left(\mathscr{E}^{c}_n\right)\\
 & \leqslant\mathrm{P}\left(\mathscr{S}^{c}_n\right)+\mathrm{P}\left(\mathscr{L}^{c}_n\right)+4t^{-r}+C/\ln^2(pn)+n^{-1},
\end{align*}
where the first inequality follows from the union bound and the second from
Lemma \ref{lem:EmpiricalErrorBound}, whose application is justified by the
assumptions of the theorem.
\end{proof}

We now turn to the proofs of the lemmas.

\begin{proof}
[\sc{Proof of Lemma \ref{lem:StrongBallConsequences}}] We proceed as in 
\citet[][p.~551]{negahban_unified_2012}. The first claim follows from
\[
s_{q}\geqslant\sum_{j=1}^{p}|\theta_{0,j}|^{q}\geqslant\sum_{j\in T\left(\eta\right)}|\theta_{0,j}|^{q}\geqslant|T(\eta)|\eta^{q}
\]
upon rearrangement. The second claim follows from
\[
\|\btheta_{0T\left(\eta\right)^{c}}\|_{1}=\sum_{j\in T\left(\eta\right)^{c}}|\theta_{0,j}|^{1-q}|\theta_{0,j}|^{q}\leqslant\eta^{1-q}\sum_{j\in T\left(\eta\right)^{c}}|\theta_{0,j}|^{q}\leqslant s_{q}\eta^{1-q}.
\]
\end{proof}

\begin{proof}
[\sc{Proof of Lemma \ref{lem:RestrictedSetConsequence}}] The claim
follows from
\begin{align*}
\left\Vert \bdelta\right\Vert _{1}=\left\Vert \bdelta_{T(\eta)}\right\Vert _{1}+\left\Vert \bdelta_{T(\eta)^{c}}\right\Vert _{1} & \leqslant\left(1+\widetilde{c}\right)\left(\left\Vert \bdelta_{T(\eta)}\right\Vert _{1}+\left\Vert \btheta_{0T(\eta)^{c}}\right\Vert _{1}\right)\\
 & \leqslant\left(1+\widetilde{c}\right)\left(|T(\eta)|^{1/2}\left\Vert \bdelta_{T(\eta)}\right\Vert _{2}+\left\Vert \btheta_{0T(\eta)^{c}}\right\Vert _{1}\right)\\
 & \leqslant\left(1+\widetilde{c}\right)\Big(\sqrt{s_{q}\eta^{-q}}\left\Vert \bdelta\right\Vert _{2}+s_{q}\eta^{1-q}\Big),
\end{align*}
where the first inequality follows from $\bdelta\in\mathcal R(\widetilde c,\eta)$, the second from the Cauchy-Schwarz inequality, and the third from Lemma \ref{lem:StrongBallConsequences}.
\end{proof}

\begin{proof}
[\sc{Proof of Lemma \ref{lem:NonasymptoticDeterministicBounds}}] 
Let $\widehat{\btheta}\in\widehat\Theta(\lambda)$ be arbitrary.
We proceed in two steps. In the first step, we show that
$\widehat{\btheta}-\btheta_0\in\mathcal{R}(\overline{c}_{0},\eta_{n})$ on
the event $\mathscr{S}_n$. In the second step, we derive bounds on
$\|\widehat\btheta - \btheta_0\|_2$ and
$\|\widehat\btheta - \btheta_0\|_1$ on the event
$\mathscr{S}_n\cap\mathscr{L}_n\cap\mathscr{E}_n$.

\medskip
\textbf{Step 1: }Abbreviate $\widehat{\bdelta}:=\widehat{\btheta}-\btheta_{0}.$ By minimization in
\eqref{eq:ell1PenalizedMEstimationIntro},
\begin{align*}
\mathbb{E}_{n}[m(\bX_{i}^{\top}\widehat{\btheta},\bY_{i})-m\left(\bX_{i}^{\top}\btheta_{0},\bY_{i}\right)] & \leqslant\lambda\left(\|\btheta_{0}\|_{1}-\|\widehat{\btheta}\|_{1}\right).
\end{align*}
Let $J\subseteq[p] $ be a for now arbitrary index set. By
convexity in Assumption \ref{assu:Convexity} followed by H\"{o}lder's
inequality, score domination $(\mathscr{S}_n),$ and the triangle inequality,
\[
\mathbb{E}_{n}[m(\bX_{i}^{\top}\widehat{\btheta},\bY_{i})-m\left(\bX_{i}^{\top}\btheta_{0},\bY_{i}\right)]\geqslant \bS_n^{\top}(\widehat{\btheta}-\btheta_{0})\geqslant-\|\bS_n\|_{\infty}\|\widehat{\bdelta}\|_{1}\geqslant-\frac{\lambda}{c_{0}}\left(\|\widehat{\bdelta}_{J}\|_{1}+\|\widehat{\bdelta}_{J^{c}}\|_{1}\right).
\]
Moreover, since
$\widehat{\btheta}=\btheta_{0}+\widehat{\bdelta}=\btheta_{0J}+\btheta_{0J^{c}}+\widehat{\bdelta}_{J}+\widehat{\bdelta}_{J^{c}},$
the triangle inequality shows that
\begin{align*}
\|\widehat{\btheta}\|_{1}-\|\btheta_{0}\|_{1} & \geqslant\|\btheta_{0J}+\widehat{\bdelta}_{J^{c}}\|_{1}-\|\btheta_{0J^{c}}+\widehat{\bdelta}_{J}\|_{1}-\|\btheta_{0}\|_{1}\\
 & =\|\btheta_{0J}\|_{1}+\|\widehat{\bdelta}_{J^{c}}\|_{1}-\left(\|\btheta_{0J^{c}}\|_{1}+\|\widehat{\bdelta}_{J}\|_{1}\right)-\left(\|\btheta_{0J}\|_{1}+\|\btheta_{0J^{c}}\|_{1}\right)\\
 & =\|\widehat{\bdelta}_{J^{c}}\|_{1}-\|\widehat{\bdelta}_{J}\|_{1}-2\|\btheta_{0J^{c}}\|_{1}.
\end{align*}
Combining the three previous displays, we get
\[
\|\widehat{\bdelta}_{J^{c}}\|_{1}-\|\widehat{\bdelta}_{J}\|_{1}-2\|\btheta_{0J^{c}}\|_{1}\leqslant\|\widehat{\btheta}\|_{1}-\|\btheta_{0}\|_{1}\leqslant\frac{1}{c_{0}}\left(\|\widehat{\bdelta}_{J}\|_{1}+\|\widehat{\bdelta}_{J^{c}}\|_{1}\right),
\]
which implies that
\[
\|\widehat{\bdelta}_{J^{c}}\|_{1}\leqslant\frac{c_{0}+1}{c_{0}-1}\|\widehat{\bdelta}_{J}\|_{1}+\frac{2c_{0}}{c_{0}-1}\|\btheta_{0J^{c}}\|_{1}=\overline{c}_{0}\|\widehat{\bdelta}_{J}\|_{1}+\left(1+\overline{c}_{0}\right)\|\btheta_{0J^{c}}\|_{1}.
\]
Choosing $J=T(\eta_n)$, we see that the event 
$\mathscr{S}_n$ implies
$\widehat{\bdelta}\in\mathcal{R}(\overline{c}_{0},\eta_{n})$, as claimed.

\medskip
\textbf{Step 2: } Define the (random) function $\widehat{\mathcal F}\colon\R^p\to\R$ by
\[
   \widehat{\mathcal F} (\bdelta) := \mathbb{E}_{n}\left[m\left(\bX_{i}^{\top}\left(\btheta_{0}+\bdelta\right),\bY_{i}\right)-m\left(\bX_{i}^{\top}\btheta_{0},\bY_{i}\right)\right]+\lambda\left(\|\btheta_{0}+\bdelta\|_{1}-\|\btheta_{0}\|_{1}\right).
\]
Then $\widehat{\mathcal F}$ is convex and $\widehat{\mathcal F}(\mathbf{0}_p)=0$.
Moreover, since $\Theta$ is convex (Assumption \ref{assu:ParameterSpace}),
$\bdelta\in\mathcal{R}(\overline{c}_{0},\eta_{n})$ and $t\in[0,1]$
imply $\btheta_{0}+t\bdelta\in\Theta$ and
\begin{align*}
\|\left(t\bdelta\right)_{T\left(\eta_{n}\right)^{c}}\|_{1}=t\|\bdelta_{T\left(\eta_{n}\right)^{c}}\|_{1} & \leqslant t\left(\overline{c}_{0}\|\bdelta_{T\left(\eta_{n}\right)}\|_{1}+\left(1+\overline{c}_{0}\right)\|\btheta_{0T\left(\eta_{n}\right)^{c}}\|_{1}\right)\\
 & =\overline{c}_{0}\|\left(t\bdelta\right)_{T\left(\eta_{n}\right)}\|_{1}+\left(1+\overline{c}_{0}\right)t\|\btheta_{0T\left(\eta_{n}\right)^{c}}\|_{1}\\
 & \leqslant\overline{c}_{0}\|\left(t\bdelta\right)_{T\left(\eta_{n}\right)}\|_{1}+\left(1+\overline{c}_{0}\right)\|\btheta_{0T\left(\eta_{n}\right)^{c}}\|_{1},
\end{align*}
which shows that $t\bdelta\in\mathcal{R}\left(\overline{c}_{0},\eta_{n}\right)$. Hence,
$\mathcal{R}(\overline{c}_{0},\eta_{n})$ is star-shaped with vantage point $\mathbf{0}_p$. Seeking a
contradiction, suppose that we are on the event $\mathscr{S}_n\cap\mathscr{L}_n\cap\mathscr{E}_n$,
but $\|\widehat{\bdelta}\|_{2}>\widetilde{u}_{n}.$ Since
$\widehat{\bdelta}\in\mathcal{R}(\overline{c}_{0},\eta_{n})$ by Step 1,
$\mathcal{R}(\overline{c}_{0},\eta_{n})$ being star-shaped implies $(\widetilde u_n /
\|\widehat{\bdelta}\|_2)\widehat{\bdelta}\in\mathcal{R}(\overline{c}_{0},\eta_{n})$. By the
definition \eqref{eq:ell1PenalizedMEstimationIntro} of $\widehat\btheta$ as a minimizer, we have
$\widehat{\mathcal F}(\widehat\bdelta)\leqslant0$. Convexity of $\widehat{\mathcal F}$ and
$\widehat{\mathcal F}(\mathbf{0}_p)=0$ then show that $\widehat{\mathcal F}((\widetilde u_n /
\|\widehat{\bdelta}\|_2)\widehat{\bdelta}) \leqslant (\widetilde u_n /
\|\widehat{\bdelta}\|_2)\widehat{\mathcal F}(\widehat{\bdelta})\leqslant0$. Unpacking
$\widehat{\mathcal F}$, these findings imply
\[
0\geqslant\inf_{\substack{\bdelta\in\mathcal{R}\left(\overline{c}_{0},\eta_{n}\right),\\
\|\bdelta\|_{2}=\widetilde{u}_{n}
}
}\left\{ \mathbb{E}_{n}\left[m\left(\bX_{i}^{\top}\left(\btheta_{0}+\bdelta\right),\bY_{i}\right)-m\left(\bX_{i}^{\top}\btheta_{0},\bY_{i}\right)\right]+\lambda\left(\|\btheta_{0}+\bdelta\|_{1}-\|\btheta_{0}\|_{1}\right)\right\}.
\]
By superadditivity of infima and the definition of the empirical error function,
on the event $\mathscr{L}_n,$ the right-hand side of the previous display is
bounded from below by
\begin{align*}
 & \inf_{\substack{\bdelta\in\mathcal{R}\left(\overline{c}_{0},\eta_{n}\right),\\
\|\bdelta\|_{2}=\widetilde{u}_{n}
}
}\mathrm{E}\left[m\left(\bX^{\top}\left(\btheta_{0}+\bdelta\right),\bY\right)-m\left(\bX^{\top}\btheta_{0},\bY\right)\right]\\
 & \qquad +\inf_{\substack{\bdelta\in\mathcal{R}\left(\overline{c}_{0},\eta_{n}\right),\\
\left\Vert \bdelta\right\Vert _{2}=\widetilde{u}_{n}
}
}\left(\mathbb{E}_{n}-\mathrm{E}\right)\left[m\left(\bX_{i}^{\top}\left(\btheta_{0}+\bdelta\right),\bY_{i}\right)-m\left(\bX_{i}^{\top}\btheta_{0},\bY_{i}\right)\right]\\
&\qquad + \lambda\inf_{\substack{\bdelta\in\mathcal{R}\left(\overline{c}_{0},\eta_{n}\right),\\
 \left\Vert \bdelta\right\Vert _{2}=\widetilde{u}_{n}
}
}\left\{ \|\btheta_{0}+\bdelta\|_{1}-\|\btheta_{0}\|_{1}\right\} \\
 &\quad  \geqslant\inf_{\substack{\bdelta\in\mathcal{R}\left(\overline{c}_{0},\eta_{n}\right),\\
\|\bdelta\|_{2}=\widetilde{u}_{n}
}
}\mathcal{E}\left(\btheta_{0}+\bdelta\right)-\epsilon_{n}\left(\widetilde{u}_{n}\right)-\overline{\lambda}_{n}\sup_{\substack{\bdelta\in\mathcal{R}\left(\overline{c}_{0},\eta_{n}\right),\\
\left\Vert \bdelta\right\Vert _{2}=\widetilde{u}_{n}
}
}\left|\|\btheta_{0}+\bdelta\|_{1}-\|\btheta_{0}\|_{1}\right|.
\end{align*}
Now, since $\widetilde u_{n}\leqslant c_{M}'$ by hypothesis,
Assumption \ref{assu:Margin} yields
\[
\inf_{\substack{\bdelta\in\mathcal{R}\left(\overline{c}_{0},\eta_{n}\right),\\
\|\bdelta\|_{2}=\widetilde{u}_{n}
}
}\mathcal{E}\left(\btheta_{0}+\bdelta\right)\geqslant\inf_{\substack{\btheta_{0}+\bdelta\in\Theta,\\
\|\bdelta\|_{2}=\widetilde{u}_{n}
}
}\mathcal{E}\left(\btheta_{0}+\bdelta\right)\geqslant c_{M}\inf_{\substack{\bdelta\in\R^{p},\\
\|\bdelta\|_{2}=\widetilde{u}_{n}
}
}\left\Vert \bdelta\right\Vert _{2}^{2}=c_{M}\widetilde{u}_{n}^{2}.
\]
Also, the triangle inequality followed by Lemma
\ref{lem:RestrictedSetConsequence} show that
\begin{align*}
\sup_{\substack{\bdelta\in\mathcal{R}\left(\overline{c}_{0},\eta_{n}\right),\\
\left\Vert \bdelta\right\Vert _{2}=\widetilde{u}_{n}
}
}\left|\|\btheta_{0}+\bdelta\|_{1}-\|\btheta_{0}\|_{1}\right| & \leqslant\sup_{\substack{\bdelta\in\mathcal{R}\left(\overline{c}_{0},\eta_{n}\right),\\
\left\Vert \bdelta\right\Vert _{2}=\widetilde{u}_{n}
}
}\|\bdelta\|_{1}\leqslant\left(1+\overline{c}_{0}\right)\left(\widetilde{u}_{n}\sqrt{s_{q}\eta_{n}^{-q}}+s_{q}\eta_{n}^{1-q}\right).
\end{align*}
In addition, $\epsilon_{n}(\widetilde{u}_{n})\leqslant
a_{\epsilon,n}\widetilde{u}_{n}+b_{\epsilon,n}$ on the event $\mathscr{E}_n.$
Harvesting the results, it follows that
\begin{align*}
0 & \geqslant\inf_{\substack{\bdelta\in\mathcal{R}\left(\overline{c}_{0},\eta_{n}\right),\\
\|\bdelta\|_{2}=\widetilde{u}_{n}
}
}\mathcal{E}\left(\btheta_{0}+\bdelta\right)-\epsilon_{n}\left(\widetilde{u}_{n}\right)-\overline{\lambda}_{n}\sup_{\substack{\bdelta\in\mathcal{R}\left(\overline{c}_{0},\eta_{n}\right),\\
\left\Vert \bdelta\right\Vert _{2}=\widetilde{u}_{n}
}
}\left|\|\btheta_{0}+\bdelta\|_{1}-\|\btheta_{0}\|_{1}\right|\\
 & \geqslant c_{M}\widetilde{u}_{n}^{2}-\left(a_{\epsilon,n}\widetilde{u}_{n}+b_{\epsilon,n}\right)-\left(1+\overline{c}_{0}\right)\overline{\lambda}_{n}\left(\widetilde{u}_{n}\sqrt{s_{q}\eta_{n}^{-q}}+s_{q}\eta_{n}^{1-q}\right)\\
 & =c_{M}\widetilde{u}_{n}^{2}-\left(a_{\epsilon,n}+\left(1+\overline{c}_{0}\right)\overline{\lambda}_{n}\sqrt{s_{q}\eta_{n}^{-q}}\right)\widetilde{u}_{n}-\left(b_{\epsilon,n}+\left(1+\overline{c}_{0}\right)\overline{\lambda}_{n}s_{q}\eta_{n}^{1-q}\right)\\
 & =:A_n\widetilde{u}_{n}^{2}-B_n\widetilde{u}_{n}-C_n.
\end{align*}
The right-hand side quadratic in $\widetilde{u}_{n}$ has
$A_n,B_n,C_n\in(0,\infty)$. Observing that $\widetilde{u}_{n}$ is actually equal
to $2B_n/A_n$ by the definition in \eqref{eq: definition un tilde key}, the
right-hand side equals $2B_n^2/A_n-C_n$. Since $B_n^2/A_n\geqslant C_n$ by
\eqref{eq: lemma c3 additional constraint} and $C_n>0$, we arrive at the desired
contradiction. We therefore conclude that provided $\widetilde{u}_{n}\leqslant
c_{M}',$ on the event $\mathscr{S}_n\cap\mathscr{L}_n\cap\mathscr{E}_n,$ we have
$\|\widehat{\bdelta}\|_{2}\leqslant\widetilde{u}_{n},$ which establishes the
$\ell_{2}$ bound. The $\ell_{1}$ bound then follows from $\widehat{\bdelta}$
belonging to $\mathcal{R}(\overline{c}_{0},\eta_{n})$ and Lemma
\ref{lem:RestrictedSetConsequence}.
\end{proof}

\begin{proof}
[\sc{Proof of Lemma \ref{lem:EmpiricalErrorBound}}] The claim will follow from
an application of the maximal inequality in Theorem
\ref{lem:MaxIneqBasedOnContraction}. First, fix $n\in\N$, $t\in[1,\infty)$ and
$u\in(0,\infty)$ satisfying \eqref{eq: extra condition to apply contraction
principle} and denote
$\Delta(u,\eta_n):=\mathcal{R}(\overline{c}_{0},\eta_n)\cap\left\{ \left\Vert
\cdot\right\Vert _{2}\leqslant u\right\} .$ Lemma
\ref{lem:RestrictedSetConsequence} shows that
\begin{equation}
\left\Vert \Delta\left(u,\eta_n\right)\right\Vert _{1}:=\sup_{\bdelta\in\Delta\left(u,\eta_n\right)}\left\Vert \bdelta\right\Vert _{1}\leqslant\left(1+\overline{c}_{0}\right)\left(u\sqrt{s_{q}\eta_n^{-q}}+s_{q}\eta_n^{1-q}\right) =: \overline\Delta_n(u).\label{eq:DeltaEtaOfuEll1Bnd new}
\end{equation}
Setting up for an application of Theorem \ref{lem:MaxIneqBasedOnContraction},
define $h:\R\times\mathcal{X}\times\mathcal Y\to\R$ by
$h(t,\bx,\by):=m(\bx^{\top}\btheta_{0}+t,\by)-m(\bx^{\top}\btheta_{0},\by)$ for
all $t\in\R$ and $(\bx,\by)\in\mathcal{X}\times\mathcal Y$. By construction,
$h(0,\cdot,\cdot)\equiv0$. By Assumption
\ref{assu:LossLocallyLipschitzAndMore}.\ref{enu:LossLocallyLipschitz}, the
restriction $h:[-c_{L},c_{L}]\times\mathcal{X}\times\mathcal Y\to\R$ is
$L(\bx,\by)$-Lipschitz in its first argument, thus verifying Condition
\ref{enu:ContractionConsequenceLocallyLipschitz} of Theorem
\ref{lem:MaxIneqBasedOnContraction} with $C_h=c_L$. H\"{o}lder's inequality, Assumption
\ref{assu:LossLocallyLipschitzAndMore}.\ref{enu:LossLocallyLipschitz},
(\ref{eq:DeltaEtaOfuEll1Bnd new}), and \eqref{eq: extra condition to apply
contraction principle} imply that
\[
\max_{1\leqslant i\leqslant n}\sup_{\bdelta\in\Delta\left(u,\eta_{n}\right)}\left|\bX_{i}^{\top}\bdelta\right|
   \leqslant\max_{1\leqslant i\leqslant n}\left\Vert \bX_{i}\right\Vert _{\infty}\left\Vert \Delta\left(u,\eta_{n}\right)\right\Vert _{1}
   \leqslant t n^{1/r} B_{n}\overline{\Delta}_{n}\left(u\right)\leqslant c_{L}
\]
with probability at least $1-t^{-r}$, where the bound $\P(\max_{1\leqslant
i\leqslant n}\|\bX_i\|_{\infty}>t n^{1/r} B_n)\leqslant t^{-r}$ follows from Markov's
inequality, since Assumption
\ref{assu:LossLocallyLipschitzAndMore}.\ref{enu:LossLocallyLipschitz} implies
that $\E[\|\bX\|_{\infty}^r]\leqslant B_n^r$. 
Condition
\ref{enu:ContractionConsequenceXprimeDeltaBounded} of Theorem \ref{lem:MaxIneqBasedOnContraction}
therefore holds with $C_h=c_L$ and $\zeta_n=t^{-r}$.
Given that
$s_{q},t\in[1,\infty)$, and $B_{n}\in[1,\infty)$, 
\eqref{eq: extra condition to apply contraction principle}
implies that $u\leqslant c_L$. Therefore, it follows from Assumption
\ref{assu:LossLocallyLipschitzAndMore}.\ref{enu:LossMeanSquareEll2Conts} that
\begin{align*}
& \sup_{\bdelta\in\Delta\left(u,\eta_{n}\right)}\mathrm{E}\big[h(\bX^{\top}\bdelta,\bX,\bY)^{2}\big]\\
& \leqslant\sup_{\substack{\btheta_{0}+\bdelta\in\Theta,\\
\|\bdelta\|_{2}\leqslant u
}
}\mathrm{E}\big[\left|m(\bX^{\top}\left(\btheta_{0}+\bdelta\right),\bY)-m(\bX^{\top}\btheta_{0},\bY)\right|^{2}\big]\leqslant C_{L}^{2}u^{2},
\end{align*}
and so Condition \ref{enu:ContractionConsequencehBoundedInL2} of Theorem
\ref{lem:MaxIneqBasedOnContraction} holds for $B_{1n}=C_{L}u.$ For (the final)
Condition \ref{enu:ContractionConsequenceLtimesXBoundedInEmpL2} of Theorem
\ref{lem:MaxIneqBasedOnContraction}, we invoke
Theorem \ref{thm: max inequality nonnegative}. Observe first that
$L(\bX,\bY)^2X_{j}^2\geqslant0$ for all $j\in[p]$. Assumption
\ref{assu:LossLocallyLipschitzAndMore}.\ref{enu:LossLocallyLipschitz} entails
$\max_{1\leqslant j\leqslant p}\E[L(\bX,\bY)^2X_{j}^2]\leqslant C_L^2$, so
Condition \ref{enu:MaxIneqIIIFirstMomentBnd} of Theorem \ref{thm: max inequality
nonnegative} holds with $\mu_n=C_L^2$, a constant. Assumption
\ref{assu:LossLocallyLipschitzAndMore}.\ref{enu:LossLocallyLipschitz} also
implies that $\E[\max_{1\leqslant j\leqslant p}L(\bX,\bY)^4X_{j}^4]\leqslant
B_n^4$, so Condition \ref{enu:MaxIneqIIIMomentOfMaxBnd} of Theorem \ref{thm: max inequality nonnegative} holds
with $q=2$ and $M_n=B_n^2$. Equation \eqref{eq: extra condition to apply contraction principle}
shows that $B_n^2\ln(pn)/n^{1/2}\leqslant C_L^2$,
which verifies Condition \ref{enu:MaxIneqIIIGrowthCond} of Theorem \ref{thm: max inequality nonnegative}.
Theorem \ref{thm: max inequality nonnegative} therefore shows that
there is a universal constant $C\in[1,\infty)$ such that 
\begin{equation}\label{eq: lemma c4 useful inequality to be used later}
\max_{1\leqslant j\leqslant p}\mathbb{E}_{n}\left[L\left(\bX_{i},\bY_i\right)^{2}X_{i,j}^{2}\right]\leqslant (CC_L)^2
\end{equation}
with probability at least $1-C/\ln^2(pn)$. Condition
\ref{enu:ContractionConsequenceLtimesXBoundedInEmpL2} of Theorem
\ref{lem:MaxIneqBasedOnContraction} therefore holds with $B_{2n}=CC_{L}$ and
$\gamma_{n}=C/\ln^2(pn)$. Theorem \ref{lem:MaxIneqBasedOnContraction} combined
with the bound $\overline{\Delta}_{n}(u)$ on $\|\Delta(u,\eta_{n})\|_{1}$ from
(\ref{eq:DeltaEtaOfuEll1Bnd new}) and $\ln(8pn)\leqslant4\ln(pn)$ (which follows
from $p\geqslant2$) now show that
\begin{equation}\label{eq: application of contraction principle lemma c4}
 \mathrm{P}\left(\sqrt{n}\epsilon_{n}\left(u\right)>\left\{ 4C_{L}u\right\} \lor\left\{ 16\sqrt{2}CC_{L}\overline{\Delta}_{n}\left(u\right)\sqrt{\ln\left(pn\right)}\right\} \right)
   \leqslant 4t^{-r} + 4C/\ln^2(pn) + n^{-1}.
\end{equation}
Now, given that $s_{q},C\in[1,\infty),p\in[2,\infty),n\in[3,\infty)$ and
$\eta_n\in(0,1]$, it follows that
$4\sqrt{2}C(1+\overline{c}_{0})[s_{q}\eta_n^{-q}\ln(pn)]^{1/2}\geqslant1$, and
so
$16\sqrt{2}CC_{L}\overline{\Delta}_{n}\left(u\right)\sqrt{\ln\left(pn\right)}\geqslant4C_{L}u.$
Therefore, the asserted claim follows from \eqref{eq: application of contraction
principle lemma c4} upon recognizing that
\begin{align*}
\overline{\Delta}_{n}\left(u\right)\sqrt{\frac{\ln\left(pn\right)}{n}} & =(1+\overline c_0)\left(u\sqrt{s_{q}\eta_{n}^{2-q}}+s_{q}\eta_{n}^{2-q}\right)
\end{align*}
and redefining the universal constant $C$ appropriately.
\end{proof}

\subsection{Proofs for Section \ref{sub: boot penalty level}}\label{sec: proofs for bcv} 
For the arguments in this section, we first introduce some additional notation.
Since $\btheta_{0}$ is interior to $\Theta$ (Assumption
\ref{assu:ParameterSpace}), there is a radius $\overline{r}_{n}\in(0,\infty)$
such that $\overline r_n \leqslant \min(c_L,c_M')$ and the ball  $\mathcal
B_{\btheta_0}(\overline r_n):=
\left\{\btheta\in\R^{p};\|\btheta-\btheta_{0}\|_{2}\leqslant\overline{r}_{n}\right\}$
is a subset of $\Theta$, with $c_{L},c_{M}'\in(0,\infty]$ provided by
Assumptions \ref{assu:Margin} and \ref{assu:LossLocallyLipschitzAndMore},
respectively. Fix $\btheta \in \mathcal B_{\btheta_0}(\overline r_n)$, and define
\begin{align}
f\left(\tau,\bz\right) & :=m \left(\bx^{\top}\btheta_\tau,\by\right)-m\left(\bx^{\top}\btheta_{0},\by\right)\;\text{for each}\;(\tau,\bz)\in\R\times\mathcal{Z},\label{eq:fOfwandrDudley}\\
g\left(\tau\right) & :=\mathrm{E}\left[f\left(\tau,\bZ\right)\right]\;\text{for each}\;\tau\in\left[-1,1\right]\label{eq:gOfrDudley},
\end{align}
where we employ the shorthand notations,
$\btheta_\tau:=\btheta_0+\tau(\btheta-\btheta_0)$, $\bz:=(\bx,\by)$,
$\bZ:=(\bX,\bY)$ and $\mathcal Z:=\mathcal X\times\mathcal Y$. Below we show
that $g$ is well defined. Also, let $f_1'(\tau,\bz):=(\partial/\partial
\tau)f(\tau,\bz)$ denote the partial derivative of $f$ with respect to its first
argument, when it exists.

We next state and prove the following non-asymptotic version of Lemma
\ref{thm:RatesBM}.

\begin{thm}
[\textbf{Non-Asymptotic Error Bounds: Generic Bootstrap
Method}]\label{thm:NonasyHighProbBndsBM} Let Assumptions
\ref{assu:ParameterSpace}--\ref{assu:Approximate-Sparsity} and
\ref{assu:ResidualBootstrapMethod} hold, let $\beta_n$ and $\delta_n$ be non-random
sequences in $[0,1]$ and $[0,\infty)$, respectively,
such that
\begin{equation}\label{eq: residual estimation generic bound}
\P\Big(\En[(\widehat{U}_{i}-U_{i})^{2}] > \delta_{n}^{2}/\ln^{2}\left(pn\right)\Big) \leqslant \beta_n,
\end{equation}
let
$\widehat{\Theta}(\widehat{\lambda}^{\mathtt{bm}}_{\alpha})$
be the solutions to the $\ell_{1}$-penalized M-estimation problem
(\ref{eq:ell1PenalizedMEstimationIntro})\textcolor{red}{{} }with penalty level
$\lambda=\widehat{\lambda}^{\mathtt{bm}}_{\alpha}$ given in
(\ref{eq:BootstrapPenaltyLevel}), and define
\begin{align}
u_{n,\alpha}^{\mathtt{bm}}  :=\frac{4c_{0}\sqrt{s_{q}\eta_{n}^{-q}}}{(c_{0}-1)c_M}\left(C_{L}\eta_{n}+c_{0}C_{U}\sqrt{\frac{\ln\left(p/\alpha\right)}{n}}\right). \label{eq:BootstrapPenaltyBound}
\end{align}
Then there is a constant $C_1\in[1,\infty)$, depending only on $c_U$ and
$C_U$, and a universal constant $C_2\in[1,\infty)$ such that with
\[
\rho_{n}:=C_1\max\left\{ \beta_{n}+t^{-r},tn^{1/r}B_{n}\delta_{n},\left(\frac{\widetilde{B}_{n}^{4}\ln^{7}\left(pn\right)}{n}\right)^{1/6},\frac{1}{\ln^{2}\left(pn\right)}\right\},
\]
for $n\in\N$ and $t\in[1,\infty)$ satisfying
$$
\eta_{n}\leqslant1,\quad
C_2u_{n,\alpha}^{\mathtt{bm}}\leqslant c_{M}',\quad
\frac{B_n^2\ln(pn)}{\sqrt n}\leqslant C_L^2 ,\quad
\frac{\widetilde{B}^2_n\ln(pn)}{\sqrt n}\leqslant C_U^2
$$
and
$$
t n^{1/r} B_n\left(C_2u_{n,\alpha}^{\mathtt{bm}}\sqrt{s_{q}\eta_{n}^{-q}}+s_{q}\eta_{n}^{1-q}\right)\leqslant\frac{\left(c_{0}-1\right)c_{L}}{2c_{0}},
$$
we have both
\begin{align*}
\sup_{\widehat\btheta\in\widehat{\Theta}(\widehat{\lambda}^{\mathtt{bm}}_{\alpha})}\|\widehat{\btheta}-\btheta_{0}\|_{2} & \leqslant C_2u_{n,\alpha}^{\mathtt{bm}}\quad\text{and}\quad
\sup_{\widehat\btheta\in\widehat{\Theta}(\widehat{\lambda}^{\mathtt{bm}}_{\alpha})}\|\widehat{\btheta}-\btheta_{0}\|_{1}\leqslant\frac{2c_{0}}{c_{0}-1}\left(C_2 u_{n,\alpha}^{\mathtt{bm}}\sqrt{s_{q}\eta_{n}^{-q}}+s_{q}\eta_{n}^{1-q}\right)
\end{align*}
with probability at least $1-\alpha-\rho_n - \beta_{n}-2C_2/\ln^2(pn)-5t^{-r}-n^{-1}$.
\end{thm}

In proving Theorem \ref{thm:NonasyHighProbBndsBM}, we rely on the following
six lemmas, whose proofs can be found at the end of this section. Recall that
$P$ denotes the distribution of $(\bX,\bY)=\bZ$.

\begin{lem}[\textbf{$\boldsymbol{L_1}$-Boundedness of
Levels}]\label{lem:L1BoundednessOfLevels} Let Assumptions
\ref{assu:ParameterSpace}, \ref{assu:Margin} and
\ref{assu:LossLocallyLipschitzAndMore} hold. Then
\[
\sup_{\tau\in\left[-1,1\right]}\mathrm{E}\left[\left|f\left(\tau,\bZ\right)\right|\right]\leqslant c_L C_L<\infty.
\]
\end{lem}

\begin{lem}[\textbf{Almost Sure Existence of
Partials}]\label{lem:AlmostSureExistenceOfPartials} Let Assumptions
\ref{assu:ParameterSpace}, \ref{as: diff and int}, \ref{assu:Margin} and
\ref{assu:LossLocallyLipschitzAndMore} hold. Then for any $\tau\in[-1,1]$, the
partial $f_{1}'(\tau,\bz)$ exists for $P$-almost every ($P$-a.e.)
$\bz\in\mathcal Z$, and then
\begin{align}
f_{1}'\left(\tau,\bz\right)=m_{1}'\left(\bx^{\top}\btheta_\tau,\by\right)\bx^{\top}\left(\btheta-\btheta_{0}\right).\label{eq:f2wrExpressionAlmostSure}
\end{align}
\end{lem}

\begin{lem}[\textbf{$\boldsymbol{L_1}$-Boundedness of
Partials}]\label{lem:L1BoundednessOfPartials} Let Assumptions
\ref{assu:ParameterSpace}, \ref{as: diff and int}, \ref{assu:Margin},
\ref{assu:LossLocallyLipschitzAndMore} and \ref{assu:ResidualBootstrapMethod}
hold. Then
\[
   \sup_{\tau\in\left[-1,1\right]}\mathrm{E}\left[\left|f_{1}'\left(\tau,\bZ\right)\right|\right]
   \leqslant \overline{r}_{n}\sqrt{p}\left(\widetilde{B}_{n} + C_LB_n \sqrt{\overline{r}_{n}}\right)<\infty.   
\]
\end{lem}

\begin{lem}[{\textbf{Difference Quotient Domination}}]\label{lem:DominationOfDifferenceQuotients}
Let Assumptions
\ref{assu:ParameterSpace}--\ref{assu:LossLocallyLipschitzAndMore} and
\ref{assu:ResidualBootstrapMethod} hold. Then there is a $P$-integrable
function $\overline{f}:\mathcal{Z}\to[0,\infty]$ such that
\[
\left|f\left(\tau_0+h,\bz\right)-f\left(\tau_0,\bz\right)\right|\leqslant\left|h\right|\overline{f}\left(\bz\right)\;\text{for each}\;\tau_0\in\left[-1,1\right],\;h\in\left[-d\left(\tau_0\right),d\left(\tau_0\right)\right],\;\bz\in\mathcal{Z},
\]
where $d\left(\tau_0\right):=|\tau_0+1|\wedge|\tau_0-1|$ denotes the distance to the nearest interval endpoint.
\end{lem}

\begin{lem}[\textbf{Existence of Derivatives}]\label{lem:ExistenceOfDerivatives}
Let Assumptions
\ref{assu:ParameterSpace}--\ref{assu:LossLocallyLipschitzAndMore} and
\ref{assu:ResidualBootstrapMethod} hold. Then $g$ given in \eqref{eq:gOfrDudley}
is well defined as a mapping from $[-1,1]$ to $\R$. This mapping is differentiable
on $(-1,1)$ with derivative given by
$$g'(\tau)=\E\left[f_1'\left(\tau,\bZ\right)\right],\quad\tau\in\left(-1,1\right).$$
In particular,
$$g'(0)=\E\left[U\bX^\top(\btheta-\btheta_0)\right]=0.$$
\end{lem}

\begin{lem}[\textbf{Zero Derivative}]\label{lem:EofUXis0}
Let Assumptions \ref{assu:ParameterSpace}--\ref{assu:LossLocallyLipschitzAndMore} and
\ref{assu:ResidualBootstrapMethod} hold. Then $\E[U\bX]=\mathbf{0}_p$.
\end{lem}

\begin{proof}[\sc{Proof of Theorem \ref{thm:NonasyHighProbBndsBM}}] 
We first set up for an application of the multiplier bootstrap consistency
result in Theorem \ref{thm:MultiplierBootstrapConsistency}. To this end, note
that   
Assumption \ref{assu:ResidualBootstrapMethod} implies the moment conditions
(\ref{eq:MomentConds}) for $Z_{i,j}$, $b$ and $B_{n}$ there equal to
$U_{i}X_{i,j}$, $c_{U}^2$ and $C_U\widetilde{B}_{n}$, respectively. Fix
$t\in[1,\infty)$. Assumption
\ref{assu:LossLocallyLipschitzAndMore}.\ref{enu:LossLocallyLipschitz} implies
that $\E[\|\bX\|_{\infty}^r]\leqslant B_n^r$, so $\P(\max_{1\leqslant i\leqslant
n}\|\bX_i\|_{\infty}>tn^{1/r}B_n)\leqslant t^{-r}$ by Markov's inequality. It
then follows from \eqref{eq: residual estimation generic bound} that the
estimation error condition (\ref{eq:L2PnEstErrCond}) for
$\widehat{Z}_{i,j}=\widehat{U}_{i}X_{i,j}$ holds with $\delta_{n}$ and
$\beta_{n}$ there replaced by $tn^{1/r}B_{n}\delta_{n}$ and $\beta_{n}+t^{-r},$
respectively. Since $U\bX$ is centered (cf.~Lemma \ref{lem:EofUXis0}), Theorem
\ref{thm:MultiplierBootstrapConsistency} therefore shows that there is a
constant $C_1\in[1,\infty)$, depending only on $c_{U}$ and $C_U$, such
that\footnote{We here invoke the scaling property that
$q_{tV}(\alpha)=tq_{V}(\alpha)$ for $t\in(0,\infty)$ and $\alpha\in(0,1)$ both
non-random and $q_{V}\left(\alpha\right)$ denoting the $\alpha$ quantile of the
random variable $V.$} 
\begin{align*}
 & \sup_{\alpha\in\left(0,1\right)}\big|\mathrm{P}\big(\Vert \bS_n\Vert_{\infty}>\widehat{q}^{\texttt{bm}}\left(1-\alpha\right)\big)-\alpha\big|\\
 & \qquad\leqslant C_1\max\left\{ \beta_{n}+t^{-r},tn^{1/r}B_{n}\delta_{n},\left(\frac{\widetilde{B}_{n}^{4}\ln^{7}\left(pn\right)}{n}\right)^{1/6},\frac{1}{\ln^{2}\left(pn\right)}\right\} .
\end{align*}
Taking $\rho_{n}$ to be this upper bound, it thus follows by construction of the
bootstrap penalty level
$\widehat{\lambda}^{\mathtt{bm}}_{\alpha}=c_{0 }\widehat{q}^{\texttt{bm}}\left(1-\alpha\right)$
that
$\mathrm{P}(\widehat{\lambda}^{\mathtt{bm}}_{\alpha}<c_{0}\|\bS_n\|_{\infty})\leqslant\alpha+\rho_{n}.$
We proceed to establish the claimed bounds on the estimation error for this
constant $C_1$.

To this end, assume without loss of generality that $tn^{1/r}B_n\delta_n\leqslant
1$ (otherwise $\rho_n\geqslant 1$ and the probabilistic claim becomes vacuous) and
use the observation that, conditional on $\{(\bX_i,\bY_i,\widehat{U}_{i})\}_{i=1}^{n}$, the
random vector $\mathbb{E}_{n}[e_{i}\widehat{U}_{i}\bX_{i}]$ is centered Gaussian
in $\R^{p}$ with $j$th coordinate variance
$n^{-1}\mathbb{E}_{n}[\widehat{U}_{i}^{2}X_{i,j}^{2}]$ in combination with
Theorem \ref{lem:GaussianQuantileBnd} to see that
\[
\widehat{q}^{\texttt{bm}}\left(1-\alpha\right)\leqslant(2+\sqrt{2})\sqrt{\frac{\ln\left(p/\alpha\right)}{n}\max_{1\leqslant j\leqslant p}\mathbb{E}_{n}[\widehat{U}_{i}^{2}X_{i,j}^{2}]}.
\]
In addition, there is a universal constant $\widetilde C\in[1,\infty)$ such
that with probability at least $1 - \widetilde C/\ln^2(pn) - t^{-r} - \beta_n,$
\begin{align*}
\max_{1\leqslant j\leqslant p}\mathbb{E}_{n}[\widehat{U}_{i}^{2}X_{i,j}^{2}] & 
\leqslant2\max_{1\leqslant j\leqslant p}\Big(\mathbb{E}_{n}[U_{i}^{2}X_{i,j}^{2}]+\mathbb{E}_{n}[(\widehat{U}_{i}-U_{i})^{2}X_{i,j}^{2}]\Big)\\
 & \leqslant 2(\widetilde CC_U)^2+2(tn^{1/r}B_{n})^{2}\mathbb{E}_{n}[(\widehat{U}_{i}-U_{i})^{2}]\\
 & \leqslant2(\widetilde CC_{U})^{2}+2(tn^{1/r}B_{n})^{2}\delta_{n}^{2}/\ln^{2}(pn)\\
 & \leqslant4(\widetilde CC_{U})^{2},
\end{align*}
where the first inequality follows from the elementary inequality
$(a+b)^{2}\leqslant2a^{2}+2b^{2},$ the second follows from the already
established bound $\P(\max_{1\leqslant i\leqslant
n}\|\bX_i\|_{\infty}>tn^{1/r}B_n)\leqslant t^{-r}$ and Thereom \ref{thm: max
inequality nonnegative} applied with $Z_j=U^2X_j^2$, $\mu_n=C_U^2$, $q=2$ and
$M_n=\widetilde B_n^2$ [which is justified by Assumptions
\ref{assu:ResidualBootstrapMethod}.\ref{enu:ZijSecondMomentsBndAwayZero} and
\ref{assu:ResidualBootstrapMethod}.\ref{enu:maxZijFourthMomentBnd} and the condition
$\widetilde B_n^2\ln(pn)/\sqrt n\leqslant C_U^2$], the third from \eqref{eq:
residual estimation generic bound}, and the fourth and final inequality follows
from the inequalities $\widetilde CC_U\geqslant 1\geqslant tn^{1/r}B_n\delta_n$.
Hence, with the same probability, 
\[
\widehat{\lambda}^{\mathtt{bm}}_{\alpha}\leqslant 2(2+\sqrt 2)c_{0}\widetilde CC_{U}\sqrt{\frac{\ln\left(p/\alpha\right)}{n}}=:\overline{\lambda}_n.
\]
With $\overline\lambda_n$ defined as such the implied $u_n$ and the
universal constant $C$ appearing in the statement of Theorem
\ref{thm:NonAsymptoticProbabilisticBounds} satisfy $Cu_n \leqslant 2(2+\sqrt
2)C\widetilde C u_{n,\alpha}^{\mathtt{bm}}$.
We therefore take $C_2:=2(2+\sqrt 2)C\widetilde C$ as the universal constant. Up
to this point, the choice of $t\in[1,\infty)$ has been arbitrary. With the
restrictions placed on $n$ and $t$ in the statement of Theorem
\ref{thm:NonasyHighProbBndsBM} for this choice of $C_2$, the asserted
probabilistic bounds on the estimation error follow from Theorem
\ref{thm:NonAsymptoticProbabilisticBounds}.
\end{proof}

\begin{proof}
[\sc{Proof of Lemma \ref{thm:RatesBM}}] 
We set up for an application of Theorem \ref{thm:NonasyHighProbBndsBM}. First,
to satisfy \eqref{eq: residual estimation generic bound}, for $\delta_n$
provided by \eqref{eq: residual estimation side condition}, we set
\[
\beta_n:=\P\Big(\En[(\widehat{U}_{i}-U_{i})^{2}] > \delta_{n}^{2}/\ln^{2}\left(pn\right)\Big),
\]
such that $\beta_n\to0$. Next, from $n^{1/r}B_ns_q\eta_n^{1-q}\to0$ in
\eqref{eq: simple restriction bm}, we deduce that $s_q\eta_n^{1-q}\to0$ (recall
that $B_n\geqslant1$) and, thus, $\eta_n\to0$ (recall that $s_q\geqslant1$). It
follows that $\eta_n\leqslant1$ for sufficiently large $n$. Since $\alpha_n$
satisfies $\ln(1/\alpha_n)\lesssim\ln(pn)$, we also have
$u_{n,\alpha_n}^{\mathtt{bm}}\lesssim\sqrt{s_q\eta_n^{2-q}}\to0$. Letting $C_2$
be the universal constant from Theorem \ref{thm:NonasyHighProbBndsBM}, we
eventually have $C_2u_{n,\alpha_n}^{\mathtt{bm}}\leqslant c_M'$. From
$B_n^2\ln(pn)/\sqrt n\to0$ and $\widetilde B_n^4\ln^7(pn)/n\to0$  in \eqref{eq:
simple restriction bm} we deduce that $B_n^2\ln(pn)/\sqrt n\leqslant C_L^2$
eventually and $\widetilde B_n^2\ln(pn)/\sqrt n\leqslant C_U^2$ eventually,
respectively. From $u_{n,\alpha_n}^{\mathtt{bm}}\lesssim\sqrt{s_q\eta_n^{2-q}}$,
we further deduce
\[
   u_{n,\alpha_n}^{\mathtt{bm}}s_q\eta_n^{-q}+ s_q\eta_n^{1-q}\lesssim s_q\eta_n^{1-q}.
\]
Choose now $t=t_n:=1\lor [n^{1/r}B_n(\delta_n\lor s_q\eta_n^{1-q})]^{-1/2}$.
Then \eqref{eq: simple restriction bm} guarantees $t_n=[n^{1/r}B_n(\delta_n\lor
s_q\eta_n^{1-q})]^{-1/2}\in[1,\infty)$ for sufficiently large $n$, and
$t_n\to\infty$. By choice of $t_n$ and \eqref{eq: simple restriction bm}, we
have
\[
   t_n n^{1/r} B_n\left(C_2u_{n,\alpha}^{\mathtt{bm}}\sqrt{s_{q}\eta_{n}^{-q}}+s_{q}\eta_{n}^{1-q}\right)\to0
\]
implying that for sufficiently large $n$,
\[
   t_n n^{1/r} B_n\left(C_2u_{n,\alpha}^{\mathtt{bm}}\sqrt{s_{q}\eta_{n}^{-q}}+s_{q}\eta_{n}^{1-q}\right)\leqslant\frac{\left(c_{0}-1\right)c_{L}}{2c_{0}}.
\]
The previous observations and \eqref{eq: residual estimation side condition}
combine to show that the $\rho_n$ in Theorem \ref{thm:NonasyHighProbBndsBM}
implied by this choice of $t_n$ converges to zero. Since also $\alpha_n\to0$,
the estimation error bounds provided by Theorem \ref{thm:NonasyHighProbBndsBM}
hold with probability approaching one.
\end{proof}

\begin{proof}[\sc{Proof of Lemma \ref{lem:L1BoundednessOfLevels}}]
Fix $\tau\in[-1,1]$. Since $\btheta\in\mathcal B_{\btheta_0}(\overline
r_n)\subset\Theta$ and $\overline r_n\leqslant c_L$, we have
$\|\btheta_\tau-\btheta_0\|_2=|\tau|\|\btheta-\btheta_0\|_2\leqslant c_L$, so
using the Cauchy-Schwarz inequality followed by Assumption
\ref{assu:LossLocallyLipschitzAndMore}.\ref{enu:LossMeanSquareEll2Conts}, we get
\begin{align*}
\mathrm{E}\left[\left|f\left(\tau,\bZ\right)\right|\right] 
   & =\mathrm{E}\left[\left|m\left(\bX^{\top}\btheta_{\tau},\bY\right)-m\left(\bX^{\top}\btheta_{0},\bY\right)\right|\right]\\
   & \leqslant\sqrt{\mathrm{E}\left[\left|m\left(\bX^{\top}\btheta_{\tau},\bY\right)-m\left(\bX^{\top}\btheta_{0},\bY\right)\right|^{2}\right]}\\
   & \leqslant\sqrt{C_L^2\|\btheta_\tau-\btheta_0\|_2^2} \leqslant c_LC_L<\infty.
\end{align*}
This bound holds for all $\tau\in[-1,1]$.
\end{proof}

\begin{proof}[\sc{Proof of Lemma \ref{lem:AlmostSureExistenceOfPartials}}]
By Assumption \ref{as: diff and int}, $m_{1}'(\bx^{\top}\btheta,\by)$ exists for
$P$-a.e. $\bz\in\mathcal Z$. Also, for any $\tau\in[-1,1]$ we have
$\Vert\btheta_\tau -\btheta_{0}\Vert_{2}=|\tau|\Vert
\btheta-\btheta_{0}\Vert_{2} \leqslant \overline{r}_{n}$, so that every
$\btheta_\tau,\tau\in[-1,1]$, lies in
$\mathcal{B}_{\btheta_0}(\overline{r}_{n})$ and, thus, $\Theta$. Hence, for any
$\tau\in[-1,1]$ there is a set $A\subseteq\mathcal{Z}$ [possibly depending on
$(\btheta,\tau)$] such that $P(A)=0$, and for all $\bz\in\mathcal{Z}\backslash
A$, $m_{1}'(\bx^{\top}\btheta_\tau,\by)$ exists. In this
case, the chain rule of differentiation shows that the partial
$f_{1}'\left(\tau,\bz\right)$ exists and equals
\begin{align*}
f_{1}'\left(\tau,\bz\right) & =\left.\frac{\partial}{\partial\widetilde{\tau}}m\left(\bx^{\top}[\btheta_{0}+\widetilde{\tau}(\btheta-\btheta_{0})],\by\right)\right|_{\widetilde{\tau}=\tau}=m_{1}'\left(\bx^{\top}\btheta_\tau,\by\right)\bx^{\top}\left(\btheta-\btheta_{0}\right)
\end{align*}
which gives the asserted claim.
\end{proof}

\begin{proof}[{\sc{Proof of Lemma \ref{lem:L1BoundednessOfPartials}}}]
From Lemma \ref{lem:AlmostSureExistenceOfPartials} we know that for any
$\tau\in[-1,1]$, $f'_1(\tau,\bZ)$ exists a.s.~and then takes the form in
\eqref{eq:f2wrExpressionAlmostSure}. Setting $\tau=0$, we get
\begin{align*}
\mathrm{E}\left[\left|f_{1}'\left(0,\bZ\right)\right|\right] & =\mathrm{E}\left[\left|m_{1}'\left(\bX^{\top}\btheta_{0},\bY\right)\bX^{\top}\left(\btheta-\btheta_{0}\right)\right|\right]\\
   & =\mathrm{E}\left[\left|U\bX^{\top}\left(\btheta-\btheta_{0}\right)\right|\right]\\
   & \leqslant\left\Vert \btheta-\btheta_{0}\right\Vert _{1}\mathrm{E}\left[\left\Vert U\bX\right\Vert _{\infty}\right]\tag{H{\"o}lder}\\
   & \leqslant\widetilde{B}_{n}\overline{r}_{n}\sqrt{p}<\infty,
\end{align*}
where the last inequality stems from the Cauchy-Schwarz and Jensen inequalities
and Assumption
\ref{assu:ResidualBootstrapMethod}.\ref{enu:maxZijFourthMomentBnd}. Also, for
any $\tau\in[-1,1]$, $\|\btheta_\tau
-\btheta_{0}\|_{2}=|\tau|\|\btheta-\btheta_{0}\|_{2}\leqslant \overline{r}_{n}$,
so that every $\btheta_\tau,\tau\in[-1,1]$, lies in
$\mathcal{B}_{\btheta_0}(\overline{r}_{n})$ and, thus, $\Theta$.
Since $\overline r_n \leqslant c_L$, from the Cauchy-Schwarz and H{\"o}lder
inequalities and Assumptions
\ref{assu:LossLocallyLipschitzAndMore}.\ref{enu:LossLocallyLipschitz} and
\ref{assu:LossLocallyLipschitzAndMore}.\ref{enu:ResidualMeanSquareEll2Conts}
[recall that $L$ maps into $[1,\infty)$], we get
\begin{align*}
   \E\left[\left|f_1'\left(\tau,\bZ\right)-f_1'\left(0,\bZ\right)\right|\right]
   &=\E\left[\left|m'_1\left(\bX^{\top}\btheta_\tau,\bY\right)-m'_1\left(\bX^{\top}\btheta_0,\bY\right)\right|\left|\bX^{\top}\left(\btheta-\btheta_0\right)\right|\right] \\
   &\leqslant \sqrt{\E\left[\left|m'_1\left(\bX^{\top}\btheta_\tau,\bY\right)-m'_1\left(\bX^{\top}\btheta_0,\bY\right)\right|^2\right]}\sqrt{\E\left[\left|\bX^{\top}\left(\btheta-\btheta_0\right)\right|^2\right]}\\
   &\leqslant \sqrt{C_L^2 \overline{r}_n} B_n \overline{r}_n \sqrt{p}.
\end{align*}   
The claim now follows from the triangle inequality and $\tau\in[-1,1]$ being
arbitrary.
\end{proof}

\begin{proof}[{\sc{Proof of Lemma \ref{lem:DominationOfDifferenceQuotients}}}]
Since $m$ is (finite) convex in its first argument (Assumption
\ref{assu:Convexity}) and
$f\left(\tau,\bz\right)=m\left(\bx^{\top}\btheta_\tau,\by\right)-m\left(\bx^{\top}\btheta_{0},\by\right)$,
$f:\R\times\mathcal Z\to\R$ is (finite) convex and, thus, everywhere
subdifferentiable in its first argument with compact-valued subdifferential
\citep[Theorem 23.4]{rockafellar_convex_1970}. Letting $\partial_1f(\tau,\bz)$
denote the subdifferential of $f$ with respect to its first argument evaluated
at $(\tau,\bz)$, it follows that for any $(\tau_1,\tau_2)\in\R^2$ and any
$(\tau_1^{\ast},\tau_2^{\ast})\in\partial_{1}f\left(\tau_1,\bz\right)\times\partial_{1}f\left(\tau_2,\bz\right)$,
\begin{align*}
   f\left(\tau_1,\bz\right)-f\left(\tau_2,\bz\right) & \geqslant \tau_{2}^{\ast}\left(\tau_1-\tau_2\right)\geqslant-\left(\left|\tau_{1}^{\ast}\right|+\left|\tau_{2}^{\ast}\right|\right)\left|\tau_1-\tau_2\right|\quad\text{and}\\
   f\left(\tau_2,\bz\right)-f\left(\tau_1,\bz\right) & \geqslant \tau_{1}^{\ast}\left(\tau_2-\tau_1\right)\geqslant-\left(\left|\tau_{1}^{\ast}\right|+\left|\tau_{2}^{\ast}\right|\right)\left|\tau_1-\tau_2\right|,
\end{align*}
which combine to yield
\begin{align*}
\left|f\left(\tau_1,\bz\right)-f\left(\tau_2,\bz\right)\right| & \leqslant\left|\tau_1-\tau_2\right|\left(\sup_{\tau_{1}^{\ast}\in\partial_{1}f\left(\tau_1,\bz\right)}\left|\tau_{1}^{\ast}\right|+\sup_{\tau_{2}^{\ast}\in\partial_{1}f\left(\tau_2,\bz\right)}\left|\tau_{2}^{\ast}\right|\right).
\end{align*}   
Setting $\tau_1=\tau_0+h$ and $\tau_2=\tau_0$, we see that
\[
\left|f\left(\tau_0+h,\bz\right)-f\left(\tau_0,\bz\right)\right|\leqslant\left|h\right|\overline{f}\left(\bz\right)\;\text{for each}\;\tau_0\in\left[-1,1\right],\;h\in\left[-d(\tau_0),d(\tau_0)\right],\;\bz\in\mathcal{Z}, 
\]
for $\overline{f}:\mathcal{Z}\to[0,\infty]$ defined by
\[
\overline{f}\left(\bz\right):=2\sup_{\tau\in\left[-1,1\right]}\sup_{\tau^{\ast}\in\partial_{1}f\left(\tau,\bz\right)}\left|\tau^{\ast}\right|,\quad \bz\in\mathcal{Z}.
\]
Note that $\overline{f}$ thus defined depends on neither $\tau_0$ nor $h$. It
remains to show that $\overline{f}$ is $P$-integrable. To this end, note that
for each $\tau\in\R$ and $\bz\in\mathcal Z$, $\partial_{1}f(\tau,\bz)$ is the
non-empty compact interval
\[
\partial_{1}f\left(\tau,\bz\right)=\left[f_{1-}'\left(\tau,\bz\right),f_{1+}'\left(\tau,\bz\right)\right]
\]
\citep[p.~216]{rockafellar_convex_1970}, with $f_{1-}'\left(\tau,\bz\right)$ and
$f_{1+}'\left(\tau,\bz\right)$ being the left and right (partial) derivatives
\begin{align*}
f_{1-}'\left(\tau,\bz\right) & :=\lim_{\tau'\uparrow \tau}\frac{f\left(\tau',\bz\right)-f\left(\tau,\bz\right)}{\tau'-\tau},\\
f_{1+}'\left(\tau,\bz\right) & :=\lim_{\tau'\downarrow \tau}\frac{f\left(\tau',\bz\right)-f\left(\tau,\bz\right)}{\tau'-\tau},
\end{align*}
respectively, and both limits exist as real numbers per (finite) convexity of
$f(\cdot,\bz)$. It follows that the ``inner'' supremum in $\overline f(\bz)$ is attained at an interval
endpoint, i.e.~for all $(\tau,\bz)\in\R\times\mathcal Z$,
\[
\sup_{\tau^{\ast}\in\partial_{1}f\left(\tau,\bz\right)}\left|\tau^{\ast}\right|=\left|f_{1-}'\left(\tau,\bz\right)\right|\lor\left|f_{1+}'\left(\tau,\bz\right)\right|.
\]
Both the left and right derivatives are non-decreasing functions of $\tau$, cf.~
\citet[Theorem 24.1]{rockafellar_convex_1970} and $f(\cdot,\bz)$ being finite
convex (hence closed and proper). It follows that the ``outer'' supremum over
$\tau\in[-1,1]$ is also attained at an interval endpoint, so
\begin{align*}
\sup_{\tau\in\left[-1,1\right]}\sup_{\tau^{\ast}\in\partial_{1}f\left(\tau,\bz\right)}\left|\tau^{\ast}\right| 
   &= \sup_{\tau\in\left[-1,1\right]}\left|f_{1-}'\left(\tau,\bz\right)\right|\lor\sup_{\tau\in\left[-1,1\right]}\left|f_{1+}'\left(\tau,\bz\right)\right|\\
   & =\max\left\{ \left|f_{1-}'\left(-1,\bz\right)\right|,\left|f_{1-}'\left(1,\bz\right)\right|,\left|f_{1+}'\left(-1,\bz\right)\right|,\left|f_{1+}'\left(1,\bz\right)\right|\right\}. 
\end{align*}
Since $\partial_{1}f(-1,\bZ)=\left\{ f'_{1}(-1,\bZ)\right\} $ and
$\partial_{1}f(1,\bZ)=\left\{f'_{1}(1,\bZ)\right\} $ a.s.~(cf. Lemma
\ref{lem:AlmostSureExistenceOfPartials}), we have
$f_{1-}'\left(-1,\bZ\right)=f_{1+}'\left(-1,\bZ\right)$ and
$f_{1-}'\left(1,\bZ\right)=f_{1+}'\left(1,\bZ\right)$ a.s., and, thus,
\begin{align*}
   \mathrm{E}\left[\sup_{\tau\in\left[-1,1\right]}\sup_{\tau^{\ast}\in\partial_{1}f\left(\tau,\bZ\right)}\left|\tau^{\ast}\right|\right]
      & =\mathrm{E}\left[\left|f'_{1}\left(-1,\bZ\right)\right|\lor\left|f'_{1}\left(1,\bZ\right)\right|\right]\\
      & \leqslant\mathrm{E}\left[\left|f'_{1}\left(-1,\bZ\right)\right|\right]+\E\left[\left|f'_{1}\left(1,\bZ\right)\right|\right]\\
      & \leqslant 2\overline{r}_n\sqrt{p}\left(\widetilde B_n + C_LB_n\sqrt{\overline{r}_n}\right)<\infty,\tag{Lemma \ref{lem:L1BoundednessOfPartials}}
\end{align*}
implying that $\overline f$ is $P$-integrable.
\end{proof}

\begin{proof}[{\sc{Proof of Lemma \ref{lem:ExistenceOfDerivatives}}}]
Lemma \ref{lem:L1BoundednessOfLevels} shows that $g$ in \eqref{eq:gOfrDudley} is
well defined as a map from $[-1,1]$ to $\R$. Lemmas
\ref{lem:L1BoundednessOfLevels}, \ref{lem:AlmostSureExistenceOfPartials} and
\ref{lem:L1BoundednessOfPartials} combine to verify \citet[Condition
(A.2)]{dudley2014uniform} for our $f$ in \eqref{eq:fOfwandrDudley} for any
$\tau_0\in(-1,1)$ with $\delta$ there being our $d(\tau_0)$. Combining the
difference quotient domination by a $P$-integrable function in Lemma
\ref{lem:DominationOfDifferenceQuotients} with \citet[Corollary
A.3]{dudley2014uniform} now show that $g$ is differentiable at every
$\tau_0\in(-1,1)$ with
\[
g'(\tau_0)=\E[f_1'(\tau_0,\bZ)]=\E[m_1'(\bX^\top\btheta_{\tau_0},\bY)\bX^\top(\btheta-\btheta_0)].
\]
It remains to show that $g'(0)=0$. Since
$\btheta\in\mathcal{B}_{\btheta_0}\left(\overline{r}_{n}\right)$ and $\overline r_n \leqslant c_M'$,
Assumption \ref{assu:Margin} tells us that for all $h\in\left[-1,1\right]$,
\begin{align*}
g\left(h\right) - g\left(0\right) & =\mathcal{E}\left(\btheta_{0}+h\left(\btheta-\btheta_{0}\right)\right)\geqslant0.
\end{align*}
Seeking a contradiction, suppose first that
$g'(0)\in\left(0,\infty\right)$. Since the derivative exists, letting
$\left\{ h_{m}\right\} _{m=1}^{\infty}\subset[-1,0)$ be the strictly negative
vanishing sequence $h_{m}=-1/m$, we have 
\[
g'(0)=\lim_{m\to\infty}\frac{g\left(h_{m}\right)-g\left(0\right)}{h_{m}}\in(0,\infty).
\]
Hence, for all $m$ sufficiently large,
\[
\frac{g\left(h_{m}\right)-g\left(0\right)}{h_{m}}\in(0,\infty),
\]
which by $h_{m}<0$ implies $g\left(h_{m}\right)-g\left(0\right)\in(-\infty,0)$, a
contradiction. If we instead suppose that $g'(0)\in(-\infty,0)$, then
letting $\left\{ h_{m}\right\} _{m=1}^{\infty}\subset(0,1]$ be the
strictly positive vanishing sequence $h_{m}=1/m$, we again reach a
contradiction. It follows that $g'(0)=0$.
\end{proof}

\begin{proof}[{\sc{Proof of Lemma \ref{lem:EofUXis0}}}]
From Lemma \ref{lem:ExistenceOfDerivatives} we know that
\[
\mathrm{E}\left[U\bX^{\top}\left(\btheta-\btheta_{0}\right)\right]=0
\]
for any $\btheta\in \mathcal
B_{\btheta_{0}}\left(\overline{r}_{n}\right)\subseteq\Theta$. Seeking a
contradiction, suppose that $\mathrm{E}\left[U\bX\right]$ is non-zero. Then
$\mathrm{E}\left[U\bX\right]\in(0,\infty)$, so we may define
\[
\btheta:=\btheta_{0}+\frac{\overline{r}_{n}}{\left\Vert \mathrm{E}\left[U\bX\right]\right\Vert _{2}}\mathrm{E}\left[U\bX\right].
\]
This $\btheta$ belongs to $\mathcal
B_{\btheta_{0}}(\overline{r}_{n})$, but
\[
\mathrm{E}\left[U\bX^{\top}\right]\left(\btheta-\btheta_{0}\right)=\overline{r}_{n} \left\Vert \mathrm{E}\left[U\bX\right]\right\Vert _{2}\in(0,\infty),
\]
a contradiction. Conclude that $\mathrm{E}\left[U\bX\right]=\mathbf{0}_p$.
\end{proof}
   
\subsection{Proofs for Section \ref{sec:ResidualEstimationViaCrossValidation}}

For the arguments in this section, we introduce some additional notation.
For any non-empty $I\subsetneq[n]$, define
the \emph{subsample score} by
\[
\bS_{I}:=\mathbb{E}_{I}\left[m'_{1}\left(\bX_{i}^{\top}\btheta_{0},\bY_{i}\right)\bX_{i}\right]
\]
and the (random) \emph{subsample} \emph{empirical error} function $\epsilon_I:[0,\infty)\to[0,\infty)$ by
\begin{align*}
\epsilon_{I}\left(u\right) & :=\sup_{\substack{\bdelta\in\mathcal{R}\left(\overline{c}_{0},\eta_n\right),\\
\Vert\bdelta\Vert_{2}\leqslant u}} \left|\left(\mathbb{E}_{I}-\E\right)\left[m\left(\bX_{i}^{\top}\left(\btheta_{0}+\bdelta\right),\bY_{i}\right)-m\left(\bX_{i}^{\top}\btheta_{0},\bY_{i}\right)\right]\right|.
\end{align*}
Also, recall the notations $\eta_n = \sqrt{\ln(pn)/n}$ and $\overline
c_0=(c_0+1)/(c_0-1)$ for the user-chosen constant $c_0\in(1,\infty)$.

In proving Theorem \ref{cor: convergence rate bootstrap after cv}, we will
rely on the following eight lemmas, whose proofs can be found at the end of this section.

\begin{lem}\label{lem:ExistenceGoodCandidatePenalty} 
Let Assumption \ref{assu:CandidatePenalties} hold. Then for any constant
$C\in(0,\infty)$ satisfying $\eta_n\leqslant C_{\Lambda}a/C$ and
$n\eta_n\geqslant c_{\Lambda}C_{\Lambda}/C$, the candidate penalty set
$\Lambda_{n}$ and the interval $\left[C\eta_n,C\eta_n/a\right]$ have an element
in common.
\end{lem}
\begin{lem}\label{lem:SubsampleEmpiricalErrorBnd} 
Let Assumptions \ref{as: diff and int}, \ref{assu:LossLocallyLipschitzAndMore},
\ref{assu:Approximate-Sparsity} and \ref{assu:DataPartition} hold. Then
there is a universal constant $C\in[1,\infty)$, such that for any $n\in\N$,
$t\in[1,\infty)$ and $u\in(0,\infty)$ satisfying
\begin{equation}\label{eq: lemma e2 side condition}
\eta_n\leqslant 1,\quad\frac{B_n^2\ln(pn)}{n^{1/2}}\leqslant C_L^2\quad\text{and}\quad tn^{1/r}B_n\Big(u\sqrt{s_{q}\eta_{n}^{-q}}+s_{q}\eta_{n}^{1-q}\Big)\leqslant\frac{c_{L}}{1+\overline{c}_{0}},
\end{equation}
we have
\begin{align*}
\max_{1\leqslant k\leqslant K}\epsilon_{I_{k}^{c}}\left(u\right) & \leqslant \frac{(1+\overline c_0)CC_L}{(K-1)c_D}\left(u\sqrt{s_{q}\eta_{n}^{2-q}}+s_{q}\eta_{n}^{2-q}\right)
\end{align*}
with probability at least $1-K\left(4t^{-r} +  C/\ln^2(pn)
+\left[\left(K-1\right)c_{D}n\right]^{-1}\right)$.
\end{lem}

\begin{lem}\label{lem:SubsampleScoreBnd} 
Let Assumptions
\ref{assu:ParameterSpace}--\ref{assu:LossLocallyLipschitzAndMore},
\ref{assu:ResidualBootstrapMethod} and \ref{assu:DataPartition} hold.
Then there is a universal constant $C\in[1,\infty)$, such that
for any $n\in\N$ satisfying
\begin{equation}\label{eq: lemma e3 side condition}
\frac{\widetilde B_n^2 \ln(pn)}{\sqrt n} \leqslant C_U^2\sqrt{(K-1)c_D}\quad\text{and}\quad  pn\geqslant\left(\frac{1}{(K-1)c_D}\right)^2,
\end{equation}
we have
\[
\max_{1\leqslant k\leqslant K}\|\bS_{I_{k}^{c}}\|_{\infty}\leqslant \frac{CC_U\eta_n}{\sqrt{(K-1)c_D}}
\]
with probability at least $1-CK/\ln^4(pn)$.
\end{lem}

\begin{lem}\label{lem: theorem 1 extension}
Let Assumptions \ref{assu:ParameterSpace}--\ref{assu:Margin} and
\ref{assu:Approximate-Sparsity} hold, and let $a_{\epsilon,n},b_{\epsilon,n}$
and $\overline{\lambda}_n$ be non-random sequences in $(0,\infty)$. Define
$\widetilde u_n$ as in \eqref{eq: definition un tilde key} and suppose that
$\widetilde{u}_n\leqslant c_{M}'$ and \eqref{eq: lemma c3 additional constraint}
are satisfied. Then for any $k\in[K]$ and any (possibly random) $\lambda \in
\Lambda_n$, on the event 
\begin{equation}\label{eq: main event modified for lemma e4}
\{\lambda\geqslant c_{0}\Vert \bS_{I_k^c}\Vert_{\infty}\}\cap\{\lambda\leqslant\overline{\lambda}_n\}\cap\{\epsilon_{I_k^c}\left(\widetilde u_{n}\right)\leqslant a_{\epsilon,n}\widetilde u_n + b_{\epsilon,n}\},
\end{equation}
we have
\[
\sup_{\widehat\btheta\in\widehat\Theta_{I_k^c}(\lambda)}\mathcal{E}(\widehat\btheta) \leqslant\left(1+\overline{c}_{0}\right) \overline{\lambda}_n\left(\widetilde u_{n}\sqrt{s_q\eta_n^{-q}}+s_q\eta_n^{1-q}\right) + a_{\epsilon,n}\widetilde u_n + b_{\epsilon,n}.
\]
\end{lem}

\begin{lem}\label{lem:High-Probability-Subsample-Risk-Bound-Lambda-Star} 
Let Assumptions \ref{assu:ParameterSpace}--\ref{assu:Approximate-Sparsity} and
\ref{assu:ResidualBootstrapMethod}--\ref{assu:CandidatePenalties} hold, such
that Lemmas \ref{lem:SubsampleEmpiricalErrorBnd} and \ref{lem:SubsampleScoreBnd}
apply, and let $C\in[1,\infty)$ be the largest universal constant appearing in
these two lemmas. Define the constants
\begin{align*}
   C_1&:=\frac{C C_U}{\sqrt{(K-1)c_D}},
   \\C_2&:=\frac{C C_L}{(K-1)c_D},\\
   C_{S}&:=\frac{2\left(1+\overline c_0\right)}{c_M}\left(\frac{c_0 C_1}{a} + C_2\right)\quad\text{and}\quad\\
   C_{\mathcal E}&:=\frac{c_M C_S (1+C_S)}{2}.
\end{align*}
Then for any $n\in\N$ and $t\in[1,\infty)$ satisfying both \eqref{eq: lemma e3
side condition} and
\begin{equation}\label{eq:SideCondsHighProbSubsampleRiskBndLambdaStar}
   \left\{
   \begin{aligned}
      &\frac{c_{\Lambda}C_{\Lambda}}{c_0 C_1 n}\leqslant \eta_n\leqslant 1\wedge \frac{C_{\Lambda}a}{c_0 C_1},\quad \frac{B_n^2\ln(pn)}{n^{1/2}}\leqslant C_L^2,\\
      & C_S\sqrt{s_q\eta_n^{2-q}}\leqslant c_M' \quad\text{and}\quad (1+C_S)tn^{1/r}B_ns_q\eta_n^{1-q} \leqslant\frac{c_{L}}{1+\overline{c}_{0}},      
   \end{aligned}
   \right\}
\end{equation}
there is a non-random candidate penalty level $\lambda_{\ast} \in\Lambda_{n}$,
such that
\[
\max_{1\leqslant k\leqslant K}\sup_{\widehat\btheta\in\widehat{\Theta}_{I_{k}^{c}}\left(\lambda_{\ast}\right)}\mathcal{E}(\widehat\btheta)\leqslant C_{\mathcal E}s_q\eta_n^{2-q}
\]
with probability at least $1-K(4t^{-r} + 2C/\ln^2(pn)+[(K-1)c_{D}n]^{-1})$.
\end{lem}

\begin{lem}\label{lem: convexity minoration implication} Let Assumptions
\ref{assu:ParameterSpace}--
\ref{assu:Margin} hold. Then for all $\btheta\in\Theta$ such that $\mathcal
E(\btheta) \leqslant c_M(c_M')^2$, we have $\|\btheta - \btheta_0\|_2^2 \leqslant
\mathcal E(\btheta) / c_M$.
\end{lem}

\begin{lem}\label{lem:High-Probability-Subsample-Risk-Bound-Lambda-CV}
Let Assumptions \ref{assu:ParameterSpace}--\ref{assu:Approximate-Sparsity} and
\ref{assu:ResidualBootstrapMethod}--\ref{assu:CandidatePenalties} hold, let
$C$, $C_S$, and $C_{\mathcal E}$ be the constants defined in Lemma
\ref{lem:High-Probability-Subsample-Risk-Bound-Lambda-Star}, and define the
non-random sequence $\overline{\mathcal{E}}^{\ast}_{n}:=C_{\mathcal E} s_q
\eta_n^{2-q}$. For each $(k,\lambda)\in[K]\times\Lambda_n$, fix a solution
$\widehat\btheta_{I_k^c}(\lambda)\in\widehat\Theta_{I_k^c}(\lambda)$ to
\eqref{eq:SubsampleEstimator} and a solution $\widehat{\lambda}^{\mathtt{cv}}$
to \eqref{eq:LambdaCVDefn} based on
$\{\widehat\btheta_{I_k^c}(\lambda)\}_{(k,\lambda)\in[K]\times\Lambda_n}$. Then for
any $(n,t)\in\N\times[1,\infty)$ satisfying \eqref{eq: lemma e3 side condition},
\eqref{eq:SideCondsHighProbSubsampleRiskBndLambdaStar} and
\begin{equation}
\left\{ n\geqslant\frac{1}{c_{\Lambda}\wedge a}\quad\text{and}\quad \frac{11C_L}{4c_Dc_M}\sqrt{\frac{3t\ln n}{\ln(1/a)n}} + \frac{3}{2}\sqrt{\frac{\overline{\mathcal{E}}^{\ast}_{n}}{c_Dc_M}}\leqslant c_M'\wedge c_L\right\},\label{eq:AddedSideCondsHighProbSubsampleRiskBndLambdaCV}
\end{equation}
we have
\begin{equation}\label{eq: lemma a6 main bound}
\max_{1\leqslant k\leqslant K}\|\widehat\btheta_{I_{k}^{c}}(\widehat{\lambda}^{\mathtt{cv}}) - \btheta_0 \|_2 \leqslant \frac{11C_L}{4c_Dc_M}\sqrt{\frac{3t\ln n}{\ln(1/a)n}} + \frac{3}{2}\sqrt{\frac{\overline{\mathcal{E}}^{\ast}_{n}}{c_Dc_M}}
\end{equation}
with probability at least $1 - K(4t^{-r}+3t^{-1}+2C/\ln^2(pn)+[(K-1)c_{D}n]^{-1})$.
\end{lem}

\begin{lem}\label{thm:High-Probability-L2n-Error-Bnd-CV-Residuals}
Let Assumptions \ref{assu:ParameterSpace}--\ref{assu:Approximate-Sparsity} and
\ref{assu:ResidualBootstrapMethod}--\ref{assu:CandidatePenalties} hold, let $C$,
$C_S$, and $C_{\mathcal E}$ be the constants defined in Lemma
\ref{lem:High-Probability-Subsample-Risk-Bound-Lambda-Star}, and define the
non-random sequence $\overline{\mathcal{E}}^{\ast}_{n}=C_{\mathcal E} s_q
\eta_n^{2-q}$. For each $(k,\lambda)\in[K]\times\Lambda_n$, fix a solution
$\widehat\btheta_{I_k^c}(\lambda)\in\widehat\Theta_{I_k^c}(\lambda)$ to
\eqref{eq:SubsampleEstimator} and a solution $\widehat{\lambda}^{\mathtt{cv}}$
to \eqref{eq:LambdaCVDefn} based on
$\{\widehat\btheta_{I_k^c}(\lambda)\}_{(k,\lambda)\in[K]\times\Lambda_n}$, and define
$\widehat{U}_i^{\mathtt{cv}}$ as in \eqref{eq:ResidualEstimatorCV} based on
$\{\widehat\btheta_{I_k^c}(\widehat{\lambda}^{\mathtt{cv}})\}_{k\in[K]}$. Then
for any $(n,t)\in\N\times[1,\infty)$ satisfying \eqref{eq: lemma e3 side
condition}, \eqref{eq:SideCondsHighProbSubsampleRiskBndLambdaStar} and
\eqref{eq:AddedSideCondsHighProbSubsampleRiskBndLambdaCV}, we have
\begin{equation}
\En\big[(\widehat{U}_{i}^{\mathtt{cv}}-U_{i})^{2}\big]\leqslant\frac{3C_{L}^{2}t\ln n}{\ln\left(1/a\right)}\left( \frac{11C_L}{4c_Dc_M}\sqrt{\frac{3t\ln n}{\ln(1/a)n}} + \frac{3}{2}\sqrt{\frac{\overline{\mathcal{E}}_n^{\ast}}{c_Dc_M}} \right)\label{eq:High-Prob-CV-Residual-L2n-Bnd}
\end{equation}
with probability at least $1 -
K(4t^{-r}+4t^{-1}+2C/\ln^2(pn)+[(K-1)c_{D}n]^{-1})$
\end{lem}

\begin{proof}
[\sc{Proof of Theorem \ref{cor: convergence rate bootstrap after cv}}] 
The proof
will follow from Lemma \ref{thm:RatesBM}, with \eqref{eq: residual estimation
side condition} being verified via Lemma
\ref{thm:High-Probability-L2n-Error-Bnd-CV-Residuals}. Observe that \eqref{eq:
difficult restriction} ensures that there is a non-random sequence $t_n$ in
$[1,\infty)$ such that
\begin{equation*}
t_n\to\infty, \quad t_n n^{1/r} B_ns_q\eta_n^{1-q}\to0\quad\text{and}\quad  \frac{t_n^3 B_n^4 s_q(\ln(pn))^{5-q/2}(\ln n)^2}{n^{1-q/2-4/r}}\to 0.
\end{equation*}
Therefore, for $\overline{\mathcal{E}}_n^{\ast}$ appearing in the statement of
Lemma \ref{thm:High-Probability-L2n-Error-Bnd-CV-Residuals}, setting
\[
\delta_n^2:=\frac{3C_{L}^{2}t_n\ln n}{\ln\left(1/a\right)}\left( \frac{11C_L}{4c_Dc_M}\sqrt{\frac{3t_n\ln n}{\ln(1/a)n}} + \frac{3}{2}\sqrt{\frac{\overline{\mathcal{E}}_n^{\ast}}{c_Dc_M}} \right)\ln^2(pn),
\]
a calculation shows that $n^{1/r}B_n\delta_n\to 0$. 
Together with \eqref{eq: difficult restriction}, this implies that \eqref{eq:
simple restriction bm} is satisfied. Also, \eqref{eq: difficult restriction} and
$t_n n^{1/r} B_ns_q\eta_n^{1-q}\to0$ imply that, setting $t=t_n$, \eqref{eq:
lemma e3 side condition} and
\eqref{eq:SideCondsHighProbSubsampleRiskBndLambdaStar} will hold for all $n$
large enough. In addition, given that $\delta_n\to 0$ and $t_n\to\infty$, it
follows that, setting $t=t_n$,
\eqref{eq:AddedSideCondsHighProbSubsampleRiskBndLambdaCV} will hold for all $n$
large enough as well. Thus, Lemma
\ref{thm:High-Probability-L2n-Error-Bnd-CV-Residuals} implies that, setting
$\widehat U_i = \widehat{U}_{i}^{\mathtt{cv}}$ for all $i\in[n]$, \eqref{eq:
residual estimation side condition} is satisfied. The asserted claim now follows
from applying Lemma \ref{thm:RatesBM}.
\end{proof}

\begin{proof}[\sc{Proof of Theorem \ref{thm: convergence rates post estimator}}] 
The asserted claim will follow from an application of Theorem \ref{thm; post
penalized estimator non asymptotic}, which we state and prove in a separate
appendix (see Section \ref{sec:analysis post estimator}), due to its length.

Observe that \eqref{eq: even more difficult restriction} ensures that there is a
non-random sequence $t_n$ in $[1,\infty)$ such that
\begin{equation*}
t_n\to\infty, \quad t_n n^{1/r} B_ns_q\eta_n^{1-q}\to0\quad\text{and}\quad  \frac{t_n^3 B_n^4 s_q(\ln(pn))^{5-q/2}(\ln n)^2}{n^{1-q/2-4/r}}\to 0.
\end{equation*}
Therefore, for $\overline{\mathcal{E}}_n^{\ast}$ appearing in the statement of
Lemma \ref{thm:High-Probability-L2n-Error-Bnd-CV-Residuals}, setting
$$
\delta_n^2:=\frac{3C_{L}^{2}t_n\ln n}{\ln\left(1/a\right)}\left( \frac{11C_L}{4c_Dc_M}\sqrt{\frac{3t_n\ln n}{\ln(1/a)n}} + \frac{3}{2}\sqrt{\frac{\overline{\mathcal{E}}_n^{\ast}}{c_Dc_M}} \right)\ln^2(pn),
$$
we have $n^{1/r}B_n\delta_n\to 0$.  Also, \eqref{eq: even more difficult
restriction} and $t_n n^{1/r} B_ns_q\eta_n^{1-q}\to0$ imply that \eqref{eq:
lemma e3 side condition} and
\eqref{eq:SideCondsHighProbSubsampleRiskBndLambdaStar} with $t=t_n$ hold for all
$n$ large enough. In addition, given that $\delta_n\to 0$ and $t_n\to\infty$, it
follows that \eqref{eq:AddedSideCondsHighProbSubsampleRiskBndLambdaCV} with
$t=t_n$ holds for all $n$ large enough as well. Thus, given that
$\E[\max_{1\leqslant i\leqslant n}\|\bX_i\|_{\infty}^r]\leqslant n B_n^r$ by
Assumption
\ref{assu:LossLocallyLipschitzAndMore}.\ref{enu:LossLocallyLipschitz}, Lemma
\ref{thm:High-Probability-L2n-Error-Bnd-CV-Residuals} together with
$n^{1/r}B_n\delta_n\to 0$ imply that
\begin{equation}\label{eq:ZhatsApproxZsPostBCVPf}
\P\left( \max_{1\leqslant j\leqslant p}\En\left[(\widehat{U}_{i}^{\mathtt{cv}}X_{i,j}-U_{i}X_{i,j})^2\right] > \frac{\widetilde\delta_n^2}{\ln^2(pn)} \right)\to 0
\end{equation}
for some non-random sequence $\widetilde\delta_n$ in $(0,1)$ satisfying
$\widetilde \delta_n \to 0$. Let $\boldsymbol{\mathcal{N}}$ be a centered normal
random vector in $\R^p$ with covariance matrix $\E[U^2\bX\bX^{\top}]$ and for
all $\beta\in(0,1)$, let $q^{\boldsymbol{\mathcal{N}}}(\beta)$ be the $\beta$th
quantile of $\|\boldsymbol{\mathcal{N}}\|_{\infty}$. Theorem
\ref{thm:QuantileComparison} together with Assumption
\ref{assu:ResidualBootstrapMethod} and \eqref{eq: even more difficult
restriction} then imply that with probability $1 - o(1)$,
$$
\frac{q^{\boldsymbol{\mathcal{N}}}(1-\alpha_n - \rho_n)}{\sqrt{n}} \leqslant \widehat{q}^{\mathtt{bcv}}\left(1-\alpha_n\right)\quad\text{and}\quad \widehat{q}^{\mathtt{bcv}}(1/2) \leqslant \frac{q^{\boldsymbol{\mathcal{N}}}(1/2 + \rho_n)}{\sqrt{n}}
$$
for some non-random sequence $\rho_n$ in $(0,\infty)$ satisfying $\rho_n \to 0$.
Also, from (the median version of) Borell's inequality \citep[Proposition
A.2.1]{van_der_vaart_weak_1996}, we get
\begin{align*}
\widehat{q}^{\mathtt{bcv}}(1-\alpha_n) 
\leqslant \widehat{q}^{\mathtt{bcv}}(1/2) + \sqrt{\max_{1\leqslant j\leqslant p}\En[(\widehat{U}_{i}^{\mathtt{cv}}X_{i,j})^2]}\sqrt{\frac{2\ln(1/\alpha_n)}{n}}.
\end{align*}
Using \eqref{eq:ZhatsApproxZsPostBCVPf} and the same arguments as those in the
proof of Theorem \ref{thm:NonasyHighProbBndsBM}, we get $$\max_{1\leqslant
j\leqslant p}\En[(\widehat{U}_{i}^{\mathtt{cv}}X_{i,j})^2]\leqslant
4(\widetilde{C}C_U)^2$$ with probability $1-o(1)$ for some universal constant
$\widetilde C\in[1,\infty)$. Since $\ln(1/\alpha_n)\lesssim \ln(pn)$, we thus
have
$$
\frac{q^{\mathcal N}(1-\alpha_n - \rho_n)}{\sqrt{n}} \leqslant \widehat{q}^{\mathtt{bcv}}\left(1-\alpha_n\right) \leqslant \frac{q^{\mathcal N}(1/2 + \rho_n)}{\sqrt{n}} + C\eta_n
$$
with probability $1 - o(1)$ for some constant $C\in[1,\infty)$. Setting
$\underline\lambda_n$ and $\overline\lambda_n$ to be $c_0$ times the left-hand
side and the right-hand side of this chain of inequalities, respectively, we
thus have $\P(\widehat{\lambda}^{\mathtt{bcv}}_{\alpha_n} > \overline\lambda_n)
\to0$ and $\P(\widehat{\lambda}^{\mathtt{bcv}}_{\alpha_n} < \underline\lambda_n)
\to0$. The same arguments as those used to invoke the multiplier
bootstrap consistency result (Theorem \ref{thm:MultiplierBootstrapConsistency})
in the proof Theorem \ref{thm:NonasyHighProbBndsBM} show that
$\P(\widehat{\lambda}^{\mathtt{bcv}}_{\alpha_n} <
c_0\|\bS_n\|_{\infty})\leqslant \alpha_n+\rho_n\to 0$. Since $\rho_n\to 0$, we
must eventually have $q^{\mathcal N}(1/2 + \rho_n)\leqslant q^{\mathcal
N}(2/3)$. Using the Gaussian quantile bound (Lemma
\ref{lem:GaussianQuantileBnd}) alongside Assumption
\ref{assu:ResidualBootstrapMethod}, we therefore get $\overline\lambda_n
\lesssim \eta_n$. Similarly, because $\alpha_n\to0$ and $\rho_n\to0$, we must
eventually have $q^{\mathcal N}(1-\alpha_n - \rho_n)\geqslant q^{\mathcal
N}(1/2)$. Lower bounding the maximum $\|\boldsymbol{\mathcal N}\|_\infty$ by the
coordinate $|\mathcal{N}_1|$ (for example), and observing that the quantile of a
folded normal distribution scales linearly with its standard deviation, again
using Assumption \ref{assu:ResidualBootstrapMethod}, we get
$n^{-1/2}\lesssim\underline\lambda_n$. It follows that $\phi_n:= ((\eta_n^2 +
\overline\lambda_n^2)/\underline\lambda_n^2)^{1/2}\lesssim \sqrt{\ln(pn)}$. The
asserted claim now follows from an application of Theorem \ref{thm; post
penalized estimator non asymptotic}, which is justified by \eqref{eq: even more
difficult restriction}.
\end{proof}

\begin{proof}
[\sc{Proof of Lemma \ref{lem:ExistenceGoodCandidatePenalty}}] Fix any
$C\in(0,\infty)$ satisfying $\eta_n\leqslant C_{\Lambda}a/C$ and
$n\eta_n\geqslant c_{\Lambda}C_{\Lambda}/C $. Denote $b_{n}:=C\eta_n$. We will
show that there is an integer $\ell_0\in\{0,1,2,\dotsc\}$ such that 
\begin{equation}\label{eq: lemma a1 chain}
c_{\Lambda}C_{\Lambda}/n \leqslant b_n \leqslant C_{\Lambda}a^{\ell_0}\leqslant b_n/a.
\end{equation}
By Assumption \ref{assu:CandidatePenalties}, this will imply that
$C_{\Lambda}a^{\ell_0}$ belongs to both the candidate penalty set $\Lambda_n$
and the interval $[C\eta_n,C\eta_n/a]$. To prove \eqref{eq: lemma a1 chain},
note that the condition $\eta_n\leqslant C_{\Lambda}a/C$ implies that
$$
0\leqslant\frac{\ln\left(b_{n}/C_{\Lambda}\right)}{\ln a}-1.
$$
Thus, there exists an integer $\ell_0\in\{0,1,2,\dots\}$ such that
$$
\frac{\ln\left(b_{n}/C_{\Lambda}\right)}{\ln a}-1\leqslant\ell_{0}\leqslant\frac{\ln\left(b_{n}/C_{\Lambda}\right)}{\ln a}.
$$
In turn, the latter implies that $b_n \leqslant C_{\Lambda}a^{\ell_0}\leqslant
b_n/a$. Moreover, the condition $n\eta_n\geqslant c_{\Lambda}C_{\Lambda}/C$
means that $c_{\Lambda}C_{\Lambda}/n\leqslant b_n$. Combining these inequalities
gives \eqref{eq: lemma a1 chain} and completes the proof of the lemma.
\end{proof}

\begin{proof}[\sc{Proof of Lemma \ref{lem:SubsampleEmpiricalErrorBnd}}] 
The claim will follow from $K$ applications of the maximal inequality in Theorem
\ref{lem:MaxIneqBasedOnContraction} in combination with the union bound. The
proof is very similar to that of Lemma \ref{lem:EmpiricalErrorBound}. We include
the steps for the sake of completeness.
First, fix $n\in\N$, $t\in[1,\infty)$ and $u\in(0,\infty)$ satisfying \eqref{eq:
lemma e2 side condition} and denote
$\Delta(u,\eta_n):=\mathcal{R}(\overline{c}_{0},\eta_n)\cap\left\{ \left\Vert
\cdot\right\Vert _{2}\leqslant u\right\} $. Lemma
\ref{lem:RestrictedSetConsequence} shows that
\begin{equation}
\left\Vert \Delta\left(u,\eta_n\right)\right\Vert _{1}:=\sup_{\bdelta\in\Delta\left(u,\eta_n\right)}\left\Vert \delta\right\Vert _{1}\leqslant\left(1+\overline{c}_{0}\right)\left(u\sqrt{s_{q}\eta_n^{-q}}+s_{q}\eta_n^{1-q}\right)=:\overline{\Delta}_n(u).\label{eq:DeltaEtaOfuEll1Bnd}
\end{equation}
Setting up for an application of Theorem \ref{lem:MaxIneqBasedOnContraction},
define $h:\R\times\mathcal{X}\times\mathcal Y\to\R$ by
$h(t,\bx,\by):=m(\bx^{\top}\btheta_{0}+t,\by)-m(\bx^{\top}\btheta_{0},\by)$ for
all $t\in\R$ and $(\bx,\by)\in\mathcal{X}\times\mathcal Y$. By construction,
$h(0,\cdot,\cdot)\equiv0$. By Assumption
\ref{assu:LossLocallyLipschitzAndMore}.\ref{enu:LossLocallyLipschitz}, the
restriction $h:[-c_{L},c_{L}]\times\mathcal{X}\times\mathcal Y\to\R$ is
$L(\bx,\by)$-Lipschitz in its first argument, thus verifying Condition
\ref{enu:ContractionConsequenceLocallyLipschitz} of Theorem
\ref{lem:MaxIneqBasedOnContraction} with $C_h=c_L$. H\"{o}lder's inequality,
Assumption
\ref{assu:LossLocallyLipschitzAndMore}.\ref{enu:LossLocallyLipschitz},
(\ref{eq:DeltaEtaOfuEll1Bnd}), and \eqref{eq: lemma e2 side condition} imply
that
\[
\max_{1\leqslant k\leqslant K}\max_{i\in I_k^c}\sup_{\bdelta\in\Delta\left(u,\eta_{n}\right)}\left|\bX_{i}^{\top}\bdelta\right|\leqslant\max_{1\leqslant i\leqslant n}\left\Vert \bX_{i}\right\Vert _{\infty}\left\Vert \Delta\left(u,\eta_{n}\right)\right\Vert _{1}\leqslant tn^{1/r}B_{n}\overline{\Delta}_{n}\left(u\right)\leqslant c_{L}
\]
with probability at least $1-t^{-r}$, where the bound $\P(\max_{1\leqslant
i\leqslant n}\|\bX_i\|_{\infty}>tn^{1/r}B_n)\leqslant t^{-r}$ follows from
Markov's inequality since Assumption
\ref{assu:LossLocallyLipschitzAndMore}.\ref{enu:LossLocallyLipschitz} implies
that $\E[\|\bX\|_{\infty}^r]\leqslant B_n^r$. Condition
\ref{enu:ContractionConsequenceXprimeDeltaBounded} of Theorem
\ref{lem:MaxIneqBasedOnContraction} thus holds with $C_h=c_L$ and
$\zeta_n=t^{-r}$. Further, given that $s_{q},t,B_n\in[1,\infty)$, and
$\eta_n\leqslant 1$ by assumption, \eqref{eq: lemma e2 side condition} implies
that $u\leqslant c_{L}.$ Therefore, it follows from Assumption
\ref{assu:LossLocallyLipschitzAndMore}.\ref{enu:LossMeanSquareEll2Conts} that
\begin{align*}
&\sup_{\bdelta\in\Delta\left(u,\eta_{n}\right)}\mathrm{E}\left[h(\bX^{\top}\bdelta,\bX,\bY)^{2}\right]\\
& \leqslant\sup_{\substack{\btheta_{0}+\bdelta\in\Theta\\
\|\bdelta\|_{2}\leqslant u}
}\mathrm{E}\left[\left|m\left(\bX^{\top}\left(\btheta_{0}+\bdelta\right),\bY\right)-m\left(\bX^{\top}\btheta_{0},\bY\right)\right|^{2}\right]\leqslant C_{L}^{2}u^{2},
\end{align*}
and so Condition \ref{enu:ContractionConsequencehBoundedInL2} of Theorem
\ref{lem:MaxIneqBasedOnContraction} holds for $B_{1n}=C_{L}u$. For
(the final) Condition \ref{enu:ContractionConsequenceLtimesXBoundedInEmpL2}
of Theorem \ref{lem:MaxIneqBasedOnContraction}, note that for some universal
constant $C\in[1,\infty)$, we have
\begin{align*}
\max_{1\leqslant k\leqslant K}\max_{1\leqslant j\leqslant p}\mathbb{E}_{I_{k}^{c}}\bracn{L(\bX_{i},\bY_i)^{2}X_{i,j}^{2}}
\leqslant\frac{1}{\left(K-1\right)c_{D}}\max_{1\leqslant j\leqslant p}\En\bracn{L(\bX_{i},\bY_i)^{2}X_{i,j}^{2}}
\leqslant\frac{(CC_{L})^{2}}{\left(K-1\right)c_{D}}
\end{align*}
with probability at least $1 - C/\ln^2(pn)$, where the first (deterministic)
inequality follows from Assumption \ref{assu:DataPartition} and the second
(probabilistic) inequality follows from the argument leading to \eqref{eq: lemma
c4 useful inequality to be used later} in the proof of Lemma
\ref{lem:EmpiricalErrorBound}.
Condition \ref{enu:ContractionConsequenceLtimesXBoundedInEmpL2} of Theorem
\ref{lem:MaxIneqBasedOnContraction} thus holds with $\gamma_{n}=C/\ln^2(pn)$ and
the now $\left(K,c_{D}\right)$-dependent $B_{2n}=
CC_L/\sqrt{\left(K-1\right)c_{D}}$. 
Theorem \ref{lem:MaxIneqBasedOnContraction} combined with the bound
$\overline{\Delta}_{n}(u)$ on $\|\Delta(u,\eta_{n})\|_{1}$ from
(\ref{eq:DeltaEtaOfuEll1Bnd}) and $\ln(8pn)\leqslant4\ln(pn)$ (which follows
from $p\geqslant2$) now show that for any given $k\in[K]$, we have
\begin{multline}
 \mathrm{P}\left(\sqrt{|I_k^c|}\epsilon_{I_{k}^{c}}\left(u\right)>\left\{ 4C_{L}u\right\} \lor\left\{ 16\sqrt{2}CC_{L}\overline{\Delta}_{n}\left(u\right)\sqrt{\frac{\ln\left(pn\right)}{(K-1)c_D}}\right\} \right) \\
 \leqslant 4t^{-r} +  4C/\ln^2(pn) +\left[\left(K-1\right)c_{D}n\right]^{-1},\label{eq: application of fancy maximal inequality, new}
\end{multline}
where we also used Assumption \ref{assu:DataPartition} to bound $|I_k^c|^{-1}$.
Next, Assumption \ref{assu:DataPartition} also implies that $(K-1)c_D\leqslant
Kc_D\leqslant 1$. Now, given that $s_q,C\in[1,\infty)$, $p\in[2,\infty)$,
$n\in[3,\infty)$, and $\eta_n\in(0,1]$, it follows that
$4\sqrt{2}C(1+\overline{c}_0)\{s_q
\eta_n^{-q}\ln(pn)/[(K-1)c_D]\}^{1/2}\geqslant 1$, and so
$16\sqrt{2}CC_{L}\overline{\Delta}_{n}\left(u\right)\{\ln\left(pn\right)/[(K-1)c_D]\}^{1/2}\geqslant4C_{L}u.$
It follows from \eqref{eq: application of fancy
maximal inequality, new} that
\begin{align*}
\epsilon_{I_{k}^{c}}\left(u\right) 
&\leqslant 16\sqrt{2}CC_{L}\overline{\Delta}_{n}\left(u\right)\sqrt{\frac{\ln\left(pn\right)}{(K-1)c_D|I_k^c|}}
\leqslant \frac{16\sqrt{2}CC_{L}}{(K-1)c_D}\overline{\Delta}_{n}\left(u\right)\eta_n
\end{align*}
with probability at least $1-(4t^{-r} + 4C/\ln^2(pn) +
[(K-1)c_{D}n]^{-1})$. The latter bound is uniform in
$k\in[K]$. Redefining the universal constant $C$ appropriately, the union bound
produces the desired result.
\end{proof}

\begin{proof}[\sc{Proof of Lemma \ref{lem:SubsampleScoreBnd}}] 
The claim will follow from $K$ applications of Theorem \ref{thm: max inequality
standard} in combination with the union bound. Fix $k\in[K]$. We invoke Theorem
\ref{thm: max inequality standard} with $\bZ_i=U_i\bX_i$, and $i\in I_k^c$, such
that $n$ there corresponds to $|I_k^c|$. Lemma \ref{lem:EofUXis0} shows that
these random variables are centered. Assumption
\ref{assu:ResidualBootstrapMethod}.\ref{enu:ZijSecondMomentsBndAwayZero}
involves $\max_{1\leqslant j\leqslant p}\E[|UX_j|^2]\leqslant C_U^2$, so
Condition \ref{enu:MaxIneqIISecondMomentBnd} of Theorem \ref{thm: max inequality
standard} is satisfied with $\sigma_n=C_U$, a constant. Assumption
\ref{assu:ResidualBootstrapMethod}.\ref{enu:maxZijFourthMomentBnd} shows that
Condition \ref{enu:MaxIneqIIMomentOfMaxBnd} of Theorem \ref{thm: max inequality
standard} is satisfied with $q=4$ and $M_n=\widetilde B_n$. For (the final)
Condition \ref{enu:MaxIneqIIGrowthCond} of Theorem \ref{thm: max inequality
standard}, observe from \eqref{eq: lemma e3 side condition} that
\begin{align*}
\frac{\widetilde B_n \sqrt{\ln(pn)}}{|I_k^c|^{1/4}}
&\leqslant \frac{\widetilde B_n \sqrt{\ln(pn)}}{[(K-1)c_Dn]^{1/4}}\leqslant C_U,
\end{align*}
where $|I_k^c|\geqslant(K-1)c_Dn$ follows from Assumption
\ref{assu:DataPartition}. Therefore, applying Theorem \ref{thm: max inequality
standard}, we obtain that there is a universal constant $C\in[1,\infty)$, such
that
\begin{align*}
\P\left( \sqrt{|I_k^c|}\|\bS_{I_k^c}\|_{\infty} > CC_U \sqrt{\ln(p|I_k^c|)} \right) 
&\leqslant \frac{C}{\ln^4(p|I_k^c|)} \leqslant \frac{C}{\ln^4(p(K-1)c_Dn)} \leqslant \frac{2^4C}{\ln^4(pn)},
\end{align*}
where the last inequality follows from $pn\geqslant 1/[(K-1)c_D]^2$. Hence, with
probability at least $1-2^4C/\ln^4(pn)$,
\begin{align*}
   \|\bS_{I_k^c}\|_{\infty} \leqslant CC_U \sqrt{\frac{\ln(p|I_k^c|)}{|I_k^c|}}
      \leqslant \frac{CC_U}{\sqrt{(K-1)c_D}} \sqrt{\frac{\ln(pn)}{n}} = \frac{CC_U\eta_n}{\sqrt{(K-1)c_D}}.
\end{align*}
The latter bound is uniform in $k\in[K]$. The claim now follows from combining
this inequality with the union bound, and redefining the universal constant $C$
appropriately.
\end{proof}

\begin{proof}
[\sc{Proof of Lemma \ref{lem: theorem 1 extension}}]
Fix $k\in[K]$ and observe that Lemma \ref{lem:NonasymptoticDeterministicBounds}
still holds if we replace $\bS_n$, $\epsilon_n(\widetilde u_n)$, and
$\mathscr{S}_n\cap\mathscr{L}_n\cap\mathscr{E}_n$ by $\bS_{I_k^c}$,
$\epsilon_{I_k^c}(\widetilde u_n)$, and \eqref{eq: main event modified for lemma
e4}, respectively. Fix $\widehat\btheta \in
\widehat\Theta_{I_k^c}(\lambda)$ and denote $\widehat\bdelta := \widehat\btheta -
\btheta_0$. Then $\widehat\bdelta\in \mathcal R(\overline c_0,\eta_n)$,
$\|\widehat\bdelta\|_2 \leqslant \widetilde u_n$ and
$\Vert\widehat{\bdelta}\Vert_{1}\leqslant(1+\overline{c}_{0})(\widetilde
u_{n}\sqrt{s_q\eta_n^{-q}}+s_q\eta_n^{1-q})$ on the event \eqref{eq: main event
modified for lemma e4}. It follows that
\begin{align*}
\mathcal{E}(\widehat{\btheta}) 
& \leqslant \mathbb E_{I_k^c}[m(\bX_i^{\top}\widehat\btheta,\bY_i)] - \mathbb E_{I_k^c}[m(\bX_i^{\top}\btheta_0,\bY_i)] + \epsilon_{I_k^c}(\widetilde u_n) \\
& \leqslant \lambda(\|\btheta_0\|_1 - \|\widehat\btheta\|_1)+\epsilon_{I_k^c}\left(\widetilde u_{n}\right)\\
& \leqslant\overline{\lambda}_n\Vert\widehat{\bdelta}\Vert_{1}+\epsilon_{I_k^c}\left(\widetilde u_{n}\right)\\
& \leqslant \left(1+\overline{c}_{0}\right) \overline{\lambda}_n\left(\widetilde u_{n}\sqrt{s_q\eta_n^{-q}}+s_q\eta_n^{1-q}\right) + a_{\epsilon,n}\widetilde u_n + b_{\epsilon,n},
\end{align*}
where the first inequality follows from the definition of
$\epsilon_{I_k^c}(\widetilde u_n)$, the second from the definition of
$\widehat\btheta$, the third from the triangle inequality and \eqref{eq: main
event modified for lemma e4}, and the fourth from \eqref{eq: main event modified
for lemma e4} again. The asserted claim now follows from taking the supremum 
over $\widehat\btheta \in \widehat\Theta_{I_k^c}(\lambda)$.
\end{proof}

\begin{proof}[\sc{Proof of Lemma
\ref{lem:High-Probability-Subsample-Risk-Bound-Lambda-Star}}] 
Let $(n,t)\in\N\times[1,\infty)$ satisfy \eqref{eq: lemma e3 side condition}
and \eqref{eq:SideCondsHighProbSubsampleRiskBndLambdaStar}. Then Lemma
\ref{lem:ExistenceGoodCandidatePenalty} with the constant there equal to
$c_0C_1$ shows that
$
[c_{0}C_1\eta_n,c_{0}C_1\eta_n/a]\cap\Lambda_{n}\neq\emptyset,
$
so we can fix a penalty $\lambda_{\ast}\in\Lambda_{n}$ satisfying
$
   c_0C_1\eta_n\leqslant\lambda_\ast\leqslant c_0C_1\eta_n/a.
$
Specify the non-random numbers
\[
\overline{\lambda}_n:=\frac{c_0C_1\eta_n}{a},\quad a_{\epsilon,n}:=(1+\overline c_0)C_2\sqrt{s_q\eta_n^{2-q}}\quad\text{and}\quad b_{\epsilon,n}:=(1+\overline c_0)C_2s_q\eta_n^{2-q}.
\]
The via \eqref{eq: definition un tilde key} implied  $\widetilde u_n$ then
reduces to $C_S(s_q\eta_n^{2-q})^{1/2}=:\check{u}_n$. Then $\check u_n\leqslant
c_M'$ by \eqref{eq:SideCondsHighProbSubsampleRiskBndLambdaStar}, the requirement
\eqref{eq: lemma c3 additional constraint} reduces to $C_S \geqslant 2$, which
holds true by construction. Define the events
\begin{align*}
\mathscr{Z}_{k} & :=\left\{ \Vert \bS_{I_{k}^{c}}\Vert_{\infty}\leqslant C_1\eta_n\right\} \quad\text{and}\quad\mathscr{E}_{k}:=\left\{ \epsilon_{I_{k}^{c}}\left(\check u_{n}\right)\leqslant a_{\epsilon,n} \check u_{n} +b_{\epsilon,n}\right\},\quad k\in[K].
\end{align*}
Then Lemma \ref{lem: theorem 1 extension} and
\eqref{eq:SideCondsHighProbSubsampleRiskBndLambdaStar} imply that, for any $k\in[K]$, on
$\mathscr{Z}_{k}\cap\mathscr{E}_{k}$, the penalty level $\lambda_\ast$ yields
\begin{equation}\label{eq:SubsampleRiskBndHoldingOutI_k}
\sup_{\widehat\btheta\in\widehat{\Theta}_{I_{k}^{c}}\left(\lambda_{\ast}\right)}\mathcal{E}(\widehat\btheta)\leqslant\left(1+\overline{c}_{0}\right) \overline{\lambda}_n\left(\check u_{n}\sqrt{s_q\eta_n^{-q}}+s_q\eta_n^{1-q}\right) + a_{\epsilon,n}\check u_n + b_{\epsilon,n}=C_{\mathcal E}s_q\eta_n^{2-q}.
\end{equation}
Further, Lemma \ref{lem:SubsampleEmpiricalErrorBnd} (with $u=\check u_n$) and
\eqref{eq:SideCondsHighProbSubsampleRiskBndLambdaStar} show that
\[
\P\big((\cap_{k=1}^{K}\mathscr{E}_{k})^{c}\big)\leqslant K\left(4t^{-r} + C/\ln^2(pn)+\left[\left(K-1\right)c_{D}n\right]^{-1}\right).
\]
Finally, Lemma \ref{lem:SubsampleScoreBnd} and
\eqref{eq: lemma e3 side condition} show that
\[
\P\big((\cap_{k=1}^{K}\mathscr{Z}_{k})^{c}\big)\leqslant CK/\ln^4(pn).
\]
It thus follows from the union bound and $pn\geqslant\mathrm e$ that
(\ref{eq:SubsampleRiskBndHoldingOutI_k}) holds simultaneously for all
$k\in[K] $ with probability at least $1-K(4t^{-r} +
2C/\ln^2(pn)+[(K-1)c_{D}n]^{-1})$.
\end{proof}

\begin{proof}[\sc{Proof of Lemma \ref{lem: convexity minoration implication}}]
Fix any $\btheta\in\Theta$ such that $\mathcal E(\btheta) \leqslant
c_M(c_M')^2$. If $\|\btheta - \btheta_0\|_2 \leqslant c_M'$, then $\|\btheta -
\btheta_0\|_2^2 \leqslant \mathcal E(\btheta) / c_M$ is immediate from
Assumption \ref{assu:Margin}. It thus suffices to prove that the case $\|\btheta
- \btheta_0\|_2 > c_M'$ is not possible. Seeking a contradiction, suppose that
$\|\btheta - \btheta_0\|_2 > c_M'$. Then we must have $\|\btheta -
\btheta_0\|_2>0$ and $c_M'\in(0,\infty)$. It follows that $u:=c_M' / \|\btheta -
\btheta_0\|_2 \in (0,1)$, and so defining $\widetilde\btheta:=\btheta_0 +
u(\btheta - \btheta_0)$, we have $\|\widetilde\btheta - \btheta_0\|_2 =
u\|\btheta - \btheta_0\|_2 = c_M'$. Using Assumptions \ref{assu:Convexity} and
\ref{assu:Margin}, we therefore see that
$$
c_M (c_M')^2 = c_M\|\widetilde\btheta - \btheta_0\|_2^2 \leqslant \mathcal E(\widetilde\btheta)\leqslant (1-u)\mathcal E(\btheta_0) + u\mathcal E(\btheta) = u\mathcal E(\btheta) \leqslant \frac{c_M(c_M')^3}{\|\btheta - \btheta_0\|_2}.
$$
This implies that $\|\btheta - \btheta_0\|_2\leqslant c_M'$, which contradicts
$\|\btheta - \btheta_0\|_2 > c_M'$. Thus, the case $\|\btheta - \btheta_0\|_2 >
c_M'$ is not possible, and the proof is complete.
\end{proof}

\begin{proof}
[\sc{Proof of Lemma \ref{lem:High-Probability-Subsample-Risk-Bound-Lambda-CV}}]
Fix $(n,t)\in\N\times[1,\infty)$ satisfying \eqref{eq: lemma e3 side condition},
\eqref{eq:SideCondsHighProbSubsampleRiskBndLambdaStar} and
\eqref{eq:AddedSideCondsHighProbSubsampleRiskBndLambdaCV} and define the
non-random sequences
$$
q_n:= C_L\sqrt{\frac{t\overline{\mathcal{E}}^{\ast}_{n}}{c_Dc_Mn}}+\overline{\mathcal{E}}^{\ast}_{n},\quad
\widetilde q_n := \frac{C_L^2}{c_M}\frac{3t\ln n}{c_D\ln(1/a)n} + C_L\sqrt{\frac{q_n}{c_M}}\sqrt{\frac{3t\ln n}{c_D\ln(1/a)n}},
$$
$$
\check u_{1,n}:= \frac{C_L}{c_M}\sqrt{\frac{3t\ln n}{c_D\ln(1/a)n}} + \sqrt{\frac{q_n}{c_M}},\quad\text{and}\quad \check u_{2,n}:= \frac{C_L}{c_M}\sqrt{\frac{3t\ln n}{c_D\ln(1/a)n}} + \sqrt{\frac{q_n + \widetilde q_n}{c_Dc_M}}.
$$
Note that since $c_D\in(0,1)$ by Assumption \ref{assu:DataPartition}, we have
$\check u_{1,n} < \check u_{2,n}$ and, as we will show at the end of this proof
via elementary inequalities, $\check u_{2,n}$ is smaller than the right-hand
side of \eqref{eq: lemma a6 main bound}. Therefore, $\check u_{1,n}\vee\check
u_{2,n}\leqslant c_M'\wedge c_L$ by
\eqref{eq:AddedSideCondsHighProbSubsampleRiskBndLambdaCV}. The latter inequality
will be used below to justify applications of Assumptions \ref{assu:Margin} and
\ref{assu:LossLocallyLipschitzAndMore}.\ref{enu:LossMeanSquareEll2Conts}.

For each $(k,\lambda)\in[K]\times\Lambda_n$, fix a solution
$\widehat\btheta_{I_k^c}\left(\lambda\right)\in\widehat\Theta_{I_k^c}\left(\lambda\right)$.
For $k\in\left[K\right], j\in\{1,2\}$ and $\lambda\in\Lambda_n$, denote
\begin{equation}\label{eq: interpolation theta tilde}
\widetilde\btheta_{I_k^c,j}(\lambda):=\begin{cases}
   \widehat\btheta_{I_k^c}(\lambda) & \text{if}\quad\|\widehat\btheta_{I_k^c}(\lambda) - \btheta_0\|_2 \leqslant \check u_{j,n},\\
   \btheta_0 + \frac{\check u_{j,n}}{\| \widehat\btheta_{I_k^c}(\lambda) - \btheta_0 \|_2}\big(\widehat\btheta_{I_k^c}(\lambda) - \btheta_0\big) & \text{if}\quad\|\widehat\btheta_{I_k^c}(\lambda) - \btheta_0\|_2 > \check u_{j,n}.
\end{cases}
\end{equation}
In addition, for $\btheta\in\Theta$ and $k\in[K]$, let
\begin{align*}
   f_k(\btheta)&:= (\mathbb E_{I_k} - \E)[m(\bX_i^{\top}\btheta,\bY_i) - m(\bX_i^{\top}\btheta_0,\bY_i)]
\end{align*}
let $\lambda_{\ast}\in\Lambda_{n}$ be a penalty level satisfying the bound
of Lemma \ref{lem:High-Probability-Subsample-Risk-Bound-Lambda-Star}, and define
events
\begin{align*}
\mathscr{R}_n&:=\left\{\max_{1\leqslant k\leqslant K}\mathcal{E}(\widehat{\btheta}_{I_{k}^{c}}(\lambda_{\ast}))\leqslant \overline{\mathcal{E}}^{\ast}_{n}\right\},\\   
\mathscr{F}_n&:=\left\{\max_{1\leqslant k\leqslant K}|f_k(\widehat{\btheta}_{I_{k}^{c}}(\lambda_{\ast}))| \leqslant C_L \sqrt{\frac{t\overline{\mathcal{E}}^{\ast}_{n}}{c_Dc_Mn}}\right\}\quad\text{and}\\
\mathscr{C}_n&:=\left\{\big| f_k(\widetilde{\btheta}_{I_{k}^{c},j}(\lambda))\big| \leqslant C_L\sqrt{\frac{3t\ln n}{c_D\ln(1/a)n}}\|\widetilde{\btheta}_{I_{k}^{c},j}(\lambda)   - \btheta_0\|_2,\text{ all }k\in[K],j\in\{1,2\},\lambda\in\Lambda_n\right\}.
\end{align*}
We first derive a lower bound for $\P(\mathscr{R}_n\cap \mathscr{F}_n\cap
\mathscr{C}_n)$ and then prove that \eqref{eq: lemma a6 main bound} is satisfied
on $\mathscr{R}_n\cap \mathscr{F}_n\cap \mathscr{C}_n$.

To derive a lower bound for $\P(\mathscr{R}_n\cap \mathscr{F}_n\cap
\mathscr{C}_n)$, first observe that for any $k \in [K]$, the
variance of the conditional distribution of 
$$
\mathbb{E}_{I_{k}}\left[m\big(\bX_{i}^{\top}\widehat{\btheta}_{I_{k}^{c}}\left(\lambda_{\ast}\right),\bY_{i}\big)-m\big(\bX_{i}^{\top}\btheta_0,\bY_{i}\big)\right]
$$
given $\{(\bX_i,\bY_i)\}_{i\in I_k^c}$ is bounded from above by
$$
\left|I_{k}\right|^{-1}\E_{\bX,\bY}\left[\big| m\big(\bX^{\top}\widehat{\btheta}_{I_{k}^{c}}\left(\lambda_{\ast}\right),\bY\big)-m\big(\bX^{\top}\btheta_0,\bY\big)\big|^{2}\right],
$$
and so, by Assumption \ref{assu:DataPartition} and Chebyshev's inequality
applied conditional on $\{(\bX_i,\bY_i)\}_{i\in I_k^c}$,
$$
\P\bigg( |f_k(\widehat\btheta_{I_k^c}(\lambda_{\ast}))| > \sqrt{\frac{t}{c_D n}}\sqrt{\E_{\bX,\bY}\left[\big| m\big(\bX^{\top}\widehat{\btheta}_{I_{k}^{c}}\left(\lambda_{\ast}\right),\bY\big)-m\big(\bX^{\top}\btheta_0,\bY\big)\big|^{2}\right]} \bigg)\leqslant \frac{1}{t}.
$$
Also, by \eqref{eq:AddedSideCondsHighProbSubsampleRiskBndLambdaCV} and Lemma
\ref{lem: convexity minoration implication}, on the event $\mathscr{R}_n$, we
have $\|\widehat\btheta_{I_k^c}(\lambda_{\ast}) - \btheta_0\|_2^2 \leqslant
\overline{\mathcal{E}}^{\ast}_{n}/c_M$, and so by Assumption
\ref{assu:LossLocallyLipschitzAndMore}.\ref{enu:LossMeanSquareEll2Conts},
$$
\E_{\bX,\bY}\left[\big| m\big(\bX^{\top}\widehat{\btheta}_{I_{k}^{c}}\left(\lambda_{\ast}\right),\bY\big)-m\big(\bX^{\top}\btheta_0,\bY\big)\big|^{2}\right] \leqslant \frac{C_L^2 \overline{\mathcal{E}}^{\ast}_{n}}{c_M}.
$$
Further, observe that for any $\lambda\in\Lambda_{n}$ and $j\in\{1,2\}$, the
variance of the conditional distribution of
$$
\mathbb{E}_{I_{k}}\left[m\big(\bX_{i}^{\top}\widetilde{\btheta}_{I_{k}^{c},j}\left(\lambda\right),\bY_{i}\big)-m\big(\bX_{i}^{\top}\btheta_0,\bY_{i}\big)\right]
$$
given $\{(\bX_i,\bY_i)\}_{i\in I_k^c}$ is bounded from above by
$$
\left|I_{k}\right|^{-1}\E_{\bX,\bY}\left[\big| m\big(\bX^{\top}\widetilde{\btheta}_{I_{k}^{c},j}\left(\lambda\right),\bY\big)-m\big(\bX^{\top}\btheta_0,\bY\big)\big|^{2}\right],
$$
and so, by Assumption \ref{assu:DataPartition} and Chebyshev's inequality
applied conditional on $\{(\bX_i,\bY_i)\}_{i\in I_k^c}$,
\begin{align*}
   &\P\bigg( |f_k(\widetilde\btheta_{I_k^c,j}(\lambda))| > \sqrt{\frac{3t\ln n}{c_D \ln(1/a)n}}\sqrt{\E_{\bX,\bY}\left[\big| m\big(\bX^{\top}\widetilde{\btheta}_{I_{k}^{c},j}\left(\lambda\right),\bY\big)-m\big(\bX^{\top}\btheta_0,\bY\big)\big|^{2}\right]} \bigg)\\
   &\leqslant \ln(1/a) / (3t\ln n).   
\end{align*}
In addition, given that we have
$\max_{j\in\{1,2\}}\|\widetilde{\btheta}_{I_{k}^{c},j}\left(\lambda\right) - \btheta_0\|_2 \leqslant
\check u_{1,n}\vee\check u_{2,n}\leqslant c_L$, it follows from Assumption
\ref{assu:LossLocallyLipschitzAndMore}.\ref{enu:LossMeanSquareEll2Conts} that
$$
\E_{\bX,\bY}\left[\big| m\big(\bX^{\top}\widetilde{\btheta}_{I_{k}^{c},j}\left(\lambda\right),\bY\big)-m\big(\bX^{\top}\btheta_0,\bY\big)\big|^{2}\right] \leqslant C_L^2\| \widetilde{\btheta}_{I_{k}^{c},j}(\lambda) - \btheta_0 \|_2^2,\quad j\in\{1,2\}.
$$
The qualifier $a^\ell\geqslant c_\Lambda/n$ in the definition of $\Lambda_n$ in
Assumption \ref{assu:CandidatePenalties} implies that
\[
   \ell\leqslant\frac{\ln(1/c_\Lambda)+\ln n}{\ln(1/a)}.
\]
Since also $\ell\in\{0,1,2,\dotsc\}$, from
\eqref{eq:AddedSideCondsHighProbSubsampleRiskBndLambdaCV} and the previous display, one can deduce that
$$
\left|\Lambda_{n}\right|\leqslant\frac{2\ln n}{\ln\left(1/a\right)} + 1 \leqslant \frac{3 \ln n}{\ln(1/a)}.
$$
Combining the presented results with Lemma
\ref{lem:High-Probability-Subsample-Risk-Bound-Lambda-Star} and the union bound,
we obtain
\[
\P\left(\mathscr{R}_n\cap\mathscr{F}_n\cap\mathscr{C}_n\right)\geqslant 1 - K\left(4t^{-r}+3t^{-1}+2C/\ln^2(pn)+\left[\left(K-1\right)c_{D}n\right]^{-1}\right),
\]
which is the desired bound.

We next prove that \eqref{eq: lemma a6 main bound} holds on
$\mathscr{R}_n\cap\mathscr{F}_n\cap\mathscr{C}_n$. For the rest of proof, we
therefore fix a realization of the data $\{(\bX_i,\bY_i)\}_{i=1}^n$ and assume that
$\mathscr{R}_n\cap\mathscr{F}_n\cap\mathscr{C}_n$ is satisfied. Given
\[
\widehat{\lambda}^{\mathtt{cv}}\in\argmin_{\lambda\in\Lambda_{n}}\sum_{k=1}^{K}\sum_{i\in I_{k}}m\big(\bX_{i}^{\top}\widehat{\btheta}_{I_{k}^{c}}\left(\lambda\right),\bY_{i}\big),
\]
a problem for which $\lambda_{\ast}$ is feasible, we must have
$$
\sum_{k=1}^{K}\left|I_{k}\right|\mathbb{E}_{I_{k}}\big[m\big(\bX_{i}^{\top}\widehat{\btheta}_{I_{k}^{c}}(\widehat{\lambda}^{\mathtt{cv}}),\bY_{i}\big)\big]
\leqslant \sum_{k=1}^{K}\left|I_{k}\right|\mathbb{E}_{I_{k}}\big[m\big(X_{i}^{\top}\widehat{\btheta}_{I_{k}^{c}}\left(\lambda_{\ast}\right),\bY_{i}\big)\big].
$$
Here, by $\mathscr{R}_n$ and $\mathscr{F}_n$, for each $k\in[K]$ we
have
\begin{align*}
\mathbb{E}_{I_{k}}\big[m\big(\bX_{i}^{\top}\widehat{\btheta}_{I_{k}^{c}}\left(\lambda_{\ast}\right),\bY_{i}\big)\big] 
& = \mathbb{E}_{I_{k}}\big[m\big(\bX_{i}^{\top}\btheta_0,\bY_{i}\big)\big] + f_k(\widehat{\btheta}_{I_{k}^{c}}\left(\lambda_{\ast}\right)) + \mathcal{E}(\widehat{\btheta}_{I_{k}^{c}}\left(\lambda_{\ast}\right)) \\ 
& \leqslant \mathbb{E}_{I_{k}}\big[m\big(\bX_{i}^{\top}\btheta_0,\bY_{i}\big)\big] + q_n.
\end{align*}
Therefore,
\begin{equation}\label{eq: cv implication simplified}
\sum_{k=1}^{K}\frac{\left|I_{k}\right|}{n}\mathbb{E}_{I_{k}}\big[m\big(\bX_{i}^{\top}\widehat{\btheta}_{I_{k}^{c}}(\widehat{\lambda}^{\mathtt{cv}}),\bY_{i}\big)\big]
\leqslant \sum_{k=1}^{K}\frac{\left|I_{k}\right|}{n}\mathbb{E}_{I_{k}}\big[m\big(\bX_{i}^{\top}\btheta_0,\bY_{i}\big)\big] + q_n.
\end{equation}
Now, define $\widehat{\mathcal K}$ as
\begin{equation}\label{eq: cv bound simple set}
\widehat{\mathcal K}:=\left\{k\in\left[K\right];\mathbb{E}_{I_{k}}\big[m\big(\bX_{i}^{\top}\widehat{\btheta}_{I_{k}^{c}}(\widehat{\lambda}^{\mathtt{cv}}),\bY_{i}\big)\big]
\leqslant \mathbb{E}_{I_{k}}\big[m\big(\bX_{i}^{\top}\btheta_0,\bY_{i}\big)\big] + q_n\right\}
\end{equation}
and $\widehat{\mathcal K}^c:=\left[K\right]\backslash\widehat{\mathcal K}$. We
will prove that
\begin{equation}\label{eq: split bounds for two sets of indices}
\max_{k\in\widehat{\mathcal K}}\|\widehat\btheta_{I_k^c}(\widehat\lambda^{\mathtt{cv}}) - \btheta_0 \|_2\leqslant \check u_{1,n}
\quad\text{and}\quad
\max_{k\in\widehat{\mathcal K}^c}\|\widehat\btheta_{I_k^c}(\widehat\lambda^{\mathtt{cv}}) - \btheta_0 \|_2\leqslant \check u_{2,n}
\end{equation}

To prove the \textit{first} inequality in \eqref{eq: split bounds for two sets of
indices}, seeking a contradiction, suppose that the inequality is not true, and
fix any $k\in\widehat{\mathcal K}$ such that
$\|\widehat\btheta_{I_k^c}(\widehat\lambda^{\mathtt{cv}}) - \btheta_0 \|_2 >
\check u_{1,n}$. Then $\check u_{1,n} =
\|\widetilde\btheta_{I_k^c,1}(\widehat\lambda^{\mathtt{cv}}) - \btheta_0 \|_2 <
\|\widehat\btheta_{I_k^c}(\widehat\lambda^{\mathtt{cv}}) - \btheta_0 \|_2$, and so
by \eqref{eq: interpolation theta tilde}, \eqref{eq: cv bound simple set}, and
convexity (Assumption \ref{assu:Convexity}),
$$
\mathbb{E}_{I_{k}}\big[m\big(\bX_{i}^{\top}\widetilde{\btheta}_{I_{k}^{c},1}(\widehat{\lambda}^{\mathtt{cv}}),\bY_{i}\big)\big]
< \mathbb{E}_{I_{k}}\big[m\big(\bX_{i}^{\top}\btheta_0,\bY_{i}\big)\big] + q_n.
$$
Therefore, by $\mathscr{C}_n$ and $k\in\widehat{\mathcal K}$,
\begin{align*}
\mathcal E(\widetilde{\btheta}_{I_{k}^{c},1}(\widehat{\lambda}^{\mathtt{cv}}))
& \leqslant |f_k(\widetilde{\btheta}_{I_{k}^{c},1}(\widehat{\lambda}^{\mathtt{cv}}))| + \mathbb{E}_{I_{k}}\big[m\big(\bX_{i}^{\top}\widetilde{\btheta}_{I_{k}^{c},1}(\widehat{\lambda}^{\mathtt{cv}}),\bY_{i}\big)\big]
- \mathbb{E}_{I_{k}}\big[m\big(\bX_{i}^{\top}\btheta_0,\bY_{i}\big)\big] \\
& < C_L\sqrt{\frac{3t\ln n}{c_D\ln(1/a)n}}\| \widetilde{\btheta}_{I_{k}^{c},1}(\widehat{\lambda}^{\mathtt{cv}}) - \btheta_0\|_2 + q_n,
\end{align*}
and so given that $\check u_{1,n}\leqslant c_M'$, by the margin condition
(Assumption \ref{assu:Margin}),
$$
\| \widetilde{\btheta}_{I_{k}^{c},1}(\widehat{\lambda}^{\mathtt{cv}}) - \btheta_0\|_2^2 < \frac{C_L}{c_M}\sqrt{\frac{3t\ln n}{c_D\ln(1/a)n}} \| \widetilde{\btheta}_{I_{k}^{c},1}(\widehat{\lambda}^{\mathtt{cv}}) - \btheta_0\|_2 + \frac{q_n}{c_M}.
$$
Using the elementary inequality that $x^2< Bx+C$ and $B,C\geqslant0$
imply $x<B+\sqrt{C}$, 
we see
$$
\check u_{1,n} = \| \widetilde{\btheta}_{I_{k}^{c},1}(\widehat{\lambda}^{\mathtt{cv}}) - \btheta_0\|_2 < \frac{C_L}{c_M}\sqrt{\frac{3t\ln n}{c_D\ln(1/a)n}} +\sqrt{\frac{q_n}{c_M}} = \check u_{1,n}.
$$
This contradiction proves the first inequality in \eqref{eq: split bounds for
two sets of indices}.

To prove the \textit{second} inequality in \eqref{eq: split bounds for two sets
of indices}, observe that from the first inequality in \eqref{eq: split bounds
for two sets of indices} and the definition of
$\widetilde\btheta_{I_k^c,1}(\widehat\lambda^{\mathtt{cv}})$ in \eqref{eq:
interpolation theta tilde}, we know that for all $k\in\widehat{\mathcal
K}$, $\widetilde\btheta_{I_k^c,1}(\widehat\lambda^{\mathtt{cv}}) =
\widehat\btheta_{I_k^c}(\widehat\lambda^{\mathtt{cv}})$. It therefore follows
from $\mathscr{C}_n$ that for all $k\in\widehat{\mathcal K}$,
\begin{align*}
&\mathbb{E}_{I_{k}}\big[m\big(\bX_{i}^{\top}\widehat{\btheta}_{I_{k}^{c}}(\widehat{\lambda}^{\mathtt{cv}}),\bY_{i}\big)\big] - \mathbb{E}_{I_{k}}\big[m\big(\bX_{i}^{\top}\btheta_0,\bY_{i}\big)\big] \\
& \qquad  = f_k(\widehat{\btheta}_{I_{k}^{c}}(\widehat{\lambda}^{\mathtt{cv}})) + \mathcal E(\widehat{\btheta}_{I_{k}^{c}}(\widehat{\lambda}^{\mathtt{cv}})) 
 \geqslant f_k(\widetilde{\btheta}_{I_{k}^{c},1}(\widehat{\lambda}^{\mathtt{cv}})) 
  \geqslant - C_L\sqrt{\frac{3t\ln n}{c_D\ln(1/a)n}}\check u_{1,n}.
\end{align*} 
Rearranging and using the definitions of $\check u_{1,n}$ and
 $\widetilde q_n$, it follows from \eqref{eq: cv implication simplified} that
$$
\sum_{k\in\widehat{\mathcal K}^c}\frac{|I_k|}{n}\mathbb{E}_{I_{k}}\big[m\big(\bX_{i}^{\top}\widehat{\btheta}_{I_{k}^{c}}(\widehat{\lambda}^{\mathtt{cv}}),\bY_{i}\big)\big] \leqslant \sum_{k\in\widehat{\mathcal K}^c}\frac{|I_k|}{n}\mathbb{E}_{I_{k}}\big[m\big(\bX_{i}^{\top}\btheta_0,\bY_{i}\big)\big] + q_n + \widetilde q_n.
$$
Hence, given that
$$
\min_{k\in\widehat{\mathcal K}^c}\left\{\mathbb{E}_{I_{k}}\big[m\big(\bX_{i}^{\top}\widehat{\btheta}_{I_{k}^{c}}(\widehat{\lambda}^{\mathtt{cv}}),\bY_{i}\big)\big] - \mathbb{E}_{I_{k}}\big[m\big(\bX_{i}^{\top}\btheta_0,\bY_{i}\big)\big]\right\} \geqslant 0
$$
by construction, it follows from Assumption \ref{assu:DataPartition} that for
all $k\in\widehat{\mathcal K}^c$,
$$
\mathbb{E}_{I_{k}}\big[m\big(\bX_{i}^{\top}\widehat{\btheta}_{I_{k}^{c}}(\widehat{\lambda}^{\mathtt{cv}}),\bY_{i}\big)\big] \leqslant \mathbb{E}_{I_{k}}\big[m\big(\bX_{i}^{\top}\btheta_0,\bY_{i}\big)\big] + \frac{q_n + \widetilde q_n}{c_D}.
$$
We now establish the second inequality in \eqref{eq: split bounds for two sets
of indices} using an argument parallel to that used to establish the first
inequality with $\widetilde\btheta_{I_k^c,2}( \widehat\lambda^{\mathtt{cv}})$,
$(q_n + \widetilde q_n)/c_D$, and $\check u_{2,n}$  playing the roles of
$\widetilde\btheta_{I_k^c,1}( \widehat\lambda^{\mathtt{cv}})$, $q_n$ and $\check
u_{1,n}$, respectively. This observation finishes the proof of 
the inequalities in \eqref{eq: split bounds for two sets of indices}.

To complete the proof, note that since $\check u_{1,n} < \check u_{2,n}$, the
inequalities in \eqref{eq: split bounds for two sets of indices} imply that
$\|\widehat\btheta_{I_k^c}(\widehat\lambda^{\mathtt{cv}}) - \btheta_0
\|_2\leqslant \check u_{2,n}$ \textit{for all} $k\in[K]$. It thus remains to
simplify the expression for $\check u_{2,n}$. To do so, we denote $T_n := (3t\ln
n / [c_D\ln(1/a)n])^{1/2}$ and use the following elementary inequalities, where
the very first inequality uses $n\geqslant 1/a$ in
\eqref{eq:AddedSideCondsHighProbSubsampleRiskBndLambdaCV}:
$$
q_n \leqslant \left( \sqrt{\overline{\mathcal{E}}^{\ast}_n} + \frac{C_LT_n}{2\sqrt{c_M}} \right)^{2}, \quad
\widetilde q_n \leqslant \left( \frac{\sqrt{q_n}}{2} + \frac{C_LT_n}{\sqrt{c_M}} \right)^2,
$$
$$
\sqrt{q_n + \widetilde q_n}\leqslant \sqrt{q_n} + \sqrt{\widetilde q_n} \leqslant \frac{3}{2}\sqrt{\overline{\mathcal{E}}^{\ast}_n}  + \frac{7C_LT_n}{4\sqrt{c_M}}\quad\text{and}
$$
$$
\check u_{2,n} \leqslant \frac{C_LT_n}{c_M} + \frac{7C_LT_n}{4c_M\sqrt{c_D}} + \frac{3}{2}\sqrt{\frac{\overline{\mathcal{E}}^{\ast}_n}{c_Dc_M}} \leqslant \frac{11C_LT_n}{4c_M\sqrt{c_D}} + \frac{3}{2}\sqrt{\frac{\overline{\mathcal{E}}^{\ast}_n}{c_Dc_M}}.
$$
Therefore, $\check u_{2,n}$ is smaller than the right-hand
side of \eqref{eq: lemma a6 main bound}, which completes the proof.
\end{proof}

\begin{proof}
[\sc{Proof of Lemma \ref{thm:High-Probability-L2n-Error-Bnd-CV-Residuals}}] Let
$(n,t)\in\N\times[1,\infty)$ satisfy \eqref{eq: lemma e3 side condition},
\eqref{eq:SideCondsHighProbSubsampleRiskBndLambdaStar} and
\eqref{eq:AddedSideCondsHighProbSubsampleRiskBndLambdaCV}, and denote
$\Lambda_{n,k}:= \{\lambda\in\Lambda_n; \| \widehat\btheta_{I_k^c}(\lambda) -
\btheta_0 \|_2 \leqslant c_L\}$ for each $k\in[K]$. Then for each $k\in[K]$ and
each $\lambda\in\Lambda_{n,k}$, by Assumption
\ref{assu:LossLocallyLipschitzAndMore}.\ref{enu:ResidualMeanSquareEll2Conts} and
Markov's inequality applied conditional on $\{(\bX_i,\bY_i)\} _{i\in I^c_{k}}$,
we have
\begin{align*}
 & \P\Big(\mathbb{E}_{I_{k}}\Big[\big| m_{1}'\big(\bX_{i}^{\top}\widehat{\btheta}_{I_{k}^{c}}(\lambda),\bY_{i}\big)-m_{1}'\big(\bX_{i}^{\top}\btheta_{0},\bY_{i}\big)\big|^{2}\Big]>C_L^2t\| \widehat\btheta_{I_k^c}(\lambda) - \btheta_0 \|_2\Big)\leqslant t^{-1}.
\end{align*}
Also, since $n\geqslant 1/\left(c_{\Lambda}\wedge a\right)$ by
\eqref{eq:AddedSideCondsHighProbSubsampleRiskBndLambdaCV}, Assumption
\ref{assu:CandidatePenalties} implies that
$\left|\Lambda_{n}\right|\leqslant3\left(\ln n\right)/\ln\left(1/a\right).$ (See
the proof of Lemma \ref{lem:High-Probability-Subsample-Risk-Bound-Lambda-CV} for
more details.) Therefore, by the union bound, the probability that there exists
$k\in[K]$ and $\lambda\in\Lambda_{n,k}$ such that
$$
\mathbb{E}_{I_{k}}\Big[\big| m_{1}'\big(\bX_{i}^{\top}\widehat{\btheta}_{I_{k}^{c}}(\lambda),\bY_{i}\big)-m_{1}'\big(\bX_{i}^{\top}\btheta_{0},\bY_{i}\big)\big|^{2}\Big]  >\frac{3C_{L}^{2}t\ln n}{\ln\left(1/a\right)}\| \widehat\btheta_{I_k^c}(\lambda) - \btheta_0 \|_2
$$
is bounded from above by $K/t$. In addition, by Lemma
\ref{lem:High-Probability-Subsample-Risk-Bound-Lambda-CV} and
\eqref{eq:AddedSideCondsHighProbSubsampleRiskBndLambdaCV},
$$
\max_{1\leqslant k\leqslant K}\|\widehat\btheta_{I_k^c}(\widehat\lambda^{\mathtt{cv}}) - \btheta_0 \|_2 \leqslant \frac{11C_L}{4c_Dc_M}\sqrt{\frac{3t\ln n}{\ln(1/a)n}} + \frac{3}{2}\sqrt{\frac{\overline{\mathcal{E}}_n^{\ast}}{c_Dc_M}} \leqslant c_L
$$
with probability at least $1 -
K\left(4t^{-r}+3t^{-1}+2C/\ln^2(pn)+\left[\left(K-1\right)c_{D}n\right]^{-1}\right)$.
Hence, with the same probability, all $\Lambda_{n,k}$ are non-empty.
It follows from the union bound that
\begin{align*}
\En\big[(\widehat{U}_{i}^{\mathtt{cv}}-U_{i})^{2}\big] & =\sum_{k=1}^{K}\frac{\left|I_{k}\right|}{n}\mathbb{E}_{I_{k}}\Big[\big| m_{1}'\big(\bX_{i}^{\top}\widehat{\btheta}_{I_{k}^{c}}\big(\widehat{\lambda}^{\mathtt{cv}}\big),\bY_{i}\big)-m_{1}'\big(\bX_{i}^{\top}\btheta_{0},\bY_{i}\big)\big|^{2}\Big]\\
 & \leqslant \frac{3C_{L}^{2}t\ln n}{\ln\left(1/a\right)}\left( \frac{11C_L}{4c_Dc_M}\sqrt{\frac{3t\ln n}{\ln(1/a)n}} + \frac{3}{2}\sqrt{\frac{\overline{\mathcal{E}}_n^{\ast}}{c_Dc_M}} \right)
\end{align*}
with probability at least $1 -
K\left(4t^{-r}+4t^{-1}+2C/\ln^2(pn)+\left[\left(K-1\right)c_{D}n\right]^{-1}\right)$,
as desired.
\end{proof}

\subsection{Proofs for Section \ref{sec: inference}}


Note that according to Assumption \ref{as: smoothness inference}, the second
derivative $m_{11}''(t,\by)$ may not exist for some $(t,\by)\in\R\times\mathcal
Y$. With some abuse of notation, for such $t$ and $\by$, throughout this
section, we set $m_{11}''(t,\by):= 0$, which is consistent with our convention
in Algorithm \ref{alg:ThreeStepDebiasing}. With this convention, the function
$m_{11}''$ is now defined on the entire set $\R\times\mathcal Y$, although it
may not have the interpretation as to the derivative of $m_1'$ with respect to
its first argument.

In proving Theorem \ref{thm: asymptotic distribution}, we rely on the
following eight lemmas, whose proofs can be found at the end of this section.

\begin{lem}[\textbf{Second Derivative of Loss}]\label{lem: loss second derivative}
Let Assumption \ref{as: smoothness inference} hold. Then
$|m_{11}''(t,\by)|\leqslant C_m$ for all $(t,\by)\in\R\times\mathcal Y$.
\end{lem}

\begin{lem}[\textbf{Interpolation}]\label{lem: second order interpolation}
Let Assumption \ref{as: smoothness inference} hold. Then for any
$(\widetilde\btheta_0,\widetilde\btheta_1)\in\Theta\times\Theta$ and
$(\bx,\by)\in\R^p\times\mathcal Y$, we have
$$
m_1'(\bx^{\top}\widetilde\btheta_1,\by) - m_1'(\bx^{\top}\widetilde\btheta_0,\by) = \int_0^1 m_{11}''(\bx^{\top}\widetilde\btheta_\tau,\by)\bx^{\top}(\widetilde\btheta_1 - \widetilde\btheta_0)d\tau,
$$
where we denote $\widetilde\btheta_{\tau} := \widetilde\btheta_0 +
\tau(\widetilde\btheta_1 - \widetilde\btheta_0)$ for all $\tau\in(0,1)$.
As a consequence,
$$
|m_1'(\bx^{\top}\widetilde\btheta_1,\by) - m_1'(\bx^{\top}\widetilde\btheta_0,\by)| 
\leqslant C_m|\bx^\top(\widetilde\btheta_1 - \widetilde\btheta_0)|.
$$
\end{lem}

\begin{lem}[\textbf{First- and Second-Order Conditions}]\label{lem:ExistenceOfSecondOrderDerivatives}
Let Assumptions
\ref{assu:ParameterSpace}--\ref{assu:LossLocallyLipschitzAndMore}, \ref{as:
integrability inference}, and \ref{as: smoothness inference} hold. Then
\begin{equation}\label{eq: first order condition inference lemma}
\E[m'_1(\bX^{\top}\btheta_0,\bY)\bX] = \mathbf 0_p,
\end{equation}
and
\begin{equation}\label{eq: second order condition inference lemma}
\E\left[m_{11}''\left(\bX^\top\btheta_0,\bY\right)|\bX^\top\left(\btheta-\btheta_0\right)|^2\right]\geqslant 2c_M\|\btheta-\btheta_0\|^2_2\quad\text{for all}\quad\btheta\in\R^p.
\end{equation}
\end{lem}

\begin{lem}[\textbf{Existence and Uniqueness of
}$\bmu_0$]\label{lem:ExistenceAndUniquenessOfMu0} Let Assumptions
\ref{assu:ParameterSpace}--\ref{assu:LossLocallyLipschitzAndMore}, \ref{as:
integrability inference}, and \ref{as: smoothness inference} hold. Then there is
a unique solution to \eqref{eq: definition of mu0}, namely
\begin{equation}\label{eq: mu definition proof lemma}
  \bmu_0=\left(\E\left[m_{11}''\left(\bX^\top\btheta_0,\bY\right)\bW\bW^\top\right]\right)^{-1}\E\left[m_{11}''\left(\bX^\top\btheta_0,\bY\right)\bW D\right].
\end{equation}
\end{lem}

\begin{lem}[\textbf{Variance Denominator
Bound}]\label{lem:VarianceDenominatorBound} Let Assumptions
\ref{assu:ParameterSpace}--\ref{assu:LossLocallyLipschitzAndMore}, \ref{as: integrability inference}, and \ref{as: smoothness inference} hold. Then
\begin{equation}\label{eq: denominator bound inference}
   \E\left[m''_{11}(\bX^{\top}\btheta_0,\bY)(D - \bW^{\top}\bmu_0)D\right] \geqslant 2c_M.
   \end{equation}
\end{lem}

\begin{lem}[\textbf{Remainder Term, I}]\label{lem: inference proof remainder 1}
Under the conditions of Theorem \ref{thm: asymptotic distribution},
$$
R_{n,1}:=\En\left[\left(m''_{11}(\bX_i^{\top}\widetilde\btheta, \bY_i) - m''_{11}(\bX_i^{\top}\btheta_0, \bY_i)\right)(D_i - \bW_i^{\top}\bmu_0)D_i\right] = o_{\P}(1).
$$
\end{lem}

\begin{lem}[\textbf{Remainder Term, II}]\label{lem: inference proof remainder 2}
For all $\tau\in[0,1]$, 
denote $\check\btheta_{\tau}:= (\beta_0 + \tau(\widetilde\beta - \beta_0), \widetilde\bgamma^{\top})^{\top}$. 
Under the conditions of Theorem \ref{thm: asymptotic distribution},
\begin{align*}
R_{n,2} & :=  \int_0^1\En\left[\left(m''_{11}(\bX_i^{\top}\check\btheta_{\tau}, \bY_i) - m''_{11}(\bX_i^{\top}\widetilde\btheta, \bY_i)\right)(D_i - \bW_i^{\top}\widetilde\bmu)D_i\right]d\tau \\
& \lesssim_\P a_n + \overline\Delta_n^{1/2} + (B_na_n /\overline\Delta_n)^{r/2}.
\end{align*}
\end{lem}

\begin{lem}[\textbf{Remainder Term, III}]\label{lem: inference proof remainder 3}
For all $\tau\in[0,1]$, 
denote $\mathring\btheta_{\tau} := (\beta_0,\bgamma_0^{\top} + \tau(\widetilde\bgamma - \bgamma_0)^{\top})^{\top}$. 
Under the conditions of Theorem \ref{thm: asymptotic distribution},
\begin{align*}
R_{n,3} & := \int_0^1\En\left[\left(m''_{11}(\bX_i^{\top}\mathring\btheta_{\tau}, \bY_i) - m''_{11}(\bX_i^{\top}\btheta_{0}, \bY_i)\right)(D_i - \bW_i^{\top}\bmu_0)(\bW_i^{\top}\widetilde\bgamma - \bW_i^{\top}\bgamma_0)\right]d\tau \\
& \lesssim_\P a_n^2 + B_na_n \left(\overline\Delta_n^{1/2} + (B_na_n /\overline\Delta_n)^{r/2}\right).
\end{align*}
\end{lem}

\begin{proof}[\sc{Proof of Theorem \ref{thm: asymptotic distribution}}]
Denote the (normalized) \emph{denominator} in the one-step update of Algorithm
\ref{alg:ThreeStepDebiasing} by
\[
   \mathcal D_n:= \En\left[m''_{11}(\bX_i^{\top}\widetilde\btheta, \bY_i)(D_i - \bW_i^{\top}\widetilde\bmu)D_i\right].
\]
We decompose $\mathcal D_n$ as
$$
\mathcal D_n = \E\left[ m''_{11}(\bX^{\top}\btheta_0,\bY)(D - \bW^{\top}\bmu_0)D \right] + I_{n,1} + I_{n,2} + I_{n,3} + I_{n,4},
$$
with remainder terms
\begin{align}
I_{n,1} &:= \left(\En-\E\right)\left[m''_{11}(\bX_i^{\top}\btheta_0, \bY_i)(D_i - \bW_i^{\top}\bmu_0)D_i\right],\nonumber\\
I_{n,2} &:= \En\left[m''_{11}(\bX_i^{\top}\btheta_0,\bY_i)\bW_i^{\top}(\bmu_0 - \widetilde\bmu)D_i\right], \nonumber\\
I_{n,3} &:= \En\left[\left(m''_{11}(\bX_i^{\top}\widetilde\btheta, \bY_i) - m''_{11}(\bX_i^{\top}\btheta_0, \bY_i)\right)(D_i - \bW_i^{\top}\bmu_0)D_i\right],\nonumber\\
I_{n,4} &:= \En\left[\left(m''_{11}(\bX_i^{\top}\widetilde\btheta, \bY_i) - m''_{11}(\bX_i^{\top}\btheta_0, \bY_i)\right)\bW_i^{\top}(\bmu_0 - \widetilde\bmu)D_i\right].\nonumber
\end{align}
We handle each remainder term in turn. First, by Assumption \ref{as:
integrability inference}, Lemma \ref{lem: loss second derivative}, and the
Chebyshev and Cauchy-Schwarz inequalities, we get $I_{n,1} = o_{\P}(1)$. Second,
we have
\begin{align}
   |I_{n,2}| 
   &\leqslant C_m \En\left[|\bW_i^{\top}(\bmu_0 - \widetilde\bmu)D_i|\right]\nonumber\\
   &\leqslant C_m \|\widetilde\bmu-\bmu_0\|_1 \sqrt{\En\left[\|\bW_i\|_\infty^2\right]\En\left[D_i^2\right]} \lesssim_{\P} a_nB_n,\label{eq: i am still here}
\end{align}
where the first line follows from the triangle inequality and Lemma \ref{lem:
loss second derivative}, and the second from the H{\"o}lder and Cauchy-Schwarz
inequalities and Assumptions
\ref{assu:LossLocallyLipschitzAndMore}.\ref{enu:LossLocallyLipschitz}, \ref{as:
integrability inference} and \ref{as: convergence rates inference}. Given that
$a_nB_n\to0$ by assumption, it thus follows that $I_{n,2} = o_{\P}(1)$. Third,
we have $I_{n,3} = o_{\P}(1)$ by Lemma \ref{lem: inference proof remainder 1}.
Fourth, we have
\begin{align*}
   |I_{n,4}| 
   &\leqslant 2C_m\En\left[|\bW_i^{\top}(\widetilde\bmu-\bmu_0)D_i|\right] \lesssim_{\P} a_n B_n,
\end{align*}
where the inequality follows from the triangle inequality and Lemma \ref{lem:
loss second derivative}, and the $\lesssim_\P$ follows as in \eqref{eq: i am
still here}. Given that $a_nB_n\to0$ by assumption, it thus follows that
$I_{n,4} = o_{\P}(1)$. Conclude that
\begin{equation}\label{eq: inference proof linearization denominator 1x}
\mathcal D_n = \E\left[ m''_{11}(\bX^{\top}\btheta_0, \bY)(D - \bW^{\top}\bmu_0)D \right] + o_{\P}(1).
\end{equation}
In addition, Lemma \ref{lem:VarianceDenominatorBound} shows that 
\begin{equation}\label{eq: denominator limit lower bound}
\E[ m''_{11}(\bX^{\top}\btheta_0, \bY)(D - \bW^{\top}\bmu_0)D]\geqslant 2c_M,
\end{equation}
and so $\mathcal D_n$ is bounded away from zero with probability tending to one,
which implies in particular that the one-step update in Algorithm
\ref{alg:ThreeStepDebiasing} is well-defined with probability tending to one as
well and that $1/\mathcal D_n\lesssim_{\P}1$. 

Next, consider the (normalized) \emph{numerator} in the one-step update of
Algorithm \ref{alg:ThreeStepDebiasing}. For all $\tau\in[0,1]$, denote
$\check\btheta_{\tau}:= (\beta_0 + \tau(\widetilde\beta - \beta_0),
\widetilde\bgamma^{\top})^{\top}$ and
\begin{align*}
&\mathcal N_n:= \int_0^1\En\left[m''_{11}(\bX_i^{\top}\check\btheta_{\tau}, \bY_i)(D_i - \bW_i^{\top}\widetilde\bmu)D_i\right]d\tau.
\end{align*}
Then by Lemma \ref{lem: second order interpolation}, we have
\begin{align}
&\En\left[m'_1(D_i\widetilde\beta + \bW_i^{\top}\widetilde\bgamma, \bY_i)(D_i - \bW_i^{\top}\widetilde\bmu)\right] \nonumber\\
&\qquad\qquad  =\En\left[m'_1(D_i\beta_0 + \bW_i^{\top}\widetilde\bgamma, \bY_i)(D_i - \bW_i^{\top}\widetilde\bmu)\right] + \mathcal N_n(\widetilde\beta - \beta_0)\label{eq: proof inference who knows what}.
\end{align}
Hence, by definition of the one-step update in \eqref{eq: debiased estimator
betahat},
\begin{equation}\label{eq: one step error decomposition}
\widehat\beta - \beta_0 = (\widetilde\beta - \beta_0)\left(1 - \frac{\mathcal N_n}{\mathcal D_n}\right) - \frac{1}{\mathcal D_n}\En\left[m'_1(D_i\beta_0 + \bW_i^{\top}\widetilde\bgamma, \bY_i)(D_i - \bW_i^{\top}\widetilde\bmu)\right].
\end{equation}
By Lemma \ref{lem: inference proof remainder 2}, we know
\begin{align*}
   |\mathcal N_n-\mathcal D_n|
    &\lesssim_{\P} a_n + \overline\Delta_n^{1/2} + (B_na_n /\overline\Delta_n)^{r/2}.
\end{align*}
Recalling $1/\mathcal D_n\lesssim_{\P} 1$, it follows that
\[
   \left|\sqrt{n}(\widetilde\beta - \beta_0)\left(1 - \frac{\mathcal N_n}{\mathcal D_n}\right)\right|\lesssim_\P  \sqrt{n}a_n\left(a_n + \overline\Delta_n^{1/2} + (B_na_n /\overline\Delta_n)^{r/2}\right),
\]
which $\to0$ by Assumption \ref{as: convergence rates inference} and by
hypothesis of the theorem. It follows from \eqref{eq: one step error
decomposition} that
\begin{equation}\label{eq: inference proof linearization 1x}
  \sqrt{n}(\widehat\beta - \beta_0) = - \frac{1}{\mathcal D_n}\sqrt{n}\En\left[m'_1(D_i\beta_0 + \bW_i^{\top}\widetilde\bgamma, \bY_i)(D_i - \bW_i^{\top}\widetilde\bmu)\right] + o_{\P}(1). 
\end{equation}
We next further analyze the right-hand side numerator. To this end, denote
$$
\widetilde I_{n,1} : = \sqrt{n}\En\left[m'_1(\bX_i^{\top}\btheta_0, \bY_i)(\bW_i^{\top}\widetilde\bmu - \bW_i^{\top}\bmu_0)\right]
$$
and observe that
$$
|\widetilde I_{n,1}| \leqslant \|\widetilde\bmu - \bmu_0\|_1 \| \sqrt{n}\En[m'_1(\bX_i^{\top}\btheta_0, \bY_i)\bW_i ] \|_{\infty}.
$$
Setting up for an application of Theorem \ref{thm: max inequality standard}, 
observe that for all $j\in[p-1]$, we have
$\E[|m_{1}'(\bX^{\top}\btheta_0,\bY)W_j|^2]\leqslant C_M^4$ by the
Cauchy-Schwarz inequality and Assumption \ref{as: integrability inference}. In
addition, for $q:=(r\wedge\widetilde r)/2\in(2,\infty)$,
\begin{align*}
& \E[\| m_{1}'(\bX^{\top}\btheta_0,\bY)\bW \|_{\infty}^q]  \leqslant \Big( \E[| m'_1(\bX^{\top}\btheta_0,\bY) |^{2q}] \Big)^{1/2}\Big( \E[\|\bW\|_{\infty}^{2q}] \Big)^{1/2} \leqslant C_M^q B_n^q
\end{align*}
by the Cauchy-Schwarz inequality and Assumptions
\ref{assu:LossLocallyLipschitzAndMore}.\ref{enu:LossLocallyLipschitz} and
\ref{as: integrability inference}. Therefore, given that $B_n^2\ln(pn) = o(n^{1
- 4/(r\wedge \widetilde r)})$ by hypothesis of the theorem, (the mean-zero part
of) Lemma \ref{lem:ExistenceOfSecondOrderDerivatives} and Theorem \ref{thm: max
inequality standard} combine to show that
$$
\left\| \sqrt n\En\left[ m'_{1}(\bX_i^{\top}\btheta_0,\bY_i)\bW_i\right] \right\|_{\infty} \lesssim_{\P} \sqrt{\ln(pn)}
$$
Hence, given that $\|\widetilde\bmu - \bmu_0\|_1 \lesssim_{\P} a_n$ by
Assumption \ref{as: convergence rates inference} and $a_n\sqrt{\ln(pn)}\to0$ by
assumption, it follows that  $\widetilde I_{n,1} = o_{\P}(1)$. Moreover, for all
$\tau\in[0,1]$, denote $\mathring\btheta_{\tau} := (\beta_0,\bgamma_0^{\top} +
\tau(\widetilde\bgamma - \bgamma_0)^{\top})^{\top}$ and
$$
\widetilde I_{n,2}:= \sqrt{n}\En\left[m'_1(D_i\beta_0 + \bW_i^{\top}\widetilde\bgamma, \bY_i)(\bW_i^{\top}\widetilde\bmu - \bW_i^{\top}\bmu_0)\right] - \widetilde I_{n,1}.
$$
Then
\begin{align}
|\widetilde I_{n,2}|
&\leqslant \sqrt n \int_0^1 \left| \En\left[m_{11}''( \bX_i^{\top}\mathring\btheta_{\tau},\bY_i)(\bW_i^{\top}\widetilde\bgamma - \bW_i^{\top}\bgamma_0)(\bW_i^{\top}\widetilde\bmu - \bW_i^{\top}\bmu_0)\right]\right|d\tau \nonumber\\ 
&= \sqrt n \int_0^1 \left| (\widetilde\bgamma - \bgamma_0)^\top\En\left[m_{11}''( \bX_i^{\top}\mathring\btheta_{\tau},\bY_i)\bW_i\bW_i^{\top}\right](\widetilde\bmu-\bmu_0)\right|d\tau \nonumber\\ 
&\leqslant C_m\sqrt n\|\widetilde\bgamma - \bgamma_0\|_1\|\widetilde\bmu - \bmu_0\|_1\max_{1\leqslant j,k \leqslant p-1}\En[|W_{i,j}W_{i,k}|], \label{eq: inference proof who knows what 3}
\end{align}
where the first line follows from Lemma \ref{lem: second order interpolation}
and Jensen's inequality, and the third from Lemma \ref{lem: loss second
derivative} and the H{\"o}lder and triangle inequalities. Setting up for an
application of Theorem \ref{thm: max inequality nonnegative}, observe that for
all $j,k\in[p-1]$, we have $\E[ |W_jW_k|]\leqslant C_M^2$  by the Cauchy-Schwarz
inequality and Assumption \ref{as: integrability inference}. In addition, for
$q:=r/2\in(1,\infty)$, we have $\E[\|\bW\|_{\infty}^{q}\|\bW\|_{\infty}^{q}] \leqslant
B_n^{2q}$ by Assumptions
\ref{assu:LossLocallyLipschitzAndMore}.\ref{enu:LossLocallyLipschitz}. Hence,
given that $B_n^2\ln(p n) = o(n^{1-2/r})$ by hypothesis of the theorem,
it follows from Theorem \ref{thm: max inequality nonnegative} that
$$
\max_{1\leqslant j,k\leqslant p-1}\En[ |W_{i,j}W_{i,k}|] \lesssim_{\P} 1.
$$
Therefore, given that $\|\widetilde\bgamma - \bgamma_0\|_{1}\|\widetilde\bmu -
\bmu_0\|_1 \lesssim_{\P} a_n^2$ by Assumption \ref{as: convergence rates
inference}, it follows that $|\widetilde I_{n,2}| \lesssim_{\P} \sqrt n a_n^2$,
which $\to0$ by assumption. We therefore have $\widetilde I_{n,1} + \widetilde
I_{n,2} = o_{\P}(1)$, and it follows from \eqref{eq: inference proof
linearization 1x} and $1/\mathcal D_n \lesssim_{\P} 1$ that
\begin{equation}\label{eq: inference proof linearization 2x}
  \sqrt{n}(\widehat\beta - \beta_0) = - \frac{1}{\mathcal D_n}\sqrt{n}\En\left[m'_1(D_i\beta_0 + \bW_i^{\top}\widetilde\bgamma, \bY_i)(D_i - \bW_i^{\top}\bmu_0)\right] + o_\P(1). 
\end{equation}
Further, by Lemma \ref{lem: second order interpolation}, we have
\begin{align}
& \sqrt n\En[(m'_1(D_i\beta_0 + \bW_i^{\top}\widetilde\bgamma, \bY_i) - m'_1(\bX_i^{\top}\btheta_0, \bY_i))(D_i - \bW_i^{\top}\bmu_0)] \nonumber\\
& \quad = \sqrt n \int_0^1\En[m''_{11}(\bX_i^{\top}\mathring\btheta_{\tau}, \bY_i)(D_i - \bW_i^{\top}\bmu_0)(\bW_i^{\top}\widetilde\bgamma - \bW_i^{\top}\bgamma_0)]d\tau =: \check I_{n,1} \label{eq: proof inference who knows what 2}.
\end{align}
In addition, denote
$$
\check I_{n,2} := \sqrt n\En[m''_{11}(\bX_i^{\top}\btheta_0,\bY_i)(D_i - \bW_i^{\top}\bmu_0)(\bW_i^{\top}\widetilde\bgamma - \bW_i^{\top}\bgamma_0)].
$$
Then, by Lemma \ref{lem: inference proof remainder 3} and hypothesis of the
theorem, we have
$$|\check I_{n,1} - \check I_{n,2}| \lesssim_{\P} \sqrt n \left(a_n^2 + B_na_n \left(\overline\Delta_n^{1/2} + (B_na_n /\overline\Delta_n)^{r/2}\right) \right) \to 0.
$$
Moreover,
$$
|\check I_{n,2}| \leqslant \|\widetilde\bgamma-\bgamma_0\|_1 \|\sqrt n\En[m''_{11}(\bX_i^{\top}\btheta_0,\bY_i)(D_i - \bW_i^{\top}\bmu_0)\bW_i]\|_{\infty}.
$$
Setting up for an application of Theorem \ref{thm: max inequality standard},
observe that for all $j\in[p]$,
$$
\E[|m_{11}''(\bX^{\top}\btheta_0,\bY)(D-\bW^{\top}\bmu_0)W_j|^2]\leqslant C_m^2C_M^4
$$
where we have used Lemma \ref{lem: loss second derivative}, the Cauchy-Schwarz inequality, and Assumption \ref{as: integrability inference}. In addition, for $q:=(r\wedge\widetilde r)/2\in(2,\infty)$,
\begin{align*}
& \E[\| m_{11}''(\bX^{\top}\btheta_0,\bY)(D-\bW^{\top}\bmu_0)\bW \|_{\infty}^q] \\
&\qquad \leqslant C_m^q\Big( \E[| D - \bW^{\top}\bmu_0 |^{2q}] \Big)^{1/2}\Big( \E[\|\bW\|_{\infty}^{2q}] \Big)^{1/2} \leqslant C_m^qC_M^q B_n^q
\end{align*}
by Lemma \ref{lem: loss second derivative}, the Cauchy-Schwarz inequality, and Assumptions \ref{assu:LossLocallyLipschitzAndMore}.\ref{enu:LossLocallyLipschitz} and \ref{as: integrability inference}. Therefore, given that $B_n^2\ln(pn) = o(n^{1 - 4/(r\wedge \widetilde r)})$ by assumption,
$$
\left\| \sqrt n\En\left[ m''_{11}(\bX_i^{\top}\btheta_0,\bY_i)(D_i - \bW_i^{\top}\bmu_0)\bW_i\right] \right\|_{\infty} \lesssim_{\P} \sqrt{\ln(pn)}
$$
by \eqref{eq: definition of mu0} and Theorem \ref{thm: max inequality standard}. 
Hence, given that $\|\widetilde\bgamma - \bgamma_0\|_1\lesssim_{\P}a_n$ by 
Assumption \ref{as: convergence rates inference}, it follows that 
$|\check I_{n,2}| \lesssim_{\P} a_n\sqrt{\ln(pn)}$, which $\to0$ by hypothesis, 
and so $\check I_{n,1} = o_{\P}(1)$. Thus, it follows from 
\eqref{eq: inference proof linearization 2x} and $1/\mathcal D_n \lesssim_{\P} 1$ 
that
\begin{equation}\label{eq: inference proof linearization 3x}
  \sqrt{n}(\widehat\beta - \beta_0) = - \frac{1}{\mathcal D_n}\sqrt{n}\En\left[m'_1(\bX_i^{\top}\btheta_0, \bY_i)(D_i - \bW_i^{\top}\bmu_0)\right] + o_\P(1). 
\end{equation}
Combining this bound with \eqref{eq: inference proof linearization denominator 1x} and \eqref{eq: denominator limit lower bound} in turn yields
$$
\sqrt n(\widehat\beta - \beta_0) = -\frac{n^{-1/2}\sum_{i=1}^n m'_1(\bX_i^{\top}\btheta_0, \bY_i)(D_i - \bW_i^{\top}\bmu_0)}{\E\Big[ m''_{11}(\bX_i^{\top}\btheta_0, \bY)(D - \bW^{\top}\bmu_0)D \Big]} + o_\P(1).
$$
Using Assumptions \ref{as: identifiability inference} and \ref{as: integrability inference}
and Lemmas \ref{lem: loss second derivative} and \ref{lem:VarianceDenominatorBound},
we see that the asymptotic variance $\sigma_0$ is bounded from above and away
from zero, and from the Cauchy-Schwarz inequality, we have
\[
\E\left[\left|m'_1(\bX^{\top}\btheta_0, \bY)(D - \bW^{\top}\bmu_0)\right|^{2+(\widetilde r - 4)/2}\right]\leqslant C_M^{\widetilde r}\in(0,\infty).
\]
It therefore follows that
$$
\frac{\sqrt n(\widehat\beta - \beta_0)}{\sigma_0} = -\frac{{n}^{-1/2}\sum_{i=1}^n m'_1(\bX_i^{\top}\btheta_0, \bY_i)(D_i - \bW_i^{\top}\bmu_0)}{(\E[(m'_1(\bX_i^{\top}\btheta_0, \bY)(D - \bW^{\top}\bmu_0))^2])^{1/2}} + o_\P(1),
$$
and the asserted claim of the theorem now follows from Lyapunov's version of 
the Central Limit Theorem in combination with Slutsky's lemma.
\end{proof}

\begin{proof}[\sc{Proof of Lemma \ref{lem: loss second derivative}}]
Observe that for all $t_1,t_2\in\R$ and $\by\in\mathcal Y$, we have 
$|m_1'(t_2,\by) - m_1'(t_1,\by)|\leqslant C_m|t_2-t_1|$ by 
Assumption \ref{as: smoothness inference}. Hence, for all $t\in\R$ and 
$\by\in\mathcal Y$ such that the derivative of $m_{1}'(t,\by)$ with respect to 
the first argument exists, it is equal to $m_{11}''(t,\by)$ by definition and 
satisfies $|m_{11}''(t,\by)|\leqslant C_m$. Also, for all $t\in\R$ and 
$\by\in\mathcal Y$ such that the derivative of $m_{1}'(t,\by)$ with respect to 
the first argument does not exist, we have $m_{11}''(t,\by)=0$ by convention, 
and so $|m_{11}''(t,\by)|\leqslant C_m$ as well. This gives the asserted claim.
\end{proof}

\begin{proof}[{\sc Proof of Lemma \ref{lem: second order interpolation}}]
Fix $(\widetilde\btheta_0,\widetilde\btheta_1)\in\Theta\times\Theta$ and 
$(\bx,\by)\in\R^p\times \mathcal Y$. 
We provide the proof for the case where
$\bx^{\top}\widetilde\btheta_0 \leqslant \bx^{\top}\widetilde\btheta_1$. The
other case is analogous.
Let $\mathcal J := \{j\in[J]; \ \bx^{\top}\widetilde\btheta_0 \leqslant t_{\by,j}\leqslant \bx^{\top}\widetilde\btheta_1\}$. 
If $\mathcal J$ is empty (i.e.~no threshold was encountered or crossed), then
\begin{equation}\label{eq: second derivative why we are doing it}
m_1'(\bx^{\top}\widetilde\btheta_1,\by) - m_1'(\bx^{\top}\widetilde\btheta_0,\by) = \int_{\bx^{\top}\widetilde\btheta_0}^{\bx^{\top}\widetilde\btheta_1} m_{11}''(t,\by)dt = \int_0^1 m_{11}''(\bx^{\top}\widetilde\btheta_{\tau},\by)\bx^{\top}(\widetilde\btheta_1 - \widetilde\btheta_0)d\tau,
\end{equation}
where the first equality follows from the fundamental theorem of calculus and 
Assumption \ref{as: smoothness inference} and the second from the change of 
variables $t = \bx^{\top}(\widetilde\btheta_0 + \tau(\widetilde\btheta_1 - \widetilde\btheta_0))$.

If the set $\mathcal J$ is non-empty, we let $j_0$ and $j_1$ be its smallest and 
largest element, respectively. (These elements could coincide.) Then
\begin{align}
m_1'(\bx^{\top}\widetilde\btheta_1,\by) - m_1'(\bx^{\top}\widetilde\btheta_0,\by)
& =  m_1'(\bx^{\top}\widetilde\btheta_1,\by) - m_1'(t_{\by,j_1},\by)  \nonumber\\
&\quad + \sum_{j=j_0+1}^{j_1} m_1'(t_{\by,j},\by) - m_1'(t_{\by,j-1},\by)  \nonumber\\
&\quad + m_1'(t_{\by,j_0},\by) - m_1'(\bx^{\top}\widetilde\btheta_0,\by), \label{eq: eq: second derivative why we are doing it again}
\end{align}
where the sum is omitted if $j_0 = j_1$. Applying the same argument as that 
leading to \eqref{eq: second derivative why we are doing it} to each of the 
terms on the right-hand side of \eqref{eq: eq: second derivative why we are doing it again} shows that
\begin{align*}
m_1'(\bx^{\top}\widetilde\btheta_1,\by) - m_1'(\bx^{\top}\widetilde\btheta_0,\by)
& =  \int_{\tau_{j_1}}^1 m_{11}''(\bx^{\top}\widetilde\btheta_{\tau},\by)\bx^{\top}(\widetilde\btheta_1 - \widetilde\btheta_0)d\tau \\
& \quad + \sum_{j=j_0 + 1}^{j_1} \int_{\tau_{j-1}}^{\tau_j} m_{11}''(\bx^{\top}\widetilde\btheta_{\tau},\by)\bx^{\top}(\widetilde\btheta_1 - \widetilde\btheta_0)d\tau \\
& \quad + \int_{0}^{\tau_{j_0}} m_{11}''(\bx^{\top}\widetilde\btheta_{\tau},\by)\bx^{\top}(\widetilde\btheta_1 - \widetilde\btheta_0)d\tau \\
& = \int_{0}^1 m_{11}''(\bx^{\top}\widetilde\btheta_{\tau},\by)\bx^{\top}(\widetilde\btheta_1 - \widetilde\btheta_0)d\tau,
\end{align*}
where we denoted $\tau_j := (t_{\by,j} - \bx^{\top}\widetilde\btheta_0)/(\bx^{\top}(\widetilde\btheta_1 - \widetilde\btheta_0))$
for all $j\in\{j_0,\dots,j_1\}$. The previous display yields the first claim.
The second claim follows from the first and Lemma \ref{lem: loss second derivative}.
\end{proof}

\begin{proof}[{\sc Proof of Lemma \ref{lem:ExistenceOfSecondOrderDerivatives}}]
Since $\btheta_{0}$ is interior to $\Theta$ (Assumption
\ref{assu:ParameterSpace}), there is a radius $\overline{r}_{n}\in(0,\infty)$
such that $\overline r_n \leqslant c_M'$ and the ball $\mathcal B_{\btheta_0}(\overline r_n) := \{
\btheta\in\R^{p};\|\btheta-\btheta_{0}\|_{2}\leqslant\overline{r}_{n}\}$ is a 
subset of $\Theta$, with $c_{M}'\in(0,\infty]$ being provided in Assumption \ref{assu:Margin}. 
Fix any $\btheta\in \mathcal{B}_{\btheta_0}(\overline{r})$ and define
$$
f\left(\tau,\bz\right)
:=m\left(\bx^{\top}\btheta_\tau,\by\right)-m\left(\bx^{\top}\btheta_{0},\by\right)\;\text{for each}\;(\tau,\bz)\in\R\times\mathcal{Z}.
$$
As in Section \ref{sec: proofs for bcv}, we here employ the 
shorthand notations 
$\btheta_\tau=\btheta_0+\tau(\btheta-\btheta_0)$, $\bz=(\bx,\by)$, 
$\bZ=(\bX,\bY)$ and $\mathcal Z=\mathcal X\times\mathcal Y$.

Now, for any $\tau\in(-1,1)$, we have 
$\btheta_\tau\in \mathcal B_{\btheta_0}(\overline r_n) \subset \Theta$ 
by Assumption \ref{assu:ParameterSpace}, and so $\E[|f(\tau,\bZ)|]<\infty$ 
by Assumption \ref{as: diff and int}. 
Hence, $g(\tau) := \E[f(\tau,\bZ)]$, $\tau\in(-1,1)$, is a well-defined map 
from $(-1,1)$ to $\R$. Further, for any $\tau\in(-1,1)$ and 
$\bz\in\mathcal Z$, there is an $\alpha\in[0,1]$ such that
\begin{align*}
|f(\tau,\bz)|
& = |\tau|\cdot|m_1'(\bx^{\top}\btheta_0 + \alpha\tau\bx^{\top}(\btheta - \btheta_0),\by)|\cdot|\bx^{\top}(\btheta - \btheta_0)| \\
& \leqslant |\tau|\cdot\Big( |m_1'(\bx^{\top}\btheta_0 + \bx^{\top}(\btheta - \btheta_0),\by)| + |m_1'(\bx^{\top}\btheta_0 - \bx^{\top}(\btheta - \btheta_0),\by)| \Big) \cdot|\bx^{\top}(\btheta - \btheta_0)| \\
& \leqslant 2|\tau|\cdot\Big(|m_1'(\bx^{\top}\btheta_0,\by)| + C_m |\bx^{\top}(\btheta - \btheta_0)|\Big)\cdot|\bx^{\top}(\btheta - \btheta_0)|,
\end{align*}
where the first line follows from the Mean Value Theorem and Assumption 
\ref{as: smoothness inference}, the second from $m(\cdot,y)$ being convex and
differentiable (Assumptions \ref{assu:Convexity} and \ref{as: smoothness inference})
and the fact that the inequality $a\leqslant b\leqslant c$ implies 
$|b|\leqslant |a|+|c|$, and the third follows from the triangle inequality
(adding and subtracting $m_1'(\bx^\top\btheta_0,\by)$ twice) and
two applications of the Lemma \ref{lem: second order interpolation} inequality.
Hence,
$
\E[ \sup_{\tau\in(-1,1)}|f(\tau,\bZ)|] <\infty
$ 
by the Cauchy-Schwarz inequality and Assumption \ref{as: integrability inference}.
Also, for any $\tau\in(-1,1)$ and $\bz=(\bx,\by)\in\mathcal Z$, $f_1'(\tau,\bz)$ 
exists by Assumption \ref{as: smoothness inference} 
and
\begin{align*}
|f_1'(\tau,\bz)| & = |m_1'(\bx^{\top}\btheta_0 + \tau\bx^{\top}(\btheta - \btheta_0),\by)|\cdot|\bx^{\top}(\btheta - \btheta_0)|\\
&  \leqslant 2\Big(|m_1'(\bx^{\top}\btheta_0,\by)| + C_m |\bx^{\top}(\btheta - \btheta_0)|\Big)\cdot|\bx^{\top}(\btheta - \btheta_0)|
\end{align*}
by the same arguments as above,
so $\E[\sup_{\tau\in(-1,1)}|f'_1(\tau,\bZ)|]<\infty$ as well.
It then follows from \citet[Corollary A.3]{dudley2014uniform}
that $g$ is differentiable on $(-1,1)$ with derivative given by
$g'(\tau)=\E[f_1'(\tau,\bZ)]$. 
In particular, $g'(0) = \E[m'_1(\bX^{\top}\btheta_0,\bY)\bX^{\top}(\btheta - \btheta_0)]$. 

Further, $f''_{11}(0,\bZ)$ exists almost surely by Assumption \ref{as: smoothness inference}. 
Also, for any $\tau \in(-1,1)$ and $\bz = (\bx,\by) \in\mathcal Z$,
\begin{align}
|f'_1(\tau,\bz) - f'_1(0,\bz)|  & \leqslant | m_1'(\bx^{\top}\btheta_0 + \tau\bx^{\top}(\btheta - \btheta_0),\by) - m_1'(\bx^{\top}\btheta_0 ,\by) |\cdot|\bx^{\top}(\btheta - \btheta_0)| \nonumber \\
& \leqslant C_m|\tau|\cdot|\bx^{\top}(\btheta - \btheta_0)|^2 \label{eq: dudley application 2}
\end{align}
by Assumption \ref{as: smoothness inference} and the Lemma \ref{lem: second order interpolation}
inequality as well.
It follows that whenever $f''_{11}(0,\bZ)$ exists, it satisfies
$|f_{11}''(0,\bZ)| \leqslant C_m|\bX^{\top}(\btheta - \btheta_0)|^2$,
and so $\E[|f''_{11}(0,\bZ)|]<\infty$ by 
Assumption \ref{as: integrability inference}. 
In addition, denoting $\overline f(\bz):=C_m|\bx^{\top}(\btheta - \btheta_0)|^2$, 
we have from \eqref{eq: dudley application 2} that 
$|f_1'(\tau,\bz) - f_1'(0,\bz)|\leqslant |\tau|\overline f(\bz)$, 
where $\E[\overline f(\bZ)]<\infty$ by Assumption \ref{as: integrability inference}. 
It thus follows that $g'(\tau)$ is differentiable at $\tau=0$ 
with derivative $g''(0) = \E[f_{11}''(0,\bZ)]$ by 
\citet[Corollary A.3]{dudley2014uniform} applied with $f_1'$ instead of $f$, 
and so $g(\tau)$ is twice differentiable at $\tau = 0$ with second derivative 
$g''(0) = \E[f_{11}''(0,\bZ)] = \E[m''_{11}(\bX^{\top}\btheta_0,\bY)|\bX^{\top}(\btheta - \btheta_0)|^2]$.

Next, from Taylor's theorem (with Peano's form of remainder),
we know that
\begin{equation}\label{eq: calculus number 1}
   g(\tau)-g(0)=g'(0)\tau+\frac{1}{2}g''(0)\tau^2+h(\tau)\tau^2,\quad\tau\in(-1,1),
\end{equation}
where the function $h:(-1,1)\to\R$ satisfies $h(\tau)\to0$ as $\tau\to 0$. On the other hand, since $\overline r_n \leqslant c_M'$, Assumption
\ref{assu:Margin} and $\btheta\in\mathcal B_{\btheta_0}(\overline r_n)$ imply $\|\btheta_\tau-\btheta_0\|_2=|\tau|\|\btheta-\btheta_0\|_2\leqslant c_M'$, and thus
\begin{equation}\label{eq: calculus number 2}
   g(\tau)-g(0)\geqslant c_M\tau^2\|\btheta -\btheta_0\|_2^2,\quad\tau\in(-1,1).
\end{equation}
We claim that \eqref{eq: calculus number 2} implies that $g'(0)=0$ and $g''(0)\geqslant
2c_M\|\btheta -\btheta_0\|_2^2$. To see the \emph{former}, 
combine \eqref{eq: calculus number 1} and \eqref{eq: calculus number 2} to obtain
$$
g'(0) + \frac{1}{2}g''(0)\tau + h(\tau)\tau \geqslant c_M \tau \|\btheta - \btheta_0\|_2^2,\quad\tau\in(0,1),
$$
and
$$
g'(0) + \frac{1}{2}g''(0)\tau + h(\tau)\tau \leqslant c_M \tau \|\btheta - \btheta_0\|_2^2,\quad\tau\in(-1,0).
$$ 
Take the limits as $\tau\to0_{+}$ and $\tau\to0_{-}$, respectively, to see that both
$g'(0)\geqslant0$ and $g'(0)\leqslant0$, and so $g'(0)=0$.
To see the \emph{latter}, combine 
\eqref{eq: calculus number 1}, \eqref{eq: calculus number 2}, and $g'(0) = 0$ 
to obtain
$$
\frac{1}{2}g''(0)+h(\tau) \geqslant c_M\|\btheta - \btheta_0\|_2^2,\quad\tau\in(-1,1)\setminus\{0\}.
$$
Using $h(\tau)\to0$ as $\tau\to0$, the claim follows from taking the limit as
$|\tau|\to0_+$. In turn, 
$g'(0)=\E[m_1'(\bX^{\top}\btheta_0,\bY)\bX^{\top}(\btheta - \btheta_0)]=0$ 
for all $\btheta\in\mathcal B_{\btheta_0}(\overline r_n)$
gives \eqref{eq: first order condition inference lemma} by varying 
$\btheta$ over $\mathcal B_{\btheta_0}(\overline r_n)$. (See the proof of
Lemma \ref{lem:EofUXis0} for details.) Finally,
$g''(0)=\E[m_{11}''(\bX^{\top}\btheta_0,\bY)|\bX^{\top}(\btheta - \btheta_0)|^2]\geqslant 2c_M\|\btheta - \btheta_0\|_2^2$  
for all $\btheta\in\mathcal B_{\btheta_0}(\overline r_n)$
gives \eqref{eq: second order condition inference lemma} by varying $\btheta$ 
over $\mathcal B_{\btheta_0}(\overline r_n)$ and rescaling $\btheta - \btheta_0$.
\end{proof}

\begin{proof}[\sc{Proof of Lemma \ref{lem:ExistenceAndUniquenessOfMu0}}]
Using Lemma \ref{lem: loss second derivative} followed by the triangle, 
Cauchy-Schwarz and Jensen inequalities along with Assumption \ref{as: integrability inference},
we see that $\E[|m_{11}''(\bX^\top\btheta_0,\bY)||D-\bW^\top\bmu||W_j|]<\infty$ for
any $\bmu\in\R^{p-1}$ and $j\in[p-1]$, so the problem of solving the system of equations
\eqref{eq: definition of mu0} is well-defined. Also, it 
follows from Lemma \ref{lem:ExistenceOfSecondOrderDerivatives} that for all
$\btheta\in\R^p$,   
\[
   \E\left[m_{11}''\left(\bX^\top\btheta_0,\bY\right)|\bX^\top(\btheta-\btheta_0)|^2\right]\geqslant 2c_M\|\btheta-\btheta_0\|^2_2.
\]
For any $\btheta = (\beta,\bgamma^{\top})^{\top}\in\R^p$
such that $\beta=\beta_0$, we therefore obtain
\[
   \left(\bgamma-\bgamma_0\right)^\top\E\left[m_{11}''\left(\bX^\top\btheta_0,\bY\right)\bW\bW^\top\right]\left(\bgamma-\bgamma_0\right)\geqslant 2c_M\|\bgamma-\bgamma_0\|^2_2.
\]
The previous display implies that for any $\bv\in\R^{p-1}$ such that
$\|\bv\|_2 = 1$,
\[
   \bv^\top\E\left[m_{11}''\left(\bX^\top\btheta_0,\bY\right)\bW\bW^\top\right]\bv\geqslant 2c_M>0,
\]
which further implies that $\E[m_{11}''(\bX^\top\btheta_0,\bY)\bW\bW^\top]$ is
positive definite. Hence, a solution $\bmu_0$ to \eqref{eq: definition of mu0}
exists, is unique, and is given by \eqref{eq: mu definition proof lemma}.
\end{proof}

\begin{proof}[\sc{Proof of Lemma \ref{lem:VarianceDenominatorBound}}]
From Lemma \ref{lem:ExistenceOfSecondOrderDerivatives}, we know that for all
$\btheta\in\R^p$,   
\[
   \E\left[m_{11}''\left(\bX^\top\btheta_0,\bY\right)|\bX^\top(\btheta-\btheta_0)|^2\right]\geqslant 2c_M\|\btheta-\btheta_0\|^2_2.
\]
Let $\Delta:=1/\sqrt{1+\|\bmu_0\|_2^2}\in(0,1]$ with
$\bmu_0$ given by \eqref{eq: mu definition proof lemma}. Then
$\btheta':=(\beta_0+\Delta,\bgamma_0^\top-\Delta\bmu_0^\top)^\top$ satisfies
$\|\btheta'-\btheta_0\|_2^2=(1+\|\bmu_0\|_2^2)\Delta^2=1$, 
and $\bX^\top(\btheta'-\btheta_0)=\Delta(D-\bW^\top\bmu_0)$.
With $\btheta=\btheta'$, the previous display produces
\[
   \Delta^2\E\left[m_{11}''\left(\bX^\top\btheta_0,\bY\right)\left(D-\bW^\top\bmu_0\right)^2\right]\geqslant 2c_M,
\]
and so, since $\Delta\in(0,1]$, we get
\[
   \E\left[m_{11}''\left(\bX^\top\btheta_0,\bY\right)\left(D-\bW^\top\bmu_0\right)^2\right]\geqslant 2c_M.
\]
On the other hand, it follows from \eqref{eq: definition of mu0} that
\begin{align*}
   &\E\left[m_{11}''\left(\bX^\top\btheta_0,\bY\right)\left(D-\bW^\top\bmu_0\right)^2\right]=\E\left[m_{11}''\left(\bX^\top\btheta_0,\bY\right)\left(D-\bW^\top\bmu_0\right)D\right].
\end{align*}
The asserted claim follows by combining the last two displays.
\end{proof}

\begin{proof}[\sc{Proof of Lemma \ref{lem: inference proof remainder 1}}]
Suppose first that $J = 1$, such that $m$ is everywhere twice continuously 
differentiable in its first argument. Then
\begin{align}
   |R_{n,1}| 
   &\leqslant C_m\En\left[|(D_i - \bW_i^{\top}\bmu_0)D_i\bX_i^\top(\widetilde\btheta-\btheta_0)|\right]\nonumber \\
   &\leqslant C_m\|\widetilde\btheta-\btheta_0\|_1\En\left[|D_i - \bW_i^{\top}\bmu_0||D_i|\|\bX_i\|_\infty\right]\nonumber\\
   &\leqslant C_m\|\widetilde\btheta-\btheta_0\|_1\left(\En\left[|D_i - \bW_i^{\top}\bmu_0|^4\right]\En\left[|D_i|^4\right]\right)^{1/4}\sqrt{\En\left[\|\bX_i\|_\infty^2\right]}\nonumber\\
   &\lesssim_{\P} a_nB_n, \label{eq: non-smooth case term 1}
\end{align}
where the first inequality follows from the mean-value theorem and 
Assumption \ref{as: smoothness inference},
the second from H{\"o}lder's inequality, 
the third from the Cauchy-Schwarz inequality,
and the $\lesssim_\P$ from Markov's inequality and 
Assumption \ref{assu:LossLocallyLipschitzAndMore}.\ref{enu:LossLocallyLipschitz}
(recall that we take $L\geqslant1$) in combination with Assumptions
\ref{as: integrability inference} and \ref{as: convergence rates inference}. 
Given that $a_nB_n\to0$ by assumption, it thus follows that 
$R_{n,1} = o_{\P}(1)$.

Suppose now that $J\geqslant 2$. For all $i\in[n]$, we denote $\Delta_i:= |\bX_i^{\top}(\widetilde\btheta - \btheta_0)|$, 
\begin{equation}\label{eq: caligraph r1}
\mathcal R_{i,1} = 
\begin{cases} 
1,\quad \text{if $\bX_i^{\top}\btheta_0 - \Delta_i \leqslant t_{Y_i,j}\leqslant \bX_i^{\top}\btheta_0 + \Delta_i$ for some $j\in[J-1]$},\\
0,\quad\text{otherwise},
\end{cases}
\end{equation}
and
\begin{equation}\label{eq: caligraph r2}
\mathcal R_{i,2} = 
\begin{cases} 
1,\quad \text{if $\bX_i^{\top}\btheta_0 - \overline\Delta_n \leqslant t_{Y_i,j}\leqslant \bX_i^{\top}\btheta_0 + \overline\Delta_n$ for some $j\in[J-1]$},\\
0,\quad\text{otherwise},
\end{cases}
\end{equation}
Observe that for all $i\in[n]$, we have $\mathcal R_{i,1}^2 = \mathcal R_{i,1} \leqslant \mathcal R_{i,2} + \mathbf 1(\Delta_i > \overline\Delta_n)$ and 
$$
\mathbf 1(\Delta_i > \overline\Delta_n)\leqslant \mathbf 1(\|\bX_i\|_{\infty}\|\widetilde\btheta - \btheta_0\|_1 > \overline\Delta_n)\leqslant \left(\|\bX_i\|_{\infty}\|\widetilde\btheta - \btheta_0\|_1/\overline\Delta_n\right)^{r},
$$
with $r\in(4,\infty)$ provided by 
Assumption \ref{assu:LossLocallyLipschitzAndMore}.\ref{enu:LossLocallyLipschitz}.
Hence,
\begin{align}
\En[\mathcal R_{i,1}^2] 
&  \leqslant \En[\mathcal R_{i,2}] +\En[\| \bX_i \|^r_{\infty}] \left(\|\widetilde\btheta - \btheta_0\|_1/\overline\Delta_n\right)^{r}  \lesssim_{\P}  \overline\Delta_n + \left(B_n a_n/\overline\Delta_n\right)^{r}, \label{eq: non-smooth bound main appr 1}
\end{align}
where the $\lesssim_\P$ follows from Markov's inequality and 
Assumptions \ref{assu:LossLocallyLipschitzAndMore}.\ref{enu:LossLocallyLipschitz}, 
\ref{as: conditional density} and \ref{as: convergence rates inference}.

Next, decompose as $R_{n,1} = R_{n,1,1} + R_{n,1,2}$, where
\begin{align*}
R_{n,1,1} & := \En\left[(1-\mathcal R_{i,1})\left(m''_{11}(\bX_i^{\top}\widetilde\btheta, \bY_i) - m''_{11}(\bX_i^{\top}\btheta_0, \bY_i)\right)(D_i - \bW_i^{\top}\bmu_0)D_i\right],\\
R_{n,1,2} & := \En\left[\mathcal R_{i,1}\left(m''_{11}(\bX_i^{\top}\widetilde\btheta, \bY_i) - m''_{11}(\bX_i^{\top}\btheta_0, \bY_i)\right)(D_i - \bW_i^{\top}\bmu_0)D_i\right].
\end{align*}
Then $R_{n,1,1} = o_{\P}(1)$ as in the case of $J=1$ treated previously. Also,
\begin{align}
|R_{n,1,2}|
& \leqslant 2C_m\En[\mathcal R_{i,1} |(D_i - \bW_i^{\top}\bmu_0)D_i |] \nonumber\\
& \leqslant 2C_m\Big(\En[\mathcal R_{i,1}^2]\En[(D_i - \bW_i^{\top}\bmu_0)^2D_i^2]\Big)^{1/2}\\ 
& \lesssim_{\P} \overline\Delta_n^{1/2} + \left(B_n a_n/\overline\Delta_n\right)^{r/2}\label{eq: bound in32}
\end{align}
where the first inequality follows from the triangle inequality
and Lemma \ref{lem: loss second derivative}, 
the second from the Cauchy-Schwarz inequality, 
and the $\lesssim_\P$ from \eqref{eq: non-smooth bound main appr 1},
Assumption \ref{as: integrability inference}, and the Markov and Cauchy-Schwarz
inequalities. Given that $\overline\Delta_n\to0$ and 
$B_na_n/\overline\Delta_n\to 0$ by assumption, it follows that 
$R_{n,1,2} = o_{\P}(1)$, and so $R_{n,1} = R_{n,1,1} + R_{n,1,2} = o_{\P}(1)$ 
as well.
\end{proof}

\begin{proof}[\sc{Proof of Lemma \ref{lem: inference proof remainder 2}}]
Suppose first that $J=1$. Then 
$$
   |R_{n,2}| \leqslant C_m|\widetilde\beta-\beta_0|\En\left[|D_i-\bW_i^\top\widetilde\bmu|D_i^2\right]
$$
by Assumption \ref{as: smoothness inference}. The right-hand side average satisfies
\begin{align}
   \En\left[|D_i-\bW_i^\top\widetilde\bmu|D_i^2\right]
   &\leqslant \En\left[|D_i-\bW_i^\top\bmu_0|D_i^2\right]+\En\left[|\bW_i^\top(\widetilde\bmu-\bmu_0)|D_i^2\right]\nonumber\\
   &\leqslant \sqrt{\En\left[D_i^4\right]}\left(\sqrt{\En\left[|D_i-\bW_i^\top\bmu_0|^2\right]}+\|\widetilde\bmu-\bmu_0\|_1\sqrt{\En\left[\|\bW_i\|^2_\infty\right]}\right)\nonumber\\
   &\lesssim_{\P} 1 + a_nB_n\lesssim 1,\label{eq: why am i here}
\end{align}
where the first inequality follows from the triangle inequality,
the second from the Cauchy-Schwarz and H{\"o}lder inequalities, the $\lesssim_\P$ from
Assumptions \ref{assu:LossLocallyLipschitzAndMore}.\ref{enu:LossLocallyLipschitz},
\ref{as: integrability inference} and \ref{as: convergence rates inference},
and the $\lesssim$ from $a_nB_n\to0$, which holds by hypothesis. Since 
$|\widetilde\beta-\beta_0|\lesssim_{\P} a_n$ by 
Assumption \ref{as: convergence rates inference}, 
it follows that $|R_{n,2}|\lesssim_{\P} a_n$.

Suppose now that $J\geqslant 2$. For all $i\in[n]$, denote 
$\Delta_i:= |\bX_i^{\top}(\widetilde\btheta - \btheta_0)| + |D_i(\widetilde\beta - \beta_0)|$ 
and define $\mathcal R_{i,1}$ and $\mathcal R_{i,2}$ 
as in \eqref{eq: caligraph r1} and \eqref{eq: caligraph r2}, respectively. 
Then the rate in \eqref{eq: non-smooth bound main appr 1} follows from 
the argument used in the proof of Lemma \ref{lem: inference proof remainder 1}. 
Next, decompose as $R_{n,2} = R_{n,2,1} + R_{n,2,2}$, where
\begin{align*}
R_{n,2,1} &:= \int_0^1\En\left[(1-\mathcal R_{i,1})\left(m''_{11}(  \bX_i^{\top}\check\btheta_{\tau}, \bY_i) - m''_{11}(\bX_i^{\top}\widetilde\btheta, \bY_i)\right)(D_i - \bW_i^{\top}\widetilde\bmu)D_i\right]d\tau,\\
R_{n,2,2} &:= \int_0^1 \En\left[\mathcal R_{i,1}\left(m''_{11}( \bX_i^{\top}\check\btheta_{\tau}, \bY_i) - m''_{11}(\bX_i^{\top}\widetilde\btheta, \bY_i)\right)(D_i - \bW_i^{\top}\widetilde\bmu)D_i\right]d\tau.
\end{align*}
Then $|R_{n,2,1}| \lesssim_{\P} a_n$ as in the case $J=1$. Also,
\begin{align*}
|R_{n,2,2}|
& \leqslant 2C_m\En[\mathcal R_{i,1} |(D_i - \bW_i^{\top}\widetilde\bmu)D_i |]  \lesssim_{\P} \overline\Delta_n^{1/2} + (B_na_n /\overline\Delta_n)^{r/2},
\end{align*}
where the inequality follows from triangle inequality and Lemma \ref{lem: loss second derivative} 
and the $\lesssim_\P$ follows from the Cauchy-Schwarz inequality, 
\eqref{eq: non-smooth bound main appr 1}, and an argument similar to that
leading to \eqref{eq: why am i here}. Thus, 
$|R_{n,2}| \leqslant |R_{n,2,1}| + |R_{n,2,2}| \lesssim_{\P} a_n + \overline\Delta_n^{1/2} + (B_na_n /\overline\Delta_n)^{r/2}$, 
as claimed.
\end{proof}

\begin{proof}[\sc{Proof of Lemma \ref{lem: inference proof remainder 3}}]
Suppose first that $J=1$. Then 
\begin{align}
|R_{n,3}| 
& \leqslant C_m\En[|D_i - \bW_i^{\top}\bmu_0|(\bW_i^{\top}\widetilde\bgamma - \bW_i^{\top}\bgamma_0)^2] \nonumber\\
& \leqslant C_m \|\widetilde \bgamma - \bgamma_0\|_1^2  \max_{1\leqslant j,k\leqslant p-1}\En[|(D_i - \bW_i^{\top}\bmu_0)W_{i,j}W_{i,k}|] \label{eq: non-smooth case term 3}
\end{align}
by Assumption \ref{as: smoothness inference} followed by H{\"o}lder's inequality.
Setting up for an application of Theorem \ref{thm: max inequality nonnegative},
note that for all $j,k\in[p-1]$, we have  
$\E[|(D-\bW^{\top}\bmu_0)W_jW_k|]\leqslant C_M^3$ 
by the Cauchy-Schwarz inequality and Assumption \ref{as: integrability inference}. 
In addition, for $q:=(r\wedge \widetilde r)/4\in(1,\infty)$, 
$$
\E[|D - \bW^{\top}\bmu_0|^q\|\bW\|_{\infty}^{q}\|\bW\|_{\infty}^{q}] \leqslant C_M^q B_n^{2q}
$$ 
by the Cauchy-Schwarz inequality and 
Assumptions \ref{assu:LossLocallyLipschitzAndMore}.\ref{enu:LossLocallyLipschitz} 
and \ref{as: integrability inference}. Therefore, given that 
$B_n^2\ln(p n) = o(n^{1-4/(r\wedge \widetilde r)})$ by assumption,
Theorem \ref{thm: max inequality nonnegative} produces
$$
\max_{1\leqslant j,k\leqslant p-1}\En[|(D_i - \bW_i^{\top}\bmu_0)W_{i,j}W_{i,k}|] \lesssim_{\P} 1.
$$
Hence, given that $\|\widetilde\bgamma - \bgamma_0\|_{\infty}\lesssim_{\P} a_n$ by Assumption \ref{as: convergence rates inference}, it follows that $|R_{n,3}| \lesssim_{\P} a_n^2$.

Suppose now that $J\geqslant 2$. For all $i\in[n]$, denote 
$\Delta_i:= |\bW_i^{\top}(\widetilde\bgamma - \bgamma_0)|$ and define 
$\mathcal R_{i,1}$ and $\mathcal R_{i,2}$ as in \eqref{eq: caligraph r1} 
and \eqref{eq: caligraph r2}, respectively. Then the rate in \eqref{eq: non-smooth bound main appr 1} 
follows as in the proof of Lemma \ref{lem: inference proof remainder 1}. 
Next, decompose as $R_{n,3} = R_{n,3,1} + R_{n,3,2}$, where
\begin{align*}
R_{n,3,1} &: = \int_0^1\En\big[(1 - \mathcal R_{i,1})\big(m''_{11}(\bX_i^{\top}\mathring\btheta_\tau, \bY_i) -  m''_{11}(\bX_i^{\top}\btheta_0,\bY_i)\big)(D_i - \bW_i^{\top}\bmu_0)\bW_i^{\top}(\widetilde\bgamma-\bgamma_0)\big]d\tau,\\
R_{n,3,2} &: = \int_0^1\En\big[\mathcal R_{i,1}\big(m''_{11}(\bX_i^{\top}\mathring\btheta_\tau, \bY_i) -  m''_{11}(\bX_i^{\top}\btheta_0,\bY_i)\big) (D_i - \bW_i^{\top}\bmu_0)\bW_i^{\top}(\widetilde\bgamma-\bgamma_0)\big]d\tau,
\end{align*}
Then $|R_{n,3,1}| \lesssim_{\P}a_n^2$ follows from the argument in the case $J=1$. 
Also,
\begin{align*}
|R_{n,3,2}| 
&\leqslant 2C_m\En\left[\mathcal R_{i,1}|(D_i - \bW_i^{\top}\bmu_0)\bW_i^{\top}(\widetilde\bgamma - \bgamma_0)|\right] \\
& \lesssim_{\P} B_n a_n \left( \overline\Delta_n^{1/2} + (B_n a_n /\overline\Delta_n)^{r/2} \right),
\end{align*}
where the inequality follows from the triangle inequality and 
Lemma \ref{lem: loss second derivative},
and the $\lesssim_{\P}$ follows from the Cauchy-Schwarz and H{\"o}lder inequalities, 
\eqref{eq: non-smooth bound main appr 1}, and 
Assumptions \ref{assu:LossLocallyLipschitzAndMore}.\ref{enu:LossLocallyLipschitz}, 
\ref{as: integrability inference}, and \ref{as: convergence rates inference}. 
Hence, $|R_{n,3}| \leqslant |R_{n,2,1}| + |R_{n,3,2}| \lesssim_{\P} a_n^2 + B_na_n(\overline\Delta_n^{1/2} + (B_n a_n /\overline\Delta_n)^{r/2})$, as claimed.
\end{proof}

\section{Analysis of Post-Penalized M-Estimation}\label{sec:analysis post estimator} 
In this section, we derive an analog of Theorem
\ref{thm:NonAsymptoticProbabilisticBounds} for the post-$\ell_1$-penalized
M-estimator (post-$\ell_1$-ME). The main message of this section is similar to
that of Section \ref{sec:Deterministic-Bounds}: like in the case of the
$\ell_1$-ME, in order to obtain the post-$\ell_1$-ME with small estimation
errors, we should choose the penalty parameter $\lambda$ such that it is as
small as possible but larger than the (slightly inflated) maximum of the score
$c_0\|\bS_n\|_{\infty}$ with high probability, with $\bS_n$ defined in
\eqref{eq:Score}. Theorem \ref{thm:NonAsymptoticProbabilisticBounds 2}, which is
the main result of this section, serves as the key building block for the proof
of Theorem \ref{thm: convergence rates post estimator} in the main text.

Recall the definitions of $\widetilde{\Theta}(\suppn{\overline\btheta})$ and
$\widetilde\Theta(\lambda)$ in \eqref{eq:post-ell1-ME-thetabar} and
\eqref{eq:AllPostEstimators}, respectively, and that $\eta_n=\sqrt{\ln(pn)/n}$.
The following theorem yields estimation error bounds for the post-$\ell_1$-ME.
\begin{thm}
[\textbf{Non-Asymptotic Error Bounds for
Post-$\boldsymbol{\ell_1}$-ME}]\label{thm:NonAsymptoticProbabilisticBounds 2}\label{thm; post penalized estimator non asymptotic} 
Let Assumptions \ref{assu:ParameterSpace}--\ref{assu:Approximate-Sparsity},
\ref{as: post estimator smoothness}, and \ref{as: post estimator moments} hold,
let $\overline\lambda_n$ and $\underline\lambda_n$ be non-random sequences in
$(0,\infty)$ such that $\overline\lambda_n\geqslant\underline\lambda_n$, and let
$\phi_n:= \sqrt{(\eta_n^2 + \overline\lambda_n^2)/\underline\lambda_n^2}$. In
addition, suppose that
\begin{equation}\label{eq: post ell1 rate growth condition}
\frac{B_n^2\ln(pn)}{\sqrt n}\to 0\quad\text{and}\quad n^{1/r}B_ns_q\eta_n^{-q}\left( \overline\lambda_n\phi_n+ \eta_n\phi_n^2 \right)\to0.
\end{equation}
Then there is a constant $C\in[1,\infty)$, depending only on $c_0$, $C_{ev}$,
$C_L$, $c_M$ and $C_m$, such that
\begin{equation}\label{eq: post estimator ell 2 bound}
\sup_{\widetilde{\btheta}\in\widetilde{\Theta}\left(\lambda\right)}\|\widetilde{\btheta}-\btheta_{0}\|_{2}  \leqslant C\sqrt{s_q\eta_n^{-q}}\left( \overline\lambda_n + \eta_n \phi_n \right)
\end{equation}
and
\begin{equation}\label{eq: post estimator ell 1 bound}
\sup_{\widetilde{\btheta}\in\widetilde{\Theta}\left(\lambda\right)}\|\widetilde{\btheta}-\btheta_{0}\|_{1}\leqslant Cs_q\eta_n^{-q}\left( \overline\lambda_n\phi_n + \eta_n\phi_n^2 \right)
\end{equation}
with probability at least
$1-C(\mathrm{P}(\lambda<c_{0}\|\bS_n\|_{\infty})+\mathrm{P}(\lambda>\overline{\lambda}_{n})
+\mathrm{P}(\lambda<\underline{\lambda}_{n})) - o(1).$
\end{thm}

Before we prove this theorem, we introduce some extra notation. For $k\in\mathbb
N$, define the ($\ell_0$-) restricted set
\begin{align*}
   \widetilde{\mathcal R}(k)&:=\left\{\bdelta\in\mathbb R^p; \ \|\btheta_0 + \bdelta\|_0\leqslant k\text{ and }\btheta_0+\bdelta\in\Theta\right\},
\end{align*}
and the associated (random) empirical error function
$\widetilde\epsilon:[0,\infty)\times\mathbb N\to[0,\infty)$ by
\begin{equation}\label{eq: empirical error function for post estimator}
\widetilde{\epsilon}_{n}(u,k):=\sup_{\substack{\bdelta\in\widetilde{\mathcal{R}}(k),\\
\left\Vert \bdelta\right\Vert _{2}\leqslant u
}
}\left|\left(\mathbb{E}_{n}-\mathrm{E}\right)\left[m\left(\bX_{i}^{\top}\left(\btheta_{0}+\bdelta\right),\bY_{i}\right)-m\left(\bX_{i}^{\top}\btheta_{0},\bY_{i}\right)\right]\right|.
\end{equation}
Also, let $a_{\epsilon,n}$, $b_{\epsilon,n}$, $\widetilde a_{\epsilon,n}$,
$\widetilde b_{\epsilon,n}$ and $\overline\lambda_n$ be non-random sequences in
$(0,\infty)$, and let $\overline k_n$ be a non-random sequence in $\mathbb N$,
all to be specified later. Based on $a_{\epsilon,n}$, $b_{\epsilon,n}$  and
$\overline\lambda_n$, define the non-random sequence $\widetilde u_n$ in
$(0,\infty)$ as in \eqref{eq: definition un tilde key} and the events
$\mathscr{S}_n$, $\mathscr{L}_n$, and $\mathscr{E}_n$ as in \eqref{eq: three
events}. Moreover, define the events
$$
\mathscr{\widetilde E}_n :=\Big\{\widetilde \epsilon_n(\widetilde u_n + \check u_n,\overline k_n) \leqslant \widetilde a_{\epsilon,n}(\widetilde u_n + \check u_n) + \widetilde b_{\epsilon,n}\Big\}\quad\text{and}\quad \mathscr{K}_n:= \left\{\sup_{\widehat\btheta\in\widehat\Theta(\lambda)}\|\widehat\btheta\|_0\leqslant \overline k_n\right\},
$$
where
\begin{equation}\label{eq: un check definition}
\check u_n:= 4\widetilde u_n + 2\widetilde a_{\epsilon,n}/c_M.
\end{equation}

The proof of Theorem \ref{thm:NonAsymptoticProbabilisticBounds 2} will be based
on the following four lemmas, whose proofs can be found at the end of this
section.

\begin{lem}[\textbf{Restricted Set Consequence, II}]\label{lem:RestrictedSetConsequence 2}
For any $\eta\in(0,\infty)$ and $k\in\mathbb N$, $\bdelta\in\widetilde{\mathcal{R}}(k)$ implies
$\|\bdelta\|_1\leqslant \|\bdelta\|_2 \sqrt{k+s_q\eta^{-q}}+s_q\eta^{1-q}$.
\end{lem}
\begin{lem}
[\textbf{Non-Asymptotic Deterministic Bounds for
Post-$\boldsymbol{\ell_1}$-ME}]\label{lem:NonasymptoticDeterministicBounds 2} Let Assumptions
\ref{assu:ParameterSpace}--\ref{assu:Approximate-Sparsity} hold and suppose that
$\widetilde{u}_{n}+\check u_n\leqslant c_{M}'$,
\begin{equation}\label{eq: lemma c3 additional constraint 2}
\Big( a_{\epsilon,n}+(1+\overline c_0)\overline\lambda_n\sqrt{s_q\eta^{-q}_n} \Big)^2 \geqslant c_M\Big( b_{\epsilon,n} + (1+\overline c_0)\overline\lambda_ns_q \eta_n^{1-q} \Big)
\end{equation}
and
\begin{equation}\label{eq: constant constraints for post estimator}
( 2c_M\widetilde u_n + \widetilde a_{\epsilon,n})^2 \geqslant c_M\Big( \widetilde a_{\epsilon,n}\widetilde u_n + \widetilde b_{\epsilon,n} + \overline \lambda_n(1+\overline c_0)\widetilde u_n\sqrt{s_q\eta_n^{-q}} + \overline \lambda_n s_q\eta_n^{1-q}\Big).
\end{equation}
Then on the event
$\mathscr{S}_n\cap\mathscr{L}_n\cap\mathscr{E}_n\cap\mathscr{\widetilde
E}_n\cap\mathscr K_n$, for all
$\widetilde{\btheta}\in\widetilde{\Theta}(\lambda)$, we have $\widetilde{\btheta}-\btheta_{0}\in\widetilde{\mathcal{R}}(\overline k_n)$,
$$
\|\widetilde{\btheta} -\btheta_{0}\|_{2}  \leqslant \widetilde u_n + \check{u}_{n}\ \text{ and }\ \|\widetilde{\btheta} - \btheta_{0}\|_{1}  \leqslant (\widetilde u_n + \check u_n)\sqrt{\overline k_n + s_q\eta_n^{-q}} + s_q\eta_n^{1-q}.
$$
\end{lem}

\begin{lem}
[\textbf{Empirical Error Bound, II}]\label{lem:EmpiricalErrorBound 2} Let
Assumptions \ref{assu:LossLocallyLipschitzAndMore} and
\ref{assu:Approximate-Sparsity} hold. Then
there is a universal constant $C\in[1,\infty)$, such that for any $k\in\mathbb
N$, $n\in\N$, $t\in[1,\infty)$ and $u\in(0,\infty)$ satisfying
\begin{equation}\label{eq: extra condition to apply contraction principle 2}
\eta_{n}\leqslant1,\quad \frac{B_n^2 \ln(pn)}{\sqrt n}\leqslant C_L^2\quad\text{and}\quad t n^{1/r} B_n\Big(u\sqrt{k+s_q\eta_n^{-q}} +s_{q}\eta_{n}^{1-q}\Big)\leqslant c_{L},
\end{equation}
we have
\begin{align*}
\widetilde{\epsilon}_{n}\left(u, k\right) & \leqslant CC_{L}\left( u\eta_n\sqrt{k+s_q\eta_n^{-q}} +s_{q}\eta_{n}^{2-q}\right)
\end{align*}
with probability at least $1-4t^{-r} - C/\ln^2(p n) - n^{-1}.$
\end{lem}

\begin{lem}
[\textbf{Sparsity Bound}]\label{lem: sparsity} Let Assumptions
\ref{assu:ParameterSpace}--\ref{assu:Approximate-Sparsity}, \ref{as: post estimator smoothness} and \ref{as: post estimator moments} hold and let $\overline\lambda_n$ and $\underline\lambda_n$ be
non-random sequences in $(0,\infty)$ such that
$\overline\lambda_n\geqslant\underline\lambda_n$, and set $\phi_n:= \sqrt{(\eta_n^2 +
\overline\lambda_n^2)/\underline\lambda_n^2}$. In addition, assume that
\begin{equation}\label{eq: sparsity growth conditions}
   \frac{B^2_n\ln(pn)}{\sqrt n}\to0 \quad \text{and}\quad n^{1/r}B_ns_q\eta_n^{-q}\left(\overline\lambda_n + \eta_n\phi_n^2\right)\to0.
\end{equation}
Then there is a constant $C\in[1,\infty)$, depending only on $c_0$,
$C_{ev}$, $C_L$, $c_M$ and $C_m$, such that
$$
\sup_{\widehat{\btheta}\in\widehat{\Theta}\left(\lambda\right)}\|\widehat{\btheta}\|_{0}  \leqslant C s_q\eta_n^{-q}\phi_n^2
$$
with probability at least $1 -
2\P(\lambda<c_{0}\|\bS_n\|_{\infty})-\mathrm{P}(\lambda>\overline{\lambda}_{n})
- \P(\lambda<\underline\lambda_n) - o(1)$.
\end{lem}

\begin{proof}
[\sc{Proof of Theorem \ref{thm:NonAsymptoticProbabilisticBounds 2}}] Let
$\widetilde C$ and $\overline C$ be universal constants $C$ from Lemmas
\ref{lem:EmpiricalErrorBound} and \ref{lem:EmpiricalErrorBound 2}, respectively,
and let $\check C$ be the constant from Lemma \ref{lem: sparsity}. Define
$$
a_{\epsilon,n}:=\widetilde C(1+\overline c_0)C_L\sqrt{s_q\eta_n^{2-q}},\quad b_{\epsilon,n} :=\widetilde C(1+\overline c_0)C_Ls_q\eta_n^{2-q}, \quad \overline k_n:=\lceil\check Cs_q\eta_n^{-q}\phi_n^2\rceil,
$$
$$
\widetilde a_{\epsilon,n}:= \overline CC_L\eta_n\sqrt{\overline k_n + s_q\eta_n^{-q}},\quad\text{and}\quad \widetilde b_{\epsilon,n}:=\overline CC_Ls_q\eta_n^{2-q}.
$$
Then, under \eqref{eq: post ell1 rate growth condition}, Lemmas
\ref{lem:EmpiricalErrorBound} and \ref{lem:EmpiricalErrorBound 2} imply (via
appropriate choices of $t_n$ sequences) that $\P(\mathscr E_n^c)\to0$ and
$\P(\mathscr{\widetilde E}_n^c)\to 0$, respectively, and Lemma \ref{lem:
sparsity} implies that $\P(\mathscr K_n^c)\leqslant
2\P(\lambda<c_{0}\|\bS_n\|_{\infty}) +
\mathrm{P}(\lambda>\overline{\lambda}_{n}) + \P(\lambda<\underline\lambda_n) +
o(1)$. In addition, again under \eqref{eq: post ell1 rate growth condition}, we
have that $\widetilde u_n + \check u_n \to 0$, and using the above definitions
and rearranging shows that both \eqref{eq: lemma c3
additional constraint 2} and \eqref{eq: constant constraints for post estimator}
are satisfied since $\widetilde C,\overline C,C_L\geqslant 1$ and $c_M\leqslant
1$.\footnote{To verify \eqref{eq: constant constraints for post estimator}, 
for example, note that 
$(2c_M \widetilde u_n + \widetilde a_{\epsilon,n})^2 = 4c_M^2\widetilde u_n^2 + 4c_M \widetilde u_n \widetilde a_{\epsilon,n} + \widetilde a_{\epsilon,n}^2$
and then $4c_M^2\widetilde u_n^2 \geqslant c_M \overline\lambda_n (1+\bar c_0)\widetilde u_n\sqrt{s_q \eta_n^{-q}}$, 
$c_M\widetilde u_n \widetilde a_{\epsilon,n} \geqslant c_M \widetilde b_{\epsilon,n}$ and $c_M \widetilde u_n \widetilde a_{\epsilon,n} \geqslant c_M \overline\lambda _n s_q \eta_n^{1-q}$.
Condition \eqref{eq: lemma c3 additional constraint 2} follows similarly.} 
Hence, for a sufficiently large constant $C\in[1,\infty)$, that can be chosen
to depend only on $c_0$, $C_{ev}$, $C_L$ and $C_m$, Lemma
\ref{lem:NonasymptoticDeterministicBounds 2} implies that the bounds \eqref{eq:
post estimator ell 2 bound} and \eqref{eq: post estimator ell 1 bound} hold with
probability at least
$\P(\mathscr{S}_n\cap\mathscr{L}_n\cap\mathscr{E}_n\cap\mathscr{\widetilde
E}_n\cap\mathscr K_n)$ for all $n$ large enough (so that $\widetilde u_n +
\check u_n\leqslant c_M'$). Combining these bounds gives the asserted claim.
\end{proof}

\begin{proof}[\sc{Proof of Lemma \ref{lem:RestrictedSetConsequence 2}}]
Recall that, given a vector $\bdelta\in\R^{p}$ and a set of indices
$J\subseteq\left[p\right]$, we let $\bdelta_{J}$ denote the vector in $\R^{p}$
with coordinates given by $\delta_{J,j}=\delta_{j}$ if $j\in J$ and
$\delta_{J,j}= 0$ otherwise. Now, fix $\eta\in(0,\infty)$ and $k\in\N$,
let $\bdelta\in\widetilde{\mathcal{R}}(k)$, and denote
$\btheta:=\btheta(\bdelta):=\btheta_0+\bdelta$ and $J:=\suppn{\btheta}$. Then we
may decompose as follows
\[
\|\bdelta\|_1=\|\bdelta_{J\cup T(\eta)}\|_1+\|\bdelta_{J^c\cap T(\eta)^c}\|_1,
\]
with $T(\eta)=\{j\in[p]; \ |\theta_{0,j}|>\eta\}$. The \emph{first} term on the
right-hand side satisfies
\begin{align*}
\|\bdelta_{J\cup T(\eta)}\|_1
   &\leqslant\|\bdelta_{J\cup T(\eta)}\|_2\sqrt{|J\cup T(\eta)|}\tag{Cauchy-Schwarz}\\
   &\leqslant\|\bdelta\|_2\sqrt{|J|+|T(\eta)|}\\
   &\leqslant\|\bdelta\|_2\sqrt{k + s_q\eta^{-q}}.\tag{$|J|=\|\btheta_0+\bdelta\|_0$ and Lemma \ref{lem:StrongBallConsequences}}
\end{align*}
Since $\suppn{\btheta}=J$, the \emph{second} term on the right-hand side satisfies
\begin{align*}
   \|\bdelta_{J^c\cap T(\eta)^c}\|_1=\|(\btheta-\btheta_0)_{J^c\cap T(\eta)^c}\|_1=\|\btheta_{0J^c\cap T(\eta)^c}\|_1\leqslant \|\btheta_{0T(\eta)^c}\|_1\leqslant s_q\eta^{1-q}.\tag{Lemma \ref{lem:StrongBallConsequences}}
\end{align*}
The claimed $\ell_1$ bound now arises from combining the previous three displays.
\end{proof}    
   
\begin{proof}
[\sc{Proof of Lemma \ref{lem:NonasymptoticDeterministicBounds 2}}] 
Let $\widehat\btheta\in\widehat\Theta(\lambda)$ and
$\widetilde\btheta\in\widetilde\Theta(\suppn{\widehat\btheta})$ be arbitrary,
and let the event
$\mathscr{S}_n\cap\mathscr{L}_n\cap\mathscr{E}_n\cap\mathscr{\widetilde
E}_n\cap\mathscr K_n$ hold. Abbreviate $\widetilde\bdelta:=\widetilde\btheta -
\btheta_0$. Then $\suppn{\widetilde\btheta}\subseteq\suppn{\widehat\btheta}$ and
$\mathscr K_n$ imply
\[
   \|\btheta_0+\widetilde\bdelta\|_0=\|\widetilde\btheta\|_0\leqslant\|\widehat\btheta\|_0\leqslant\overline k_n
\]
and, thus, $\widetilde\bdelta\in\widetilde{\mathcal{R}}(\overline k_n)$, as claimed. The
stated $\ell_1$ bound will therefore follow from Lemma \ref{lem:RestrictedSetConsequence
2} and the $\ell_2$ bound. It remains to show the latter bound. 

Suppose to
the contrary of the asserted claim that $\|\widetilde\btheta - \btheta_0\|_2 >
\widetilde u_n + \check u_n$. Then we must have $\|\widetilde\btheta -
\widehat\btheta\|_2 > \check u_n$, since $\|\widehat\btheta -
\btheta_0\|_2\leqslant \widetilde u_n$ by Lemma
\ref{lem:NonasymptoticDeterministicBounds}. By definition of $\widetilde\Theta(\suppn{\widehat\btheta)}$ in
\eqref{eq:post-ell1-ME-thetabar},
$$
\En[m(\bX_i^\top\widetilde\btheta,\bY_i)]\leqslant\En[m(\bX_i^\top\widehat\btheta,\bY_i)].
$$
For $t:=\check u_n / \|\widetilde\btheta - \widehat\btheta\|_2\in(0,1)$, it
then follows from convexity (Assumption \ref{assu:Convexity}) that
$$
\En[m(\bX_i^\top(\widehat\btheta + t(\widetilde\btheta - \widehat\btheta)),\bY_i)]\leqslant\En[m(\bX_i^\top\widehat\btheta,\bY_i)].
$$
In addition, by definition of $\widehat\Theta(\lambda)$ in
\eqref{eq:ell1PenalizedMEstimationIntro}, the triangle inequality, 
and the $\ell_1$-bound of Lemma \ref{lem:NonasymptoticDeterministicBounds}, 
we have
$$
\En[m(\bX_i^{\top}\widehat\btheta,\bY_i)] - \En[m(\bX_i^{\top}\btheta_0,\bY_i)] \leqslant \lambda(\|\btheta_0\|_1 - \|\widehat\btheta\|_1) \leqslant \overline \lambda_n \overline\Delta_n,
$$
where $\overline\Delta_n:= (1+\overline c_0)(\widetilde u_n\sqrt{s_q
\eta_n^{-q}} + s_q\eta_n^{1-q})$.
Combining these bounds, we obtain
$$
\overline \lambda_n \overline\Delta_n \geqslant \En[m(\bX_i^\top(\widehat\btheta + t(\widetilde\btheta - \widehat\btheta)),\bY_i)] - \En[m(\bX_i^{\top}\btheta_0,\bY_i)].
$$
From the triangle inequality, we see that
\[
\check u_n - \widetilde u_n \leqslant \| \widehat\btheta + t(\widetilde\btheta - \widehat\btheta) - \btheta_0 \|_2 \leqslant \check u_n + \widetilde u_n.
\]
Convexity of the parameter space (Assumption
\ref{assu:ParameterSpace}) shows that $\widehat\btheta + t(\widetilde\btheta -
\widehat\btheta)\in\Theta$, and
$\suppn{\widetilde\btheta}\subseteq\suppn{\widehat\btheta}$ and the event
$\mathscr K_n$ combine to show that
\[
\|\widehat\btheta + t(\widetilde\btheta - \widehat\btheta)\|_0\leqslant \|\widehat\btheta\|_0 \leqslant \overline k_n.
\]
Deduce that $\widehat\btheta + t(\widetilde\btheta - \widehat\btheta) -
\btheta_0\in\widetilde{\mathcal{R}}(\overline{k}_n)$. It follows from
superadditivity of infima and the definition of the empirical error function
$\widetilde\epsilon_n$ in \eqref{eq: empirical error function for post
estimator} that
\begin{align*}
\overline \lambda_n \overline\Delta_n 
&\geqslant \En[m(\bX_i^\top(\widehat\btheta + t(\widetilde\btheta - \widehat\btheta)),\bY_i)] - \En[m(\bX_i^{\top}\btheta_0,\bY_i)]\\
&\geqslant \inf_{\substack{\bdelta\in\widetilde{\mathcal{R}}(\overline k_n),\\
   \check u_n - \widetilde{u}_{n} \leqslant \|\bdelta\|_{2}\leqslant \check u_n + \widetilde{u}_{n}}}\Big\{ \En[m(\bX_i^\top(\btheta_0 + \bdelta),\bY_i)] - \En[m(\bX_i^{\top}\btheta_0,\bY_i)] \Big\} \\
&\geqslant \inf_{\substack{\bdelta\in\widetilde{\mathcal{R}}(\overline k_n),\\
   \check u_n - \widetilde{u}_{n} \leqslant \|\bdelta\|_{2}\leqslant \check u_n + \widetilde{u}_{n}}} \mathcal E(\btheta_0 + \bdelta) - \widetilde\epsilon_n(\widetilde u_n + \check u_n, \overline k_n) \\
   &\geqslant c_M(\check u_n - \widetilde u_n)^2 - \widetilde a_{\epsilon,n}(\check u_n + \widetilde u_n) - \widetilde b_{\epsilon,n},
\end{align*}
where the final inequality uses Assumption \ref{assu:Margin} and the event
$\mathscr{\widetilde E}_n$. Expanding the square, rearranging terms, and using
$c_M\widetilde u_n^2\geqslant0$, it further follows that
$$
A_n\check u_n^2 - B_n\check u_n - C_n \leqslant 0,
$$
where $A_n:= c_M$, $B_n := 2c_M\widetilde u_n + \widetilde a_{\epsilon,n}$, and
$C_n:=\widetilde a_{\epsilon,n}\widetilde u_n +\widetilde b_{\epsilon,n} +
\overline\lambda_n\overline\Delta_n$. The definition of $\check u_n$ in
\eqref{eq: un check definition} means that $\check u_n = 2B_n / A_n$, so the
displayed inequality can be written as $2B_n^2\leqslant A_nC_n$. On the other
hand, \eqref{eq: constant constraints for post estimator} can be rewritten as
$B_n^2 \geqslant A_n C_n$, yielding the desired contradiction. Conclude that
$\|\widetilde\btheta - \btheta_0\|_2 \leqslant \widetilde u_n + \check u_n$. 
\end{proof}
     
\begin{proof}
[\sc{Proof of Lemma \ref{lem:EmpiricalErrorBound 2}}]
The proof of this lemma is closely related to that of Lemma
\ref{lem:EmpiricalErrorBound}. In particular, as in the case of Lemma
\ref{lem:EmpiricalErrorBound}, the proof will follow from an application of the
maximal inequality in Theorem \ref{lem:MaxIneqBasedOnContraction}. First, fix
any $k\in\mathbb N$, $n\in\N$, $t\in[1,\infty)$ and $u\in(0,\infty)$ satisfying 
\eqref{eq: extra condition to apply contraction principle 2} and denote
$\Delta(u,k):=\widetilde{\mathcal{R}}(k)\cap\left\{ \left\Vert \cdot\right\Vert
_{2}\leqslant u\right\} .$ If $\Delta(u,k)=\emptyset$, then the postulated bound
holds with probability one (interpreting the supremum over an empty set as
$-\infty$). We therefore assume that $\Delta(u,k)\neq\emptyset$.
Lemma \ref{lem:RestrictedSetConsequence 2} shows that
\begin{equation}
\left\Vert \Delta\left(u,k\right)\right\Vert _{1}:=\sup_{\bdelta\in\Delta\left(u,k\right)}\left\Vert \bdelta\right\Vert _{1}\leqslant u\sqrt{k + s_q\eta_n^{-q}} + s_q\eta_n^{1-q} =: \overline\Delta_n(u,k).\label{eq:DeltaEtaOfuEll1Bnd new 2}
\end{equation}
Setting up for an application of Theorem \ref{lem:MaxIneqBasedOnContraction},
define $h:\R\times\mathcal{X}\times\mathcal Y\to\R$ by
$h(t,\bx,\by):=m(\bx^{\top}\btheta_{0}+t,\by)-m(\bx^{\top}\btheta_{0},\by)$ for all
$t\in\R$ and $(\bx,\by)\in\mathcal{X}\times\mathcal Y$. By construction,
$h(0,\cdot,\cdot)\equiv0$.
By Assumption
\ref{assu:LossLocallyLipschitzAndMore}.\ref{enu:LossLocallyLipschitz}, the
restriction $h:[-c_{L},c_{L}]\times\mathcal{X}\times\mathcal Y\to\R$ is
$L(\bx,\by)$-Lipschitz in its first argument, thus verifying Condition
\ref{enu:ContractionConsequenceLocallyLipschitz} of Theorem
\ref{lem:MaxIneqBasedOnContraction} with $C_h=c_L$. 
H\"{o}lder's inequality, Assumption
\ref{assu:LossLocallyLipschitzAndMore}.\ref{enu:LossLocallyLipschitz},
(\ref{eq:DeltaEtaOfuEll1Bnd new 2}), and \eqref{eq: extra condition to apply
contraction principle 2} imply that
\[
\max_{1\leqslant i\leqslant n}\sup_{\bdelta\in\Delta\left(u,k\right)}\left|\bX_{i}^{\top}\bdelta\right|
   \leqslant\max_{1\leqslant i\leqslant n}\left\Vert \bX_{i}\right\Vert _{\infty}\left\Vert \Delta\left(u,k\right)\right\Vert _{1}
   \leqslant t n^{1/r} B_{n}\overline{\Delta}_{n}\left(u,k\right)\leqslant c_{L}
\]
with probability at least $1-t^{-r}$, where the bound $\P(\max_{1\leqslant
i\leqslant n}\|\bX_i\|_{\infty}>t n^{1/r} B_n)\leqslant t^{-r}$ follows from
Markov's inequality, since Assumption
\ref{assu:LossLocallyLipschitzAndMore}.\ref{enu:LossLocallyLipschitz} implies
that $\E[\|\bX\|_{\infty}^r]\leqslant B_n^r$. Condition
\ref{enu:ContractionConsequenceXprimeDeltaBounded} of Theorem
\ref{lem:MaxIneqBasedOnContraction} therefore holds with $C_h=c_L$ and
$\zeta_n=t^{-r}$. Further, swapping $\Delta(u,\eta_n)$
for the $\Delta(u,k)$ in the proof of Lemma \ref{lem:EmpiricalErrorBound},
we verify Conditions
\ref{enu:ContractionConsequencehBoundedInL2} and
\ref{enu:ContractionConsequenceLtimesXBoundedInEmpL2} of Theorem
\ref{lem:MaxIneqBasedOnContraction} in exactly the same way,
and with the same constants as those appearing in the proof of Lemma
\ref{lem:EmpiricalErrorBound}. 
Therefore, Theorem
\ref{lem:MaxIneqBasedOnContraction} combined with the bound
$\overline{\Delta}_{n}(u,k)$ on $\|\Delta(u,k)\|_{1}$ from
(\ref{eq:DeltaEtaOfuEll1Bnd new 2}) and $\ln(8pn)\leqslant4\ln(pn)$ (which
follows from $p\geqslant2$) now show that
$$
   \mathrm{P}\left(\sqrt{n}\widetilde{\epsilon}_{n}\left(u,k\right)>\left\{ 4C_{L}u\right\} \lor\left\{ 16\sqrt{2}CC_{L}\overline{\Delta}_{n}\left(u,k\right)\sqrt{\ln\left(pn\right)}\right\} \right)
   \leqslant 4t^{-r} + 4C/\ln(pn)^2 + n^{-1}
$$
for some universal constant $C\in[1,\infty)$. Now, given that
$k,s_{q},C\in[1,\infty)$, $p\in[2,\infty)$, $n\in[3,\infty)$, and
$\eta_n\in(0,1]$, it follows that $4\sqrt{2}C\sqrt{\ln(pn)}\geqslant1$ and
$\overline{\Delta}_n(u,k)\geqslant u$. Hence,
$16\sqrt{2}CC_{L}\overline{\Delta}_{n}\left(u,k\right)\sqrt{\ln\left(pn\right)}\geqslant4C_{L}u.$
Using this bound in the previous display and redefining the
universal constant $C$ appropriately, we arrive at the asserted claim.
\end{proof}

\begin{proof}
[\sc{Proof of Lemma \ref{lem: sparsity}}] The proof of this lemma adapts 
arguments developed by \cite{belloni_sparse_2012} for post-LASSO
to our setting with a more general loss function. For all $k\in\mathbb N$, let
$$
\mathcal S^p_{k}:=\left\{\bdelta\in\mathbb R^p;
\|\bdelta\|_2\leqslant 1,\|\bdelta\|_0\leqslant k\right\}
\quad \text{and} \quad
\widehat\phi(k):=\sup_{\bdelta\in\mathcal S^p_k}\En[(\bX^{\top}_i\bdelta)^2].
$$
We proceed in five steps.

\medskip
\textbf{Step 1: } Let $a_{1,n}$ and $a_{2,n}$ be non-random sequences in
$(0,\infty)$ for which $n^{1/r}B_na_{1,n}\to0$, and let
$$
\Delta_n:=\left\{\bdelta\in\mathbb R^p; \ \|\bdelta\|_1\leqslant a_{1,n}, \ \|\bdelta\|_2\leqslant a_{2,n}\right\}.
$$
In this step, we show that there is a constant $C\in[1,\infty)$, depending only on
$C_{ev}$ and $C_L$, such that, with probability $1-o(1)$,
$$
\sup_{\bdelta\in\Delta_n}\En[(\bX_i^\top\bdelta)^2] \leqslant C(a_{2,n}^2+a_{1,n}\eta_n).
$$ To do so, we
set up for an application of Theorem \ref{lem:MaxIneqBasedOnContraction} with
$h(t,\bx,\by) = t^2$ and
$\Delta=\Delta_n$. By construction, we have
$\|\Delta_n\|_1:=\sup_{\bdelta\in\Delta_n}\|\bdelta\|_1\leqslant a_{1,n}$. Since
$h(t_1,\bx,\by)-h(t_2,\bx,\by)=(t_1+t_2)(t_1-t_2)$, Condition
\ref{enu:ContractionConsequenceLocallyLipschitz} of Theorem
\ref{lem:MaxIneqBasedOnContraction} holds with $C_h=1/2$ and
$L(\cdot,\cdot)\equiv1$. Condition
\ref{enu:ContractionConsequenceXprimeDeltaBounded} of Theorem
\ref{lem:MaxIneqBasedOnContraction} holds for some $\zeta_n=o(1)$, since
$$
\max_{1\leqslant i\leqslant n}\sup_{\bdelta\in\Delta_n}|\bX_i^{\top}\bdelta|
   \leqslant \max_{1\leqslant i\leqslant n}\|\bX_i\|_{\infty}\|\Delta_n\|_1
   \lesssim_{\P}n^{1/r}B_na_{1,n}\to0
$$
which can be argued from H{\"o}lder's inequality, Assumption
\ref{assu:LossLocallyLipschitzAndMore}.\ref{enu:ContractionConsequenceLocallyLipschitz}
and Markov's inequality, and the $\ell_1$-restriction in $\Delta_n$. 
Condition \ref{enu:ContractionConsequencehBoundedInL2}
of Theorem \ref{lem:MaxIneqBasedOnContraction} holds with
$B_{1n}=\sqrt{C_{ev}}a_{2,n}^2$ by Assumption \ref{as: post estimator moments}. To
verify Condition \ref{enu:ContractionConsequenceLtimesXBoundedInEmpL2} of
Theorem \ref{lem:MaxIneqBasedOnContraction}, we set up for an application of
Theorem \ref{thm: max inequality nonnegative} with
$Z_{i,j}=X_{i,j}^2\geqslant0$. Assumption
\ref{assu:LossLocallyLipschitzAndMore}.\ref{enu:LossLocallyLipschitz} shows that
Conditions \ref{enu:MaxIneqIIIFirstMomentBnd} and
\ref{enu:MaxIneqIIIMomentOfMaxBnd} of Theorem \ref{thm: max inequality
nonnegative} hold with $\mu_n = C_L^2$, $q=2$, and $M_n=B_n^2$. Condition
\ref{enu:MaxIneqIIIGrowthCond} of Theorem \ref{thm: max inequality
nonnegative} then translates to $B_n^2\ln(pn)/\sqrt{n}\leqslant C_L^2$, which
holds eventually, cf.~\eqref{eq: sparsity growth conditions}. Theorem 
\ref{thm: max inequality nonnegative} therefore implies that
$
\max_{1\leqslant j\leqslant p}\En[X_{i,j}^2]\leqslant (\widetilde CC_L)^2
$
with probability $1-o(1)$ for some universal constant $\widetilde
C\in(0,\infty)$. It follows that Condition
\ref{enu:ContractionConsequenceLtimesXBoundedInEmpL2} of Theorem
\ref{lem:MaxIneqBasedOnContraction} holds for some $\gamma_n=o(1)$ and $B_{2n} =
\widetilde CC_L$. Applying Theorem \ref{lem:MaxIneqBasedOnContraction}, it
follows that for some universal constant $C\in[1,\infty)$, with probability $1-o(1)$,
\begin{equation}\label{eq: empirical process prediction norm}
\sup_{\bdelta\in\Delta_n}\left|\En[(\bX_i^{\top}\bdelta)^2] - \E[(\bX^\top \delta)^2] \right| \leqslant C(\sqrt{C_{ev}}a_{2,n}^2/\sqrt n + C_La_{1,n}\eta_n).
\end{equation}
In addition, from Assumption \ref{as: post estimator moments}, we get
$$
\sup_{\bdelta\in\Delta_n}\E[(\bX^\top \bdelta)^2] \leqslant \sup_{\bdelta\in\Delta_n}(\E[(\bX^{\top}\bdelta)^4])^{1/2}\leqslant \sqrt{C_{ev}}a_{2,n}^2.
$$
Combining these bounds, the asserted claim of this step follows from the
triangle inequality.

\medskip
\textbf{Step 2:} Let $k_n$ be a non-random sequence in $\mathbb N$ satisfying
$k_n n^{1/r}B_n\eta_n\to0$, whose existence follows from \eqref{eq: sparsity growth conditions}. In this step, we show that for such a sequence, with
probability $1 - o(1)$,
$$
\widehat\phi(k_n) \leqslant \sqrt{C_{ev}} + o(1).
$$
To this end, let $t_n$ be a non-random sequence in $(0,\infty)$ such that
$\sqrt{k_n} n^{1/r}B_n/t_n\to0$ and $t_n\sqrt{k_n}\eta_n\to0$, whose existence follows from $k_n n^{1/r}B_n\eta_n\to0$.
Since $\|\bdelta\|_1\leqslant \sqrt{k_n}\|\bdelta\|_2$ for any
$\bdelta\in\mathcal S^p_{k_n}$ (per Cauchy-Schwarz inequality), it follows that
the set $\mathcal S^p_{k_n}/t_n$ is contained in $\Delta_n$ introduced in Step 1
provided we set $a_{1,n}=\sqrt{k_n}/t_n$ and $a_{2,n}=1/t_n$. In this case, the
premise $n^{1/r}B_na_{1,n}\to0$ of Step 1 is satisfied, and it follows as in
\eqref{eq: empirical process prediction norm} that
$$
\sup_{\bdelta\in\mathcal S^p_{k_n}/t_n}\left| \En[ (\bX_i^{\top}\bdelta)^2] - \E[(\bX^\top \bdelta)^2] \right| \leqslant C\left(\sqrt{C_{ev}}/(t_n^2\sqrt n) + C_L\sqrt{k_n}\eta_n/t_n\right)
$$
with probability $1-o(1)$ for some universal constant $C\in[1,\infty)$. Hence,
with the same constant $C$,
$$
\sup_{\bdelta\in\mathcal S^p_{k_n}}\left| \En[(\bX_i^{\top}\bdelta)^2] - \E[(\bX^\top \bdelta)^2] \right| \leqslant C\left(C_{ev}/\sqrt n + C_Lt_n\sqrt{k_n}\eta_n\right)\to0
$$
with probability $1 - o(1)$. In addition, using the Cauchy-Schwarz inequality and Assumption \ref{as: post estimator moments}, we get
$$
\sup_{\bdelta\in\mathcal S^p_{k_n}}\E[(\bX^{\top}\bdelta)^2] \leqslant \sup_{\bdelta\in\mathcal S^p_{k_n}}(\E[(\bX^{\top}\bdelta)^4])^{1/2} \leqslant \sqrt{C_{ev}}.
$$
Combining these bounds gives the asserted claim of this step.

\medskip
\textbf{Step 3:} In this step, we show that there exists a constant $\widetilde
C\in(0,\infty)$, depending only on $c_0$, $C_{ev}$, $C_L$, $c_M$ and $C_m$, such
that
\begin{equation}\label{eq: sparsity bound original}
\|\widehat\btheta\|_0 \leqslant \widetilde C\widehat\phi(\|\widehat\btheta\|_0) s_q\eta_n^{-q}\phi_n^2\;\text{for all}\;\widehat\btheta \in \widehat{\Theta}(\lambda),
\end{equation}
with probability at least $1 -
2\P(\lambda<c_{0}\|\bS_n\|_{\infty})-\mathrm{P}(\lambda>\overline{\lambda}_{n})
- \P(\lambda<\underline\lambda_n) - o(1)$. To do so, fix any $\widehat\btheta =
(\widehat\theta_1,\dots,\widehat\theta_p)^{\top}\in\widehat\Theta(\lambda)$ and
observe that the first-order conditions for the optimization problem
\eqref{eq:ell1PenalizedMEstimationIntro} imply that
$$
\En[m_1'(\bX_i^{\top}\widehat\btheta,\bY_i)X_{i,j}] + \lambda Z_j = 0
$$
for all $j\in[p]$ and some $(Z_1,\dots,Z_p)^{\top}\in[-1,1]^p$ satisfying
$Z_j=1$ if $\widehat\theta_j>0$ and $Z_j=-1$ if $\widehat\theta_j < 0$.
Therefore, denoting $\widehat T:=\suppn{\widehat\btheta}$, it follows that for
all $j\in\widehat T$,
$$
|\En[m_1'(\bX_i^{\top}\widehat\btheta,\bY_i)X_{i,j}]|^2 = \lambda^2,
$$
and so
$$
\lambda^2|\widehat T| = \sum_{j\in\widehat T}|\En[m_1'(\bX_i^{\top}\widehat\btheta,\bY_i)X_{i,j}]|^2.
$$
Hence, by the triangle inequality,
\begin{align}
\lambda\sqrt{|\widehat T|} 
& \leqslant \left( \sum_{j\in\widehat T}|\En[m_1'(\bX_i^{\top}\btheta_0,\bY_i)X_{i,j}]|^2 \right)^{1/2} \nonumber\\
& \quad + \left(\sum_{j\in\widehat T}|\En[(m_1'(\bX_i^{\top}\widehat\btheta,\bY_i) - m_1'(\bX_i^{\top}\btheta_0,\bY_i))X_{i,j}]|^2\right)^{1/2}.\label{eq: pre-sparsity KKT consequence}
\end{align}
The \emph{first} term on the right-hand side is bounded from above by
$\sqrt{|\widehat T|}\|\bS_n\|_{\infty} \leqslant \sqrt{|\widehat T|}\lambda/c_0$
with probability at least $1 - \P(\lambda < c_0\|\bS_n\|_{\infty})$. 
By the dual norm inequality, the \emph{second}
term is bounded from above by
\begin{align*}
& \sup_{\bdelta\in\mathcal S^p_{|\widehat T|}} \En[(m_1'(\bX_i^{\top}\widehat\btheta,\bY_i) - m_1'(\bX_i^{\top}\btheta_0,\bY_i))\bX_{i}^{\top}\bdelta] \\
&\qquad \leqslant \sup_{\bdelta\in\mathcal S^p_{|\widehat T|}} \left(\En[(m_1'(\bX_i^{\top}\widehat\btheta,\bY_i) - m_1'(\bX_i^{\top}\btheta_0,\bY_i))^2]\right)^{1/2}\left(\En[(\bX_i^{\top}\bdelta)^2]\right)^{1/2} \\
&\qquad \leqslant C_m\sqrt{\widehat\phi(|\widehat T|)}\left( \En[(\bX_i^{\top}\widehat\btheta - \bX_i^{\top}\btheta_0)^2] \right)^{1/2},
\end{align*}
where the first inequality follows from the Cauchy-Schwarz inequality, and the
second from the definition of $\widehat\phi(|\widehat T|)$ and Assumption
\ref{as: post estimator smoothness}. To control the right-hand side average,
observe that condition \eqref{eq: sparsity growth conditions} allows us to pick
a $t=t_n\to\infty$ so as to eventually satisfy all side conditions of Theorem
\ref{thm:NonAsymptoticProbabilisticBounds}. Using Theorem
\ref{thm:NonAsymptoticProbabilisticBounds}, a calculation then shows that
\begin{align*}
   \|\widehat\btheta-\btheta_0\|_2
   \leqslant C_2 \sqrt{s_q\eta_n^{-q}}\left(\eta_n+\overline\lambda_n\right)\quad\text{and}\quad
   \|\widehat\btheta-\btheta_0\|_1
   \leqslant C_1 s_q\eta_n^{-q}\left(\eta_n + \overline\lambda_n\right),
\end{align*}
with probability at least $1 -
\P(\lambda<c_{0}\|\bS_n\|_{\infty})-\mathrm{P}(\lambda>\overline{\lambda}_{n}) -
o(1)$ for some constants $C_1,C_2\in(0,\infty)$, depending only on $c_0$, $C_L$
and $c_M$. Call the two right-hand sides of the previous display $a_{1,n}$ and
$a_{2,n}$, respectively, and consider $\Delta_n$ as defined in Step 1, but now
based on these specific sequences. From \eqref{eq: sparsity growth conditions}
we see that $n^{1/r}B_na_{1,n}\to0$, so Step 1 implies
\begin{align*}
\En[|\bX_i^{\top}(\widehat\btheta-\btheta_0)|^2]
\leqslant \sup_{\widehat\bdelta\in\Delta_n}\En[(\bX_i^{\top}\bdelta)^2]
\leqslant C(a_{2,n}^2+a_{1,n}\eta_n)
\leqslant C' s_q\eta_n^{-q}(\eta_n^2+\overline\lambda_n^2),
\end{align*}
with probability at least $1 -
\P(\lambda<c_{0}\|\bS_n\|_{\infty})-\mathrm{P}(\lambda>\overline{\lambda}_{n}) -
o(1)$, where $C$ is the constant from Step 1, and the constant $C'$ depends
only on $c_0$, $C_{ev}$, $C_L$ and $c_M$.

Continuing the inequality \eqref{eq: pre-sparsity KKT consequence} with the
upper bounds thus developed, we see that
$$
\lambda\sqrt{|\widehat T|} \leqslant \lambda\sqrt{|\widehat T|}/c_0 + \widetilde C\sqrt{\widehat\phi(|\widehat T|) s_q\eta_n^{-q}(\eta_n^2 + \overline\lambda_n^2)}
$$
with probability at least $1 -
2\P(\lambda<c_{0}\|\bS_n\|_{\infty})-\mathrm{P}(\lambda>\overline{\lambda}_{n})
- o(1)$, where $\widetilde C$ is a constant depending only on $c_0$, $C_{ev}$,
$C_L$, $c_M$, and $C_m$. Observe that all probabilistic bounds in this step hold
simultaneously for all $\widehat\btheta\in\widehat\Theta(\lambda)$. Noting that
$c_0>1$ and $|\widehat T| = \|\widehat\btheta\|_0$, the desired inequality
follows from rearranging the previous display, bounding $\lambda$ from below by
$\underline\lambda_n$ and recasting the constant $\tilde C\in(0,\infty)$.

\medskip
\textbf{Step 4:} For the same constant $\widetilde C$ as that in the Step 3, let
\begin{equation}\label{eq: knbar def}
\widetilde k_n:= \min\left\{k\in\N; \ k>2\widetilde C\widehat\phi(k)s_q\eta_n^{-q}\phi_n^2\right\}.
\end{equation}
Such a $\widetilde{k}_n$ exists in $\N$, since
$\widehat\phi(k)=\widehat\phi(p)$ for $k\geqslant p$. In this step, we show
that
\begin{equation}\label{eq: second sparsity bound}
\sup_{\widehat\btheta\in\widehat\Theta(\lambda)}\|\widehat\btheta\|_0 \leqslant \widetilde k_n
\end{equation}
with probability at least $1 -
2\P(\lambda<c_{0}\|\bS_n\|_{\infty})-\mathrm{P}(\lambda>\overline{\lambda}_{n})
- \P(\lambda<\underline\lambda_n) - o(1)$. To this end, let $\widetilde s_n: =
s_q\eta_n^{-q}\phi_n^2$, and let the event \eqref{eq: sparsity bound original}
hold. Seeking a contradiction, note that if $\widetilde
k_n<\|\widehat\btheta\|_0$ for some $\widehat\btheta\in\widehat\Theta(\lambda)$,
then for such $\widehat\btheta$,
\begin{align*}
\|\widehat\btheta\|_0 & \leqslant \widetilde C\widehat\phi(\|\widehat\btheta\|_0) \widetilde s_n = \widetilde C\widehat\phi((\|\widehat\btheta\|_0/\widetilde k_n)\widetilde k_n) \widetilde s_n \\
& \leqslant \widetilde C\widehat\phi(\lceil\|\widehat\btheta\|_0/\widetilde k_n\rceil\widetilde k_n) \widetilde s_n \leqslant \widetilde C\lceil\|\widehat\btheta\|_0/\widetilde k_n\rceil \widehat\phi(\widetilde k_n) \widetilde s_n  \leqslant 2\widetilde C(\|\widehat\btheta\|_0/\widetilde k_n) \widehat\phi(\widetilde k_n) \widetilde s_n,
\end{align*}
where the second inequality in the second line follows from Lemma 9 in
\cite{belloni_sparse_2012}. However, this chain of inequalities implies that
$$
\widetilde k_n \leqslant 2\widetilde C\widehat\phi(\widetilde k_n)\widetilde s_n,
$$
which contradicts the definition of $\widetilde k_n$ in \eqref{eq: knbar def}.
Hence, on the event \eqref{eq: sparsity bound original}, we have \eqref{eq:
second sparsity bound}. Combining this result with the previous step yields the
asserted claim of this step.

\medskip
\textbf{Step 5:} We now complete the proof. To this end, for the constant
$\widetilde C$ from Step 3, let
$$
k_n:= \lceil4\widetilde CC_{ev}s_q\eta_n^{-q}\phi_n^2\rceil.
$$
By \eqref{eq: sparsity growth conditions}, we have $k_n n^{1/r}B_n\eta_n\to0$,
and so it follows from Step 2 that
$$
k_n > 2\widetilde C\widehat\phi(k_n)s_q\eta_n^{-q}\phi_n^2
$$
with probability $1 - o(1)$ since $C_{ev}\geqslant 1$. Hence, $\widetilde k_n$
defined in \eqref{eq: knbar def} satisfies $\widetilde k_n \leqslant k_n$ with
probability $1 - o(1)$. Combining this result with Step 4 yields the asserted
probabilistic claim of the lemma.
\end{proof}

\section{Fundamental Tools}\label{ap: fundamental tools}

\subsection{Maximal Inequalities}

Let $\Gn[f(\bZ_{i})]:=\sqrt{n}(\En-\E)[f(\bZ_{i})]$ abbreviate the centered and
scaled empirical average of the $\{f(\bZ_i)\}_{i=1}^n$.
\begin{thm}
[\textbf{Maximal Inequality, I}]\label{lem:MaxIneqBasedOnContraction} Let
$\left\{ \bZ_{i}\right\} _{i=1}^{n}$ be independent copies of a random vector
$\bZ$, with support $\mathcal{\bZ}$, of which $\bX$ is a $p$-dimensional
subvector, let $\Delta$ be a non-empty subset of $\R^{p}$, and let
$h:\R\times\mathcal{Z}\to\R$ be a measurable function satisfying
$h\left(0,\cdot\right)\equiv0$. Suppose that there are non-random sequences
$B_{1n},B_{2n}\in[0,\infty),\zeta_n,\gamma_{n}\in\left(0,1\right)$, a constant
$C_{h}\in(0,1)$, and a measurable function $L:\mathcal{Z}\to[0,\infty)$ such
that
\begin{enumerate}
\item \label{enu:ContractionConsequenceLocallyLipschitz}for all
$\bz\in\mathcal{Z}$ and all $t_{1},t_{2}\in\R$ satisfying
$\left|t_{1}\right|\lor\left|t_{2}\right|\leqslant C_{h},$
\begin{align*}
\left|h\left(t_{1},\bz\right)-h\left(t_{2},\bz\right)\right| & \leqslant L\left(\bz\right)\left|t_{1}-t_{2}\right|;
\end{align*}
\item \label{enu:ContractionConsequenceXprimeDeltaBounded}$\max_{1\leqslant
i\leqslant n}\sup_{\bdelta\in\Delta}\left|\bX_i^{\top}\bdelta\right|\leqslant C_{h}$
with probability at least $1 - \zeta_n$;
\item
\label{enu:ContractionConsequencehBoundedInL2}$\sup_{\bdelta\in\Delta}\E\bracn{h(\bX^{\top}\bdelta,\bZ)^{2}}\leqslant
B_{1n}^{2}$; and,
\item \label{enu:ContractionConsequenceLtimesXBoundedInEmpL2}$\max_{1\leqslant
j\leqslant p}\En\bracn{L\left(\bZ_{i}\right)^{2}X_{i,j}^{2}}\leqslant B_{2n}^{2}$
with probability at least $1-\gamma_{n}$.
\end{enumerate}
Then, denoting $\norm{\Delta}_{1}:=\sup_{\bdelta\in\Delta}\norm{\bdelta}_{1},$ we
have
\begin{align*}
\P\left(\sup_{\bdelta\in\Delta}\left|\Gn\bracn{h\paran{\bX_{i}^{\top}\bdelta,\bZ_{i}}}\right|>u\right) & \leqslant4\zeta_n + 4\gamma_{n}+n^{-1},
\end{align*}
provided $u\geqslant\left\{ 4B_{1n}\right\}
\lor\bracen{8\sqrt{2}B_{2n}\norm{\Delta}_{1}\sqrt{\ln\left(8pn\right)}}.$
\end{thm}
\begin{proof}
The claim follows from the proof of \citet[Lemma
D.3]{belloni2018highdimensional}, where in Step 1 we replace the set
$\Omega$ by the intersection of $\Omega$ and $\{\max_{1\leqslant i\leqslant
n}\sup_{\bdelta\in\Delta}\left|\bX_i^{\top}\bdelta\right|\leqslant C_{h}\}$.
\end{proof}

\begin{thm}[\textbf{Maximal Inequality, II}]\label{thm: max inequality standard}
Let $\{\bZ_i\}_{i=1}^n$ be independent copies of a random vector $\bZ =
(Z_1,\dots,Z_p)^{\top}$ in $\R^p$ with $p\geqslant2$, and assume that for
non-random sequences $M_n,\sigma_n\in(0,\infty)$ and a constant $q\in(2,\infty)$,
we have
\begin{inparaenum}[(1)]
\item\label{enu:MaxIneqIISecondMomentBnd} $\max_{1\leqslant j\leqslant p}\E[Z_{j}^2]\leqslant \sigma_n^2$,
\item\label{enu:MaxIneqIIMomentOfMaxBnd} $\E[\|\bZ\|_{\infty}^q]\leqslant M_n^q$, and 
\item\label{enu:MaxIneqIIGrowthCond} $M_n\leqslant \sigma_n n^{1/2-1/q}/\sqrt{\ln(pn)}$.
\end{inparaenum}
Then
$$
\P\left(\max_{1\leqslant j\leqslant p}\left| \Gn(Z_{i,j}) \right| > C\sigma_n\sqrt{\ln(pn)}\right) \leqslant \frac{c_q}{\ln^q(p n)},
$$
where $C\in(0,\infty)$ is a universal constant and $c_q\in(0,\infty)$ is a
constant depending only on $q$.
\end{thm}

\begin{proof}
\citet[][Lemma A.3]{belloni2018highdimensional} implies that, for a universal
constant $K\in(0,\infty)$,
\[
\E\left[\max_{1\leqslant j\leqslant p}\left|\Gn(Z_{i,j})\right|\right]
\leqslant K\left(\sqrt{\max_{1\leqslant j\leqslant p}\E[Z_{j}^2]\ln p}+\sqrt{\frac{\E\left[\max_{1\leqslant i\leqslant n}\|\bZ_i\|^2_\infty\right]}{n}}\ln p\right).
\]
By Jensen's inequality, Condition \ref{enu:MaxIneqIIMomentOfMaxBnd} shows
\[
   \E\left[\max_{1\leqslant i\leqslant n}\|\bZ_i\|^2_\infty\right]
   \leqslant\left(\E\left[\max_{1\leqslant i\leqslant n}\|\bZ_i\|^q_\infty\right]\right)^{2/q}
   \leqslant\left(\E\left[\sum_{i=1}^n\|\bZ_i\|^q_\infty\right]\right)^{2/q}
   \leqslant n^{2/q}M_n^2.
\]
Using also Conditions \ref{enu:MaxIneqIISecondMomentBnd} and \ref{enu:MaxIneqIIGrowthCond}, we thus arrive at
\[
   \E\left[\max_{1\leqslant j\leqslant p}\left|\Gn(Z_{i,j})\right|\right]
   \leqslant K\left(\sigma_n\sqrt{\ln p}
      +n^{1/q-1/2}M_n\ln p\right)
   \leqslant2K\sigma_n\sqrt{\ln p}.
\]
Using \citet[][Lemma A.2]{belloni2018highdimensional} with $s=q$ and $t=\sigma_n\sqrt{3n\ln(pn)}$,
we see that
\[
\P\left(\max_{1\leqslant j\leqslant p}\left|\Gn(Z_{i,j})\right|>(4K+\sqrt{3})\sigma_n\sqrt{\ln(pn)}\right)
\leqslant\frac{1}{pn}+\frac{C_q n \E\left[\|\bZ-\E[\bZ]\|_\infty^q\right]}{\sigma_n^q [n\ln(pn)]^{q/2}},
\]
where $C_q\in(0,\infty)$ is a constant depending only on $q$. The triangle and
Jensen inequalities yield
\[
   \E\left[\|\bZ-\E[\bZ]\|_\infty^q\right]
   \leqslant\E\left[(\|\bZ\|_\infty+\|\E[\bZ]\|_\infty)^q\right]
   \leqslant\E\left[(\|\bZ\|_\infty+(\E[\|\bZ\|_\infty^q])^{1/q})^q\right].
\]
Applying the inequality $(a+b)^q\leqslant 2^{q-1}(a^q+b^q)$, which is valid for
any $a,b\in[0,\infty)$ and any $q\in[1,\infty)$,
using Condition \ref{enu:MaxIneqIIMomentOfMaxBnd} we see that 
\[
   \E\left[(\|\bZ\|_\infty+\E[\|\bZ\|_\infty])^q\right]
   \leqslant\E\left[2^{q-1}\left(\|\bZ\|_\infty^q+\E[\|\bZ\|_\infty^q]\right)\right]
   =2^q\E[\|\bZ\|_\infty^q]
   \leqslant2^qM_n^q.
\]
Gathering these bounds, and using Condition \ref{enu:MaxIneqIIGrowthCond}, we
arrive at
\[
\P\left(\max_{1\leqslant j\leqslant p}\left|\Gn(Z_{i,j})\right|>(4K+\sqrt{3})\sigma_n\sqrt{\ln(pn)}\right)
\leqslant\frac{1}{pn}+\frac{2^qC_q}{\ln^q(pn)}.
\]
Since the polynomial $x^{1/q}$ dominates the logarithm $\ln x$ as $x\to\infty$,
there is a constant $K_q\in(0,\infty)$ depending only on $q$ such that $\ln
x\leqslant K_q x^{1/q}$ for all $x\in[1,\infty)$. The claim now follows from
specifying $C=4K+\sqrt{3}$ and $c_q=K^q_q+2^qC_q$.
\end{proof}

\begin{thm}
[\textbf{Maximal Inequality, III}]\label{thm: max inequality nonnegative} Let
$\{\bZ_i\}_{i=1}^n$ be independent copies of a random vector $\bZ =
(Z_1,\dots,Z_p)^{\top}$ in $\R^p$  with $p\geqslant2$ such that $Z_j\geqslant 0$
for all $j\in[p]$ and assume that for non-random sequences
$M_n,\mu_n\in(0,\infty)$ and a constant $q\in(1,\infty)$, we have 
\begin{inparaenum}[(1)]
\item\label{enu:MaxIneqIIIFirstMomentBnd} $\max_{1\leqslant j\leqslant p}\E[Z_{j}]\leqslant \mu_n$,
\item\label{enu:MaxIneqIIIMomentOfMaxBnd} $\E[\|\bZ\|_{\infty}^q]\leqslant M_n^q$, and 
\item\label{enu:MaxIneqIIIGrowthCond} $M_n\leqslant \mu_n n^{1-1/q}/\ln(pn)$.
\end{inparaenum}
Then
$$
\P\left(\max_{1\leqslant j\leqslant p}\En\left[Z_{i,j}\right] > C\mu_n\right) \leqslant \frac{c_q}{\ln^q(p n)},
$$
where $C\in(0,\infty)$ is a universal constant and $c_q\in(0,\infty)$ is a
constant depending only on $q$.
\end{thm}

\begin{proof}
\citet[][Lemma A.5]{belloni2018highdimensional} shows that, for a universal
constant $K\in(0,\infty)$,
\[
\E\left[\max_{1\leqslant j\leqslant p}\En\left[Z_{i,j}\right]\right]
\leqslant K\left(\max_{1\leqslant j\leqslant p}\E[Z_j]+\frac{\E\left[\max_{1\leqslant i\leqslant n}\|\bZ_i\|_\infty\right]\ln p}{n}\right).
\]
Jensen's inequality and Condition \ref{enu:MaxIneqIIIMomentOfMaxBnd} imply
\[
   \E\left[\max_{1\leqslant i\leqslant n}\|\bZ_i\|_\infty\right]
   \leqslant\left(\E\left[\max_{1\leqslant i\leqslant n}\|\bZ_i\|^q_\infty\right]\right)^{1/q}   
   \leqslant\left(\E\left[\sum_{i=1}^n\|\bZ_i\|^q_\infty\right]\right)^{1/q}
   \leqslant n^{1/q}M_n.
\]
Using also Conditions \ref{enu:MaxIneqIIIFirstMomentBnd} and \ref{enu:MaxIneqIIIGrowthCond}, we thus arrive at
\[
\E\left[\max_{1\leqslant j\leqslant p}\En\left[Z_{i,j}\right]\right]
\leqslant K(\mu_n+n^{1/q-1} M_n \ln p)
\leqslant 2K\mu_n.
\]
Using \citet[][Lemma A.4]{belloni2018highdimensional} with $s=q$ and $t=n\mu_n$, we see that
\[
\P\left(\max_{1\leqslant j\leqslant p}\En\left[Z_{i,j}\right]>(4K+1)\mu_n\right)
\leqslant \frac{c_q\E\left[\max_{1\leqslant i\leqslant n}\|\bZ_i\|_\infty^q\right]}{n^q \mu_n^q},
\]
where $c_q\in(0,\infty)$ is a constant depending only on $q$. Using Condition \ref{enu:MaxIneqIIIMomentOfMaxBnd}, we see that
\[
\E\left[\max_{1\leqslant i\leqslant n}\|\bZ_i\|^q_\infty\right]   
   \leqslant\E\left[\sum_{i=1}^n\|\bZ_i\|^q_\infty\right]
   \leqslant nM_n^q,
\]
Hence, by Condition \ref{enu:MaxIneqIIIGrowthCond},
\[
\P\left(\max_{1\leqslant j\leqslant p}\En\left[Z_{i,j}\right]>(4K+1)\mu_n\right)
\leqslant \frac{c_q n^{1-q} M_n^q}{\mu_n^q}
\leqslant \frac{c_q}{\ln^q(p n)},
\]
which gives the asserted claim.
\end{proof}

\subsection{Gaussian Inequality}
\begin{thm}
[\textbf{Gaussian Quantile Bound}]\label{lem:GaussianQuantileBnd}
Let $\left(Y_{1},\dotsc,Y_{p}\right)$ be centered Gaussian in $\R^{p}$
with $\sigma^{2}:=\max_{1\leqslant j\leqslant p}\E\left[Y_{j}^{2}\right]$
and $p\geqslant2$. Let $q^{Y}\left(1-\alpha\right)$ denote the $\left(1-\alpha\right)$-quantile
of $\max_{1\leqslant j\leqslant p}\left|Y_{j}\right|$ for $\alpha\in\left(0,1\right)$.
Then $q^{Y}\left(1-\alpha\right)\leqslant(2+\sqrt{2})\sigma\sqrt{\ln\left(p/\alpha\right)}.$
\end{thm}
\begin{proof}
By the Borell-TIS (Tsirelson-Ibragimov-Sudakov) inequality \citep[Theorem 2.1.1]{adler_random_2007},
for any $t\in(0,\infty)$ we have
\[
\P\Big(\max_{1\leqslant j\leqslant p}\left|Y_{j}\right|>\E\big[\max_{1\leqslant j\leqslant p}\left|Y_{j}\right|\big]+\sigma t\Big)\leqslant\mathrm{e}^{-t^{2}/2}.
\]
This inequality translates to the quantile bound
\[
q^{Y}(1-\alpha)\leqslant\E\big[\max_{1\leqslant j\leqslant p}\left|Y_{j}\right|\big]+\sigma\sqrt{2\ln\left(1/\alpha\right)}.
\]
\citet[Proposition A.3.1]{talagrand_mean_2010} shows that
\[
\E\big[\max_{1\leqslant j\leqslant p}\left|Y_{j}\right|\big]\leqslant\sigma\sqrt{2\ln\left(2p\right)},
\]
thus implying
\[
q^{Y}(1-\alpha)\leqslant\sigma\Big(\sqrt{2\ln\left(2p\right)}+\sqrt{2\ln\left(1/\alpha\right)}\Big).
\]
The claim now follows from $p\geqslant2$.
\end{proof}

\subsection{Central Limit Theorem and Bootstrap in High Dimensions}

Throughout this section we let $\{\bZ_{i}\}_{i=1}^n$ be independent copies of a
centered random vector $\bZ$ in $\R^{p}$ and denote their scaled average and
(common)  variance by
\[
\cS_{n}:=\frac{1}{\sqrt{n}}\sum_{i=1}^{n}\bZ_{i}\quad\text{and}\quad\bSigma:=\E\left[\bZ\bZ^{\top}\right],
\]
respectively. (The existence of $\bSigma$ in $\R^{p\times p}$ is guaranteed by
our assumptions below.) Write $\cN_n$ for a centered $p$-dimensional Gaussian
vector with variance $\bSigma$. For $\R^{p}$-valued random variables
$\bU$ and $\bV$, define the distributional measure of distance
\[
\rho\left(\bU,\bV\right):=\sup_{A\in\mathcal{A}_{p}}\left|\P\left(\bU\in A\right)-\P\left(\bV\in A\right)\right|,
\]
where $\mathcal{A}_{p}$ denotes the collection of all hyperrectangles in $\R^{p}$.
\begin{thm}
[\textbf{High-Dimensional CLT}]\label{thm:HDCLT} If for some constant
$b\in(0,\infty)$ and a non-random sequence $B_{n}$ in $[1,\infty)$,
\begin{equation}
\min_{1\leqslant j\leqslant p}\E\left[Z_{j}^{2}\right]\geqslant b,\quad\max_{k\in\{1,2\}}\max_{1\leqslant j\leqslant p}\E\big[\left|Z_{j}\right|^{2+k}\big]/B_{n}^{k}\leqslant 1\quad\text{and}\quad\E\Big[\max_{1\leqslant j\leqslant p}Z_{j}^{4}\Big]\leqslant B_{n}^{4},\label{eq:MomentConds}
\end{equation}
then there is a constant $C_{b}\in(0,\infty)$, depending only on $b$, such that
\begin{equation}
\rho\left(\cS_{n},\cN_n\right)\leqslant C_{b}\left(\frac{B_{n}^{4}\ln^{7}\left(pn\right)}{n}\right)^{1/6}.\label{eq:HDCLTUpperBnd}
\end{equation}
\end{thm}

\begin{proof}
The claim follows from \citet[Proposition 2.1]{chernozhukov_central_2017} with $q=4$.
\end{proof}

Let $\{\widehat{\bZ}_{i}\}_{i=1}^n$ be random elements of $\R^p$, and let
$\{e_{i}\}_{i=1}^n$ be i.i.d.~standard Gaussians independent of
$\{(\bZ_{i},\widehat{\bZ}_{i})\}_{i=1}^n$. Define
\[
   \widehat{\cS}_{n}^{e}:=\frac{1}{\sqrt n}\sum_{i=1}^{n}e_{i}\widehat{\bZ}_{i},
\]
and let $\P_{e}$ denote the (conditional) probability measure computed with
respect to $\{e_{i}\}_{i=1}^n$ for fixed
$\{(\bZ_{i},\widehat{\bZ}_{i})\}_{i=1}^n$. Also, abbreviate
\[
\widetilde{\rho}(\widehat{\cS}_{n}^{e},\cN_{n}):=\sup_{A\in\mathcal{A}_{p}}\left|\P_{e}\big(\widehat{\cS}_{n}^{e}\in A\big)-\P\left(\cN_n\in A\right)\right|,
\]
with the tilde stressing that $\widetilde{\rho}(\widehat{\cS}_{n}^{e},\cN_{n})$
is a random quantity.
\begin{thm}
[\textbf{Multiplier Bootstrap for Many Approximate Means}]\label{thm:HDMultBootApprox}
Let (\ref{eq:MomentConds}) hold for some constant $b\in(0,\infty)$
and a non-random sequence $B_{n}$ in $[1,\infty)$, and let $\beta_{n}$
and $\delta_{n}$ be non-random sequences in $[0,\infty)$ such that
\begin{equation}\label{eq:L2PnEstErrCond}
\P\left(\max_{1\leqslant j\leqslant p}\En[(\widehat{Z}_{i,j}-Z_{i,j})^{2}]>\frac{\delta_{n}^{2}}{\ln^{2}\left(pn\right)}\right)\leqslant\beta_{n}.
\end{equation}
Then there is a constant $C_{b}\in(0,\infty)$, depending only on $b$,
such that with probability at least $1-\beta_{n}-1/\ln^{2}\left(pn\right)$,
\begin{equation}\label{eq:HDBootstrapUpperBnd}
\widetilde{\rho}(\widehat{\cS}_{n}^{e},\cN_{n})\leqslant C_{b}\left(\delta_{n}\lor\left(\frac{B_{n}^{4}\ln^{6}\left(pn\right)}{n}\right)^{1/6}\right).
\end{equation}
\end{thm}

\begin{proof}
The claim essentially follows from the proof of \citet[Theorem
2.2]{belloni2018highdimensional}. We include the argument for completeness and
to clarify the dependence on $b$.

First, denote $\Delta_n:=\delta_n/\sqrt{\ln(pn)}$ and consider the random element $\cS^e_n$ of
$\R^p$ defined by
\[
   \cS^e_n:=\frac{1}{\sqrt{n}}\sum_{i=1}^{n}e_{i}\bZ_{i}.
\]
Observe that, conditional on $\{(\bZ_i,\widehat{\bZ}_i)\}_{i=1}^n$, the elements
$\{\pm(\widehat{\cS}^e_{n,j}-\cS^e_{n,j})\}_{j=1}^p$ are jointly centered
Gaussian with largest (conditional) variance
\[
   \widehat\sigma_e^2:=\max_{1\leqslant j\leqslant p}\En[(\widehat{Z}_{i,j}-Z_{i,j})^{2}].
\] 
Applying Lemma \ref{lem:GaussianQuantileBnd} conditional on on
$\{(\bZ_i,\widehat{\bZ}_i)\}_{i=1}^n$ and with $\alpha=1/n$ shows that
\[
\P_e\left(\|\widehat\cS^e_n-\cS^e_n\|_\infty>K_1\widehat\sigma_e\sqrt{\ln(pn)}\right)\leqslant n^{-1}
\]
for the absolute constant $K_1:=2+\sqrt 2$. Since \eqref{eq:L2PnEstErrCond}
means that $\widehat\sigma_p^2\leqslant\Delta_n^2/\ln(pn)$ with probability at
least $1-\beta_n$, with the same probability we have
\begin{equation}\label{eq:EventMultBootApproxProof}
\P_e\left(\|\widehat\cS^e_n-\cS^e_n\|_\infty>K_1\Delta_n\right)\leqslant n^{-1}.
\end{equation}

Next, consider any (hyper)rectangle $A\in\mathcal{A}_{p}$. Then there are
$p$-dimensional vectors $\bw^l=(w_{1}^l,\dots,w_{p}^l)^{\top}$ and
$\bw^u=(w_{1}^u,\dots,w_{p}^u)^{\top}$ of (possibly extended) reals, for which
\[
A=\{\bw\in\R^{p};w_{j}^l\leqslant w_{j}\leqslant w_{j}^u\text{ for all }j\in[p]\}.
\]
Based on this representation, define the (expanded) set
\[
 A^+:=\{\bw\in\R^{p};w^l_j-K_1\Delta_n \leqslant w_{j}^l\leqslant w^u_j + K_1\Delta_n\text{ for all }j\in[p]\}.
\]
Then also $A^+\in\mathcal A_p$ and $A\subseteq A^+$. It follows that, on the
event \eqref{eq:EventMultBootApproxProof},
\begin{align*}
   \P_e\big(\widehat\cS^e_n\in A\big)
   &\leqslant \P_e\left(\cS^e_n\in A^+\right) + n^{-1}\\
   &\leqslant \P\left(\cN_{n}\in A^+\right) + \widetilde\rho(\cS^e_n,\cN_n) + n^{-1},
\end{align*}
where $\widetilde\rho(\cS^e_n,\cN_n)$ is the random variable given by
\[
\widetilde\rho(\cS^e_n,\cN_n):=\sup_{A\in\mathcal{A}_{p}}\left|\P_{e}\left(\cS^e_n\in A\right)-\P\left(\cN_n\in A\right)\right|.
\]
Repeating the anti-concentration argument on
\citet[p.~69]{belloni2018highdimensional}, we see that there is a constant
$K_2\in(0,\infty)$, depending only on $b$, such that
\[
   \P\left(\cN_n \in A^+\right)\leqslant\P\left(\cN_n \in A\right)+2K_2\delta_n.
\]
Hence, on the event \eqref{eq:EventMultBootApproxProof}, we have the
\emph{upper} bound
\[
\P_e\big(\widehat\cS^e_n\in A\big)\leqslant\P\left(\cN_n \in A\right)+\widetilde\rho(\cS^e_n,\cN_n)+K_2\delta_n+n^{-1}.
\]

To get a \emph{lower} bound, we instead consider the set
\[
A^-:=\{\bw\in\R^{p};w^l_j+K_1\Delta_n \leqslant w_{j}^l\leqslant w^u_j - K_1\Delta_n\text{ for all }j\in[p]\},
\]
which satisfies $A^-\in\mathcal A_p$ and $A^-\subseteq A$. A parallel argument
shows that on the event \eqref{eq:EventMultBootApproxProof},
\[
\P_e\big(\widehat\cS^e_n\in A\big)\geqslant\P\left(\cN_n \in A\right)-\widetilde\rho(\cS^e_n,\cN_n)-K_2\delta_n-n^{-1}.
\]
Combine the upper and lower bounds and take the supremum over $A\in\mathcal A_p$ to see that on the event
\eqref{eq:EventMultBootApproxProof},
\[
\widetilde\rho(\widehat\cS^e_n,\cN_n)\leqslant\widetilde\rho(\cS^e_n,\cN_n)+K_2\delta_n+n^{-1}.
\]
Condition \eqref{eq:MomentConds} and \citet[Theorem
A.2(ii)]{belloni2018highdimensional} [with their $q=4$ and their
$\beta=1/\ln^2(pn)$, the latter choice being justified by $p\geqslant2$ and
$n\geqslant3$] combine to show that there is a constant $K_3\in(0,\infty)$,
depending only on $b$, such that with probability at least $1-1/\ln^2(pn)$,
\[
   \widetilde\rho(\cS^e_n,\cN_n)\leqslant K_3\left(\frac{B_n^4\ln^6(pn)}{n}\right)^{1/6}.
\]
The previous two displays combine to show that with probability at least
$1-\beta_n-1/\ln^2(pn)$,
\[
\widetilde\rho(\widehat\cS^e_n,\cN_n)\leqslant K_3\left(\frac{B_n^4\ln^6(pn)}{n}\right)^{1/6}+K_2\delta_n+n^{-1}.
\]
The claim now follows from taking $C_b:=K_3+K_2+1$, which depends only on $b$.
\end{proof}

Let $\cN_\bM$ be a $p$-dimensional centered Gaussian vector with variance
matrix $\bM\in\R^{p\times p}$. Define $q_{\bM}^{\cN}:\R\to\R\cup\left\{ \pm\infty\right\}$
as the (extended) quantile function of $\Vert \cN_{\bM}\Vert_{\infty}$, i.e.
\begin{align*}
q_{\bM}^{\cN}\left(\alpha\right) & :=\inf\left\{ t\in\R;\P\left(\Vert \cN_{\bM}\Vert_{\infty}\leqslant t\right)\geqslant\alpha\right\} ,\quad\alpha\in\R.
\end{align*}
We here interpret $q_{\bM}^{\cN}\left(\alpha\right)$ as $+\infty(=\inf\emptyset)$
if $\alpha\geqslant1,$ and $-\infty(=\inf\R)$ if $\alpha\leqslant0$,
such that $q_{\bM}^{\cN}$ is monotone increasing.
\begin{thm}\label{lem:QuantileComparisonLemma} 
Let $\bM\in\R^{p\times p}$ be a symmetric positive semi-definite matrix with strictly
positive diagonal, i.e.~$M_{j,j}\in(0,\infty)$ for all $j\in[p]$, let $\bU$ be an
$\R^{p}$-valued random variable, and let $q$ denote the quantile function of
$\Vert \bU\Vert_{\infty}$. Then
\[
q_{\bM}^{\cN}\left(\alpha-2\rho\left(\bU,\cN_{\bM}\right)\right)\leqslant q\left(\alpha\right)\leqslant q_{\bM}^{\cN}\left(\alpha+\rho\left(\bU,\cN_{\bM}\right)\right)\text{ for all }\alpha\in\left(0,1\right).
\]
\end{thm}
\begin{proof}
If $\rho(\bU,\cN_\bM)=0,$ then the distributions of $\bU$ and $\cN_\bM$ agree on all
hyperrectangles. In particular, these distributions agree on all cubes
$[-t,t]^p,t\in[0,\infty),$ implying equality of quantile functions, as claimed. For
the remainder of the proof, we therefore take $\rho(\bU,\cN_\bM)\in(0,\infty).$ 
Since $\bM$ has a strictly positive diagonal, by the union bound, for any $t\in\R$ we have
\[
\P\left(\Vert \cN_{\bM}\Vert_{\infty}=t\right)\leqslant\sum_{j=1}^{p}\P\left(\left|\mathrm{N}\left(0,1\right)\right|=t/\sqrt{M_{j,j}}\right)=0.
\]
It follows that for each $\alpha\in\left(0,1\right),$
$q_{\bM}^{\cN}\left(\alpha\right)\in\R$ is uniquely defined by
\[
\P\left(\Vert \cN_{\bM}\Vert_{\infty}\leqslant q_{\bM}^{\cN}\left(\alpha\right)\right)=\alpha.
\]
Fix $\alpha\in(0,1)$. In establishing the \emph{lower} bound we may take
$2\rho\left(\bU,\cN_{\bM}\right)<\alpha$. (Otherwise
$q_{\bM}^{\cN}(\alpha-2\rho(\bU,\cN_{\bM}))=-\infty$ and there is nothing to
show.) Then
\[
   \left[-q_{\bM}^{\cN}\left(\alpha-2\rho\left(\bU,\cN_{\bM}\right)\right),q_{\bM}^{\cN}\left(\alpha-2\rho\left(\bU,\cN_{\bM}\right)\right)\right]^{p} 
\]
is a rectangle and, thus,
\begin{align*}
&\P\left(\Vert \bU\Vert_{\infty}\leqslant q_{\bM}^{\cN}\left(\alpha-2\rho\left(\bU,\cN_{\bM}\right)\right)\right)\\
&\leqslant\P\left(\Vert \cN_{\bM}\Vert_{\infty}\leqslant q_{\bM}^{\cN}\left(\alpha-2\rho\left(\bU,\cN_{\bM}\right)\right)\right)+\rho\left(\bU,\cN_{\bM}\right) <\alpha,
\end{align*}
which implies the lower bound. In establishing the \emph{upper} bound we may
assume $\rho(\bU,\cN_{\bM})<1-\alpha.$ (Otherwise
$q_{\bM}^{\cN}\left(\alpha+\rho\left(\bU,\cN_{\bM}\right)\right)=+\infty$ and there is
nothing to show.) Then from the rectangle
\[
   \left[-q_{\bM}^{\cN}\left(\alpha+\rho\left(\bU,\cN_{\bM}\right)\right),q_{\bM}^{\cN}\left(\alpha+\rho\left(\bU,\cN_{\bM}\right)\right)\right]^{p}, 
\]
a similar calculation shows
\begin{align*}
\P\left(\Vert \bU\Vert_{\infty}\leqslant q_{\bM}^{\cN}\left(\alpha+\rho\left(\bU,\cN_{\bM}\right)\right)\right) & \geqslant\alpha,
\end{align*}
which by definition of quantiles implies the upper bound.
\end{proof}

Now, define $q_{n}(\alpha)$ as the $\alpha$-quantile
of $\Vert \cS_{n}\Vert_{\infty}$
\begin{align*}
q_{n}(\alpha) & :=\inf\left\{ t\in\R;\P(\Vert \cS_{n}\Vert_{\infty}\leqslant t)\geqslant\alpha\right\} ,\quad\alpha\in\left(0,1\right),
\end{align*}
and let $\widehat{q}_{n}(\alpha)$ be the $\alpha$-quantile
of $\Vert\widehat{\cS}_{n}^{e}\Vert_{\infty}$ computed conditional
on $\{(\bZ_{i},\widehat{\bZ}_{i})\}_{i=1}^n$, i.e.
\[
\widehat{q}_{n}(\alpha):=\inf\left\{ t\in\R;\P_{e}(\Vert\widehat{\cS}_{n}^{e}\Vert_{\infty}\leqslant t)\geqslant\alpha\right\} ,\quad\alpha\in\left(0,1\right).
\]

\begin{thm}
[\textbf{Quantile Comparison}]\label{thm:QuantileComparison} If
(\ref{eq:MomentConds}) holds for some constant $b\in(0,\infty)$ and
a non-random sequence $B_{n}$ in $[1,\infty)$, and
\[
\rho_{n}:=2C_{b}\left(\frac{B_{n}^{4}\ln^{7}\left(pn\right)}{n}\right)^{1/6}
\]
denotes two times the upper bound (\ref{eq:HDCLTUpperBnd}) in Theorem \ref{thm:HDCLT}, then
\[
q_{\bSigma}^{\cN}\left(1-\alpha-\rho_{n}\right)\leqslant q_{n}\left(1-\alpha\right)\leqslant q_{\bSigma}^{\cN}\left(1-\alpha+\rho_{n}\right)\text{ for all }\alpha\in\left(0,1\right).
\]
If, in addition, (\ref{eq:L2PnEstErrCond}) holds for some non-random sequences
$\beta_{n}$ and $\delta_{n}$
in $[0,\infty)$, and
\[
\rho'_{n}:=2C_{b}'\left(\delta_{n}\lor\left(\frac{B_{n}^{4}\ln^{6}\left(pn\right)}{n}\right)^{1/6}\right)
\]
denotes two times the upper bound (\ref{eq:HDBootstrapUpperBnd}) in Theorem \ref{thm:HDMultBootApprox}, then
with probability at least $1-\beta_{n}-1/\ln^{2}(pn),$
\[
q_{\bSigma}^{\cN}\left(1-\alpha-\rho_{n}'\right)\leqslant\widehat{q}_{n}\left(1-\alpha\right)\leqslant q_{\bSigma}^{\cN}\left(1-\alpha+\rho_{n}'\right)\text{ for all }\alpha\in\left(0,1\right).
\]
\end{thm}
\begin{proof}
Observe that (\ref{eq:MomentConds}) implies that $\bSigma$ is a symmetric positive
semi-definite matrix with strictly positive diagonal. Apply Theorem
\ref{lem:QuantileComparisonLemma} with $\bU=\cS_{n}$ to obtain
\[
q_{\bSigma}^{\cN}(1-\alpha-2\rho(\cS_{n},\cN_n))\leqslant q_{n}\left(1-\alpha\right)\leqslant q_{\bSigma}^{\cN}(1-\alpha+\rho(\cS_{n},\cN_n))\text{ for all }\alpha\in\left(0,1\right).
\]
The \emph{first} pair of inequalities then follows from
$2\rho(\cS_{n},\cN_n)\leqslant\rho_{n}$, cf.~Theorem \ref{thm:HDCLT}. To
establish the \emph{second} claim, apply Theorem \ref{lem:QuantileComparisonLemma} with
$\bU=\widehat{\cS}_{n}^{e}$ and conditional on the
$\{(\bZ_{i},\widehat{\bZ}_{i})\}_{i=1}^n$ to obtain
\[
q_{\bSigma}^{\cN}(1-\alpha-2\widetilde{\rho}(\widehat{\cS}_{n}^{e},\cN_n))\leqslant\widehat{q}_{n}\left(1-\alpha\right)\leqslant q_{\bSigma}^{\cN}(1-\alpha+\widetilde{\rho}(\widehat{\cS}_{n}^{e},\cN_n))\text{ for all }\alpha\in\left(0,1\right).
\]
The second pair of inequalities then follows on the event $2\widetilde{\rho}(\widehat{\cS}_{n}^{e},\cN_n)\leqslant\rho'_{n},$
which by Theorem \ref{thm:HDMultBootApprox} occurs with probability
at least $1-\beta_{n}-1/\ln^{2}(pn).$
\end{proof}
\begin{thm}
[\textbf{Multiplier Bootstrap Consistency}]\label{thm:MultiplierBootstrapConsistency}
Let (\ref{eq:MomentConds}) and (\ref{eq:L2PnEstErrCond}) hold for some constant
$b\in(0,\infty)$ and non-random sequences $B_{n}$ in $[1,\infty)$ and
$\delta_{n}$ and $\beta_{n}$ both in $[0,\infty).$ Then there is a constant
$C_{b}\in(0,\infty)$, depending only on $b$, such that
\[
\sup_{\alpha\in\left(0,1\right)}\big|\P\big(\Vert \cS_{n}\Vert_{\infty}>\widehat{q}_{n}\left(1-\alpha\right)\big)-\alpha\big|\leqslant C_{b}\max\left\{\beta_{n},\delta_{n},\left(\frac{B_{n}^{4}\ln^{7}\left(pn\right)}{n}\right)^{1/6},\frac{1}{\ln^{2}\left(pn\right)}\right\}.
\]
\end{thm}
\begin{proof}
Fix $\alpha\in(0,1)$. By Theorems \ref{thm:HDCLT} and \ref{thm:QuantileComparison},
\begin{align*}
\P\big(\Vert \cS_{n}\Vert_{\infty}\leqslant\widehat{q}_{n}\left(1-\alpha\right)\big) & \leqslant\P\big(\Vert \cS_{n}\Vert_{\infty}\leqslant q_{\bSigma}^{\cN}\left(1-\alpha+\rho_{n}'\right)\big)+\beta_{n}+\frac{1}{\ln^{2}(pn)}\\
 & \leqslant\P\left(\Vert \cN_{n}\Vert_{\infty}\leqslant q_{\bSigma}^{\cN}\left(1-\alpha+\rho'_{n}\right)\right)+\rho_{n}+\beta_{n}+\frac{1}{\ln^{2}(pn)},
\end{align*}
with $\rho_n$ and $\rho_n'$ defined in Theorem \ref{thm:QuantileComparison}.
If $\rho_n'\geqslant\alpha$, then
$q_{\bSigma}^{\cN}\left(1-\alpha+\rho'_{n}\right)=+\infty$ and, thus,
\[
   \P\left(\Vert \cN_{n}\Vert_{\infty}\leqslant q_{\bSigma}^{\cN}\left(1-\alpha+\rho'_{n}\right)\right)=1\leqslant1-\alpha+\rho_{n}'.
\]
If $\rho_n'<\alpha$, then
\[
   \P\left(\Vert \cN_{n}\Vert_{\infty}\leqslant q_{\bSigma}^{\cN}\left(1-\alpha+\rho'_{n}\right)\right)=1-\alpha+\rho'_{n}.
\]
Continuing the initial string of inequalities, in either case we arrive at
\[
   \P\big(\Vert \cS_{n}\Vert_{\infty}\leqslant\widehat{q}_{n}\left(1-\alpha\right)\big)\leqslant1-\alpha+\rho_{n}'+\rho_{n}+\beta_{n}+\frac{1}{\ln^{2}(pn)}.
\]
Parallel reasoning shows
\begin{align*}
\P\big(\Vert \cS_{n}\Vert_{\infty}\leqslant\widehat{q}_{n}\left(1-\alpha\right)\big) & \geqslant1-\alpha-\Big(\rho_{n}'+\rho_{n}+\beta_{n}+\frac{1}{\ln^{2}(pn)}\Big).
\end{align*}
The claim now follows from combining and rearranging the previous two displays,
taking the supremum over $\alpha\in(0,1)$, and recasting the constant
$C_b\in(0,\infty)$, which can be chosen to depend only on $b$.
\end{proof}

\section{Solution Existence, Sparsity and
Uniqueness}\label{sec:ExistenceSparsityAndUniqueness} 
The relation $\widehat{\btheta}(\lambda)\in\widehat{\Theta}(\lambda)$ in
(\ref{eq:ell1PenalizedMEstimationIntro}) hinges on the fundamental property of
existence of a minimizer, i.e.~the non-emptiness of $\widehat{\Theta}(\lambda)$.
The non-emptiness, cardinality, and related properties of the solution set (a
subset of $\R^p$) generally depend on the data, penalty level, and parameter
space. While we in the main sections of the paper presume the existence of a
minimizer, for the sake of completeness, in this section we provide criteria for
existence of a solution as well as related properties. To keep matters interesting
yet statements simple, we here consider the full (hence unbounded) parameter
space $\Theta=\R^p$.\footnote{Of course, since the criterion is presumed convex
(hence continuous), non-emptiness and compactness of the parameter space suffice
for the existence of a solution, cf.~Weierstrass' extreme value theorem.} Denote
the $n\times p$ regressor matrix $\mathbf{X}:=[\bX_1:\cdots:\bX_n]^\top$ and use
$\mathbf{X}_{J}$ to denote the $n\times|J|$ submatrix of $\mathbf{X}$ with
columns indexed by $J\subseteq\left[p\right]$. (The rank of
$\mathbf{X}_\emptyset$ is interpreted as zero.)

\begin{thm}[\textbf{$\boldsymbol{\ell_1}$-ME Existence and
Sparsity}]\label{thm:ExistenceAndSparsity} Let the loss function
$m:\R\times\mathcal{Y}\to\R$ be non-negative, $m(\cdot,\by)$ convex for all
$\by\in\mathcal{Y}$, and $\Theta=\R^p$. Then for any $\lambda\in(0,\infty)$  and
any realization of $\{(\bY_i,\bX_i)\}_{i=1}^{n}$, the following properties hold:
\begin{enumerate}
\item\label{enu:Existence} The set of minimizers
$\widehat{\Theta}\left(\lambda\right)$ in
(\ref{eq:ell1PenalizedMEstimationIntro}) is non-empty, convex and compact.
\item\label{enu:Sparsity} For at least one minimizer
$\widehat{\btheta}\in\widehat{\Theta}\left(\lambda\right)$, the columns
$\{\mathbf{X}_j;{j\in\mathrm{supp}(\widehat\btheta)}\}$ of $\mathbf{X}$
corresponding to the non-zero entries of $\widehat{\btheta}$ are linearly
independent,
i.e.~$\mathrm{rank}(\mathbf{X}_{\mathrm{supp}(\widehat\btheta)})=\|\widehat{\btheta}\|_0$.
In particular, for such a minimizer, $\|\widehat{\btheta}\|_0\leqslant n\wedge
p$.
\end{enumerate}
\end{thm}

As stated in the theorem, non-negativity establishes not only existence of a
solution, but also existence of a solution for which the ``active'' columns of
$\mathbf{X}$ are linearly independent. Such a solution is therefore
\emph{sparse} in the sense of having no more than $n\wedge p$ non-zeros, with
the interesting part of the bound being
$\|\widehat{\btheta}\|_0\leqslant n$. All examples in Section
\ref{sec:Examples} concern non-negative loss functions, thus guaranteeing the
existence of a (sparse) solution.

Theorem \ref{thm:ExistenceAndSparsity} certainly has precursors in the
literature. Existence results for $\ell_1$-penalized M-estimators are mentioned
in \citet[Section 2.3]{tibshirani_lasso_2013} for differentiable and strictly
convex loss functions. As shown in Theorem \ref{thm:ExistenceAndSparsity},
neither differentiability nor strict convexity is necessary for this
conclusion.\footnote{Moreover, one can find a differentiable and strictly
convex function and a positive penalty level $\lambda$ such that the criterion
function in (\ref{eq:ell1PenalizedMEstimationIntro}) can be made arbitrarily
small and, thus, no minimizer exists. A condition (such as non-negativity)
therefore appears to be missing or implicit in the treatment in
\citet{tibshirani_lasso_2013}.}
The existence of a solution associated with linearly
independent active columns dates back to \citet{osborne_lasso_2000} for the
LASSO and \citet[Theorem B.3]{rosset_boosting_2004} for differentiable non-negative
convex loss functions. See also \citet[Section 5.2]{tibshirani_lasso_2013} for
discussion. Again, the differentiability is not necessary. In particular,
Theorem \ref{thm:ExistenceAndSparsity} applies to both quantile regression and
trimmed LAD loss functions, both of which are non-differentiable.

\begin{rem}[\textbf{Non-Negativity}]
The non-negativity of the loss function used in establishing the existence of an
${\ell_1}$-ME is actually a bit of a red herring as adding or subtracting a
constant from the loss bears no impact on the solution set. The crucial element
is that the loss is bounded from below, and non-negativity is a simple way to
ensure this. Inspecting the proof of Theorem \ref{thm:ExistenceAndSparsity}, we
see that this property is only used to ensure positivity of the ``asymptotic
slope'' of the ${\ell_1}$-penalized M-estimation criterion [assuming
$\lambda\in(0,\infty)$]. A positive asymptotic slope follows from the condition
\[
   \lim_{\tau\to\infty}\frac{m\left(\tau v,\by\right)}{\tau}\in\left[0,\infty\right]\text{ for both }v\in\left\{ -1,1\right\} \text{ and all }\by\in\mathcal{Y}.
\]
A \textit{positive} asymptotic slope means that the loss eventually increases,
and a \textit{zero} asymptotic slope means that the loss eventually flattens
and, hence, the penalty eventually dominates. In the previous display, existence
of the limits (as possibly extended real numbers) is guaranteed by convexity.
Non-negativity of the displayed limits follows trivially from non-negativity of
the loss itself.\qed
\end{rem}

\begin{rem}[{\textbf{Post-LASSO Existence}}]
The existence of a solution for which the active columns are linearly
independent is particularly interesting in the context of least squares post
variable selection based on the LASSO, also known as \emph{post-LASSO}
\citep{belloni_sparse_2012, belloni_high_2011, belloni_least_2013}.\footnote{
Post-LASSO is sometimes referred to as \emph{Gauss LASSO}. The method coincides
with the \citet{meinshausen_relaxed_2007} \emph{relaxed LASSO}, when the
relaxation parameter is set to zero.} Specifically, the linear independence
implied by such a LASSO solution $\widehat\btheta$ ensures that least squares
based on the regressors selected by $\widehat\btheta$ has a unique solution,
namely
$(\mathbf{X}_{\mathrm{supp}(\widehat\btheta)}^{\top}\mathbf{X}_{\mathrm{supp}(\widehat\btheta)})^{-1}\mathbf{X}_{\mathrm{supp}(\widehat\btheta)}^{\top}\bY$,
where $\bY$ denotes the $n\times 1$ vector of outcomes. Hence, while the
post-LASSO need not exist uniquely for \emph{every} LASSO solution, as long as
$\lambda\in(0,\infty)$, there is a solution for which it does.\qed
\end{rem}

\begin{rem}[\textbf{Non-Sparse Solutions}]
Not all minimizers are necessarily sparse. In fact, one can construct examples
with $p>n$ and some solution having all $(p)$ non-zeros. For a simple numerical
example, take the LAD loss $m(t,y)=|y-t|$ from median regression  with $n=1$
observation and $p=2$ parameters, data $Y=X_1=X_2=1$, and the penalty level
$\lambda\in(0,1)$. From the objective
$|1-(\theta_1+\theta_2)|+\lambda(|\theta_1|+|\theta_2|)$ and $\lambda\in(0,1)$
we see that it is more costly to move $\theta_1+\theta_2$ away from one than to
move $\theta_1$ or $\theta_2$ away from zero. Hence, any solution sets
$\theta_2=1-\theta_1$. The function $|\theta_1|+|1-\theta_1|$ equals one for all
$\theta_1\in[0,1]$ and is strictly greater otherwise, so the set of solutions is
the closed line segment $\{(u,1-u)\in\R^2;u\in[0,1]\}$. Note that the solution
set is both closed and bounded and involves both sparse solutions (the two end
points) and non-sparse solutions (everything in between).\qed
\end{rem}

We next turn to the question of uniqueness. We say that the columns of
$\mathbf{X}$ are in \emph{general position} if for any integer
$k\in\{0,1,\dotsc,(n\wedge p)-1\}$, any $k+1$ column indices
$j_1,\dotsc,j_{k+1}\in\left[p\right]$, and any signs
$\sigma_1,\dotsc,\sigma_{k+1}\in\{-1,1\}$, the affine span of the $k+1$ signed
columns $\{\sigma_1\mathbf{X}_{j_1},\dotsc,\sigma_{k+1}\mathbf{X}_{j_{k+1}}\}$
does not contain any element of $\{\pm
\mathbf{X}_j;j\in[p]\backslash\{j_1,\dotsc,j_{k+1}\}\}$.

\begin{thm}[\textbf{$\boldsymbol{\ell_1}$-ME Uniqueness via General Position}]\label{thm:UniquenessGeneralPosition}
Let the loss function $m:\R\times\mathcal{Y}\to\R$ be
non-negative, $m(\cdot,\by)$ strictly convex for all $\by\in\mathcal{Y}$, and
$\Theta=\R^p$. Then for any $\lambda\in(0,\infty)$ and any realization
of $\{(\bY_i,\bX_i)\}_{i=1}^{n}$ for which the columns of $\mathbf{X}$ are in
general position, there is a unique minimizer
$\widehat{\Theta}\left(\lambda\right)=\{\widehat{\btheta}\left(\lambda\right)\}$
with $\|\widehat{\btheta}\left(\lambda\right)\|_0\leqslant n\wedge p$.
\end{thm}
Sparsity of the solution follows from uniqueness, cf.~Theorem
\ref{thm:ExistenceAndSparsity}. Like the existence result, the uniqueness
theorem have precursors in earlier literature, including
\citet{osborne_lasso_2000} for the LASSO and \citet[Theorem
B.5]{rosset_boosting_2004} for differentiable non-negative convex loss functions.
See \citet{tibshirani_lasso_2013} for discussion and additional references.
Our modest contribution here lies in showing that the crucial parts of the argument can
accommodate non-differentiability by appealing to the Karush-Kuhn-Tucker (KKT)
conditions for optimality.

While general position is an abstract condition, it is satisfied with probability one
when the regressors are drawn from an absolutely continuous distribution on $\R^{p\times n}$.
This observation leads us to the following corollary.

\begin{cor}[\textbf{$\boldsymbol{\ell_1}$-ME Uniqueness via Absolute
Continuity}]\label{cor:UniquenessAbsoluteContinuity} 
Let the loss function
$m:\R\times\mathcal{Y}\to\R$ be non-negative, $m(\cdot,\by)$ strictly convex for
all $\by\in\mathcal{Y}$, $\Theta=\R^p$, and the elements of $\mathbf{X}$
absolutely continuously distributed with respect to Lebesgue measure on
$\R^{np}$. Then for any $\lambda\in(0,\infty)$ and no matter the distribution of
$\{\bY_i\}_{i=1}^{n}$, with probability one, there is a unique $\ell_1$-ME,
which then has at most $n\wedge p$ non-zero components.
\end{cor}

See \citet[Lemma 5]{tibshirani_lasso_2013} for a similar statement for
differentiable and strict convex (non-negative) loss functions and \emph{ibid.}
(p.~1463) for the argument of almost surely sufficiency of absolute continuity
for general position of $\mathbf X$.

\begin{rem}[{\textbf{Necessity of General Position}}]
One cannot in general achieve uniqueness without the columns of $\mathbf{X}$
being in general position. For a simple numerical example, take the square loss
$m(t,y)=(1/2)(y-t)^2$ from mean regression with $n=1$ observation and $p=2$
parameters, data $Y=X_{1}=X_{2}=1$, and penalty level
$\lambda=\textstyle{\frac{1}{2}}$. While the square loss is strictly convex,
since $X_1=X_2$, the columns of $\mathbf{X}$ are not in general position. The
KKT conditions associated with minimizing
$\textstyle{\frac{1}{2}}(1-\theta_1-\theta_2)^2+\textstyle{\frac{1}{2}}(|\theta_1|+|\theta_2|)$
are satisfied by any pair $(u,\textstyle{\frac{1}{2}}-u)$ with
$u\in(0,\textstyle{\frac{1}{2}})$, so the solution is not unique.\qed
\end{rem}

\begin{rem}[{\textbf{Necessity of Strict Convexity}}]
One cannot in general achieve uniqueness without strict convexity of the loss
function. For a simple numerical example, take the (not strictly) convex LAD
loss $m(t,y)=|y-t|$ from median regression with $n=1$ observation and $p=2$
parameters, data $Y=X_{1}=1$ and $X_{2}=0$, and penalty level $\lambda=1$. Since
$X_1$ and $X_2$ differ in terms of more than just their signs, the columns of
$\mathbf{X}$ are in general position. (Recall that the affine span of a singleton is
the singleton itself.) However, the objective
$|1-(\theta_1+0\cdot\theta_2)|+|\theta_1|+|\theta_2|$ is minimized at any
$\theta_1\in[0,1]$ with $\theta_2=0$, so the solution is not unique.\qed
\end{rem}

\begin{rem}[\textbf{Unpenalized Coefficients}]\label{rem:UnpenalizedCoefficients}
In cases where one leaves one or more coefficients out of the penalty, the
existence of solution can no longer be guaranteed independently of the data. For
example, consider including an unpenalized intercept in a binary response
setting with negative log-likelihood loss. In the event that all outcomes are of
the same label (all zeros or all ones), one can achieve complete separation
based on the constant regressor alone. Statements such as ``$\ell_1$-penalized
logistic regression always has a solution'' appearing in the literature must
therefore explicitly or implicitly penalize even the intercept (provided one is
present). For a treatment of existence in the so-called generalized LASSO
problem (with possibly non-square loss), where the penalty is the $\ell_1$ norm
of a (not necessarily identity) matrix times the coefficient vector, see
\citet{ali_generalized_2019}. \qed
\end{rem}

\section{Proofs for Appendix \ref{sec:ExistenceSparsityAndUniqueness}}

For this section, we let $\widehat M:\Theta\to\R$ abbreviate the average loss,
defined by
\[
\widehat{M}(\btheta):=\En[m(\bX_i^{\top}\btheta,\bY_i)],\quad\btheta\in\Theta=\R^p.   
\]

\begin{proof}[\sc{Proof of Theorem} \ref{thm:ExistenceAndSparsity}]
First, we consider Part \ref{enu:Existence} of the theorem. To show existence,
consider the objective function $f$ on $\R^p$ defined by
\[
f\left(\btheta\right):=\widehat{M}\left(\btheta\right)+\lambda\left\|\btheta\right\|_{1}.
\]
Then $f$ is convex and finite (i.e. real-valued). The recession cone
$R_f:=\{\bvartheta\in\R^p;f^{\infty}(\bvartheta)\leqslant 0\}$ of $f$ consists
of the vectors $\bvartheta\in\R^p$ such that the recession function $f^{\infty}$
at $\bvartheta$ is non-positive. To see that $R_f=\{\mathbf{0}_p\}$, invoke
\citet[Corollary 8.5.2]{rockafellar_convex_1970} to evaluate $f^{\infty}$ as
\[
f^{\infty}\left(\bvartheta\right) = \lim_{\tau\to\infty}\frac{f\left(\tau\bvartheta\right)}{\tau}
   = \lim_{\tau\to\infty}\frac{\widehat{M}\left(\tau\bvartheta\right)+\lambda\left\|\tau\bvartheta\right\|_1}{\tau}
   = \lim_{\tau\to\infty}\frac{\widehat{M}\left(\tau\bvartheta\right)}{\tau} + \lambda\left\|\bvartheta\right\|_1.
\]
Loss non-negativity implies $f^{\infty}(\bvartheta)\geqslant
\lambda\|\bvartheta\|_1$, so from $\lambda\in(0,\infty)$ we conclude that
$f^{\infty}(\bvartheta)>0$ for all $\bvartheta\in\R^p\backslash\{\mathbf{0}\}$.
\citet[Theorem 27.1(d)]{rockafellar_convex_1970} now shows that the set
$\widehat{\Theta}(\lambda)$ of minimizers of $f$ is non-empty and bounded.
Convexity of $\widehat{\Theta}(\lambda)$ follows from convexity of $f$. That the
level set $\widehat{\Theta}(\lambda)$ is closed (hence compact) follows from
continuity of $f$, which is a consequence of its convexity and finiteness on
$\R^p$ \citep[Corollary 10.1.1]{rockafellar_convex_1970}.

Now we consider Part \ref{enu:Sparsity} of the theorem. To show linear
independence, let
$\widehat{\btheta}\in\widehat{\Theta}\left(\lambda\right)(\neq\emptyset)$ be any
solution. We know that
$\mathrm{rank}(\mathbf{X}_{\mathrm{supp}(\widehat\btheta)})\leqslant\|\widehat{\btheta}\|_0$.
If equality holds, then we are done, so suppose that
$\mathrm{rank}(\mathbf{X}_{\mathrm{supp}(\widehat\btheta)})<\|\widehat{\btheta}\|_0$.
Let $\widehat{T}:=\mathrm{supp}(\widehat\btheta)$ abbreviate the support of
$\widehat{\btheta}$. Then there is a
$\bv\in\R^{\|\widehat{\btheta}\|_0}\backslash\{\mathbf{0}\}$ such that
$\mathbf{X}_{\widehat{T}}\bv=\mathbf{0}$. Since $\bv$ is non-zero, there is an
index $j\in\widehat{T}$ such that $v_j\neq0$. Fix such an index $j$. It then
follows that
\begin{equation}\label{eq:XjLinearInXks}
   \mathbf{X}_j=\sum_{k\in\widehat{T}\backslash\{j\}}c_k\mathbf{X}_k,\quad\text{where}\quad c_k:=-\frac{v_{k}}{v_{j}},\quad k\in\widehat{T}\backslash\{j\}.
\end{equation}
Per convexity, optimality of $\widehat{\btheta}$ is equivalent to
$\mathbf{0}\in\partial f(\widehat{\btheta})$, where $\partial
f(\widehat{\btheta})$ denotes the subdifferential of $f$ at $\widehat{\btheta}$.
Since all functions involved are finite convex, \citet[Theorems 23.8 and
23.9]{rockafellar_convex_1970} combine to show
\[
\partial f(\widehat{\btheta})=\partial\widehat{M}(\widehat{\btheta})+\lambda \partial\|\widehat{\btheta}\|_1
   = \mathbb{E}_n\left[\bX_i\partial_1 m(\bX_i^{\top}\widehat{\btheta},\bY_i)\right]+\lambda\sum_{k=1}^{p}\partial|\widehat{\theta}_k|,
\]
where summation is understood in the Minkowski (i.e.~set) sense, and $\partial_1
m(t,\by)$ denotes the subdifferential of $m(\cdot,\by)$ at $t$. Since
$\widehat{\btheta}$ is a solution, we can find $z_i\in-\partial_1
m(\bX_i^{\top}\widehat{\btheta},\bY_i),i\in[n]$, and
$\gamma_k\in\partial|\widehat{\theta}_k|,k\in[p]$, such that
\[
\mathbb{E}_n\left[z_i X_{i,k}\right] = \lambda \gamma_k,\quad k\in\left[p\right].
\]
Fix such $\{z_i\}_{i=1}^{n}$ and $\{\gamma_k\}_{k=1}^{p}$. Note that
$\gamma_k=\mathrm{sgn}(\widehat{\theta}_k)$ for $\widehat{\theta}_k\neq0$. From
(\ref{eq:XjLinearInXks}) we know that
$X_{i,j}=\sum_{k\in\widehat{T}\backslash\{j\}}c_kX_{i,k}$ for all $i\in[n]$, so
multiplying each side of (\ref{eq:XjLinearInXks}) by $\gamma_j$ and using
$\gamma_k^2=1$ for all $k\in\widehat{T}$, we arrive at
\begin{equation}\label{eq:gammaXjLinearIngammaXks}
   \gamma_j X_{i,j}=\sum_{k\in\widehat{T}\backslash\{j\}}a_k \gamma_k X_{i,k},\quad i\in[n],\;\quad\text{where}\quad a_k:=c_k \gamma_j \gamma_k,\quad k\in\widehat{T}\backslash\{j\}.
\end{equation}
Multiplying each side of \eqref{eq:gammaXjLinearIngammaXks} by $z_i$ and
then averaging over $i\in[n]$, we therefore arrive at
\[
\underbrace{\gamma_j \mathbb{E}_n\left[z_i X_{i,j}\right]}_{=\lambda} = \sum_{k\in\widehat{T}\backslash\{j\}}a_k \underbrace{\gamma_k\mathbb{E}_n\left[z_i X_{i,k}\right]}_{=\lambda}.
\]
The previous display and $\lambda\in(0,\infty)$ imply
$\sum_{k\in\widehat{T}\backslash\{j\}} a_k=1$. We now follow the argument on
\citet[p.~969]{rosset_boosting_2004}, included here for completeness. Define
$\bvartheta\in\R^p$ by $\vartheta_j:=-\gamma_j$,
$\vartheta_k:=a_k\gamma_k,k\in\widehat{T}\backslash\{j\}$, and
$\vartheta_k=0,k\notin\widehat{T}$. We construct $\widetilde{\btheta}\in\R^p$ by
moving $\widehat{\btheta}$ in the direction $\bvartheta$ until we hit a new
zero. That is, we let
\[
\widetilde{\btheta}:=\widehat{\btheta}+\tau_0\bvartheta\quad\text{where}\quad\tau_0:=\inf\left\{\tau\geqslant 0; \widehat{\theta}_k+\tau\vartheta_{k}=0\text{ for some }k\in\widehat{T}\right\}.
\]
Note that $\tau_0$ is finite, since $\widehat{\theta}_j+\tau\vartheta_j=0$ is
solved by $\tau=|\widehat{\theta}_j|$. 
For indices $J\subseteq[p]$, we let $\bdelta_{(J)}$ denote the
$|J|$-dimensional subvector of $\bdelta\in\R^p$ with indices indexed by $J$.
Then
\begin{align*}
   \mathbf{X}\widetilde{\btheta}=\mathbf{X}_{\widehat{T}}\widetilde{\btheta}_{(\widehat{T})}
   &=\mathbf{X}_{\widehat{T}}\widehat{\btheta}_{(\widehat{T})}+\tau_0\mathbf{X}_{\widehat{T}}\bvartheta_{(\widehat{T})}\\
   &=\mathbf{X}_{\widehat{T}}\widehat{\btheta}_{(\widehat{T})}+\tau_0\underbrace{\Big(-\gamma_j\mathbf{X}_{j}+\sum_{k\in\widehat{T}\backslash\{j\}}a_k \gamma_k \mathbf{X}_k\Big)}_{=\mathbf{0}_n\text{ cf.~\eqref{eq:gammaXjLinearIngammaXks}}}=\mathbf{X}\widehat{\btheta}.
\end{align*}
It follows that
$\widehat{M}(\widetilde\btheta)=\widehat{M}(\widehat\btheta)$, so
$\widetilde\btheta$ achieves the same loss as $\widehat\btheta$. 

Moreover, from
\begin{align*}
\|\widetilde\btheta\|_1
   &=|\widetilde\theta_j|+\sum_{k\in\widehat{T}\backslash\{j\}}|\widetilde\theta_k|\tag{\ensuremath{\mathrm{supp}(\widetilde\btheta)\subseteq\widehat T}}\\
   &=(\widehat\theta_j+\tau_0\vartheta_j)\mathrm{sgn}(\widehat\theta_j)+\sum_{k\in\widehat{T}\backslash\{j\}}(\widehat\theta_k+\tau_0\vartheta_k)\mathrm{sgn}(\widehat\theta_k)\tag{no sign flips}\\
   &=(\widehat\theta_j-\tau_0\gamma_j)\mathrm{sgn}(\widehat\theta_j)+\sum_{k\in\widehat{T}\backslash\{j\}}(\widehat\theta_k+\tau_0a_k\gamma_k)\mathrm{sgn}(\widehat\theta_k)\\
   &=\widehat\theta_j\mathrm{sgn}(\widehat\theta_j)-\tau_0+\sum_{k\in\widehat{T}\backslash\{j\}}\widehat\theta_k\mathrm{sgn}(\widehat\theta_k) + \tau_0 \sum_{k\in\widehat{T}\backslash\{j\}}a_k\tag{\ensuremath{\gamma_k = \mathrm{sgn}(\widehat\theta_k),k\in\widehat T}}\\
   &=|\widehat\theta_j|+\sum_{k\in\widehat{T}\backslash\{j\}}|\widehat\theta_k|\tag{\ensuremath{\sum_{k\in\widehat{T}\backslash\{j\}}a_k=1}}\\
   &=\|\widehat\btheta\|_1
\end{align*}
we see that $\widetilde\btheta$ achieves the same $\ell_1$ norm as well. It
follows that $\widetilde{\btheta}$ is also a solution and that
$\widetilde{\btheta}$ has (at least) one more zero than $\widehat{\btheta}$. We
can repeat the above argument until we arrive at a solution for which the
columns indexed by its support are linearly independent.

Finally, letting $\widehat\btheta$ be a solution for which
$\mathrm{rank}(\mathbf{X}_{\mathrm{supp}(\widehat\btheta)})=\|\widehat{\btheta}\|_0$,
sparsity then follows from
$\mathrm{rank}(\mathbf{X}_{\mathrm{supp}(\widehat\btheta)})\leqslant n\wedge p$.
\end{proof}

\begin{proof}[\sc{Proof of Theorem} \ref{thm:UniquenessGeneralPosition}]
Existence follows from Theorem \ref{thm:ExistenceAndSparsity}. We argue
uniqueness in four steps. In \textit{Step 1}, we show that strict convexity of
$m(\cdot,\by)$ implies that every solution
$\widehat{\btheta}\in\widehat{\Theta}(\lambda)$ leads to the same linear forms
$\mathbf{X}\widehat\btheta$. In \textit{Step 2}, we use this observation in
combination with the optimality conditions to define the so-called
equicorrelation set $\mathscr E$, which (as we establish) contains the supports
of all solutions $\widehat{\btheta}\in\widehat{\Theta}(\lambda)$. In
\textit{Step 3}, we leverage the columns of $\bfX$ being in general position to
show that the columns picked out by the equicorrelation set are linearly
independent. In \textit{Step 4}, we show that every solution can be
characterized as the solution to the same strictly convex programming problem
and must therefore coincide.

\textbf{Step 1}. To establish equality of linear forms, seeking a contradiction,
suppose that we can find solutions $\widehat{\btheta}^{(0)}$ and
$\widehat{\btheta}^{(1)}$ for which
$\mathbf{X}\widehat\btheta^{(0)}\neq\mathbf{X}\widehat\btheta^{(1)}$. Then
$\bX_i^{\top}\widehat\btheta^{(0)}\neq \bX_i^{\top}\widehat\btheta^{(1)}$ for
some $i\in[n]$. Consider the objective function $f$ on $\R^p$ defined by
$f\left(\btheta\right):=\widehat{M}\left(\btheta\right)+\lambda\left\|\btheta\right\|_{1}$.
Per strict convexity of $m(\cdot,\by),\by\in\mathcal Y$, and convexity of
$\left\|\cdot\right\|_1$, for any $\tau\in(0,1)$ and
$\widehat\btheta^{(\tau)}:=(1-\tau)\widehat{\btheta}^{(0)}+\tau\widehat{\btheta}^{(1)}$
we have
\begin{align*}
f(\widehat\btheta^{(\tau)})
   &=\mathbb{E}_n\big[m\big((1-\tau)\bX_i^{\top}\widehat{\btheta}^{(0)}+\tau \bX_i^{\top}\widehat{\btheta}^{(1)},\bY_i\big)\big]+\lambda\|(1-\tau)\widehat{\btheta}^{(0)}+\tau\widehat{\btheta}^{(1)}\|_1\\
   &<\left(1-\tau\right)\mathbb{E}_n\big[m\big(\bX_i^{\top}\widehat{\btheta}^{(0)},\bY_i\big)\big]+\tau\mathbb{E}_n\big[m\big(\bX_i^{\top}\widehat{\btheta}^{(1)},\bY_i\big)\big]
      +(1-\tau)\lambda\|\widehat{\btheta}^{(0)}\|_1+\tau\lambda\|\widehat{\btheta}^{(1)}\|_1\\
   &=\left(1-\tau\right)\big[\widehat{M}(\widehat\btheta^{(0)})+\lambda\|\widehat{\btheta}^{(0)}\|_1\big]
      +\tau\big[\widehat{M}(\widehat\btheta^{(1)})+\lambda\|\widehat{\btheta}^{(1)}\|_1\big]=\inf_{\R^p} f,
\end{align*}
which contradicts optimality of $\widehat{\btheta}^{(0)}$ and
$\widehat{\btheta}^{(1)}$. Hence, $\mathbf{X}\widehat\btheta$ is constant across
solutions $\widehat{\btheta}\in\widehat{\Theta}(\lambda)$.

\textbf{Step 2.} As described in the proof of Theorem
\ref{thm:ExistenceAndSparsity}, optimality of $\widehat{\btheta}$ is equivalent
to there being $z_i\in-\partial_1
m(\bX_i^{\top}\widehat{\btheta},\bY_i),i\in[n]$, and
$\gamma_j\in\partial|\widehat{\theta}_j|,j\in[p]$, such that
\[
\mathbb{E}_n\left[z_i X_{i,j}\right] = \lambda \gamma_j\quad j\in\left[p\right],
\]
where $\partial_1 m(t,\by)$ denotes the subdifferential of $m(\cdot,\by)$ at
$t$. Note that $\gamma_j=\mathrm{sgn}(\widehat{\theta}_j)$ for
$\widehat{\theta}_j\neq0$. From Step 1 we know that every solution
$\widehat{\btheta}\in\widehat{\Theta}(\lambda)$ leads to the same linear forms
$\mathbf{X}\widehat\btheta$. We can therefore define the \emph{equicorrelation
set} $\mathscr{E}$ independently of the solution by
\[
\mathscr{E}:=\left\{j\in\left[p\right];
   \left|\mathbb{E}_n\left[z_iX_{i,j}\right]\right|=\lambda
   \text{ for some }z_i\in -\partial_1 m(\bX_i^{\top}\widehat{\btheta},\bY_i),i\in[n]\right\}.
\]
We claim that $\mathscr{E}$ contains the support of every solution. Indeed, let
$\widehat\btheta\in\widehat\Theta(\lambda)$ and
$j\in\mathrm{supp}(\widehat\btheta)$. Then there are $z_i\in-\partial_1
m(\bX_i^{\top}\widehat{\btheta},\bY_i),i\in[n]$, and
$\gamma_j\in\partial|\widehat{\theta}_j|,j\in[p]$, such that
$\mathbb{E}_n\left[z_i X_{i,j}\right] = \lambda \gamma_j$ with
$\gamma_j=\mathrm{sgn}(\widehat\theta_j)$ being either plus or minus one. Hence
$j\in\mathscr{E}$.

\textbf{Step 3.} 
This step is similar to the linear independence argument in the proof of Theorem
\ref{thm:ExistenceAndSparsity}. We here show that the columns $\mathbf{X}$ being
in general position implies that
$\mathrm{rank}(\mathbf{X}_{\mathscr{E}})=|\mathscr{E}|$. We prove the
contra-positive statement, which is that
$\mathrm{rank}(\mathbf{X}_{\mathscr{E}})<|\mathscr{E}|$ implies that the columns
of $\mathbf{X}$ are \textit{not} in general position. To this end, we may
without loss of generality assume that $|\mathscr{E}|\leqslant (n\wedge p)+1$.
(Indeed, if $|\mathscr{E}|>(n\wedge p)+1$, then taking any subset $\mathscr{E}'$
of $\mathscr{E}$ with $|\mathscr{E}'|=(n\wedge p)+1$, we have
$\mathrm{rank}(\mathbf{X}_{\mathscr{E}'})\leqslant n\wedge p<|\mathscr{E}'|$,
and we continue the argument with $\mathscr{E}'$ in place of $\mathscr{E}$.)
Now, since $\mathrm{rank}(\mathbf{X}_{\mathscr{E}})<|\mathscr{E}|$, there is a
$\bv\in\R^{|\mathscr{E}|}\backslash\{\mathbf 0\}$ such that
$\mathbf{X}_{\mathscr{E}}\bv=\mathbf{0}$. Since $\bv$ is non-zero, there is a
$j\in\mathscr{E}$ such that $v_j\neq0$. Fix such a $j$. It then follows that
\[
\mathbf{X}_j=\sum_{k\in\mathscr{E}\backslash\{j\}}c_k\mathbf{X}_k,\quad\text{where}\quad c_k:=-\frac{v_{k}}{v_{j}},\quad k\in\mathscr{E}\backslash\{j\}.
\]
From $k\in\mathscr{E}$, we know there are $z_i\in-\partial_1
m(\bX_i^{\top}\widehat{\btheta},\bY_i),i\in[n]$ such that
$\left|\mathbb{E}_n\left[z_iX_{i,k}\right]\right|=\lambda$. Fix such
$\{z_i\}_{i=1}^n$, and abbreviate
$\xi_k:=\mathrm{sgn}(\mathbb{E}_n\left[z_iX_{i,k}\right])$ for $k\in\mathscr E$.
Since $\lambda\in(0,\infty)$, we have $\xi_k\in\{-1,1\}$ and, thus, $\xi_k^2=1$
for each $k\in\mathscr E$. It therefore follows that
\[
\xi_j \mathbf{X}_j=\sum_{k\in\mathscr{E}\backslash\{j\}} a_k \xi_k \mathbf{X}_k,\quad a_k:=\xi_j \xi_k c_k,\quad k\in\mathscr{E}\backslash\{j\}.
\]
Multiplying the $i$th equation by $z_i$ and averaging over $i\in[n]$, we see
that
\[
\underbrace{\xi_j \mathbb{E}_n\left[z_iX_{i,j}\right]}_{=\lambda}=\sum_{k\in\mathscr{E}\backslash\{j\}} a_k \underbrace{\xi_k \mathbb{E}_n\left[z_iX_{i,k}\right]}_{=\lambda},
\]
which via $\lambda\in(0,\infty)$ further implies that
$\sum_{k\in\mathscr{E}\backslash\{j\}} a_k=1$. Deduce that $\xi_j \mathbf{X}_j$
lies in the affine span of $\{\xi_k
\mathbf{X}_k;k\in\mathscr{E}\backslash\{j\}\}$. Since $|\mathscr
E|\leqslant(n\wedge p)+1$, we have
$|\mathscr{E}\backslash\{j\}|\leqslant n\wedge p$, which shows that the columns
of $\mathbf{X}$ are not in general position.

\textbf{Step 4.} Since $\mathscr{E}$ contains the support of every solution, we
can characterize any solution $\widehat\btheta\in\widehat\Theta(\lambda)$ by
$\widehat\btheta_{(\mathscr{E}^c)}=\mathbf{0}$ and
\[
\widehat{\btheta}_{\left(\mathscr{E}\right)}
   \in\argmin_{\btheta\in\R^{\left|\mathscr{E}\right|}}
   \left\{\mathbb{E}_{n}[m\left(\bX_{i\mathscr{E}}^{\top}\btheta,\bY_i\right)]
      +\lambda\left\|\btheta\right\|_1\right\}.
\]
Strict convexity of $m(\cdot,\by),\by\in\mathcal Y$, and
$\mathrm{rank}(\mathbf{X}_{\mathscr{E}})=|\mathscr{E}|$ show that the function
$\btheta\mapsto\mathbb{E}_{n}[m(\bX_{i\mathscr{E}}^{\top}\btheta,\bY_i)]$
defined on $\R^{|\mathscr{E}|}$ is strictly convex. The right-hand side problem
therefore has a unique solution, so $\widehat\Theta(\lambda)$ must be singleton.
\end{proof}

\section{Additional Examples}\label{sec:Additional-Examples}

In this section, we provide additional examples of loss functions that fit the
M-estimation framework \eqref{eq:ell1PenalizedMEstimationIntro} with the loss
function $m(t,\by)$ being convex in its first argument.

\begin{example}
[\textbf{Logistic Calibration}]\label{exa:LogisticCalibration} 
In the context of average treatment effect estimation under a conditional
independence assumption with a high-dimensional vector of controls, consider the
\textit{logit propensity score model} 
\begin{equation}
\P(Y=1|\bX)=\Lambda(\bX^{\top}\btheta_{0}),\label{eq: logit model}
\end{equation}
where $Y\in\left\{ 0,1\right\} $ is a treatment indicator, $\bX$ a vector of
controls, and $\Lambda$ the logistic CDF. Using (\ref{eq:EstimandIntro}),
$\btheta_{0}$ can be identified with the logistic loss function in
(\ref{eq:LossLogit}). However, as shown by \citet{tan2020regularized},
$\btheta_{0}$ can also be identified using (\ref{eq:EstimandIntro}) with the
logistic \textit{calibration} loss
\begin{equation}
m\left(t,y\right)=y\mathrm{e}^{-t}+\left(1-y\right)t,\label{eq:LogisticCalibration}
\end{equation}
which is convex in $t$ as well. As demonstrated by \citet{tan2020regularized},
use of this alternative loss function leads to average treatment effect
estimators with certain robustness properties. Specifically, under Tan's
conditions, these treatment effect estimators remain root-$n$ consistent and
asymptotically normal even if the model for the outcome regression function is
misspecified (\emph{ibid}.).\qed
\end{example}

\begin{example}
[\textbf{Logistic Balancing}]\label{exa:LogisticBalancing} In the same setting
as that of the previous example, the \emph{covariate balancing approach}
of \citet{imai_covariate_2014} amounts to specifying a parametric model for the
treatment indicator $Y\in\left\{ 0,1\right\} $,
\[
\P(Y=1|\bX)=F(\bX^{\top}\btheta_{0})
\]
and ensuring covariate balance in the sense that
\[
\E\left[\left\{ \frac{Y}{F\left(\bX^{\top}\btheta_{0}\right)}-\frac{1-Y}{1-F\left(\bX^{\top}\btheta_{0}\right)}\right\} \bX\right]=\mathbf{0}_p.
\]
Balancing here amounts to enforcing a collection of moment conditions and is
therefore naturally studied in a generalized method of moments (GMM) framework.
However, specifying $F$ to be the logistic CDF $\Lambda$, covariate balancing
can be achieved via M-estimation of $\btheta_{0}$ based on the loss function
\[
m\left(t,y\right)=\left(1-y\right)\mathrm{e}^{t}+y\mathrm{e}^{-t}+\left(1-2y\right)t,
\]
which is also convex in $t$. See \citet{tan2020regularized} for details.\qed
\end{example}

\begin{example}
[\textbf{Panel Logit Model}]\label{exa:PanelLogit} 
Consider the \emph{panel logit model}
\[
\P(Y_{\tau}=1|\bX,\gamma,Y_{0},\dotsc,Y_{\tau-1})=\Lambda(\gamma+\bX_{\tau}^{\top}\btheta_{0}),\quad \tau\in\{1,2\},
\]
where $\bY=(Y_{1},Y_{2})^{\top}\in\{0,1\}^2$ is a pair of binary outcome
variables, $\bX=(\bX_{1}^{\top},\bX_{2}^{\top})^{\top}$ is a vector of
regressors, and $\gamma$ is a unit-specific unobserved fixed effect.
\citet{rasch1960probabilistic} provides conditions under which $\btheta_{0}$ can be
identified by
$\btheta_{0}=\argmin\nolimits_{\btheta\in\R^{p}}\E[m(\left(\bX_{1}-\bX_{2}\right)^{\top}\btheta,\bY)],$
where
\begin{align}
m\left(t,\by\right) & =\mathbf{1}\left(y_{1}\neq y_{2}\right)\left[\ln\left(1+\mathrm{e}^{t}\right)-y_{1}t\right],\label{eq:LossConditionalLogit}
\end{align}
which is convex in $t$.\footnote{See also \citet[Section
3.2]{chamberlain_panel_1984} and \citet[Section
15.8.3]{wooldridge_econometric_2010}.}\qed
\end{example}

\begin{example}
[\textbf{Panel Duration Model}]\label{exa:PartialLikelihoodAndLogLinearHazards}
Consider the \textit{panel duration model} with the log-linear hazard specification
\[
\ln h_{\tau}\left(y\right)=\bX_{\tau}^{\top}\btheta_{0}+\ln h_{0}\left(y\right),\quad \tau\in\{1,2\}.
\]
Here $h_{\tau}$ denotes the (conditional) hazard for spell $\tau$, and both
$h_{0}$ and $h_{\tau}$ are allowed to be unit-specific. This model is a special
case of the duration models studied in \citet[Section
3.1]{chamberlain_heterogeneity_1985}. Chamberlain presumes that the spells
$Y_{1}$ and $Y_2$ are (conditionally) independent of each other and shows that
the partial log-likehood contribution is\footnote{See also \citet[Chapter 9,
Section 2.10.2]{lancaster_econometric_1992}.}
\[
\theta\mapsto\mathbf{1}\left(Y_{1}<Y_{2}\right)\ln\Lambda(\left(\bX_{1}-\bX_{2}\right)^{\top}\btheta)+\mathbf{1}\left(Y_{1}\geqslant Y_{2}\right)\ln\big(1-\Lambda(\left(\bX_{1}-\bX_{2}\right)^{\top}\btheta)\big).
\]
The implied loss function
\begin{align}
m\left(t,\by\right) & =\ln\left(1+\mathrm{e}^{t}\right)-\mathbf{1}\left(y_{1}<y_{2}\right)t\label{eq:LossPartialLogLikelihood}
\end{align}
is of the logit form (see Example \ref{exa:Logit}), hence convex in $t$. With
more than two completed spells, the partial log-likelihood takes a
conditional-logit form (\emph{ibid.}), and the resulting loss is therefore still
a convex function (albeit involving multiple indices).\qed
\end{example}

\section{Additional Simulations}\label{sec:AdditionalSimulations}

In this section we present additional results for the simulation setting in
Section \ref{sec:Simulations}. In Appendix
\ref{sec:Sparsity-of-Debiasing-Vector}, we relate the $\ell_0$ norm of the
debiasing coefficient vector $\bmu_{0}$ to that of the structural coefficients
$\bgamma_{0}$ attached to the controls. We then present additional simulation
results in Appendix \ref{sec:Additional-Simulation-Results}, which stem from a
different choice of probability tolerance $\alpha=\alpha_{n}$.

\subsection{Sparsity of Debiasing Coefficient Vector\label{sec:Sparsity-of-Debiasing-Vector}}

In this section we argue that, at least in our simulation setting,
the $\ell_{0}$ norm of the non-primitive debiasing vector $\bmu_{0}$
is bounded by that of the structural coefficients $\bgamma_{0}$. We consider
a master data-generating processes (DGP) akin to the one in the main text,
where
\begin{align*}
\bX & =\left(\begin{array}{c}
D\\
\bW
\end{array}\right)\sim\mathrm{N}\left(\left(\begin{array}{c}
0\\
\boldsymbol{0}_{\left(p-1\right)\times1}
\end{array}\right),\left(\begin{array}{cc}
\Sigma_{DD} & \bSigma_{WD}^{\top}\\
\bSigma_{WD} & \bSigma_{WW}
\end{array}\right)\right),\\
Y\mid \bX & \sim\mathrm{Ber}\left(F\left(\bX^{\top}\btheta_{0}\right)\right)\overset{d}{=}\mathrm{Ber}\left(F\left(\beta_{0}D+\bW^{\top}\bgamma_{0}\right)\right),\\
(\Sigma_{XX})_{j,k} & =\rho^{\left|j-k\right|},\quad\left|\rho\right|<1,
\end{align*}
and $F:\R\to[0,1]$ is a twice differentiable cumulative
distribution function (CDF) with everywhere positive derivative $f=F'$
and satisfying the technical condition
\[
\mathrm{E}\left[\left|\left.\frac{\mathrm{d}}{\mathrm{d}t}\frac{f\left(t\right)}{F\left(t\right)\left(1-F\left(t\right)\right)}\right|_{t=\bX^{\top}\btheta_{0}}\right|\right]<\infty,
\]
which we use to guarantee the finiteness of
$\mathrm{E}[|m''_{11}(\bX^{\top}\btheta_{0},\bY)|]$. For the logit model
$F=\Lambda$, one has $\Lambda'=\Lambda(1-\Lambda)$, and the latter condition is
trivial. For the probit model $F=\Phi$ considered in the main text, the absolute
integrability in the previous display is less obvious, although it can be shown
to follow. (See also Appendix \ref{sec: verification} for detailed derivations
for the examples in Section \ref{sec:Examples}.)

Suppose that for some integer $C\in\left[p-1\right]$ the first $C$ controls
$(W_{1},\dotsc,W_{C})^{\top}=:\bW_{[C]}$
are relevant in the sense that we have $\gamma_{0,j}\neq0,j\in\left[C\right],$
and the remaining controls
$(W_{C+1},\dotsc,W_{p-1})^{\top}=:\bW_{[p-1]\backslash[C]}$
are irrelevant, in that
$\gamma_{0,j}=0,j\in\left[p-1\right]\backslash\left[C\right]$. This structure
fits the exactly sparse and intermediate coefficient patterns presented in
Section \ref{sec:Simulations} for which we have $C=1$ and $4$,
respectively.\footnote{The structure also fits with the approximately sparse
pattern of Section \ref{sec:Simulations}, but there $\|\bmu_{0}\|_{0}\leqslant
p-1(=\|\bgamma_0\|_0)$ holds trivially.}

We wish to quantify the sparsity of
\[
\bmu_{0}=\left(\mathrm{E}\left[m''_{11}\left(\bX^{\top}\theta_{0},Y\right)\bW\bW^{\top}\right]\right)^{-1}\mathrm{E}\left[m''_{11}\left(\bX^{\top}\btheta_{0},Y\right)\bW D\right],
\]
which is well-defined under the assumptions of the main
text.\footnote{To be more precise, Lemma \ref{lem:ExistenceAndUniquenessOfMu0}
provides \emph{one} set of sufficient conditions for the existence and uniqueness a
solution to \eqref{eq: definition of mu0}. Other sets of sufficient conditions are possible.} For binary response we have
\[
m\left(t,y\right)=-y\ln F\left(t\right)-\left(1-y\right)\ln\left(1-F\left(t\right)\right).
\]
Differentiating once and simplifying, we get
\[
m'_{1}\left(t,y\right)=\frac{f\left(t\right)}{F\left(t\right)\left(1-F\left(t\right)\right)}\left(F\left(t\right)-y\right).
\]
Differentiating once more, we see that
\begin{align*}
m''_{11}\left(t,y\right) & =\frac{f\left(t\right)^{2}}{F\left(t\right)\left(1-F\left(t\right)\right)}+\left(\frac{f\left(t\right)}{F\left(t\right)\left(1-F\left(t\right)\right)}\right)'\left(F\left(t\right)-y\right),
\end{align*}
and, thus, using finiteness of $\mathrm{E}[|m''_{11}(\bX^{\top}\btheta_{0},Y)|]$,
\[
\mathrm{E}\left[m''_{11}\left(\bX^{\top}\btheta_{0},Y\right)|\bX^\top\btheta = t\right]
   =\frac{f\left(t\right)^{2}}{F\left(t\right)\left(1-F\left(t\right)\right)}=:\omega_F\left(t\right).
\]
It follows that we can express $\bmu_{0}$ as
\[
\bmu_{0}=\left(\mathrm{E}\left[\omega_F\left(\bX^{\top}\btheta_{0}\right)\bW\bW^{\top}\right]\right)^{-1}\mathrm{E}\left[\omega_F\left(\bX^{\top}\btheta_{0}\right)\bW D\right].
\]
We claim that $\bmu_{0}$ coincides with $\bmu^{\ast}$ defined as
\[
\bmu^{\ast}:=\left(\begin{array}{c}
\left(\mathrm{E}\left[\omega_F\left(\bX^{\top}\btheta_{0}\right)\bW_{\left[C\right]}\bW_{\left[C\right]}^{\top}\right]\right)^{-1}\mathrm{E}\left[\omega_F\left(\bX^{\top}\btheta_{0}\right)\bW_{\left[C\right]}D\right]\\
\boldsymbol{0}_{\left(p-1-C\right)\times1}
\end{array}\right),
\]
which, in particular, implies that $\bmu_{0}$ has no more than $C$
non-zero entries, since $\|\bmu_{0}\|_{0}\leqslant C=\|\bgamma_{0}\|_{0}$.
To establish this claim, we show that $\bmu^{\ast}$ solves the linear
system of equations
\begin{equation}
\mathrm{E}\left[\omega_F\left(\bX^{\top}\btheta_{0}\right)\bW\bW^{\top}\right]\bmu=\mathrm{E}\left[\omega_F\left(\bX^{\top}\btheta_{0}\right)\bW D\right]\label{eq:DebiasingSystemOfEquations}
\end{equation}
in $\bmu\in\mathbf{R}^{p-1}$, in which case it must coincide with
the (previously established) unique solution $\bmu_{0}$. To see that
$\bmu^{\ast}$ is indeed a solution, first note that by zero means
and Gaussianity
\begin{align*}
\mathrm{E}\left[\bW_{\left[p-1\right]\backslash\left[C\right]}\mid D,\bW_{\left[C\right]}\right] & =\bSigma_{W_{\left[p-1\right]\backslash\left[C\right]},DW_{\left[C\right]}}\bSigma_{DW_{\left(\left[C\right]\right)},DW_{\left[C\right]}}^{-1}\left(\begin{array}{c}
D\\
\bW_{\left[C\right]}
\end{array}\right).
\end{align*}
Using the Toeplitz correlation structure, we get
\[
\bSigma_{W_{\left[p-1\right]\backslash\left[C\right]},DW_{\left[C\right]}}\bSigma_{DW_{\left[C\right]},DW_{\left[C\right]}}^{-1}=\left(\begin{array}{cccc}
\boldsymbol{0}_{\left(p-\left(1+C\right)\right)\times1} & \cdots & \boldsymbol{0}_{\left(p-\left(1+C\right)\right)\times1} & \bSigma_{W_{\left[p-1\right]\backslash\left[C\right]},W_{C}}\end{array}\right).
\]
Indeed, observe that
\begin{align*}
 & \left(\begin{array}{cccc}
\boldsymbol{0}_{\left(p-1-C\right)\times1} & \cdots & \boldsymbol{0}_{\left(p-1-C\right)\times1} & \bSigma_{W_{\left[p-1\right]\backslash\left[C\right]},W_{C}}\end{array}\right)\bSigma_{DW_{\left[C\right]},DW_{\left[C\right]}}\\
 & =\left(\begin{array}{cccc}
0 & \cdots & 0 & \rho\\
0 & \cdots & 0 & \rho^{2}\\
\vdots &  & \vdots & \vdots\\
0 & \cdots & 0 & \rho^{p-1-C}
\end{array}\right)\left(\begin{array}{ccccc}
1 & \rho & \rho^{2} & \cdots & \rho^{C}\\
\rho & 1 & \rho &  & \vdots\\
\rho^{2} & \rho & \ddots & \ddots & \rho^{2}\\
\vdots &  & \ddots &  & \rho\\
\rho^{C} & \cdots &  & \rho & 1
\end{array}\right)\\
 & =\left(\begin{array}{cccc}
\rho^{C+1} & \rho^{C} & \cdots & \rho\\
\rho^{C+2} & \rho^{C+1} & \cdots & \rho^{2}\\
\vdots & \vdots &  & \vdots\\
\rho^{p-1} & \rho^{p-2} & \cdots & \rho^{p-1-C}
\end{array}\right)=\bSigma_{W_{\left[p-1\right]\backslash\left[C\right]},DW_{\left[C\right]}},
\end{align*}
as desired. We therefore get the simplification
\begin{align*}
\mathrm{E}\left[\bW_{\left[p-1\right]\backslash\left[C\right]}\mid D,\bW_{\left[C\right]}\right] & = W_{C}\bSigma_{W_{\left[p-1\right]\backslash\left[C\right]},W_{C}}.
\end{align*}
Since $\bX^{\top}\btheta_{0}$ does not depend on $\bW_{\left[p-1\right]\backslash\left[C\right]}$,
the right-hand side (RHS) vector of (\ref{eq:DebiasingSystemOfEquations})
is
\begin{align*}
\mathrm{E}\left[\omega_F\left(\bX^{\top}\btheta_{0}\right)\bW D\right] & =\mathrm{E}\left[\omega_F\left(\bX^{\top}\btheta_{0}\right)\left(\begin{array}{c}
\bW_{\left[C\right]}\\
\bW_{\left[p-1\right]\backslash\left[C\right]}
\end{array}\right)D\right]\\
 & =\mathrm{E}\left[\omega_F\left(\bX^{\top}\btheta_{0}\right)\left(\begin{array}{c}
   \bW_{\left[C\right]}\\
\mathrm{E}\left[\bW_{\left[p-1\right]\backslash\left[C\right]}\mid D,\bW_{\left[C\right]}\right]
\end{array}\right)D\right]\\
 & =\mathrm{E}\left[\omega_F\left(\bX^{\top}\btheta_{0}\right)\left(\begin{array}{c}
\bW_{\left[C\right]}\\
W_{C}\bSigma_{W_{\left[p-1\right]\backslash\left[C\right]},W_{C}}
\end{array}\right)D\right]\\
 & =\left(\begin{array}{c}
\mathrm{E}\left[\omega_F\left(\bX^{\top}\btheta_{0}\right)\bW_{\left[C\right]}D\right]\\
\mathrm{E}\left[\omega_F\left(\bX^{\top}\btheta_{0}\right)W_{C}D\right]\bSigma_{W_{\left[p-1\right]\backslash\left[C\right]},W_{C}}
\end{array}\right).
\end{align*}
Similarly, the left-hand side (LHS) matrix of (\ref{eq:DebiasingSystemOfEquations})
is
\begin{align*}
 & \mathrm{E}\left[\omega_F\left(\bX^{\top}\btheta_{0}\right)\bW\bW^{\top}\right]\\
 & =\left(\begin{array}{cc}
\mathrm{E}\left[\omega_F\left(\bX^{\top}\btheta_{0}\right)\bW_{\left[C\right]}\bW_{\left[C\right]}^{\top}\right] & \mathrm{E}\left[\omega_F\left(\bX^{\top}\btheta_{0}\right)\bW_{\left[C\right]}\bW_{\left[p-1\right]\backslash\left[C\right]}^{\top}\right]\\
\mathrm{E}\left[\omega_F\left(\bX^{\top}\btheta_{0}\right)\bW_{\left[p-1\right]\backslash\left[C\right]}\bW_{\left[C\right]}^{\top}\right] & \mathrm{E}\left[\omega_F\left(\bX^{\top}\btheta_{0}\right)\bW_{\left[p-1\right]\backslash\left[C\right]}\bW_{\left[p-1\right]\backslash\left[C\right]}^{\top}\right]
\end{array}\right)\\
 & =\left(\begin{array}{cc}
\mathrm{E}\left[\omega_F\left(\bX^{\top}\btheta_{0}\right)\bW_{\left[C\right]}\bW_{\left[C\right]}^{\top}\right] & \mathrm{E}\left[\omega_F\left(\bX^{\top}\btheta_{0}\right)\bW_{\left[C\right]}W_{C}\right]\bSigma_{W_{\left[p-1\right]\backslash\left[C\right]},W_{C}}^{\top}\\
\bSigma_{W_{\left[p-1\right]\backslash\left[C\right]},W_{C}}\mathrm{E}\left[\omega_F\left(\bX^{\top}\btheta_{0}\right)W_{C}\bW_{\left[C\right]}^{\top}\right] & \mathrm{E}\left[\omega_F\left(\bX^{\top}\btheta_{0}\right)\bW_{\left[p-1\right]\backslash\left[C\right]}\bW_{\left[p-1\right]\backslash\left[C\right]}^{\top}\right]
\end{array}\right).
\end{align*}

Insertion shows that $\bmu^{\ast}$ solves the \emph{top} $C$ equations.
To see that $\bmu^{\ast}$ solves the \emph{bottom} part of the system as
well, partition as follows:
\begin{align*}
 & \left(\mathrm{E}\left[\omega_F\left(\bX^{\top}\btheta_{0}\right)\bW_{\left[C\right]}\bW_{\left[C\right]}^{\top}\right]\right)^{-1}\\
 & =\left(\begin{array}{cc}
\mathrm{E}\left[\omega_F\left(\bX^{\top}\btheta_{0}\right)\bW_{\left[C-1\right]}\bW_{\left[C-1\right]}^{\top}\right] & \mathrm{E}\left[\omega_F\left(\bX^{\top}\btheta_{0}\right)\bW_{\left[C-1\right]}W_{C}\right]\\
\mathrm{E}\left[\omega_F\left(\bX^{\top}\btheta_{0}\right)W_{C}\bW_{\left[C-1\right]}^{\top}\right] & \mathrm{E}\left[\omega_F\left(\bX^{\top}\btheta_{0}\right)W_{C}^{2}\right]
\end{array}\right)^{-1}\\
 & =:\left(\begin{array}{cc}
\mathbf{A} & \boldsymbol{b}\\
\boldsymbol{c}^{\top} & d
\end{array}\right)^{-1}.
\end{align*}
The bottom left block of the LHS matrix of (\ref{eq:DebiasingSystemOfEquations})
is $\bSigma_{W_{\left[p-1\right]\backslash\left[C\right]},W_{C}}$
times
\begin{align*}
\mathrm{E}\left[\omega_F\left(\bX^{\top}\btheta_{0}\right)W_{C}\bW_{\left[C\right]}^{\top}\right] & =\left(\begin{array}{cc}
\mathrm{E}\left[\omega_F\left(\bX^{\top}\btheta_{0}\right)W_{C}\bW_{\left[C-1\right]}^{\top}\right] & \mathrm{E}\left[\omega_F\left(\bX^{\top}\btheta_{0}\right)W_{C}^{2}\right]\end{array}\right)\\
 & =\left(\begin{array}{cc}
\boldsymbol{c}^{\top} & d\end{array}\right).
\end{align*}
Inserting $\bmu^{\ast}$ to check and verify the bottom part of the
system, we get
\begin{align*}
 & \left(\begin{array}{cc}
\bSigma_{W_{\left[p-1\right]\backslash\left[C\right]},W_{\left(C\right)}}\mathrm{E}\left[\omega_F\left(\bX^{\top}\btheta_{0}\right)W_{C}\bW_{\left[C\right]}^{\top}\right] & \mathrm{E}\left[\omega_F\left(\bX^{\top}\btheta_{0}\right)\bW_{\left[p-1\right]\backslash\left[C\right]}\bW_{\left[p-1\right]\backslash\left[C\right]}^{\top}\right]\end{array}\right)\bmu^{\ast}\\
 & =\bSigma_{W_{\left[p-1\right]\backslash\left[C\right]},W_{C}}\mathrm{E}\left[\omega_F\left(\bX^{\top}\btheta_{0}\right)W_{C}\bW_{\left[C\right]}^{\top}\right]\\
 &\quad\times \left(\mathrm{E}\left[\omega_F\left(\bX^{\top}\btheta_{0}\right)\bW_{\left[C\right]}\bW_{\left[C\right]}^{\top}\right]\right)^{-1}\mathrm{E}\left[\omega_F\left(\bX^{\top}\btheta_{0}\right)\bW_{\left[C\right]}D\right]\\
 & =\bSigma_{W_{\left[p-1\right]\backslash\left[C\right]},W_{C}}\left(\begin{array}{cc}
\boldsymbol{c}^{\top} & d\end{array}\right)\left(\begin{array}{cc}
\mathbf{A} & \boldsymbol{b}\\
\boldsymbol{c}^{\top} & d
\end{array}\right)^{-1}\mathrm{E}\left[\omega_F\left(\bX^{\top}\btheta_{0}\right)\bW_{\left[C\right]}D\right].
\end{align*}
Since
\[
\left(\begin{array}{cc}
\boldsymbol{0}_{1\times\left(C-1\right)} & 1\end{array}\right)\left(\begin{array}{cc}
\mathbf{A} & \boldsymbol{b}\\
\boldsymbol{c}^{\top} & d
\end{array}\right)=\left(\begin{array}{cc}
\boldsymbol{c}^{\top} & d\end{array}\right),
\]
we have
\[
\left(\begin{array}{cc}
\boldsymbol{c}^{\top} & d\end{array}\right)\left(\begin{array}{cc}
\mathbf{A} & \mathbf{b}\\
\boldsymbol{c}^{\top} & d
\end{array}\right)^{-1}=\left(\begin{array}{cc}
\boldsymbol{0}_{1\times\left(C-1\right)} & 1\end{array}\right).
\]
Hence,
\begin{align*}
 & \bSigma_{W_{\left[p-1\right]\backslash\left[C\right]},W_{C}}\left(\begin{array}{cc}
\boldsymbol{c}^{\top} & d\end{array}\right)\left(\begin{array}{cc}
\mathbf{A} & \boldsymbol{b}\\
\boldsymbol{c}^{\top} & d
\end{array}\right)^{-1}\mathrm{E}\left[\omega_F\left(\bX^{\top}\btheta_{0}\right)\bW_{\left[C\right]}D\right]\\
 & =\bSigma_{W_{\left[p-1\right]\backslash\left[C\right]},W_{C}}\left(\begin{array}{cc}
\boldsymbol{0}_{1\times\left(C-1\right)} & 1\end{array}\right)\left(\begin{array}{c}
\mathrm{E}\left[\omega_F\left(\bX^{\top}\btheta_{0}\right)\bW_{\left[C-1\right]}D\right]\\
\mathrm{E}\left[\omega_F\left(\bX^{\top}\btheta_{0}\right)\bW_{C}D\right]
\end{array}\right)\\
 & =\mathrm{E}\left[\omega_F\left(\bX^{\top}\btheta_{0}\right)\bW_{C}D\right]\bSigma_{W_{\left[p-1\right]\backslash\left[C\right]},W_{C}},
\end{align*}
as desired.

To verify the above calculation, we return to the binary probit model ($F=\Phi$) in our main
simulations section (Section \ref{sec:Simulations}). Figures \ref{fig:RawDebiasingVector} and
\ref{fig:SortedAbsoluteDebiasingVector} display the raw and sorted absolute values, respectively, of
the debiasing coefficients $\bmu_{0}$ for the different coefficient patterns and correlation levels
considered in the main text. These values are obtained via simulation for $p=100$ and a sample size
of $100,000$. For clarity of plots, we only display the first and largest $10$ coefficients,
respectively. The figures show that, when the structural coefficients $\bgamma_0$ are
$\ell_0$-sparse, this sparsity is indeed inherited by the debiasing coefficients coefficients
$\bmu_0$. In addition, when the structural coefficients are only approximately sparse, the sorted
absolute values of the debiasing coefficients still decay polynomially fast in the index, which is a
form of approximate sparsity.

\begin{figure}
\caption{Raw Debiasing Coefficients $\bmu_{0}$, $p=100$, First 10 Coefficients\label{fig:RawDebiasingVector}} 
\centering{}
\includegraphics[viewport=5bp 5bp 463bp
416bp,clip,width=0.7\textwidth]{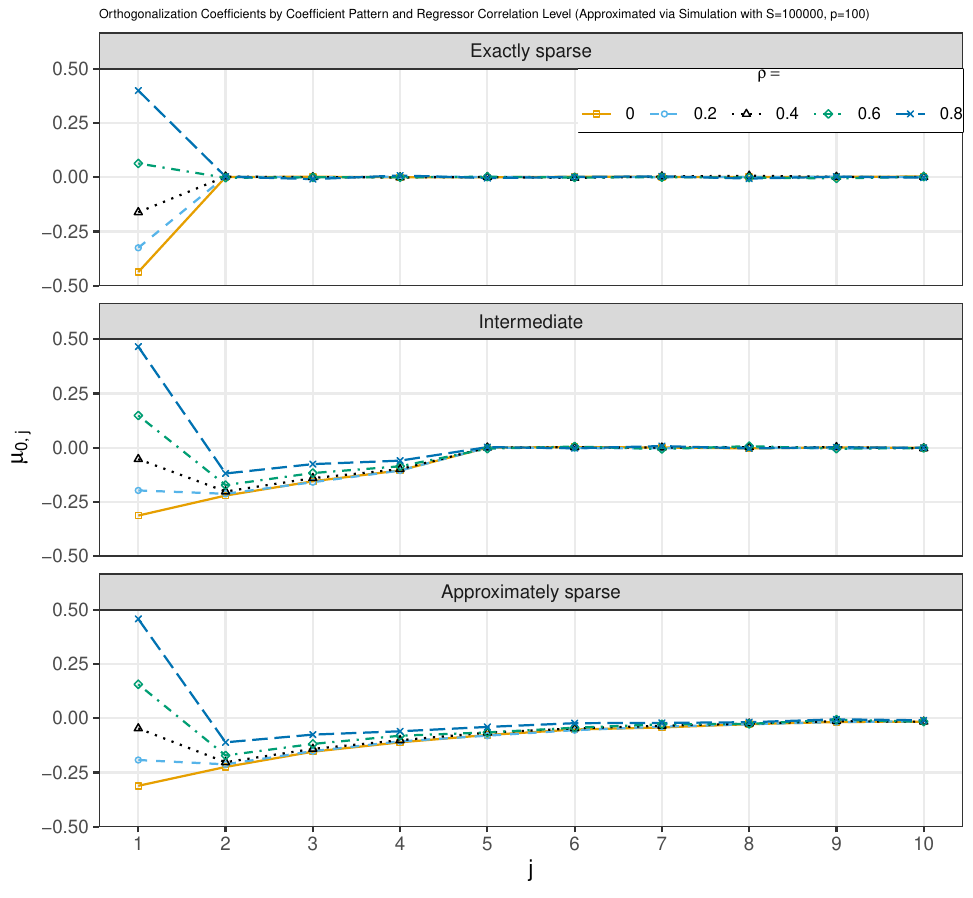}
\end{figure}

\begin{figure} 
\caption{Sorted Absolute Debiasing
  Coefficients $\bmu_{0}$, $p=100$, Largest 10 Coefficients\label{fig:SortedAbsoluteDebiasingVector}}
\centering{}
\includegraphics[viewport=5bp 5bp 463bp
416bp,clip,width=0.7\textwidth]{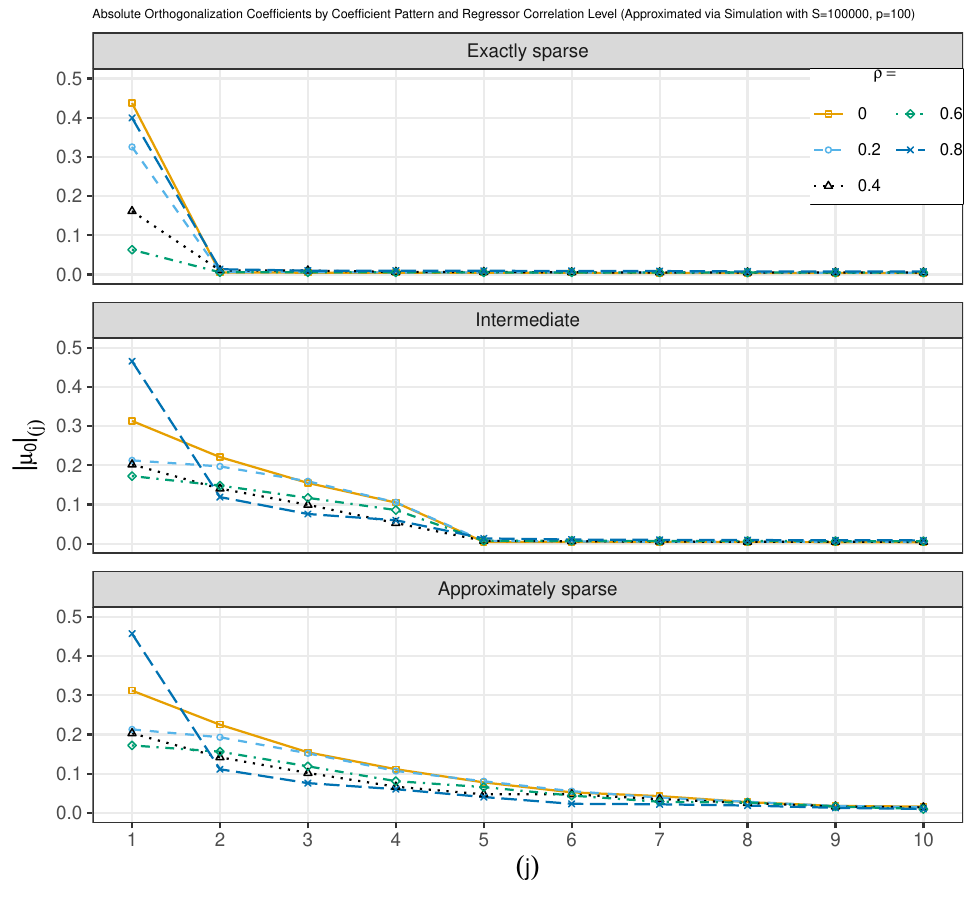}
\end{figure}

\subsection{Additional Results}\label{sec:Additional-Simulation-Results}

In this section, we give additional simulation results based on the simulation
design in the main text (Section \ref{sec:Simulations}). The only difference is
that we use the ad hoc probability tolerance rule $\alpha_{n}=10/n$ instead of
the rule $\alpha_{n}=.1/\ln(p\lor n)$ from
\citet{belloni_sparse_2012}.\footnote{Strictly speaking, the results in Section
\ref{sec:Simulations} and this section also differ in terms of their random
number generator seeds and, hence, the resulting datasets. Thus, even though CV
does not depend on $\alpha_n$, one may also see small numerical differences when
contrasting figures for CV across the two sections.} The ad hoc rule was reverse
engineered to yield $\alpha=10\%,5\%$ and $2.5\%$ for $n=100,200$ and $400$,
respectively, which can be compared with $\alpha\approx2.2\%,1.9\%$ and $1.7\%$ appearing in Section \ref{sec:Simulations}.

Following the progression of Section \ref{sec:Simulations}, we display the
estimation errors stemming from the ad hoc rule and then provide the resulting
normal approximations. Specifically, Figures
\ref{fig:MeanEll2ErrorEstimatorComparisonBCCHMarkupAdhocRule},
\ref{fig:MeanEll2ErrorBCVVaryingc0Adhocrule},
\ref{fig:MeanEll2ErrorPostBCVVaryingc0Adhocrule},
\ref{fig:NormalApproxBCVPostBCVCVVaryingSampleSizeBCCHmarkupAdhocruleZeroCorr}
and \ref{fig:NormalApproxPostBCVCVBCCHmarkupAdhocruleVaryingCorr} below should
be compared to Figures \ref{fig:MeanEll2ErrorEstimatorComparisonBCCHTuning},
\ref{fig:MeanEll2ErrorBCVVaryingc0BCCHrule},
\ref{fig:MeanEll2ErrorPostBCVVaryingc0BCCHrule},
\ref{fig:NormalApproxBCVPostBCVCVVaryingSampleSizeBCCHtuningZeroCorr} and
\ref{fig:NormalApproxPostBCVCVBCCHtuningVaryingCorr}, respectively.

\begin{figure}
\caption{Mean $\ell_{2}$ Estimation Error by Method with $c_{0}=1.1$ and
$\alpha_{n}=10/n$\label{fig:MeanEll2ErrorEstimatorComparisonBCCHMarkupAdhocRule}}

\centering{}\includegraphics[viewport=5bp 5bp 463bp 416bp,clip,width=0.7\textwidth]{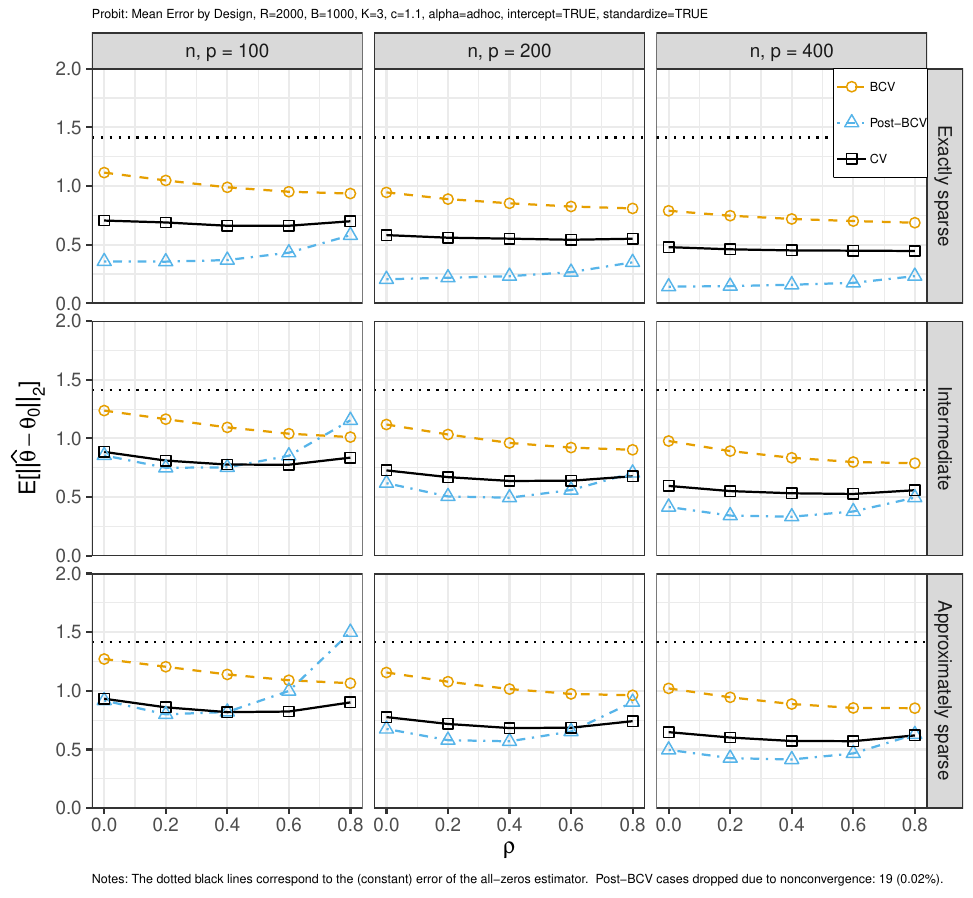}
\end{figure}
\begin{figure}
\caption{Mean $\ell_{2}$ BCV Estimation Error by Score Markup with $\alpha_{n}=10/n$\label{fig:MeanEll2ErrorBCVVaryingc0Adhocrule}}

\centering{}\includegraphics[viewport=5bp 5bp 463bp 416bp,clip,width=0.7\textwidth]{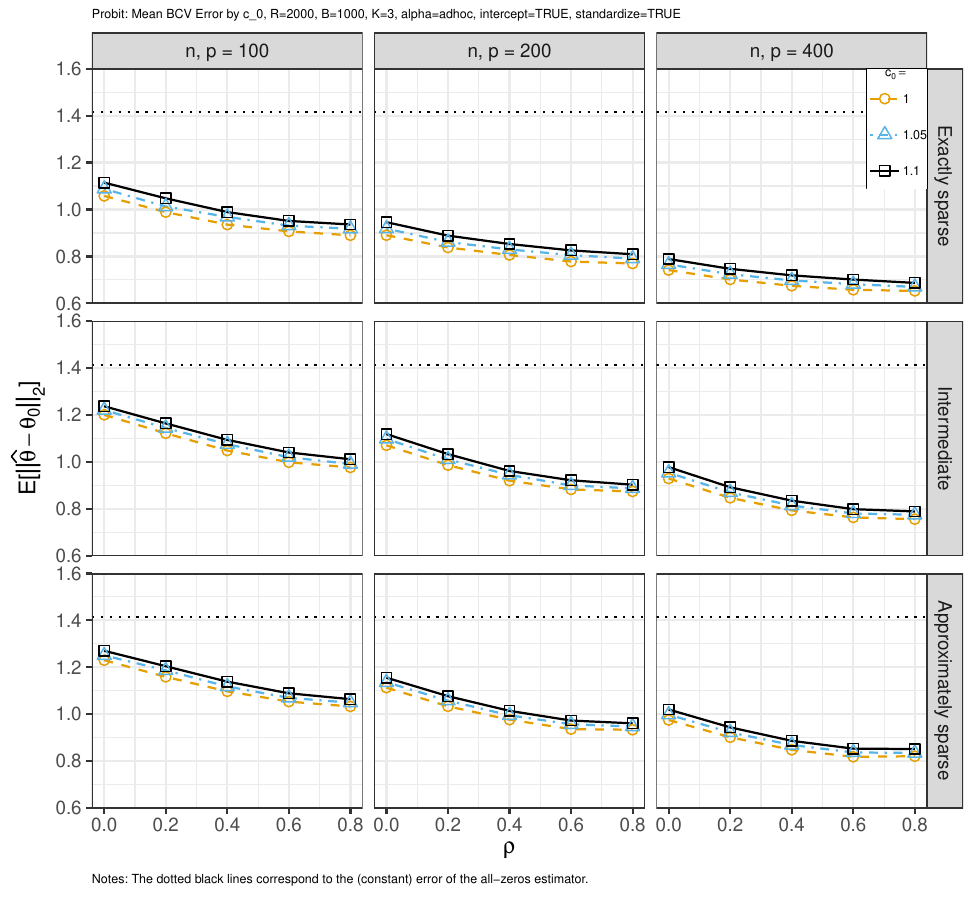}
\end{figure}
\begin{figure}
\caption{Mean $\ell_{2}$ Post-BCV Estimation Error by Score Markup with $\alpha_{n}=10/n$\label{fig:MeanEll2ErrorPostBCVVaryingc0Adhocrule}}

\centering{}\includegraphics[viewport=5bp 5bp 463bp 416bp,clip,width=0.7\textwidth]{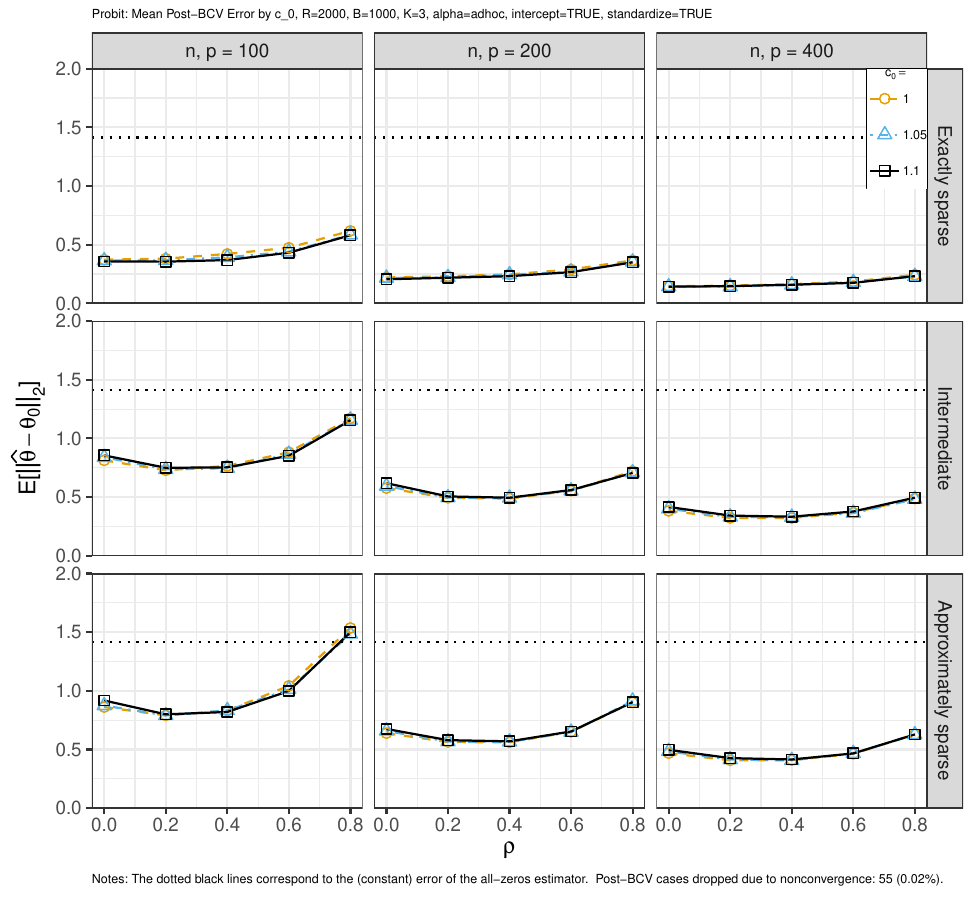}
\end{figure}

\begin{figure}
\caption{Densities of Studentized Estimates by $n(=p)$ with $\rho=0$, $c_{0}=1.1$
and $\alpha_{n}=10/n$\label{fig:NormalApproxBCVPostBCVCVVaryingSampleSizeBCCHmarkupAdhocruleZeroCorr}}

\centering{}\includegraphics[viewport=5bp 5bp 463bp 416bp,clip,width=0.68\textwidth]{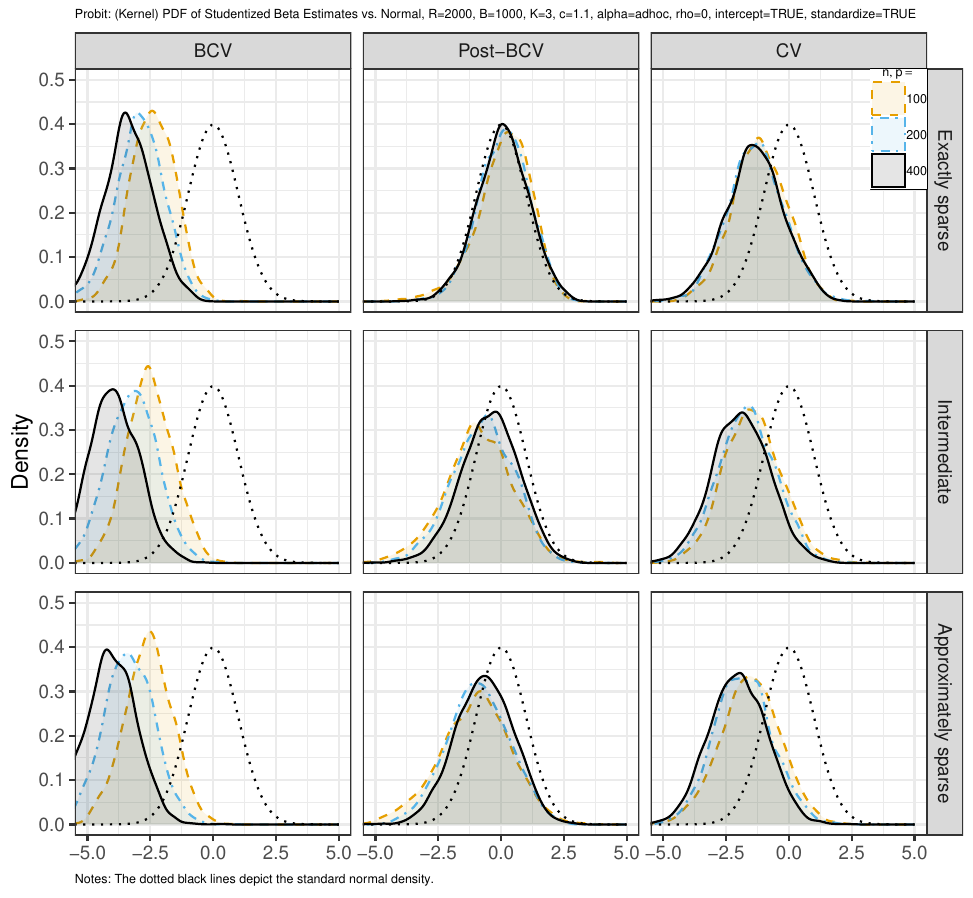}
\end{figure}
\begin{figure}
\caption{{\footnotesize{}Densities of Studentized Post-BCV and CV Estimates
for Different $\rho$ with $n(=p)=400,$ Approximately Sparse Coefficient Pattern,
$c_{0}=1.1$ and $\alpha_{n}=10/n$}\label{fig:NormalApproxPostBCVCVBCCHmarkupAdhocruleVaryingCorr}}

\centering{}\includegraphics[viewport=5bp 5bp 463bp 416bp,clip,width=0.68\textwidth]{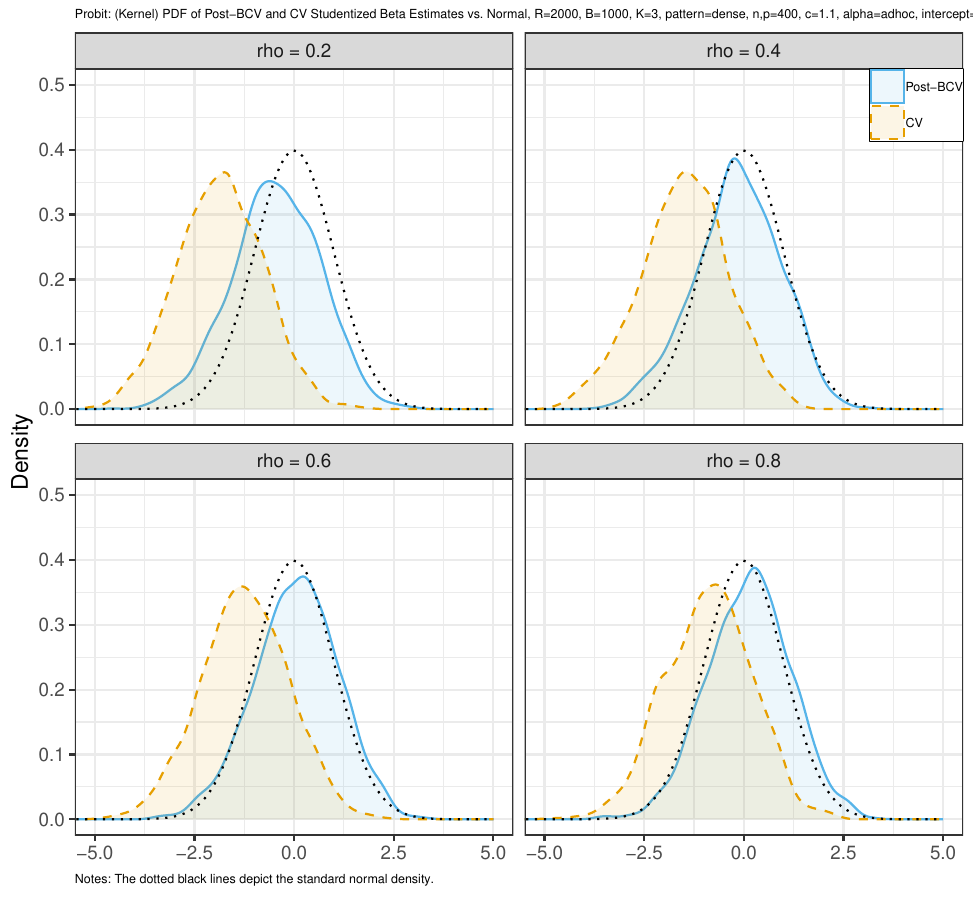}
\end{figure}

For both estimation and inference purposes, the BCV and Post-BCV estimators
resulting from the ad hoc rule perform similarly to those resulting from the
rule $\alpha_{n}=.1/\ln(p\lor n)$. The rankings and conclusions of the main text
therefore appear robust to the choice of $\alpha_{n}$.

\section{Comparison with Existing Penalty Methods}\label{sec:Comparison-with-Existing-Methods}
For specific models and loss functions, existing methods for choosing the penalty can be used to
estimate $\btheta_0$. One theoretically justifiable method is based on moderate deviation theory for
self-normalized sums as detailed in \citet{jing_self-normalized_2003} and
\citet{de_la_pena_self-normalized_2009}. Building on various structures (e.g.~mean-square projection
or Lipschitz loss), data-driven self-normalization is used in both \citet{belloni_sparse_2012} and
\citet{belloni2016post} for the LASSO and $\ell_1$-penalized logistic regression, respectively. We
next compare the performance of our BCV penalty method with the self-normalized penalty methods in
each of these papers. Our main finding is that whenever the self-normalized penalty methods apply,
they provide results which are rather similar to those obtained from our BCV method.

\subsection{Comparing with \texorpdfstring{\citet{belloni_sparse_2012}}{BCCH}
Penalty Method} We here consider data-generating processes identical to those in
Section \ref{sec:Simulations} except that the outcome is generated as
\[
Y_i=\bX_i^{\top}\btheta_0+\varepsilon_i,\quad \varepsilon_i\mid \bX_i \sim \mathrm{N}(0,1),\quad i\in\left[n\right],
\]
thus implying a linear model with Gaussian errors. Taking $m$ to be the
(one-half) square loss $(1/2)(y-t)^2$, the $\ell_1$-ME in
\eqref{eq:ell1PenalizedMEstimationIntro} is the LASSO
\citep{tibshirani_regression_1996}.

\citet[Algorithm A.1]{belloni_sparse_2012}
provides a data-driven penalization scheme allowing for non-Gaussianity,
conditionally heteroskedastic errors and regressors measured on different
scales. We here compare with a slightly simplified two-step version of their
algorithm, presuming that the researchers knows that (i) the regressors are
measured on the same scale, and (ii) the errors are conditionally homoskedastic
(with variance unknown). These simplifications are only made to ease
exposition and reduce the computational burden involved in the
comparison.\footnote{We also presume that the researcher knows that the true
intercept is zero. Hence, we do not include a constant regressor and penalize
all regressor coefficients.}

With conditionally homoskedastic errors and equivariant regressors, a two-step
version of \citet[Algorithm A.1]{belloni_sparse_2012} goes as follows:
\begin{enumerate}
   \item \textbf{Initial:} Calculate an initial penalty level
   \[
      \widehat{\lambda}^{\mathtt{brt}}_{\alpha}:=\frac{c_0 \widehat{\sigma}_y \Phi^{-1}(1-\alpha/(2p))}{\sqrt{n}},
   \]
   where $\widehat{\sigma}_y^2$ is the outcome sample variance, and store the
   residuals
   $\widehat{\varepsilon}_i:=Y_i-\bX_i^{\top}\widehat{\btheta},i\in[n],$ from
   any LASSO solution
   $\widehat{\btheta}\in\widehat{\Theta}(\widehat{\lambda}^{\mathtt{brt}}_{\alpha})$.
   \item \textbf{Refined:} Calculate a refined penalty level
   \[
      \widehat{\lambda}^{\mathtt{bcch}}_{\alpha}:=\frac{c_0 \widehat{\sigma}_{\varepsilon} \Phi^{-1}(1-\alpha/(2p))}{\sqrt{n}},
   \]
   where $\widehat{\sigma}_{\varepsilon}^2$ is the residual sample variance
   $\{\widehat{\varepsilon}_i\}_{i=1}^n$. The \citet{belloni_sparse_2012}
   estimator is then any LASSO solution
   $\widehat{\btheta}\in\widehat{\Theta}(\widehat{\lambda}^{\mathtt{bcch}}_{\alpha})$.
\end{enumerate}
As in \citet{belloni_sparse_2012}, in the above we use $c_0=1.1$ and
$\alpha=\alpha_n=.1/\ln(n\lor p)$. Note that the (simplified) initial step here
corresponds to the penalty level $\widehat{\lambda}^{\mathtt{brt}}_{\alpha}$ and
resulting estimator suggested in \citet{bickel_simultaneous_2009}.

Figure \ref{fig:MeanLASSOEll2ErrorEstimatorBCCHComparison} plots the mean
$\ell_2$ estimation errors resulting from the BCV,
\citet{bickel_simultaneous_2009} (BRT09) and in \citet{belloni_sparse_2012}
(BCCH12) penalty methods, respectively, considering the sample and problem sizes
($n$ and $p$), correlation levels $(\rho)$ and coefficient patterns considered
in Section \ref{sec:Simulations} (for the binary probit). The figure shows that
BCV performs at least as well as the other methods. However, compared to the
here theoretically justifiable BCCH12 method, the BCV improvement is modest.

\begin{figure}
\caption{Comparing with \citet{belloni_sparse_2012} Penalty Method (for LASSO)\label{fig:MeanLASSOEll2ErrorEstimatorBCCHComparison}}
\centering{}\includegraphics[viewport=5bp 5bp 463bp 416bp,clip,width=0.7\textwidth]{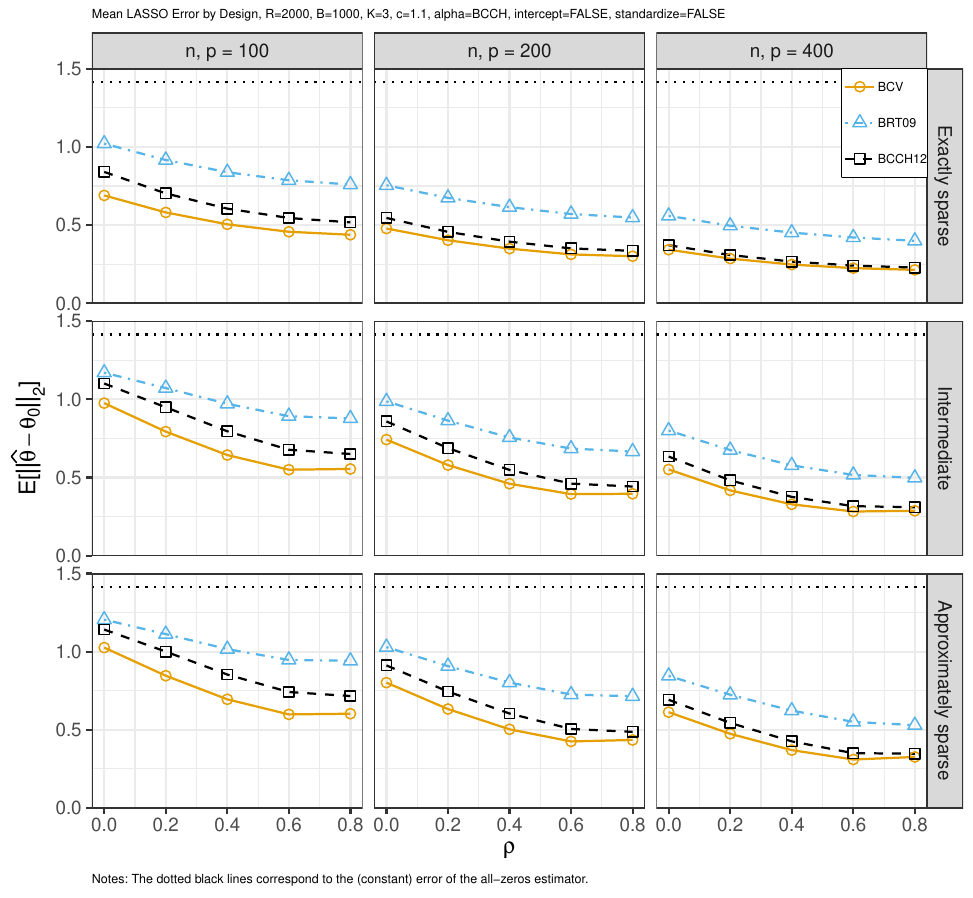}
\end{figure}

\subsection{Comparing with \texorpdfstring{\citet{belloni2016post}}{BCW}
Penalty Method} We here consider data-generating processes identical to those in
Section \ref{sec:Simulations} except that the outcome is generated as
\[
Y_i=\mathbf{1}\left(\bX_i^{\top}\btheta_0+\varepsilon_i>0\right),\quad \varepsilon_i\mid \bX_i \sim \mathrm{Logistic}(0,1),\quad i\in\left[n\right],
\]
thus implying a binary logit model as in Example \ref{exa:Logit}. We take $m$ to
be the negative logit likelihood loss in \eqref{eq:LossLogit}.

The notes of \citet[Table 1]{belloni2016post} suggest the penalty level
\[
\lambda^{\texttt{bcw}}_{\alpha,n} := \frac{c_0 \Phi^{-1}(1-\alpha/(2p))}{2\sqrt{n}},
\]
where we continue to use $c_0=1.1$ and $\alpha=\alpha_n=.1/\ln(n\lor p)$.\footnote{
See \citet[Algorithm 3]{belloni2018uniformly} for details on how to handle
regressors measured on different scales.} As in
\citet{belloni2016post}, we (correctly) presume that the regressors are
equivariant, and that the true intercept is zero. Hence, we do not include a
constant regressor and penalize all regressor coefficients.

Figure \ref{fig:MeanEll2ErrorEstimatorBCWomparison} plots the mean
$\ell_2$ estimation errors resulting from the BCV and \citet{belloni2016post}
(BCW16) penalty methods, respectively, considering the sample and problem sizes
($n$ and $p$), correlation levels $(\rho)$ and coefficient patterns considered
in Section \ref{sec:Simulations} (for the binary probit). The figure shows that
BCV performs at least as well as BCW16. However, compared to the here
theoretically justifiable BCW16 method, the BCV improvement is again modest.

\begin{figure}
\caption{Comparing with \citet{belloni2016post} Penalty Method (for
$\ell_1$-Penalized Logit)\label{fig:MeanEll2ErrorEstimatorBCWomparison}}
\centering{}\includegraphics[viewport=5bp 5bp 463bp
416bp,clip,width=0.7\textwidth]{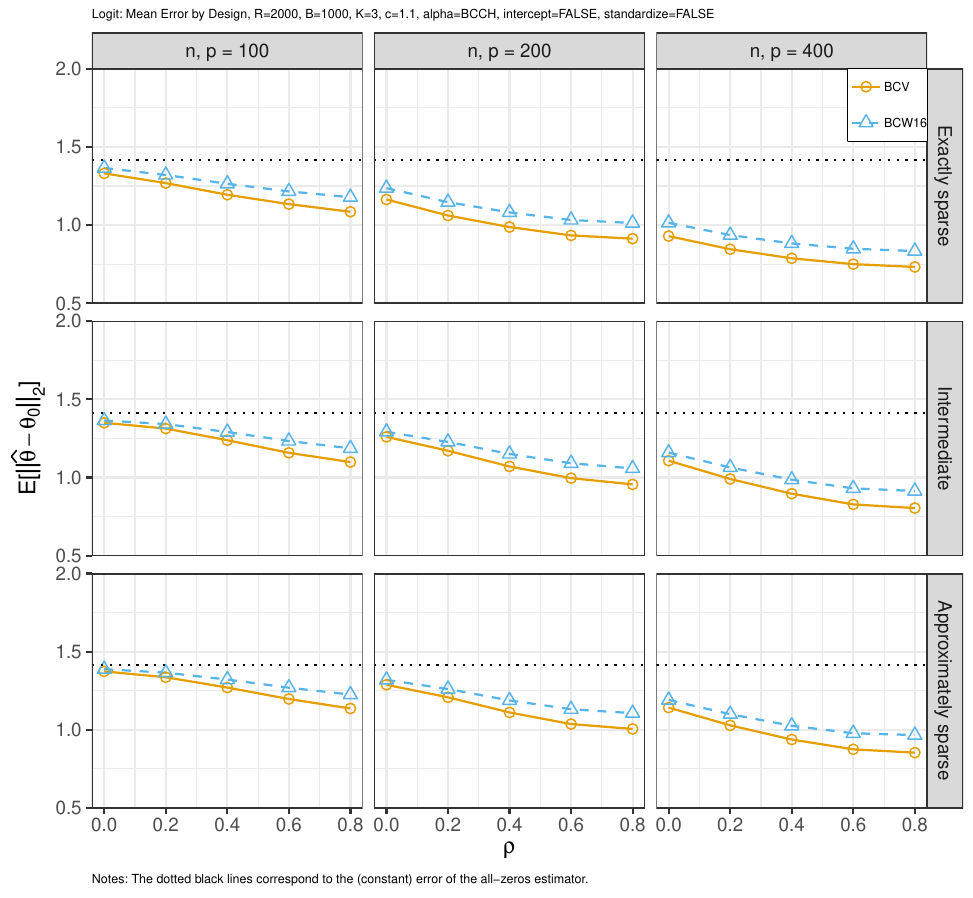}
\end{figure}


\end{document}